\documentclass[a4paper,11pt,final]{article}
\usepackage[english]{babel}
\usepackage[utf8]{inputenc}
\usepackage{subcaption}
\usepackage[T1]{fontenc}
\usepackage{lmodern}
\usepackage{amssymb,amsmath,amsfonts}
\usepackage{xspace}
\usepackage{graphicx}
\usepackage{hyperref}
\usepackage{url}
\usepackage{stmaryrd}
\usepackage[babel=true]{csquotes}
\usepackage{amsthm}
\usepackage{mathrsfs}
\usepackage{float}
\usepackage[margin=2.5cm, bottom = 3cm]{geometry}
\usepackage{enumitem}
\usepackage{verbatim}
\usepackage[dvipsnames]{xcolor}

\usepackage{tikz}
\usetikzlibrary{calc,cd,patterns,decorations.pathreplacing}
\usepackage{pgfplots}

\theoremstyle{plain}
\newtheorem{thm}{Theorem}[section]
\newtheorem{lm}[thm]{Lemma}
\newtheorem{prop}[thm]{Proposition}
\newtheorem{cor}[thm]{Corollary}

\theoremstyle{definition}
\newtheorem{dfn}[thm]{Definition}
\newtheorem{rmk}[thm]{Remark}
\newtheorem{example}[thm]{Example}
\newtheorem{algo}[thm]{Algorithm}

\newtheorem{dfn-prop}[thm]{Definition-Proposition}

\def\dpar#1#2{\frac{\partial #1}{\partial #2}}

\def\Z{\mathbb{Z}\xspace}
\def\R{\mathbb{R}\xspace}
\def\C{\mathbb{C}\xspace}

\def\P{\mathbb{P}\xspace}
\def\S{\mathbb{S}\xspace}
\def\T{\mathbb{T}\xspace}
\def\De{\Delta}

\newcommand{\om}{\omega}

\def\CP{\mathbb{CP}\xspace}

\def\bigO#1{O\mathopen{}\left(#1\right)\mathclose{}}

\def\blowup#1{\text{Bl}_{#1}}

\def\Hirzscaled#1{W_{#1}(\alpha,\beta)\xspace}

\def\poly#1{\Delta_{#1}\xspace}

\def\blpoly{wall chop\xspace}
\def\blpolys{wall chops\xspace}
\def\bldpoly{wall unchop\xspace}
\def\bldpolys{wall unchops\xspace}

\newcommand{\strep}{(\De,\vec{c},\vec{\epsilon})}

\newcommand{\dd}{{\rm d}}

\hypersetup{pdfborder={0 0 0}, colorlinks=true, linkcolor={blue!80!black}, citecolor={red!85!black},urlcolor={green!50!black}}

\newcommand{\vect}[2]{\begin{pmatrix}#1\\#2\end{pmatrix}}
\newenvironment{psmallmatrix}
  {\left(\begin{smallmatrix}}
  {\end{smallmatrix}\right)}

\newcommand{\Hred}{H_t^{\mathrm{red},j}}

\makeatletter
\newcommand{\myitem}[1]{
\item[#1]\protected@edef\@currentlabel{#1}
}
\makeatother

\newcommand\blfootnote[1]{
  \begingroup
  \renewcommand\thefootnote{}\footnote{#1}
  \addtocounter{footnote}{-1}
  \endgroup
}

\pgfplotsset{compat=1.13} 
\begin{document}

\title{Families of four-dimensional integrable systems with $S^1$-symmetries}

\author{Yohann Le Floch \qquad Joseph Palmer}

\maketitle

\begin{abstract}
The aim of this paper is to give new insights about families of integrable systems lifting a Hamiltonian $S^1$-space. Specifically, we study one-parameter families $(M^4,\omega,F_t=(J,H_t))_{0 \leq t \leq 1}$ of systems with a fixed Hamiltonian $S^1$-space $(M,\omega,J)$ and which are semitoric for certain values of the parameter $t$, with a focus on such families in which one singular point undergoes a Hamiltonian-Hopf bifurcation (also called nodal trade in the context of semitoric systems, and more generally almost toric fibrations). Beyond semitoric systems, we also study families containing hypersemitoric systems, and we investigate the local theory of a nodal trade.

Building on and generalizing the ideas of \cite{LFPfamilies}, we show how such families can be used to find explicit semitoric systems with certain desired invariants (bundled in the marked semitoric polygon). This allows us to make progress on the semitoric minimal model program by understanding and coming up with explicit systems for each strictly minimal type (i.e.~those not admitting any toric or semitoric type blowdown). In order to obtain these systems, we develop strategies for constructing and understanding explicit examples of semitoric (and hypersemitoric) systems in general. One strategy we make use of is to start from a well-understood system (such as a toric system) and to explicitly induce Hamiltonian-Hopf bifurcations to produce focus-focus singular points. This is an expanded version of the technique used in~\cite{LFPfamilies}, in order to apply it to semitoric systems which include non-trivial isotropy spheres in the underlying $S^1$-space (i.e.~$\Z_k$-spheres), which occurs in several of the strictly minimal systems.

In particular, we give an explicit one-parameter family of systems on $\CP^2$ which transitions between being of toric type, semitoric type, and hypersemitoric type depending on the value of the parameter. We study this system at each stage, computing the marked semitoric polygon of the semitoric system and determining several properties of the hypersemitoric system, including the existence of a unique flap and two parabolic orbits. Furthermore, we study the transitions between these stages.

We also come up with new explicit semitoric systems on all Hirzebruch surfaces which, together with the previous systems and the systems already contained in the literature, gives an explicit model for every type of strictly minimal system. Moreover, we show how to obtain every strictly minimal system by applying sequences of alternating toric type blowups and blowdowns to simple explicit systems. In particular, we obtain that every strictly minimal semitoric polygon can be obtained from a semitoric system which is part of a family $(M,\omega,F_t=(J,H_t))$ which is semitoric for all but a finite number of values of $t$, called a semitoric family.
\end{abstract}

\newpage

\noindent
{\bf Yohann Le Floch} \\
Institut de Recherche Math\'ematique avanc\'ee,\\
UMR 7501, Universit\'e de Strasbourg et CNRS,\\
7 rue Ren\'e Descartes,\\
67000 Strasbourg, France.\\
{\em E-mail:} \texttt{ylefloch@unistra.fr}\\

\noindent
{\bf Joseph Palmer} \\
 Department of Mathematics\\
 University of Illinois at Urbana-Champaign\\
 1409 W.~Green St\\
 Urbana, IL 61801 USA\\
 \indent \emph{and}\\
 Department of Mathematics\\
 Universiteit Antwerpen, Gebouw G\\
 Middelheimlaan 1\\
 2020 Antwerpen, Belgi{\"e}.\\
 {\em E-mail:} \texttt{jpalmer5@illinois.edu}



\newpage

\setcounter{tocdepth}{2}
\tableofcontents

\newpage

\section{Introduction}

\subsection{Context and motivation}

Semitoric systems are a type of four-dimensional integrable systems with a $S^1$-symmetry: the first component of their momentum map $F=(J,H)$ generates an effective Hamiltonian $S^1$-action (semitoric systems also have certain restrictions on their singularities, see Definition \ref{def:semitoric}). Since their introduction by V\~u Ng\d{o}c \cite{VNpoly}, semitoric systems (which are related to the almost toric fibrations introduced by Symington \cite{Sym}) have gathered a lot of attention (see for instance the recent surveys \cite{AloHoh19,Pe23} and the references therein). Semitoric systems are automatically endowed with the structure of a Hamiltonian $S^1$-space (as in Karshon's seminal paper~\cite{karshon}) which is a particular case of the complexity one spaces studied by Karshon and Tolman~\cite{KarTol01}. The full symplectic classification of semitoric systems, based on five invariants, was obtained by Pelayo and V\~u Ng\d{o}c in \cite{PVNinventiones, PVNacta} for a generic class of semitoric systems known as simple (see Definition~\ref{def:simple}), and was later generalized to all semitoric systems in~\cite{PPT}. It was explained in \cite{PVNacta} how to reconstruct a semitoric system with given invariants, using in particular semi-global considerations near each focus-focus fiber and subtle techniques of symplectic gluing. This process yields a semitoric system made from well-understood parts delicately glued together, and it is this gluing which makes it difficult to come up with semitoric systems given by simple, explicit functions globally defined on familiar manifolds. This contrasts with Delzant's theorem \cite{Del} about constructing toric integrable systems from the associated polytope, which automatically gives a global representation of the system as an explicit momentum map on a reduction of $\C^d$ by an explicit linear torus action. Note that a toric system is a particular case of semitoric system, but the class of semitoric systems is much bigger than the class of toric systems; in particular, a semitoric system may admit a type of singularities that does not appear in the toric case, namely focus-focus singularities. In this paper all semitoric systems, and in fact all integrable systems, will be assumed to be compact.

In order to produce simple, accessible examples of semitoric systems, in \cite{LFPfamilies}, we introduced strategies to construct completely explicit semitoric systems when given three of their symplectic invariants (combined in a single invariant called the marked semitoric polygon), and applied them to obtain new examples of semitoric systems on Hirzebruch surfaces. These strategies relied on the idea to consider a given semitoric system as part of a one-parameter family $(M,\omega,(J,H_t))_{0\leq t \leq 1}$ of systems, with $J$ generating an effective Hamiltonian $S^1$-action and with one fixed point transitioning between elliptic-elliptic and focus-focus, thus undergoing a Hamiltonian-Hopf bifurcation, also called a nodal trade by Symington \cite{Sym} in this situation; we called such a family a semitoric transition family. We also described the so-called toric type blowups and blowdowns, which are a certain type of blowups and blowdowns that one can perform at an elliptic-elliptic point of an integrable system to obtain another integrable system, showed that these operations send a semitoric transition family to another semitoric transition family, and used this fact to produce more systems as part of semitoric families, by applying sequences of blowups and blowdowns to fully explicit families. 

Even though our strategies to obtain nice semitoric transition families were good enough to produce interesting and new semitoric systems, they were still somewhat limited. Indeed, as we will prove in this paper, there exist some obstructions to being part of a semitoric transition family, which can be read on features of the underlying Hamiltonian $S^1$-space $(M,\omega,J)$; in particular, the presence of points with non-trivial isotropy for the $S^1$-action, which are elements of the so-called $\Z_k$-spheres, $k \geq 2$  (see Karshon~\cite{karshon}), can generate such obstructions. However, none of the systems that we constructed in \cite{LFPfamilies} included a $\Z_k$-sphere in its underlying Hamiltonian $S^1$-space, which explains why these obstructions were not present. This is also the case for the explicit system created by De Meulenaere and Hohloch \cite{HohMeu} using the recipe from \cite{LFPfamilies}, and the historical explicit system of coupled spins from the work of Sadovski\'i and Z\^{h}ilinski\'i~\cite{SZ} and its generalization constructed by Hohloch and Palmer \cite{HohPal}.

Consequently, we needed to come up with new techniques, and adapt the ones from \cite{LFPfamilies}, to construct explicit semitoric systems with given marked semitoric polygon, in order to deal with systems displaying $\Z_k$-spheres. However, even with the new strategies that we develop in the present paper, we do not produce a universal method to obtain a system from its marked semitoric polygon: we rather obtain a recipe which helps coming up with an explicit system from a given polygon, and not a precise algorithm like the one for toric systems. Therefore, the most interesting approach is to apply these new techniques to obtain new explicit systems which are particulary fundamental: strictly minimal semitoric systems.

Strictly minimal semitoric systems are the elementary bricks from which all semitoric systems can be constructed via a sequence of certain types of symplectic blowups. Indeed, there are two natural types of blowups (and corresponding blowdowns) sending a semitoric system to another semitoric system; these two types of blowups lift the $S^1$-equivariant blowups of Hamiltonian $S^1$-spaces introduced in \cite{karshon}. The first one is the toric type blowup evoked above and studied in \cite{LFPfamilies}, and is in fact not specific to semitoric systems; it consists in blowing up at an elliptic-elliptic point of the system in a way that is consistent with a choice of local toric momentum map, and does not change the number of focus-focus singularities of the system. The second one is the semitoric type blowup first studied by Zung \cite{Zung03} for general integrable systems and Symington \cite{Sym} for almost toric systems, with contributions from several authors since, see the discussion at the beginning of Section \ref{subsec:semi_blow} for more details. This type of blowup can only be performed on a semitoric system whose underlying Hamiltonian $S^1$-space has a fixed surface; it consists in using this fixed surface to add a focus-focus point to the system. Similarly to the toric case, both types of blowups can be read on marked semitoric polygons.

Strictly minimal semitoric systems are those semitoric systems which do not admit any blowdown, be it of toric or semitoric type. Equivalently, every semitoric system can be obtained by performing a sequence of blowups on one of the strictly minimal systems. In particular, strictly minimal systems form a subset of the set of minimal semitoric systems (semitoric systems that do not admit any toric type blowdown), studied and classified in \cite{KPP_min} in the compact case. In \cite{KPP_min}, minimal systems were classified by means of an object called the semitoric helix; here we extend this classification to figure out which of these helices correspond to strictly minimal systems, and we translate the result about helices into a result about polygons, thus listing every possible marked semitoric polygon that a strictly minimal semitoric system can possess.
From here, our goal is to obtain an explicit example of a semitoric system for each of these marked semitoric polygons, making use of families of integrable systems similar to the ones~\cite{LFPfamilies}.

Explicit systems with some of these marked semitoric polygons were already obtained in the literature. This is the case of the systems from \cite{HohPal}, \cite{LFPfamilies}, and of the physical example from \cite{CDEW}. However, as explained above, the systems from \cite{HohPal} and \cite{LFPfamilies} do not include $\Z_k$-spheres. There are more types of strictly minimal systems that had not been explored yet, and do include $\Z_k$-spheres. Hence they provide a fertile ground to apply the new techniques that we develop in this paper. Another technical difficulty is that some of the strictly minimal systems do not satisfy the simplicity condition, and therefore we need to adapt all the existent notions and techniques to this non-simple case.

Furthermore, examples of semitoric systems and their generalizations also appear in physics, such as the coupled angular momenta system~\cite{SZ} (see also \cite{LFP}), the Jaynes-Cummings system from quantum optics~\cite{jaynes-cummings} (see also~\cite{san-alvaro-spin}) or more recent systems such as the one in \cite{CDEW}. Some other physical examples, such as the spherical pendulum (see for instance~\cite[chapter 4]{CusBat}) or the system from~\cite{DHN21}, fail to be semitoric only because the momentum map $J$ of the $S^1$-action is not proper. Such ``generalized'' semitoric systems were studied in \cite{PelRatVN, PelRatVN-fibers}, and in \cite{DHVN} the authors study the relationship between classical and quantum rotation numbers of such systems. Another generalization of semitoric systems is the notion of faithful semitoric system~\cite{HSS-vertical}.

The recipes that we present in this manuscript allow us to come up with explicit candidates for the desired semitoric systems, but then it remains to prove that these systems are indeed semitoric. This can lead to a substantial amount of computations, as can be seen in the existing literature, for instance in \cite{LFP,HohPal,LFPfamilies,HohMeu}. Hence another of our goals is to develop tools in order to simplify these computations, under mild assumptions that are satisfied in many examples. In fact, these tools reduce the difficulty of checking the types of the singular points in a broad class of systems with underlying effective Hamiltonian $S^1$-action which includes the semitoric case.

To sum up, the aim of this paper is to perform the following tasks, that we will explain more precisely in the next few sections. We produce the list of strictly minimal marked semitoric polygons, leading to strictly minimal semitoric systems (Section \ref{subsec:strict_min_intro}). We introduce the semitoric families used in \cite{LFPfamilies} and their modified version adapted to our context, and detail the obstructions for a given semitoric system to belong to such a family (Section \ref{subsec:semi_fam_intro}). Using the techniques and tools that we develop, we obtain an explicit example of a strictly minimal system on $\C\P^2$ which belongs to such a family which transitions from semitoric to a more general kind of system called hypersemitoric (as in Hohloch and Palmer~\cite{HohPal21}, see Definition~\ref{dfn:hypersemi}), and fully analyze this family (Section \ref{subsec:CP2_intro}). Finally, we apply our tools to obtain examples of strictly minimal systems of all possible types (Section \ref{subsec:explicit_intro}).

\subsection{Strictly minimal semitoric systems}
\label{subsec:strict_min_intro}

In Section \ref{subsec:blowups_semi} we describe the toric and semitoric type blowups of semitoric systems. A strictly minimal semitoric system (see Definition \ref{def:strictly_min_system}) is a semitoric system which does not admit any toric or semitoric type blowdown, and one of the goals of this paper is to understand all strictly minimal systems. Both kind of blowups can be read on marked semitoric polygons, and therefore we first want to establish the list of strictly minimal marked semitoric polygons. Note that if we restrict ourselves to semitoric systems with no focus-focus point, called systems of toric type, only toric type blowdowns can be performed, and thus the strictly minimal polygons are the well-known minimal Delzant polygons, recalled in Section \ref{subsect:min_toric}. Hence in this paper we concentrate on semitoric systems with at least one focus-focus singularity.

In fact, it is more convenient to read the blowups on another combinatorial object called the semitoric helix of the system, which can be computed from the marked semitoric polygon, see Section \ref{sec:helix}. Roughly speaking, a semitoric helix is a sequence of vectors constructed from normal vectors to the edges of the semitoric polygon and which is periodic up to some linear map accounting for the number $s$ of focus-focus points of the system (corresponding to the total monodromy of the system). It is the semitoric version of the fan of a toric variety.

Consequently, a first step is to look for all the possible helices associated with strictly minimal semitoric systems. The helices corresponding to minimal semitoric systems, the ones which do not admit any toric type blowdown, were classified in \cite{KPP_min}. Starting from this list of minimal helices, and using the fact that a helix admits a semitoric type blowdown if and only if it has a horizontal vector (Proposition \ref{prop:blowdown-helix}), we obtain that there are exactly three types of strictly minimal helices, up to certain symmetries called $J$-reflections and $H$-reflections: types (1), (2) and (3) from Definition \ref{def:types_helices}. These strictly minimal helices are displayed in Figure \ref{fig:strictly_min_helices}.

\begin{figure}
\begin{center} 
\def\dashlength{0.08}

\begin{subfigure}[t]{.25\linewidth} 
\centering 
\begin{tikzpicture}[scale=.9]
\draw (0,-2.2)--(0,2.2); 
\draw (-2.2,0)--(2.2,0); 
\draw [very thick,->] (0,0)--(0,1); 
\draw [very thick,->] (0,0)--(-1,-2);
\draw (0,1) node [above right] {$v_0$}; 
\draw (-1,-2) node [left] {$v_1$}; 
\draw (0.5,-1) node[right] {$s = 1$};
\draw (-1,-\dashlength) -- (-1,\dashlength);
\draw (-\dashlength,1) -- (\dashlength,1);
\draw (1,-\dashlength) -- (1,\dashlength);
\draw (-\dashlength,-1) -- (\dashlength,-1);
\end{tikzpicture} 
\caption{Type (1). $v_0 = \begin{pmatrix} 0 \\ 1 \end{pmatrix}$, $v_1 = \begin{pmatrix} -1 \\ -2 \end{pmatrix}$.} 
\end{subfigure}  
\hspace{1cm} 
\begin{subfigure}[t]{.25\linewidth} 
\centering 
\begin{tikzpicture}[scale=.9]
\draw (0,-2.2)--(0,2.2); 
\draw (-2.2,0)--(2.2,0); 
\draw [very thick,->] (0,0)--(0,1); 
\draw [very thick,->] (0,0)--(-1,-1);
\draw (0,1) node [above right] {$v_0$}; 
\draw (-1,-1) node [left] {$v_1$}; 
\draw (0.5,-1) node[right] {$s = 2$};
\draw (-1,-\dashlength) -- (-1,\dashlength);
\draw (-\dashlength,1) -- (\dashlength,1);
\draw (1,-\dashlength) -- (1,\dashlength);
\draw (-\dashlength,-1) -- (\dashlength,-1);
\end{tikzpicture} 
\caption{Type (2). $v_0 = \begin{pmatrix} 0 \\ 1 \end{pmatrix}$, $v_1 = \begin{pmatrix} -1 \\ -1 \end{pmatrix}$.} 
\end{subfigure} 
\hspace{1cm} 
\begin{subfigure}[t]{.25\linewidth} 
\centering 
\begin{tikzpicture}[scale=.9]
\draw (0,-2.2)--(0,2.2); 
\draw (-2.2,0)--(2.2,0); 
\draw [very thick,->] (0,0)--(0,1); 
\draw [very thick,->] (0,0)--(-1,-1);
\draw [very thick,->] (0,0)--(0,-1);
\draw (0,1) node [right] {$v_0$}; 
\draw (-1,-1) node [above left] {$v_1$}; 
\draw (0,-1) node [left] {$v_2$}; 
\draw (0.5,-1) node[right] {$s = 1$};
\draw (-1,-\dashlength) -- (-1,\dashlength);
\draw (-\dashlength,1) -- (\dashlength,1);
\draw (1,-\dashlength) -- (1,\dashlength);
\draw (-\dashlength,-1) -- (\dashlength,-1);

\end{tikzpicture} 
\caption{Type (3). $v_0 = \begin{pmatrix} 0 \\ 1 \end{pmatrix}$, $v_1 = \begin{pmatrix} -1 \\ 1 - n \end{pmatrix}$ with $n \geq 1$, $n \neq 3$, $v_2 = \begin{pmatrix} 0 \\ -1 \end{pmatrix}$.} 
\end{subfigure} 

\end{center} 
\caption{The three strictly minimal helices.} 
\label{fig:strictly_min_helices} 
\end{figure}
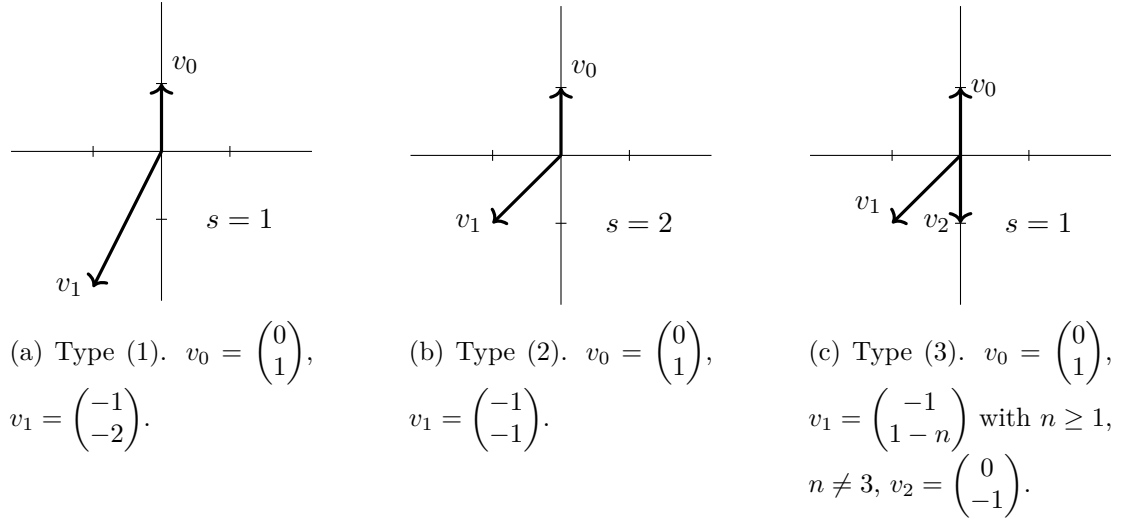

From this list of strictly minimal helices, the next step is to compute all the strictly minimal marked semitoric polygons. However, a semitoric helix can be associated with many different marked semitoric polygons, and in general it is difficult to come up with all these polygons. Nevertheless, for helices of types (1), (2) and (3) this problem is tractable, and we manage to obtain all the corresponding polygons.

\begin{thm}[Theorem \ref{thm:min_poly}]
 Let $(M,\om,F)$ be a semitoric system with at least one focus-focus singularity. This system does not admit a toric or semitoric type blowdown if and only if its marked semitoric polygon is of one of the types (1), (2a), (2b), (3a) with $n = 2$ or $n \geq 4$, (3b) and (3c) (with $n \neq 3$ for these last two) from Definition \ref{def:min-poly-123}, or the $J$-reflection or $H$-reflection of one of these, see Definition \ref{def:reflections_semitoric_poly}.
\end{thm}

Note that the marked semitoric polygons of types (2b) and (3b) with $n=2$ (and their $J$ and $H$-reflections) are exactly the polygons which admit perfect packings by equivariantly embedded balls, see \cite[Theorem 1.7]{Pal22} and~\cite{FiPaPe2016}. This gives an additional motivation to look for explicit systems with these polygons as their marked semitoric polygons.

Since the marked semitoric polygon is not the complete symplectic invariant of a semitoric system, each strictly minimal marked semitoric polygon from the above theorem corresponds to infinitely many different semitoric systems which are all strictly minimal. Hence, one way to understand all the strictly minimal semitoric systems is to construct an explicit system for each type of strictly minimal marked semitoric polygon, and then to work out the relationship between two semitoric systems having the same marked semitoric polygon. We investigate this relationship in Section \ref{subsec:marked_poly_iso} and in particular Corollary \ref{cor:marked_iso} (since the full statement is rather technical, we do not reproduce it here). Thus, our next goal is to find an explicit semitoric system for each strictly minimal marked semitoric polygon. Since the $J$-reflections and $H$-reflections can be effectively performed by changing the signs of the Hamiltonians, it suffices to find such explicit systems of types (1), (2) and (3).

\subsection{Semitoric families}
\label{subsec:semi_fam_intro}

As explained earlier, in \cite{LFPfamilies} we looked for explicit semitoric systems as parts of certain one-parameter families of integrable systems, which we review now.

\begin{dfn}[{\cite[Definition 1.4]{LFPfamilies}}]
\label{def:semitoric-family}
A \emph{semitoric family} with \emph{degenerate times} $t_1, \ldots, t_k \in [0,1]$, is a family of integrable systems $(M,\omega,F_t)$, $0 \leq t \leq 1$ on a four-dimensional symplectic manifold $(M,\omega)$ such that:
\begin{itemize}
    \item $F_t = (J,H_t)$ where $H_t$ is of the form $H_t = H(t,\cdot)$ for some smooth $H: [0,1] \times M \to \R$; 
    \item $(M,\omega,F_t)$ is semitoric if and only if $t \in [0,1] \setminus \{t_1, \ldots, t_k\}$.
\end{itemize}
\end{dfn}

By a slight abuse of terminology, and with obvious changes to the definition, we will sometimes consider semitoric families with parameter $t\in [0,b]$ for some $b>0$ (see in particular Section \ref{sec:blowup_down_type3}). Moreover, given a semitoric family with parameter $t \in [0,b]$, we will call the same systems restricted to $t \in [0,c]$ for some $c \leq b$ a sub-family.

In \cite{LFPfamilies} we considered a certain type of semitoric family, in which one point undergoes two Hamiltonian-Hopf bifurcations (in this scenario also called nodal trades) and transitions from elliptic-elliptic to focus-focus, and back to elliptic-elliptic.

\begin{dfn}[{\cite[Definition 1.6]{LFPfamilies}}]
\label{def:semitoric-transition-family}
A \emph{semitoric transition family} with \emph{transition point} $p \in M$ and \emph{transition times} $t^-, t^+$, with $0 < t^- < t^+ < 1$, is a family of integrable systems $(M,\omega,F_t)$, $0 \leq t \leq 1$ on a four-dimensional symplectic manifold $(M,\omega)$ such that:
\begin{itemize}
    \item $F_t = (J,H_t)$ where $H_t$ is of the form $H_t = H(t,\cdot)$ for some smooth $H: [0,1] \times M \to \R$; 
    \item $(M,\omega,F_t)$ is semitoric for $t \in [0,1] \setminus \{t^-,t^+\}$;
    \item for $t < t^-$ and $t > t^+$ the point $p$ is singular of elliptic-elliptic type;
    \item for $t^- < t < t^+$ the point $p$ is singular of focus-focus type;
    \item for $t = t^-$ and $t = t^+$ there are no degenerate singular points in $M \setminus \{p\}$;
    \item if $p$ is a maximum (respectively minimum) of $(H_0)_{|J^{-1}(J(p))}$ then $p$ is a minimum (respectively maximum) of $(H_1)_{|J^{-1}(J(p))}$.
\end{itemize}
\end{dfn}

Note that the first two items mean that any semitoric transition family is in particular a semitoric family with degenerate times $t^-$ and $t^+$.

The historical example of the coupled angular momenta system from \cite{SZ} (see also \cite{LFP}) is a semitoric transition family in which the system for $t \in (t^-,t^+)$ is of type (3) (up to $H$-reflection). The system from \cite{HohPal}, generalizing the latter, is also a collection of semitoric transition families including systems of types (2) and (3) (note that this system has also been adapted to the 
$b$-symplectic setting in \cite{BHMM}). In \cite{LFPfamilies}, our techniques allowed us to come up with additional explicit systems of types (2) and (3) as parts of semitoric transition families. As explained earlier, none of these systems includes $\Z_k$-spheres in its underlying Hamiltonian $S^1$-space.

Studying bifurcations of integrable systems is interesting in general (see for instance~\cite{Efs-book}), and these families are a particular case of such bifurcations. Another reason to be interested in such families of integrable systems is to understand the structure of the moduli space of integrable systems lifting a fixed Hamiltonian $S^1$-space. The moduli space of four-dimensional toric systems was endowed with a metric, and thus a topology, and studied in \cite{PPRStoric}; this was generalized to higher dimensions in~\cite{PeSa23}, where the authors also describe obstructions to the existence of a classification of minimal models in higher dimensions. In~\cite{Pal17} the toric metric and topology in dimension four were extended to the moduli space of semitoric systems. Systems with different numbers of focus-focus points are naturally in different components of this moduli space, and the explicit families discussed in the present paper can be thought of as traversing through the ``walls'' between these components.

In Section \ref{sec:obstructions}, we prove that not every semitoric system can be obtained as part of a semitoric transition family, and give some explicit obstructions, related to relative positions of focus-focus points, elliptic-elliptic points and $\Z_k$-spheres. To overcome some of these obstructions, we introduce a new type of family of integrable systems in which a point transitions from elliptic-elliptic to focus-focus and back to elliptic-elliptic, but may be non-semitoric for an interval of values of the parameter, rather than a finite number of them. This allows us to obtain certain semitoric systems in a nice family, even if they cannot appear in a semitoric transition family.

\begin{dfn}
\label{def:half-semitoric-family}
A \emph{half-semitoric transition family} with \emph{transition point} $p \in M$ and \emph{transition times} $t^-, t^+$, $t^- < t^+$, is a family of integrable systems $(M,\omega,F_t)$, $0 \leq t \leq 1$ on a four-dimensional symplectic manifold $(M,\omega)$ such that:
\begin{itemize}
    \item $F_t = (J,H_t)$ where $H_t$ is of the form $H_t = H(t,\cdot)$ for some smooth $H: [0,1] \times M \to \R$; 
    \item $(M,\omega,F_t)$ is semitoric for $t \in [0,t^+) \setminus \{t^-\}$;
    \item for $t < t^-$ and $t > t^+$ the point $p$ is singular of elliptic-elliptic type;
    \item for $t^- < t < t^+$ the point $p$ is singular of focus-focus type;
    \item for $t = t^-$ and $t = t^+$ there are no degenerate singular points in $M \setminus \{p\}$.
\end{itemize}
\end{dfn}

In fact, we will see in Section \ref{sec:obstructions} that there also exist some obstructions to be part of a half-semitoric transition family, but
these obstructions do not rule out the possibility that all the strictly minimal semitoric systems can occur in such a family.

\subsection{Explicit strictly minimal systems of type (1)}
\label{subsec:CP2_intro}

In Section \ref{sec:CP2}, we construct an explicit example of a half-semitoric transition family which is semitoric of type (1) for $t^-<t<t^+$. Furthermore, for $t>t^+$ we show that the system is \emph{hypersemitoric}, which means that all singular points of the system are either non-degenerate or of a relatively mild degenerate type called parabolic. For such $t$, we fully prove that the system is hypersemitoric with a single triangular flap. Hypersemitoric systems were introduced in~\cite{HohPal21} in the context of studying when Hamiltonian $S^1$-spaces can be extended to different types of integrable systems. In recent years hypersemitoric systems and parabolic points have been an active area of research~\cite{BolGugKud,KudMar,Kudmar2,GulHoh,HHM,GulHoh2}.

Concretely, the system that we come up with occurs on $(M,\omega) = (\C\P^2,\alpha\omega_{\text{FS}})$ where $\omega_{\text{FS}}$ is the Fubini-Study form, normalized so that the Liouville volume of $\C\P^2$ equals $2 \pi^2$, and $\alpha > 0$. We view this symplectic manifold as the symplectic reduction of $\C^3$ by the Hamiltonian
\[ N = \frac{1}{2} ( |z_1|^2 + |z_2|^2 + |z_3|^2 ) \]
at the level $N=\alpha$. The Hamiltonians 
\[ J = \frac{1}{2} ( |z_1|^2 - |z_2|^2 ), \quad R = \frac{1}{2} ( |z_1|^2 + |z_2|^2 ), \quad \mathcal{X} = \Re(z_1 z_2 \bar{z}_3^2) \]
are invariant under the action generated by $N$, and hence descend to the quotient; for this system, and every other system throughout the paper, we will slightly abuse notation and use the same coordinates and notation in the quotient.

\begin{thm}[Theorem \ref{thm:CP2_sys}]
\label{thm:intro_CP2_sys}
Let $J = \frac{1}{2} ( |z_1|^2 - |z_2|^2 )$ and let $0 < \gamma < \frac{1}{4\alpha}$ and $\delta > \frac{1}{2 \gamma \alpha}$. Let 
\[ H_t = (2t - 1) \frac{|z_3|^2}{2} - 2 \gamma t (\mathcal{X} - \delta R^2) - 2 \gamma \delta t \alpha^2 .  \]
Then the family $(\C\P^2,\alpha \omega_{FS}, F_t=(J, H_t))_{0 \leq t \leq 1}$ is 
\begin{itemize}
    \item of toric type when $0 \leq t < t^-$;
    \item semitoric with one focus-focus point (at $B = [0:0:1]$) when $t^- < t < t^+$;
    \item hypersemitoric with one triangular flap with elliptic corner $F_t(B)$ when $t^+ < t \leq 1$;
    \end{itemize}
where
\[ t^- = \frac{1}{2(1 + 2 \gamma \alpha)}, \quad t^+ = \frac{1}{2 (1 - 2 \gamma \alpha)}. \]
Moreover, when $t^- < t < t^+$, this system is of type (1) with parameter $\alpha$ as in Figure \ref{fig:min_type1}. Furthermore, for any $h_0 \in (0,\frac{\alpha}{2})$, there exists a choice of $\gamma \in (0,\frac{1}{4\alpha})$, $\delta \in (\frac{1}{2\gamma \alpha},+\infty)$ and $t \in (t^-,t^+)$ such that the height invariant of the system is $h = h_0$. 
\end{thm}

This statement means that from this family of explicit systems, by varying all the parameters, we obtain all possible marked semitoric polygons of type (1).

\begin{rmk}
In the statement of Theorem \ref{thm:CP2_sys}, the function $H_t$ is the negative of the one given in Theorem~\ref{thm:intro_CP2_sys}. The reason is that we naturally came up with the system in Section~\ref{sec:CP2}, but it is the $H$-reflection of a type (1) system when $t^- < t < t^+$ (compare the polygons in Figure \ref{fig:polygon_CP2} to the one in Figure \ref{fig:min_type1}), so changing the sign of $H_t$ turns the system into a genuine type (1) system.    
\end{rmk}

\begin{rmk}
\label{rmk:CDEW}
The system of type (1) on $\CP^2$ from Theorem~\ref{thm:intro_CP2_sys} is somewhat similar to the system given in~\cite{CDEW}, but there are several significant differences. Using our notation defined above in this section, the system in~\cite{CDEW} is given by $F_a = (J,G_a)$ with
\[ G_a = J^2 - 2 R^2 + 2 (\alpha - a^2 \omega) R + \mathcal{X} \]
on $(\C\P^2,\alpha \omega_{\text{FS}})$, with $\omega > 0$ and $a \geq 0$ a parameter. It is of toric type when $a > \sqrt{\frac{2\alpha}{\omega}}$, and semitoric with one focus-focus point when $0 < a < \sqrt{\frac{2\alpha}{\omega}}$; moreover when $a$ varies in $(0,\sqrt{\frac{2\alpha}{\omega}})$ the height invariant of the system takes every possible value in $(0,\frac{\alpha}{2})$. Hence this system also allows one to recover every marked semitoric polygon of type (1).

One crucial difference between this system and the system from Theorem \ref{thm:intro_CP2_sys} is that no flap appears in the former; instead, the point $B = [0:0:1]$ becomes degenerate exactly when $a=0$ and in this case $F_0(B) = (0,0) = F_0(m)$ for any $m \in J^{-1}(0) \cap \{z_3 = 0\}$. Adding an extra parameter playing the role of the parameter $\delta$ in the system from Theorem \ref{thm:intro_CP2_sys} could allow to make flaps appear. Indeed, in our system $\delta$ is chosen to ensure that the points of the $\Z_2$-sphere $\{z_3 = 0\}$ are always non-degenerate, but if we allowed $\gamma = \frac{1}{4\alpha}$ and $\delta = 2$ we would get a similar behavior. 
\end{rmk}

There are two aspects to obtaining a result like Theorem~\ref{thm:intro_CP2_sys}. One of them is to understand how to create a semitoric system of the desired type, and the other is to verify that it is indeed of toric type, semitoric, or hypersemitoric depending on the value of the parameter $t$. In Section \ref{sec:strategy}, we develop the strategies and tools to treat these two problems. 

For the first problem, we face the following question: how does one obtain a good candidate $(M,\omega,F=(J,H))$ for a semitoric system with a given marked semitoric polygon? As explained above, in \cite{LFPfamilies} we gave some recipes to help us find such a good candidate as part of a semitoric transition family, but they are not sufficient to deal with systems whose underlying Hamiltonian $S^1$-space $(M,\omega,J)$ includes non-trivial isotropy spheres, such as the system in Theorem~\ref{thm:intro_CP2_sys}. 

In Section \ref{sec:strategy} we refine and expand these strategies in the following ways. Starting from the normal form of a function $J$ generating an effective $S^1$-action at a fixed point $p$ with weights $\{-1,1\}$, we describe how to write, in the local coordinates of the normal form, all of the possible local Hamiltonians $H$ such that $(J,H)$ is integrable and $p$ is a critical point of $H$. This local description involves parameters, and we show how to read from these parameters the type of the singular point, either focus-focus, elliptic-elliptic, or degenerate, and how changing these parameters can induce a Hamiltonian-Hopf bifurcation, which transitions between these cases. In general, this technique can be used to understand nodal trades locally around any elliptic-elliptic point in an integrable system. In the cases that we are interested in, these local coordinates naturally come from symplectic reduction and we can extend this local Hamiltonian $H$ to a global one to get a candidate $(J,H)$. However, in a semitoric system the $\Z_k$-spheres must lie on the boundary of the momentum map image, and this may not be the case for $(J,H)$ for this choice of $H$; in Section \ref{sec:strategy} we explain how to add a suitable correction term to $H$ to overcome this problem.

Effectively checking that the system is of the desired type involves many computations that can become lengthy and tedious, in particular when finding the singular points and checking their types. In Section \ref{sec:strategy} we construct a unified framework to simplify these computations, under mild assumptions that are satisfied in all the examples we are aware of. More precisely, these tools greatly simplify the process of checking the type of non-degenerate singular points and of checking when a degenerate singular point is parabolic. Even then, the explicit computations in Sections \ref{sec:CP2} and \ref{sec:type_3} are quite involved, but this is because we include many parameters in our systems to account for all possible marked semitoric polygons of the desired types. However, if one wants to use our techniques to check a single system with fixed values for the parameters, the computations are even less complicated.
Moreover, in Section~\ref{sec:CP2}, we give a detailed analysis of a system on $\CP^2$ even past the parameter value when it is no longer semitoric (it becomes hypersemitoric); if we restricted our interest to the parameters for which it is semitoric the computations would be much shorter.

\subsection{Constructing strictly minimal semitoric systems of all types}
\label{subsec:explicit_intro}

The strategies and tools discussed in the previous section are not limited to the case of the system from Theorem \ref{thm:intro_CP2_sys}, and in the following sections we use them to construct semitoric systems corresponding to the other strictly minimal marked semitoric polygons for which no explicit semitoric system already exists in the literature.

Before stating our result, let us briefly review those already existing systems. Two-parameter families of systems are given in \cite[Theorem 1.2]{HohPal} and \cite[Section 7.4]{LFPfamilies}, and they are of type (2a), (3a) with $n=2$, or (3c) with $n=2$, depending on the choice of parameters; additionally, the system in the family from \cite[Theorem 1.2]{HohPal} may be of type (2b) or (3b) with $n=2$ (possibly up to $H$-reflection) for an approriate choice of parameters. Moreover, systems of types (3a), (3b), and (3c) with $n=2$ are given in \cite{SZ}, and \cite[Theorems 7.2 and 7.5]{LFPfamilies} as part of semitoric transition families. Note that when $n=2$, marked semitoric polygons of type (3a) and (3c) are $J$-reflections of each other, see Definition \ref{def:reflections_semitoric_poly}. 
The system from~\cite[Section 6.2]{LFPfamilies} is of type (3a) with $n=1$.
Furthermore, the system studied in~\cite{CDEW} is of type (1); we discuss this system in Remark~\ref{rmk:CDEW}. By varying the parameters in all the above systems, we can obtain every marked semitoric polygon of the corresponding type.

Combining the previously known systems with the new systems from the present paper, we can finally obtain an explicit example for each strictly minimal unmarked semitoric polygon. This is summarized in the following theorem.

\begin{thm}
\label{thm:intro_list_explicit}
For every unmarked semitoric polygon of type (1), (2a), (2b), (3a), (3b) and (3c), there exists an explicit semitoric system with this polygon as its invariant. More precisely:
\begin{itemize}
    \item the formerly known systems from~\cite{SZ, HohPal, LFPfamilies} discussed in the previous paragraph are of type (2), and (3) for $n=1$ or $n=2$. Moreover, by varying the parameters in these systems we can obtain every possible marked semitoric polygon of type (2) and (3);
    \item the system from Theorem \ref{thm:intro_CP2_sys} above is a half-semitoric transition family on $(\C\P^2, \alpha \omega_{\text{FS}})$ which is of type (1) for $t \in (t^-,t^+)$. Moreover, by varying the parameters in this system we obtain every possible marked semitoric polygon of type (1);
    \item for $n \geq 3$ and $\alpha, \beta > 0$, the system from Theorem \ref{thm:system_3a_general} is a half-semitoric transition family on $\left( W_{n-2}(\alpha + \beta,\beta), \omega_{W_{n-2}(\alpha + \beta,\beta)} \right)$ which is of type (3a) for $t \in (t^-,t^+)$;
    \item for $n \geq 3$ and $\beta > 0$, the system from Theorem \ref{thm:system_3b_general} is a half-semitoric transition family on $\left( W_{n-2}(\beta,\beta), \omega_{W_{n-2}(\beta,\beta)} \right)$ which is of type (3b) for $t \in (t^-,t^+)$;
    \item for $n \geq 3$ and $0 < \alpha < \beta$, the system from Theorem \ref{thm:system_3c_general} is a half-semitoric transition family on $\left( W_{n-2}(\beta-\alpha,\beta), \omega_{W_{n-2}(\beta-\alpha,\beta)} \right)$ which is of type (3c) for $t \in (t^-,t^+)$.
\end{itemize}
\end{thm}

\begin{rmk}
The systems of type (3a), (3b) and (3c) with parameter $n = 3$ are not strictly minimal (and not even minimal), see the discussion after Theorem \ref{thm:min_poly}. Nevertheless, being able to contruct explicit systems of these types is still interesting, in particular since they include $\Z_2$-spheres. Similarly, systems of type (3a) with $n=1$ are not strictly minimal because their marked semitoric polygons include a vertical wall.
\end{rmk}

By varying the different parameters in the new systems discussed in the last three items in the statement of Theorem \ref{thm:intro_list_explicit} (those of type (3)), we can obtain all possible scalings of the marked semitoric polygon, but not all possible values of the height invariant. Even though obtaining explicitly every possible unmarked semitoric polygon of type (3) is already significant, it would be nice to obtain every marked semitoric polygon from an explicit system.

There is, however, another way, albeit less explicit, to obtain every such marked semitoric polygon, by starting from a simple, completely explicit system and performing alternatively toric type blowups and blowdowns on this system. In \cite[Theorem 5.2]{LFPfamilies} we applied this procedure, starting from the coupled angular momenta system from \cite{SZ}, to prove that every marked semitoric polygon of type (3a) can be obtained as the polygon of a semitoric system which is part of a semitoric transition family; this is reviewed in Theorem \ref{thm:blowups_SF1}.

In Section \ref{sec:blowup_down_type3} we perform a similar procedure for systems of type (3b) and (3c).

\begin{thm}[Theorems \ref{thm:(3c)_blowups_downs} and \ref{thm:(3b)_blowups_downs}]
For every marked semitoric polygon of type (3b) or (3c), there exists a semitoric system which is part of a semitoric family obtained by alternatively performing toric type blowups and blowdowns on a one-parameter family included in the family from \cite[Theorem 1.2]{HohPal} (see Lemma \ref{lm:system_W0_HP} and the discussion before it).
\end{thm}

Since moreover the systems of type (2) from \cite[Theorem 1.2]{HohPal} and \cite[Theorem 1.9]{LFPfamilies} form semitoric families giving all the possible marked semitoric polygons (and even a semitoric transition family for (2a)), we can combine this with the results discussed above (Theorem \ref{thm:intro_CP2_sys}, \cite[Theorem 5.2]{LFPfamilies}) to obtain the following statement.

\begin{thm}
Every strictly minimal marked semitoric polygon $[(\Delta,\vec{c},\vec{\epsilon})]$ can be obtained as the marked semitoric polygon of the $t = \frac{1}{2}$ system in a semitoric family. Moreover, this semitoric family can be chosen as
\begin{itemize}
    \item a semitoric transition family if $[(\Delta,\vec{c},\vec{\epsilon})]$ is of type (2a) or (3a);
    \item a sub-family of a half-semitoric transition family if $[(\Delta,\vec{c},\vec{\epsilon})]$ is of type (1).
\end{itemize}
\end{thm}

\begin{rmk}
In view of the results contained in this paper, it is reasonable to conjecture that every strictly minimal semitoric system $(M,\omega,F)$ can be obtained as part of a semitoric family, in other words that every strictly minimal semitoric system lies in a semitoric family which connects it to a system with one less focus-focus point. More generally, it would be interesting to try to understand if this also the case for every semitoric system not satisfying the obstructions of Proposition \ref{prop:non-half-stfam} (instead of just the strictly minimal ones). We conjecture that such a result holds and the case of strictly minimal semitoric systems would constitute a good starting point, since every semitoric system is obtained by applying a sequence of toric and semitoric type blowups to one of those.
\end{rmk}

\subsection{Structure of the article}

The article is organized as follows:
\begin{itemize}
    \item in Section \ref{sec:background}, we review the necessary background on semitoric systems and their invariants, Hamiltonian $S^1$-spaces and the associated Karshon graphs, and hypersemitoric systems;
    \item in Section \ref{sec:blow_semi}, we discuss toric and semitoric type blowups and blowdowns of semitoric systems, semitoric helices, how to read blowups on helices, and list all the minimal helices;
    \item in Section \ref{sec:strictlymin} we obtain the list of all strictly minimal marked semitoric polygons, including those of toric type. We then study the relationship between two semitoric systems with the same marked semitoric polygon;
    \item in Section \ref{sec:obstructions}, we describe the obstructions preventing a given semitoric system to be part of a semitoric or half-semitoric transition family;
    \item in Section \ref{sec:blowup_down_type3} we explain how we obtained systems of type (3a) in \cite{LFPfamilies} by alternatively performing blowups and blowdowns on an explicit system, and show how to adapt this strategy to obtain systems of types (3b) and (3c); 
    \item in Section \ref{sec:strategy}, we come up with, and explain, general strategies to construct semitoric systems with a given semitoric polygon, taking into account the obstructions described in Section \ref{sec:obstructions}. We also propose efficient ways to investigate the singularities of an integrable system $(M,\omega,(J,H))$ where $J$ generates an effective Hamiltonian $S^1$-action; 
    \item in Section \ref{sec:CP2}, we apply these strategies to construct an explicit half-semitoric transition family of systems with underlying manifold $\C\P^2$, which are either of toric type, semitoric with one focus-focus point and of type (1), or hypersemitoric but not semitoric, according to the value of the parameter. In particular, this allows us to obtain an explicit system for every marked semitoric polygon of type (1);
    \item in Section \ref{sec:type_3}, we apply the same strategies and construct explicit systems for each possible unmarked semitoric polygon of type (3a), (3b) and (3c), as part of half-semitoric transition families.
\end{itemize}

\paragraph{Acknowledgements.}
While working on this paper, both authors had many nice discussions with several members of the integrable and semitoric communities, in particular Sonja Hohloch, Nikolay Martynchuk, \'{A}lvaro Pelayo, Susan Tolman and San V{\~u} Ng{\d o}c. Part of this work was performed during a visit of the first author to the second author in University of Antwerp; we thank Sonja Hohloch for funding this visit under the FWO-EoS
project G0H4518N. The second author was partially supported by an FWO senior postdoctoral fellowship 12ZW320N.

\section{Integrable systems with underlying Hamiltonian \texorpdfstring{$S^1$}{S1}-spaces} 
\label{sec:background}

In this section, we recall the necessary background about integrable and semitoric systems, and the symplectic invariants of the latter, in a way that encompasses both the simple and non-simple cases. We also recall some notions about Hamiltonian $S^1$-spaces and their relations with semitoric systems. Finally, we quickly review hypersemitoric systems, a class of integrable systems lifting Hamiltonian $S^1$-spaces and generalizing semitoric systems.

Throughout the paper, $(M,\omega)$ will be a four-dimensional compact, connected, symplectic manifold. Recall that a Liouville integrable system on $(M,\omega)$ is the data of $f_1,f_2 \in \mathcal{C}^{\infty}(M)$ such that $\{f_1,f_2\} = 0$ and the Hamiltonian vector fields $X_{f_1}, X_{f_2}$ are almost everywhere linearly independent. The map $F = (f_1,f_2): M \to \R$ is called the momentum map of the system.

\subsection{Singularities and semitoric systems}
\label{subsec:sing_semi}

In order to define semitoric systems, we first need to recall some facts about singularities of integrable systems (again, only in dimension four for this paper). These facts will also be useful when working on explicit systems in later sections. A \emph{singular point} of the integrable system $(M,\omega,F=(f_1,f_2))$ is a point $m \in M$ such that $X_{f_1}(m)$ and $X_{f_2}(m)$ (or equivalently $\dd f_1(m)$ and $\dd f_2(m)$) are linearly dependent; the \emph{rank} of $m$ is the rank of the family $(X_{f_1}(m),X_{f_2}(m))$, so it equals either zero or one. This rank equals the dimension of the orbit through $m$ for the action of $\R^2$ on $M$ given by $(t_1,t_2) \cdot m \mapsto (\phi_1^{t_1} \circ \phi_2^{t_2}(m))$, where $\phi_j$ is the Hamiltonian flow of $f_j$, $j=1,2$. One can define a notion of non-degenerate singular point as follows (note that the terminology introduced below will also apply to the corresponding singular value $F(m)$). 

\paragraph{Rank one points.} If $m \in M$ is a rank one singular point of $F$, then the orbit through $m$ of the above $\R^2$-action is one-dimensional. Let $L \subset T_m M$ be the tangent line to this orbit at $m$, and let $L^{\perp}$ be the symplectic orthogonal of $L$. Let $\nu,\mu \in \R$ be such that $\nu \dd f_1(m) + \mu \dd f_2(m) = 0$. Note that $L \subset L^{\perp}$ and that $\dd^2 (\nu  f_1 + \mu f_2)(m)$ descends to the quotient space $L^{\perp} \slash L$ since $L$ is contained in its kernel. 

\begin{dfn}[{\cite[Definition 1.21]{BolFom}, see also \cite[Section 2.1.3]{HohPal}}]
The rank one singular point $m$ is \emph{non-degenerate} if $\dd^2 (\nu  f_1 + \mu f_2)(m)$ is an isomorphism of $L^{\perp} \slash L$.    
\end{dfn}

Non-degenerate rank one points can be classified in different types (usually called Williamson types) according to the eigenvalues of the isomorphism $\dd^2 (\nu  f_1 + \mu f_2)(m)$. In practice, observe that $\omega_m$ descends to a symplectic form on $L^{\perp} \slash L$; choose any basis of $L^{\perp} \slash L$ and consider the matrix $\Omega_m$ of this quotient symplectic form in this basis. Introduce also the matrix $A_{\nu,\mu} = \Omega_m^{-1} \dd^2 (\nu  f_1 + \mu f_2)(m)$ (with the slight abuse of notation which consists in still writing $\dd^2 (\nu  f_1 + \mu f_2)(m)$ for the matrix of the Hessian in the chosen basis). As a consequence of \cite[Proposition 1.2]{BolFom}, the eigenvalues of $A_{\nu,\mu}$ are of the form $\pm \lambda$ for some $\lambda \in \C$. 

\begin{dfn}[{See \cite[Section 1.8]{BolFom} and \cite[Definition 2.5]{LFPfamilies}}]
\label{def:nondeg_rankone}
The rank one singular point $m$ is non-degenerate if $\det A_{\nu,\mu} \neq 0$, and $m$ is said to be 
\begin{itemize}
    \item \emph{elliptic-regular} if the eigenvalues of $A_{\nu,\mu}$ are of the form $\pm i \alpha$ for some $\alpha \in \R \setminus \{0\}$ (or equivalently if $\det A_{\nu,\mu} > 0$);
    \item \emph{hyperbolic-regular} if the eigenvalues of $A_{\nu,\mu}$ are of the form $\pm \alpha$ for some $\alpha \in \R \setminus \{0\}$ (or equivalently if $\det A_{\nu,\mu} < 0$).
\end{itemize}
\end{dfn}

The terminology \emph{elliptic-transverse} (respectively \emph{hyperbolic-transverse}) is also often used instead of elliptic-regular (respectively hyperbolic-regular).
 
\paragraph{Rank zero points.} Let $m \in M$ be a rank zero singular point of $F$. Let $(\mathcal{Q}_m,\{\cdot,\cdot\}_m)$ be the Lie algebra of quadratic forms on $T_m M$, with Lie bracket the Poisson bracket.

\begin{dfn}[{\cite[Definition 1.23]{BolFom}}]
The rank zero point $m$ is \emph{non-degenerate} if $\dd^2 f_1(m), \dd^2 f_2(m)$ span a Cartan subalgebra of $\mathcal{Q}_m$.   
\end{dfn}

These Cartan subalgebras can be classified (see for instance \cite[Theorem 1.3]{BolFom}), giving rise to a classification of non-degenerate rank zero points. In practice, the definition and classification of non-degenerate points can be seen from a simpler point of view, as above. This time, let $\Omega_m$ be the matrix of $\omega_m$ in any given basis of $T_m M$ and, for any $\nu, \mu \in \R$, let $A_{\nu,\mu} = \Omega_m^{-1} \left( \nu \dd^2 f_1 + \mu \dd^2 f_2)(m)\right)$. The characteristic polynomial of $A_{\nu,\mu}$ is of the form $X \mapsto \chi_{\nu,\mu}(X^2)$ with $\chi_{\nu,\mu}$ a degree two polynomial (this is a consequence of \cite[Proposition 1.2]{BolFom}, for instance), that we call the \emph{reduced characteristic polynomial} of $A_{\nu,\mu}$.

\begin{dfn}[{\cite[Section 1.8.2]{BolFom}}]
The rank zero point $m$ is non-degenerate if there exists $(\nu,\mu) \in \R^2$ such that $\chi_{\nu,\mu}$ has two distinct nonzero roots $\lambda_1, \lambda_2 \in \C$. In this case, $m$ is said to be 
\begin{itemize}
    \item \emph{elliptic-elliptic} if $\lambda_1 < 0$ and $\lambda_2 < 0$;
    \item \emph{hyperbolic-elliptic} if $\lambda_1 > 0$ and $\lambda_2 < 0$;
    \item \emph{hyperbolic-hyperbolic} if $\lambda_1 > 0$ and $\lambda_2 > 0$;
    \item \emph{focus-focus} if $\Im(\lambda_1) \neq 0$ (and hence $\Im(\lambda_2) \neq 0$).
\end{itemize}
\end{dfn}

This definition makes sense because if $m$ is non-degenerate, then the properties of the roots of $\chi_{\nu,\mu}$ in the above list do not depend on the choice of $(\nu,\mu) \in \R^2$ such that $\chi_{\nu,\mu}$ has two distinct nonzero roots (and these form an open dense set).

Near any non-degenerate singular point, one can obtain a nice symplectic normal form, which we only state for dimension four but exists in general.

\begin{thm}[Eliasson normal form \cite{Eliasson-thesis,Eli90,miranda-zung}]
\label{thm:eliasson}
Let $m \in M$ be a non-degenerate singular point of the integrable system $(M,\omega,F=(f_1,f_2))$. Then there exist local symplectic coordinates $(x,\xi) = (x_1,x_2,\xi_1,\xi_2)$ on an open neighborhood $U \subset M$ of $m$ and a map $Q = (q_1, q_2): U \to \R^n$ whose components are taken from the following list:
\begin{itemize}
    \item $q_j(x,\xi) = \frac{1}{2}(x_j^2 + \xi_j^2)$ (elliptic);
    \item $q_j(x,\xi) = x_j \xi_j$ (hyperbolic);
    \item $q_j(x,\xi) = \xi$ (regular);
    \item $q_1(x,\xi) = x_1 \xi_2 - x_2 \xi_1$, $q_2(x,\xi) = x_1 \xi_1 + x_2 \xi_2$ (focus-focus);
\end{itemize}
such that $m$ corresponds to $(x,\xi) = (0,0)$ and $\{q_j,f_{\ell}\} = 0$  for every $j, \ell \in \{1,2\}$. Moreover, if none of the components $q_j$ is hyperbolic, then there exists a local diffeomorphism $g: (\R^2,0) \to (\R^2, F(m))$ such that for every $(x,\xi) \in U$, $F(x,\xi) = (g \circ Q)(x,\xi)$.
\end{thm}

The types of the components $(q_1, q_2)$ correspond to the Williamson type of the singular point $m$. To the best of our knowledge, there does not exists a single reference containing a complete proof of this theorem, and many authors have contributed to this result (see for instance \cite{Vey,ColVey,DufMol,Mir,VNWac,Cha}). For more details, see for example \cite[Remark 4.16]{VNSepe} or the discussion following Theorem 2.1 in \cite{LFVN}.

With all these notions in mind, we are now ready to define semitoric systems.

\begin{dfn}\label{def:semitoric}
A \emph{semitoric system} is an integrable system $(M,\omega,F=(J,H))$ such that 
\begin{enumerate}
    \item $J$ is the momentum map for an effective Hamiltonian $S^1$-action;
    \item the singular points of $F$ are all non-degenerate, without any hyperbolic component.
\end{enumerate}
\end{dfn}

The original definition given by V\~{u} Ng\d{o}c in \cite[Definition 3.1]{VNpoly} was slightly more general; the version of the definition that we use here first appeared in \cite[Definition 2.1]{PVNinventiones}.

Note that one can define semitoric systems on non-compact manifolds by adding the assumption that $J$ is proper, but in this paper $M$ will always be compact.

Definition \ref{def:semitoric} implies that the singular points of a semitoric system can only be elliptic-regular, elliptic-elliptic, or focus-focus. In a semitoric system, the elliptic-regular and elliptic-elliptic values form the boundary of the image $F(M)$ of the momentum map, while the focus-focus values, in finite number, are isolated in the interior of $F(M)$ (see \cite[Theorem 1]{VNpoly}). A semitoric system with no focus-focus singularity is said to be of \emph{toric type}; this is equivalent to the existence of a local diffeomorphism $g$ near $F(M)$ such that $g \circ F$ is a toric momentum map, i.e.~the momentum map for an effective Hamiltonian $\T^2$-action. This of course includes the case of toric systems, for which $F(M)$ is a convex polygon whose edges (respectively vertices) correspond to elliptic-regular (respectively elliptic-elliptic) values \cite{Ati,GuiSte}, and which completely determines the system up to equivariant symplectomorphism \cite{Del}.

In a semitoric system, the fiber $F^{-1}(c)$ over a focus-focus value $c \in F(M)$ is a torus pinched at a finite number of points which are precisely the focus-focus singular points in this fiber. In some works, it is assumed that the situation where there are multiple pinches, and more generally the one where there are multiple focus-focus points in the same $J$-fiber, does not occur.

\begin{dfn}
\label{def:simple}
A semitoric system $(M,\omega,F=(J,H))$ is \emph{simple} if there is at most one focus-focus point in each level set of $J$.    
\end{dfn}

In this paper we will often work with non-simple systems, so we will need to adapt some definitions and constructions from earlier works on semitoric systems.  

Since we will work with semitoric systems and, more generally, with systems of the form $(M,\omega,F=(J,H))$ where $J$ generates an effective Hamiltonian $S^1$-action, it will often be practical to work in the symplectic reduction (\cite{MW}) at a given level set of $J$ to look for singular points of $(J,H)$. More precisely, let $j \in J(M)$ and let $M_j^{\text{red}} = J^{-1}(j) \slash S^1$ be the symplectic reduction of $M$ at level $j$ by the $S^1$-action generated by $J$; this set may have singular points but, if the action is free at a point $m \in J^{-1}(j)$, then $M_j^{\text{red}}$ is smooth near the class $[m] \in M_j^{\text{red}}$ of $m$. Note also that since $J$ and $H$ Poisson commute, $H$ descends to a function $H^{\text{red},j}$ on $M_j^{\text{red}}$, which is smooth on the smooth locus of $M_j^{\text{red}}$. The next result was stated in \cite[Lemma 2.6]{LFPfamilies} and follows from the proof of \cite[Lemma 2.4]{HohPal}, see also \cite[Definition 3]{ToZel}.

\begin{lm}
\label{lm:sing_red}
Let $m \in M$ be a point where the $S^1$-action generated by $J$ acts freely, and let $j = J(m)$. Then $m$ is a singular point of $F$ if and only if $[m]$ is a singular point of $H^{\text{red},j}$. In this case, $m$ is non-degenerate if and only if $[m]$ is non-degenerate in the Morse sense, and $m$ is of elliptic-regular (respectively hyperbolic-regular) type if and only if $[m]$ is elliptic (respectively hyperbolic).    
\end{lm}

\subsection{Marked semitoric polygons}
\label{subsect:msp}

Semitoric systems were originally classified in~\cite{PVNinventiones,PVNacta} (for the simple case) and this classification was extended to non necessarily simple semitoric systems in~\cite{PPT} making use of~\cite{PT}. In this paper we will only need part of the corresponding semitoric invariants; more precisely, we need what was called a \emph{marked semitoric polygon} in~\cite[Section 2.4]{LFPfamilies}. This encodes the information from the semitoric polygon, the number of focus-focus points and the height invariant; let us briefly recall some details. Note that below we include the non-simple case, which was not discussed in~\cite{LFPfamilies}; moreover, we slightly change our point of view in order to discuss the relationship with semitoric helices (see Section \ref{sec:helix}) more easily. Recall that the semitoric systems that we consider in this paper take place on compact manifolds.

\paragraph{Defining marked weighted polygons.} First, we define the notion of a marked polygon. Let $\Delta \subset \R^2$ be a rational convex polygon; here rational means that every edge is directed by an integral vector. Let
\begin{equation} T = \begin{pmatrix} 1 & 0 \\ 1 & 1\end{pmatrix} \in \text{SL}(2,\Z). \label{eqn:T}\end{equation}
Let $q$ be a vertex of $\Delta$, and let $v_1,v_2 \in \Z^2$ be the primitive inwards pointing normal vectors to the edges adjacent to $q$, ordered so that $\det(v_1,v_2) > 0$. Let $k \in \Z_{>0}$. Then we say that $q$ satisfies:
\begin{enumerate}[noitemsep]
\item the \emph{Delzant condition} if $\det(v_1,v_2) = 1$;
\item the \emph{$k$-hidden Delzant condition} if $\det( v_1, (T^*)^k v_2) = 1$;
\item the \emph{$k$-fake condition} if $v_1=(T^*)^k v_2$  (or equivalently in this context $\det( v_1,(T^*)^k v_2) = 0$).
\end{enumerate}

\begin{rmk}\label{rmk:fake-nodal-trade}
Notice that the Delzant and $1$-fake cases are not mutually exclusive. For instance, the vectors $v_1 = \begin{pmatrix} -1 \\ -1 \end{pmatrix}$, $v_2 = \begin{pmatrix} 0 \\ -1 \end{pmatrix}$
satisfy both the Delzant and $1$-fake conditions. Vertices satisfying both conditions are very important in the study of semitoric systems (and in particular semitoric families), since they allow us to transition between systems which have different numbers of focus-focus points via so-called nodal trades. Similarly, the $k$-fake and $(k-1)$-hidden Delzant conditions are also not mutually exclusive, as can be seen by taking $v_1 = \begin{pmatrix} -1 \\ -1 \end{pmatrix}$, $v_2 = \begin{pmatrix} k-1 \\ -1 \end{pmatrix}$. 
In a given semitoric polygon, there will often be several vertices which satisfy two of the conditions above, but still each vertex will be labeled as exactly one of Delzant, $k$-hidden, or $k$-fake depending on its role in the polygon.
\end{rmk}

Let $\pi_j: \R^2 \to \R$ be the canonical projection to the $j$-th factor, $j=1,2$. Let $s \in \Z_{\geq 0}$, let $\vec{c} = (c_1, \ldots, c_s) \in (\R^2)^s$, and let $\vec{\epsilon} = (\epsilon_1, \ldots, \epsilon_{s}) \in \{ -1,1 \}^{s}$. The triple $\left(\Delta, \vec{c}, \vec{\epsilon} \right)$ is called a \emph{marked weighted polygon} if the points $c_1,\ldots, c_s$ belong to the interior of $\Delta$ and are labelled with respect to the lexicographic order, which we abbreviate by $c_1 \leq_{\text{lex}} \ldots \leq_{\text{lex}} c_s$. 

Consider the subgroup
\[ \mathcal{T} = \left\{ \left(T^k, \begin{pmatrix} 0 \\ y \end{pmatrix} \right) \ \Bigg| \ k \in \Z, y \in \R \right\} \]
of $\text{GL}(2,\Z) \ltimes \R^2$ consisting of the elements leaving the vertical direction invariant, and define an action of $\mathcal{T}$ on the set of marked weighted polygons in the following way. Let $\tau \in \mathcal{T}$ and for $\vec{c} \in (\R^2)^s$, set $\tau(\vec{c}) = \left( \tau(c_1), \ldots, \tau(c_s) \right)$; observe that if the elements of $\vec{c}$ are ordered lexicographically, then so are the elements of $\tau(\vec{c})$. Let $\left(\Delta, \vec{c} = (c_1, \ldots, c_s), \vec{\epsilon} = (\epsilon_1, \ldots, \epsilon_{s}) \right)$ be a marked weighted polygon; then we define
\begin{equation}\tau \cdot \left(\Delta, \vec{c}, \vec{\epsilon} \right) = \left(\tau(\Delta), \tau(\vec{c}), \vec{\epsilon} \right). \label{eq:T_action}\end{equation} 

We also need to define some piecewise affine transformations, which requires additional notation. Fix $\lambda \in \R$ and endow the vertical line $\pi_1^{-1}(\lambda)$ with an origin $O$. Define $t_{\lambda}: \R^2 \to \R^2$ as the identity on $\{x \leq \lambda\}$ and as $T$, relative to $O$, on $\{x \geq \lambda\}$. In fact, this definition does not depend on the choice of origin $O$, as can been seen for instance from the following explicit formula:
\begin{equation} t_{\lambda}(x,y) = \begin{cases} (x,y) \text{ if } x \leq \lambda, \\ (x, y + x - \lambda) \text{ if } x \geq \lambda. \end{cases}  \label{eq:t_lamb_eq}\end{equation}
Observe that $t_{\lambda}$ preserves the lexicographic order. Moreover, for $\vec{u} = (u_1, \ldots, u_s) \in \{-1,0,1\}^s$ and $\vec{\lambda} = (\lambda_1, \ldots, \lambda_s) \in \R^s$, we define $t_{\vec{u},\vec{\lambda}} = t_{\lambda_1}^{u_1} \circ \ldots \circ t_{\lambda_s}^{u_s}$. 

Consider the group 
\[ (G_s = \{-1,1\}^s, *), \qquad \vec{\epsilon'} * \vec{\epsilon} = ({\epsilon_1}' \epsilon_1, \ldots, {\epsilon_s}' \epsilon_s) \]
and the set 
\[ \mathcal{P} = \{ (P,\vec{c},\vec{\epsilon})\mid P \text{ is a polygon }, \vec{c}\in(\R^2)^s, \vec{\epsilon}\in\{-1,1\}^s  \}.  \]
Then $G_s$ acts on $\mathcal{P}$ by
\begin{equation} \vec{\epsilon'} \cdot (P,\vec{c},\vec{\epsilon}) = \left(t_{\vec{u},\vec{\lambda}}(P), t_{\vec{u},\vec{\lambda}}(\vec{c}), \vec{\epsilon'} * \vec{\epsilon} \right), \label{eq:Gs_action}\end{equation} 
with 
\begin{equation} \vec{u} = \left(\frac{\epsilon_1 - \epsilon_1 \epsilon_1'}{2}, \ldots, \frac{\epsilon_s - \epsilon_s \epsilon_s'}{2}\right), \qquad \vec{\lambda} = (\pi_1(c_1), \ldots, \pi_1(c_s)).\label{eq:u_lambda}\end{equation}

This action does not necessarily preserve convexity of the polygon. We say that a marked weighted polygon $(\Delta,\vec{c},\vec{\epsilon})$ is \emph{admissible} if its $G_s$-orbit contains
only convex polygons, and thus obtain an action of $G_s$ on the set of admissible marked weighted polygons, which commutes with the action of $\mathcal{T}$ described above. Hence we get a $G_s \times \mathcal{T}$-action on the set of admissible marked weighted polygons.

Let $(\De,\vec{c},\vec{\epsilon})$ be a marked weighted polygon and let
\[
 \mathcal{L}^j_{\strep} := \{(x,y)\in\R^2 \mid x =  \pi_1(c_j),\, \epsilon_j y \geq \epsilon_j\pi_2(c_j)\},
\]
so $\mathcal{L}^j_{\strep}$ is the ray starting at the point $c_j$ directed upwards if $\epsilon_j =1$ and directed downwards if $\epsilon_j = -1$.
Furthermore, let
\begin{equation}\label{eqn:L-rays}
 \mathcal{L}_{\strep} := \bigcup_{j=1}^s \mathcal{L}^j_{\strep}
\end{equation}
denote the union of all such rays. We call each $\mathcal{L}^j_{\strep}$ a \emph{cut} and denote them by dotted lines in the figures, such as in the polygon at the top of Figure \ref{fig:poly_graph} for instance. 

The following definition contains a result that can be easily adapted from \cite[Lemma 4.2]{PVNacta} to the non-simple case.

\begin{dfn}
\label{def:marked_poly}
If  a marked weighted polygon $(\De,\vec{c},\vec{\epsilon})$ satisfies the following properties:
\begin{enumerate}[noitemsep]
 \item each point $q\in \partial \De \cap \mathcal{L}_{(\De,\vec{c},\vec{\epsilon})}$ is a vertex of $\Delta$ which satisfies either the $k$-fake or $k$-hidden Delzant condition, where
 \[k = \# \, \left\{j\in\{1,\ldots, s\} \ | \ q \in \mathcal{L}^j_{\strep}\right\},\]
 in which case $q$ is known as a \emph{$k$-fake} or \emph{$k$-hidden corner}, respectively;
 \item all other vertices satisfy the Delzant condition, and are called \emph{Delzant corners},
\end{enumerate}
then it is admissible. We call its $G_s \times \mathcal{T}$-orbit a \emph{marked Delzant semitoric polygon}, and we denote it by $[(\De,\vec{c},\vec{\epsilon})]$.
\end{dfn}

\begin{rmk}
 When defining the Delzant, $k$-hidden Delzant, and $k$-fake conditions at a vertex $q$, either the vectors directing the edges or the inwards pointing normal vectors to the edges may be used.
 If $u_1,u_2\in\Z^2$ are primitive vectors directing the edges emanating from $q$ ordered so that $\det (u_1,u_2) > 0$,  then $q$ satisfies the Delzant condition if $\det (u_1,u_2) = 1$, the $k$-hidden Delzant condition if $\det (u_1,T^k u_2)=1$, and the $k$-fake condition if $\det (u_1, T^k u_2)=0$.
 This point of view is taken in some references (such as~\cite{LFPfamilies}) and is equivalent to the definition of these conditions which we give above.
The point of view that we adopt in the present paper is used to simplify the relation to semitoric helices, discussed in Section~\ref{sec:helix}.
\end{rmk}

\paragraph{The marked weighted polygon of a semitoric system.} Next, we explain how to associate a marked weighted polygon to any compact semitoric system $(M,\omega,F=(J,H))$ with $m_f$ focus-focus points; this construction is already contained in~\cite{VNpoly}, but here we adopt the slightly different point of view of~\cite{PPT}. We assume that $m_f > 0$ for the sake of clarity, see~\cite[Remark 2.9]{LFPfamilies} for the case $m_f = 0$.

Let $B = F(M)$. Let $(x_j,y_j)$, $j \in \{1, \ldots, m_f\}$, be the focus-focus values labelled in lexicographic order, and note that it is possible that $(x_j,y_j)=(x_{j+1}, y_{j+1})$ for some $j$ since we are not requiring our systems to be simple. Let $\vec{\epsilon} \in \{-1,1\}^{m_f}$ and let $\ell_j^{\epsilon_j}$ be the vertical half-line starting from $(x_j,y_j)$ and going downwards if $\epsilon_j = -1$ and upwards if $\epsilon_j = 1$, and let
\[ \ell^{\vec{\epsilon}} = \bigcup_{j=1}^{m_f} \ell_j^{\epsilon_j}. \]
The choice of $\vec{\epsilon}$ corresponds to a choice of cuts (downwards or upwards) at the focus-focus values. For such a focus-focus value $c$, the number of cuts emanating from it corresponds to the number of focus-focus points in the fiber $F^{-1}(c)$.

By~\cite[Theorem 3.8]{VNpoly}, there exists a 
homeomorphism $g_{\vec{\epsilon}}: B \to  \R^2$ of the form 
\begin{equation*} g_{\vec{\epsilon}}(x,y) = \left(x, g_{\vec{\epsilon}}^{(2)}(x,y)\right), \qquad \dpar{g_{\vec{\epsilon}}^{(2)}}{y} > 0, \end{equation*}
such that the image $\poly{\vec{\epsilon}} := g_{\vec{\epsilon}}(B)$ is a rational convex polygon, and such that 
${g_{\vec{\epsilon}}}_{|B \setminus \ell^{\vec{\epsilon}}}$ is a diffeomorphism
onto its image. Then, by~\cite{VNpoly,PPT}, 
\begin{equation} \left( \Delta_{\vec{\epsilon}}, \vec{c} = (g_{\vec{\epsilon}}(x_1,y_1), \ldots, g_{\vec{\epsilon}}(x_{m_f},y_{m_f})), \vec{\epsilon} \right) \label{eqn:marked-semitoric-poly}\end{equation}
is a marked weighted polygon satisfying the conditions of Definition \ref{def:marked_poly}. The map $g_{\vec{\epsilon}}$ is called a \emph{developing map} and $g_{\vec{\epsilon}} \circ F$ is called a \emph{generalized toric momentum map}.

The freedom in this construction corresponds to the action of $G_{m_f} \times \mathcal{T}$ given in Equations \eqref{eq:T_action} and \eqref{eq:Gs_action}. The \emph{marked semitoric polygon} $\poly{(M,\omega,F)}$ of $(M,\omega,F)$ is the orbit under this group action of any of the marked weighted polygons in Equation~\eqref{eqn:marked-semitoric-poly} constructed above, which is thus independent of the choices involved, see \cite[Section 2.2]{PVNacta} or \cite[Section 2.4]{LFPfamilies} for more details in the case that the semitoric system is simple.
In this section we have also dealt with the case in which the system is possibly non-simple since the construction is essentially the same; the non-simple case is included in the original construction of the polygons~\cite{VNpoly} and was added to the classification in~\cite{PPT}.

As in~\cite{LFPfamilies} we will use the terminology \emph{unmarked semitoric polygon} to denote the semitoric polygon defined in~\cite[Section 4.3]{PVNinventiones}, which did not include the heights of the marked points. We denote by $[(\Delta,\vec{\lambda},\vec{\epsilon})]$ such an unmarked semitoric polygon; it can be obtained from the marked semitoric polygon $[(\Delta,\vec{c},\vec{\epsilon})]$ by simply forgetting the second coordinate of every element in $\vec{c}$. Alternatively, we can define an unmarked Delzant semitoric polygon as the $G_s \times \mathcal{T}$-orbit of an admissible unmarked weighted semitoric polygon satisfying the same conditions as in Definition \ref{def:marked_poly}, and then define the unmarked semitoric polygon of a semitoric system in the same fashion as its marked semitoric polygon, but by only keeping track of the first coordinate of each focus-focus value. 

\begin{rmk}
The marked semitoric polygon contains in particular the data of the height invariant of each focus-focus singularity, since this invariant corresponds to the height of the corresponding marked point in any representative of this marked semitoric polygon. This is why in general, when showing such a representative, we will not insist on this height invariant and simply picture the marked points and cuts. The only exception is Figure \ref{fig:strictly_min_poly}, where we need to insist on the height invariants and their ranges.
\end{rmk}

\subsection{Hamiltonian \texorpdfstring{$S^1$}{S1}-spaces}
\label{sec:S1-actions}

A semitoric system is an instance of an integrable system lifting a Hamiltonian $S^1$-space, as defined below.

\begin{dfn}
\label{def:S1space}
A \emph{Hamiltonian $S^1$-space} is a triple $(M,\om,J)$ such that $(M,\om)$ is a compact, connected four-dimensional symplectic manifold and $J$ is the momentum map for an effective Hamiltonian $S^1$-action.
\end{dfn}

Such spaces were classified by Karshon~\cite{karshon} in terms of a labeled graph, which we describe now. 
Let $M^{S^1}$ denote the fixed point set of the $S^1$-action. 
Each component of $M^{S^1}$ is either an isolated point or a symplectic surface.
Let $\Z_k$ denote the cyclic subgroup of $S^1$ of order $k$.
For $k\geq 2$, a \emph{$\Z_k$-sphere} is a component of the closure of the set of points
with stabilizer $\Z_k$. 
It turns out that $\Z_k$ spheres are symplectic spheres on which the group $S^1/\Z_k$ acts with two fixed points, known as the ``poles'' of the sphere. These poles are fixed points of the $S^1$-action and therefore each $\Z_k$-sphere connects two components of $M^{S^1}$.

Conversely, one can determine if a given fixed point $p$ lies at the end of a $\Z_k$-sphere for some $k$ by computing the weights of $J$ at $p$, which are defined as follows. Recall the following normal form for $J$ given in \cite[Corollary A.7]{Kar02}: there exist local complex coordinates $z_1, z_2$ near $p$, where $\C^2$ is equipped with the symplectic form $\omega_{\C} = \frac{i}{2} \left( \dd z_1 \wedge \dd \bar{z}_1 + \dd z_2 \wedge \dd \bar{z}_2  \right)$, and relatively prime integers $m, n$ such that 
\begin{equation} \label{eq:normal_J} J - J(p) = q_{m,n} := \frac{m}{2} |z_1|^2 + \frac{n}{2} |z_2|^2. \end{equation}
In other words, the $S^1$-action generated by $J$ reads in these coordinates:
\[ t \cdot (z_1,z_2) = (e^{i m t} z_1, e^{i n t} z_2). \]
These integers $m,n$ are called the \emph{weights} of $J$ at $p$. If $|m|\geq 2$ then $p$ lies in a $\Z_{|m|}$-sphere (and similarly for $n$).

The \emph{Karshon graph} of the Hamiltonian $S^1$-space $(M,\omega,J)$ is a labeled graph including two types of vertices constructed as follows:
\begin{itemize}
 \item \textbf{regular vertices:} associated to each isolated fixed point of the $S^1$-action there is a vertex of the graph labeled with the value of $J$ on that fixed point;
 \item \textbf{fat vertices:} associated to each fixed surface $\Sigma$ of the $S^1$-space there is a fat vertex of the graph (drawn as a large oval) labeled with the common $J$-value, its genus, and its normalized symplectic volume $\frac{1}{2\pi}\int_\Sigma\om$; 
 \item \textbf{edges:} associated to each $\Z_k$-sphere with $k\geq 2$ there is an edge labeled by $k$ connecting the two vertices corresponding to its poles.
\end{itemize}

When we draw such a graph the $J$-value of each vertex is represented by its horizontal position. This is unlike the convention originally used by Karshon, since this will make it easier to compare Karshon graphs to semitoric polygons for our purposes. An example of Karshon graph is displayed at the bottom of Figure \ref{fig:poly_graph}.

Given a compact semitoric system $(M,\om,(J,H))$ we call the $S^1$-space $(M,\om,J)$ the
\emph{underlying $S^1$-space}. Hohloch-Sabatini-Sepe~\cite[Theorem 3.1]{HSS} describe how to obtain the Karshon graph of the underlying $S^1$-space from the (un)marked semitoric polygon of a given semitoric system. This extends results of Karshon~\cite{karshon} regarding extracting the graph from a Delzant polygon.

Following~\cite[Section 3]{HSS} the Karshon graph may be obtained from any representative of the (un)marked semitoric polygon using the following procedure:
\begin{itemize}
 \item each vertical edge of the polygon (if any), commonly called a \emph{vertical wall}, corresponds to a fat vertex of the Karshon graph, labeled with genus zero (so most of the time, we will omit the genus label), normalized area equal to the length of the edge, and the $J$-value of that edge;
 \item each marked point of the polygon (corresponding to a focus-focus point of the system) corresponds to a vertex of the graph which is not connected to any edge, labeled with the corresponding $J$-value;
 \item each Delzant and hidden Delzant corner of the polygon corresponds to a regular vertex of the graph, labeled with the corresponding $J$-value;
 \item a $\Z_k$-sphere corresponds to a chain of edges of the polygon connecting exactly two Delzant or hidden Delzant corners, such that one, and hence all, of these edges has slope of the form $b/k$ for some $b\in\Z$ relatively prime to $k$. Necessarily, each vertex connecting adjacent edges in the chain is a fake corner. Such a chain of edges in the polygon corresponds to an edge labeled with $k$ in the graph.
\end{itemize}
This is illustrated in Figure \ref{fig:poly_graph}. Recall that the $J$-values are omitted and represented by their horizontal positions, as explained above.

\begin{figure}
\begin{center}
\begin{tikzpicture}

\filldraw[draw=black, fill=gray!60] (0,0) -- (0,1) -- (1,3) -- (2,4) -- (3,4.5) -- (4,4) -- (8,0) -- cycle;
\draw [dashed] (1,1.1) -- (1,3);
\draw (1,1.1) node[] {$\times$};
\draw [dashed] (3,2.6) -- (3,4.5);
\draw (3,2.6) node[] {$\times$};
\draw[ultra thick] (2,4) -- (3,4.5) -- (4,4);

\fill (0,-2) ellipse (2mm and 5mm);
\fill[black] (1,-2) circle (.1);
\fill[black] (2,-1.5) circle (.1);
\draw (2,-1.5) -- (4,-1.5);
\fill[black] (4,-1.5) circle (.1);
\fill[black] (3,-2) circle (.1);
\draw (3,-1.5) node [above] {$2$};
\fill[black] (8,-2) circle (.1);

\end{tikzpicture}
\end{center}
\caption{A marked semitoric polygon and the Karshon graph corresponding to the underlying Hamiltonian $S^1$-space. The chain of edges indicated in bold corresponds to a single $\Z_2$-sphere.}
\label{fig:poly_graph}
\end{figure}
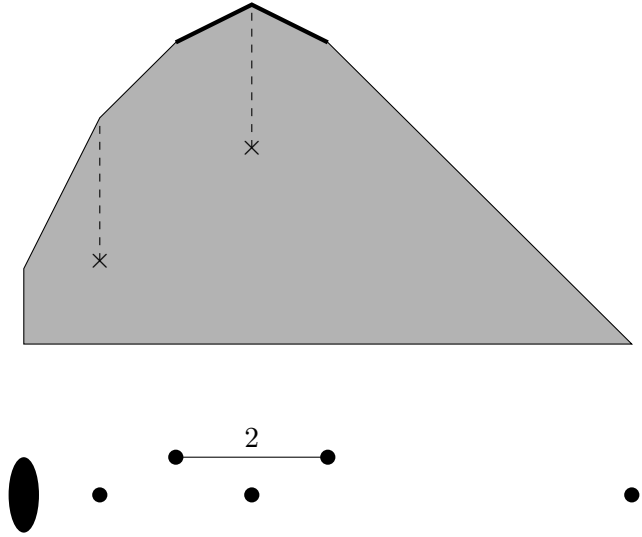

The following fact will be useful for us:

\begin{lm}
\label{lem:Zk-boundary}
 Let $(M,\om,(J,H))$ be a semitoric system and let $\Sigma\subset M$ be a $\Z_k$-sphere of the underlying
 $S^1$-space $(M,\om,J)$ where $k>1$. Then $F(\Sigma) \subset \partial F(M)$.
\end{lm}

\begin{proof}
Let $(\Delta,\vec{c},\vec{\epsilon})$ be a representative of the marked semitoric polygon of the system. Recall from Section \ref{subsect:msp} that $\Delta = g(F(M))$ where $g: F(M) \to \R^2$ is a homeomorphism. By the above discussion, as in~\cite{HSS}, $g(F(\Sigma))$ is contained in the boundary of $\Delta$ hence, since $g$ is a homeomorphism, $F(\Sigma)$ is contained in the boundary of $F(M)$.   
\end{proof}

\subsection{Hypersemitoric systems}

As before, let $(M,\omega)$ be a four-dimensional compact, connected symplectic manifold. If $J: M \to \R$ is such that $(M,\omega,J)$ is a Hamiltonian $S^1$-space (see Definition \ref{def:S1space}), one can wonder (and this is what we will try to do to construct explicit examples of semitoric systems in Sections \ref{sec:CP2} and \ref{sec:type_3}) if one can find $H$ such that $(M,\omega,(J,H))$ is semitoric. From \cite{karshon} and \cite{HSS} we know that this is not always possible (necessary and sufficient conditions will appear in \cite{HSSS2}), but it will be if we replace semitoric systems with the following more general systems defined in \cite{HohPal21}.

\begin{dfn}
\label{dfn:hypersemi}
A \emph{hypersemitoric system} is an integrable system $(M,\omega,F = (J,H))$ such that
\begin{enumerate}
    \item $J$ is the momentum map for an effective Hamiltonian $S^1$-action;
    \item the singular points of $F$ are either non-degenerate or parabolic (see Definition \ref{dfn:para}).
\end{enumerate}
\end{dfn}

These systems can also be defined in the non-compact case, by adding the assumption that $J$ is proper. In \cite{HohPal21}, the authors proved that a Hamiltonian $S^1$-space can always be lifted to such a system.

\begin{thm}[{\cite[Theorem 1.7]{HohPal21}}]
Let $(M,\omega,J)$ be a Hamiltonian $S^1$-space. Then there exists $H \in \mathcal{C}^{\infty}(M,\R)$ such that $(M,\omega,F = (J,H))$ is hypersemitoric.
\end{thm}

As explained in \cite[Remark 1.9]{HohPal21}, it is likely that hypersemitoric systems are generic among integrable systems lifting a given Hamiltonian $S^1$-space. This is consistent with the fact that in practice, when investigating an explicit Hamiltonian $H$, one comes up, in general, with a system $(M,\omega,F=(J,H)$ that is hypersemitoric (and not necessarily semitoric, see for instance the system described in \cite[Section 6.6]{LFPfamilies}).

Due to the normal form near a parabolic point (see \cite[Theorem 3.1]{KudMar} and \cite[Proposition 2.1]{BolGugKud}), the presence of parabolic singularities immediately induces the presence of hyperbolic-regular singularities, and in \cite{HohPal21} the authors studied how these parabolic and hyperbolic-regular singularities can behave in an integrable system lifting a Hamiltonian $S^1$-space. In fact, the parabolic normal form implies that to a parabolic value in $F(M)$ are attached two curves, one consisting of hyperbolic-regular values, and the other consisting of elliptic-regular values. As explained in \cite{EfsGia}, the curve of hyperbolic-regular singularities attached to a parabolic singularity (called cusp singularity in \cite{EfsGia}) can display two different behaviours, called \emph{pleat} and \emph{flap} topologies, that become transparent when unfolding the image of the momentum map to take into account that the fibers of $F$ over one side of this curve are disconnected. The curve of elliptic regular-values can either end on an elliptic-elliptic value or on another parabolic value, and in  \cite{HohPal21} the authors reserved the term \emph{flap} for the former situation. Here we need to be more precise.

\begin{dfn}[{\cite[Definition 2.45]{HohPal21}}]
Let $(M,\omega,F)$ be an integrable system. A \emph{triangular flap} is a curved triangular region in the interior of $F(M)$ whose corners are two parabolic values and one elliptic-elliptic value, and whose sides are a curve of hyperbolic-regular values joining the two parabolic values, and two curves of elliptic-regular values joining each of the parabolic values with the elliptic-elliptic value. We call the elliptic-elliptic value (respectively the parabolic values) the \emph{elliptic corner} (respectively the \emph{parabolic corners}) of the flap.
\end{dfn}

For instance, a triangular flap is clearly visible in the $t=1$ subfigure of Figure \ref{fig:image_CP2_alpha1_gamma1over8_delta5} in Section~\ref{sec:CP2}; in this figure the elliptic-elliptic corner appears in red.

The situation of a flap containing no elliptic-elliptic point is also interesting, and it will appear later in the paper, namely in Section \ref{sec:type_3}.

\begin{dfn}
Let $(M,\omega,F)$ be an integrable system. A \emph{two-sided flap} is a region in the interior of $F(M)$ whose boundary consists of one curve of hyperbolic-regular values and one curve of elliptic-regular values meeting only at two distinct parabolic values.   
\end{dfn}

A two-sided flap is visible in Figures \ref{fig:image_3c_1_2_1over26_4point5} and \ref{fig:image_3c_1_2_1over26_4point5_t089}, and one can produce more examples by choosing the parameter $\delta$ between the two quantities whose maximum gives the lower bound for this parameter in the systems from Theorems \ref{thm:system_3c}, \ref{thm:system_3c_general}, \ref{thm:system_3a}, \ref{thm:system_3a_general} and \ref{thm:system_3b_general}. 

\begin{rmk}
These definitions focus on what happens in the image of the momentum map. This does not allow us to discriminate between the different possible topologies associated with the preimage of the hyperbolic-regular line of a flap discussed in \cite[Section 4]{EfsGia}.    
\end{rmk}

\section{Blowups and minimal semitoric systems}
\label{sec:blow_semi}

We will start this section by explaining several versions of the blowup and blowdown operations, including equivariant ones. In particular, we will describe those blowups and blowdowns which are compatible with the structure of a semitoric system, called of toric type and semitoric type. 

We would like to understand the semitoric systems from which all semitoric systems can be obtained by performing a sequence of blowups. If we consider toric type blowups only, we call such a system minimal, see Definition \ref{def:min_sys}; if we consider both types of blowups, we call it strictly minimal, see Definition \ref{def:strictly_min_system}. In~\cite{KPP_min} the authors introduced an invariant of semitoric systems called the semitoric helix, which we will describe in this section. Then they used algebraic techniques extending those from~\cite{KPP_fan} to obtain a list of all helices which are associated to minimal systems up to a certain equivalence.
In fact, we noticed that with the notion of equivalence in~\cite{KPP_min} one misses a couple of cases of minimal helices (roughly reflections of the ones listed in~\cite{KPP_min}), so in this section we will amend the definition of equivalence accordingly to solve this small oversight in~\cite{KPP_min} by adding the two new operations related to these reflections.

\subsection{Blowups of semitoric systems}
\label{subsec:blowups_semi}

Recall that a blowup of size $\lambda > 0$ of a four-dimensional symplectic manifold $(M,\omega)$ at a point $p \in M$ is a new four-dimensional symplectic manifold $(\blowup{p}(M),\widetilde{\omega}_{\lambda})$ which projects to $(M,\omega)$. Roughly, it is obtained by removing the image by a symplectic embedding of a four-dimensional ball $B_{\sqrt{2\lambda}} \subset \C^2$ of radius $\sqrt{2\lambda}$ and collapsing the boundary by a Hopf fibration. For more details, see~\cite[Section 7.1]{McDuffSal}.

One can then investigate blowups required to preserve some additional structure, for instance in the presence of a Hamiltonian $S^1$-action or for semitoric systems. In both cases, the embedding of $B_{\sqrt{2\lambda}}$ is required to respect the additional structure.

In \cite[Section 6.1]{karshon}, Karshon describes $S^1$-equivariant blowups of Hamiltonian $S^1$-spaces, which we describe in Section~\ref{sec:equi_blow}, and explains how to read them on the corresponding graphs. Since a semitoric system has in particular an underlying structure of Hamiltonian $S^1$-space, blowups in the semitoric category must be $S^1$-equivariant.

There are two natural ways to perform a blowup on a semitoric system, which we will also describe in this section. These are related to certain types of $S^1$-equivariant blowups. Toric type blowups directly generalize $\T^2$-equivariant blowups of toric integrable systems and correspond to what is essentially a ``corner chop'' of the semitoric polygon. Semitoric systems also admit another type of blowup called a semitoric type blowup, which we describe in Section~\ref{subsec:semi_blow}.

\subsubsection{Blowups in the symplectic category}

We start with the case of performing a blowup on $\C^2$ at the origin, and use this to define a blowup in more general situations.
There are several equivalent ways to proceed, and we will define the blowup by making use of Lerman's symplectic cut \cite{Ler}.

Let $\lambda>0$ and consider the map 
$
\mu\colon \C^2\times \C \to \R
$
given by 
\begin{equation}
\label{eqn:CP2-S1-action}
\mu (z_1,z_2,w) = \frac{1}{2}\left(|z_1|^2+|z_2|^2+|w|^2\right)-\lambda.
\end{equation}
The flow of the Hamiltonian vector field of $\mu$ generates an $S^1$-action on $\C^2\times\C$ and  $0\in\R$ is a regular value of $\mu$, so we may perform symplectic reduction at the level $\mu=0$.

As the result of this symplectic reduction, we obtain a smooth manifold
\[
\textrm{Bl}_0 (\C^2) := \mu^{-1}(0)/S^1,
\]
not depending on $\lambda$, equipped with a symplectic form $\widetilde{\omega}_\lambda$.
We call this symplectic manifold the \emph{blowup of $\C^2$ of size $\lambda$} (at the origin).
The submanifold $E$ on which $w=0$ is called the \emph{exceptional divisor}.
Note that the symplectic blowup can also be obtained by removing the standard ball $\{(z_1,z_2)\in\C^2 \mid \frac{1}{2}(|z_1|^2+|z_2|^2) < \lambda\}$ from $\C^2$ and collapsing the boundary by the Hopf fibration, see~\cite{McDuffSal}.
Furthermore, there is a natural map $\pi\colon \mathrm{Bl}_\lambda (\C^2) \to \C^2$ which is a diffeomorphism from $\textrm{Bl}_0(\C^2) \setminus E$ onto $\C^2\setminus \{(0,0)\}$ and satisfies $\pi(E) = \{0\}$.

Given any symplectic $4$-manifold $(M,\om)$, we can use a local chart to define a blowup. Let $p\in M$ and let $U$ be a chart around $p$ with a map $\phi\colon U\to \C^2$. This induces complex coordinates $(z_1,z_2)$ on $U$. The map $\mu$ from above can now be defined (locally) using these coordinates and the same procedure as above can be used to define the symplectic manifold $(\textrm{Bl}_p (M),\widetilde{\omega}_\lambda)$, which can depend on the choice of $p$, $U$, and $\phi$. In this case there is also a smooth map $\pi\colon \textrm{Bl}_p (M)\to M$ which sends $E$ to $p$ and restricts to a diffeomorphism from $\textrm{Bl}_\lambda(M)\setminus E$ to $M\setminus \{p\}$.

\subsubsection{$S^1$-equivariant blowups}
\label{sec:equi_blow}

Suppose that $(M,\om,J)$ is a Hamiltonian $S^1$-space. 
Roughly, an $\S^1$-equivariant blowup is a symplectic blowup which respects the $S^1$-action, which amounts to choosing a chart $(U,\phi)$ which is compatible with the $S^1$-action. We now make this precise.

Recall that in Equation~\eqref{eq:normal_J} for $m,n\in\Z$ with $\mathrm{gcd}(m,n)=1$ we defined the map 
$q_{m,n}\colon \C^2 \to \R$
which generates the standard Hamiltonian $S^1$-action $\C^2$ with weights $m$ and $n$.
Note that $\{\mu,q_{m,n}\}=0$, where $\mu$ is as in Equation~\eqref{eqn:CP2-S1-action}, and therefore both $q_{m,n}$ and the $S^1$-action it generates descend to $\textrm{Bl}_0 (\C^2)$.

An \emph{$S^1$-equivariant blowup of $(M,\om,J)$ at $p\in M$ of size $\lambda$} is a blowup of size $\lambda$ with respect to a
chart $(U,\phi)$ around $p$ such that there exist $m,n\in\Z$ such that
\[
 q_{m,n}\circ \phi = J - J(p).
\]
Note that in this case $p$ is necessarily a fixed point of the $\S^1$-action, and for any such fixed point for all sufficiently small $\lambda$ such a blowup exists.
The resulting space inherits a Hamiltonian $S^1$-action with Hamiltonian $J_\lambda$ which is equal to $\pi^*(\Phi_{m,n}\circ \phi)+J(p)$ in a neighborhood of $E$ and equal to $\pi^*J$ otherwise. 

Since the result of an $S^1$-equivariant blowup of a Hamiltonian $S^1$-space is another $S^1$-space, this also induces an operation on Karshon graphs.
Karshon~\cite{karshon} explained the result of performing such a blowup on the associated Karshon graph of the Hamiltonian $S^1$-space (see also~\cite{HohPal21}).

\subsubsection{Toric type blowups}
\label{subsubsec:tori_blowups}

Toric type blowups were investigated in~\cite[Section 4]{LFPfamilies} and are not specific to semitoric systems. A \emph{toric type blowup} can be performed at any completely elliptic point $p$ of an integrable system $(M,\omega,F)$, making use of the coordinates (and local $\T^n$-action) induced by the local normal form of a completely elliptic point.
Toric type blowups for simple semitoric systems, in particular their relation to semitoric polygons, are investigated in~\cite{LFPfamilies}. In fact, the case of non-simple systems is nearly the same. 

These blowups can be read on marked semitoric polygons. Recall that such a polygon is the image of a generalized toric momentum map $ g_{\vec{\epsilon}} \circ F$. 
Any local coordinates putting $p$ in the normal form of Theorem~\ref{thm:eliasson} in a sufficiently small neighborhood of $p$ can be extended to such a toric momentum map, and thus toric type blowups on a semitoric system correspond to an operation on the polygon invariant which is similar to the ``corner chops'' which appear when performing $\T^2$-equivariant blowups of toric systems. We explain this operation now.

Recall that the action of $G_s\times \mathcal{T}$ on marked semitoric polygons and points given
in Equation~\eqref{eq:T_action} also extends to pairs $(\De,q)$, where $q\in\De$ (see~\cite[Equation (6)]{LFPfamilies}). Let $Q$ be an equivalence class of this action, which amounts to
a marked Delzant semitoric polygon with a point coherently chosen for each representative. 
Let
\[
 \mathrm{Simp}_q(u_1,u_2) = \left\{q+t_1 u_1 + t_2 u_2 \mid t_1 \geq 0, t_2 \geq 0, t_1+t_2 \leq 1 \right\}
\]
denote the simplex with one vertex at $q$ and edges directed by $u_1,u_2 \in\Z^2$.
If there exists a representative
$((\De, \vec{c}, \vec{\epsilon}),q)$ of $Q$ such that $q$ is a Delzant corner of $\De$ and 
\[
 \mathrm{Simp}_q(\lambda v_1, \lambda v_2) \cap (\mathcal{L}_{(\De,\vec{c},\vec{\epsilon})}\cup V) = \{q\}
\]
where $\mathcal{L}_{(\De,\vec{c},\vec{\epsilon})}$ is the set of cuts as in Equation~\eqref{eqn:L-rays}, $V$ is the set of all vertices
of $\De$, and $v_1,v_2\in\Z^2$ are the primitive vectors directing the edges of $\De$ adjacent to $q$, then we define the \emph{corner chop of $(\De, \vec{c}, \vec{\epsilon})$ of size $\lambda$ at $Q$} to be the $G_s\times \mathcal{T}$-orbit of $\left(\overline{\De\setminus(\mathrm{Simp}_q(\lambda v_1, \lambda v_2))}, \vec{c}, \vec{\epsilon}\right)$.

For any line segment $L\subset \R^2$ which has rational slope, there exists a matrix $A\in \mathrm{SL}(2,\Z)$ such that $A(L)$ is horizontal. The length of the resulting segment is independent of the choice of such $A$, and is called the \emph{$\mathrm{SL}(2,\Z)$-length of $L$}.
The new edge introduced by a corner chop of size $\lambda$ always has $\mathrm{SL}(2,\Z)$-length equal to $\lambda$.

The following lemma was stated for simple semitoric systems in~\cite[Lemma 4.6]{LFPfamilies}; it still holds for semitoric systems that may be non-simple.

\begin{lm}\label{lem:blowups_semitoric_chops}
The marked semitoric polygon associated to the toric type blowup of size $\lambda$ of a semitoric system $(M,\omega,F)$ at the completely elliptic point $p$ is the marked semitoric polygon obtained by performing a semitoric corner chop of size $\lambda$ on $[((\De,\vec{c},\vec{\epsilon}),(g_{\vec{\epsilon}}\circ F)(p))]$
where $g_{\vec{\epsilon}}$ is the developing map associated to $(\De,\vec{c},\vec{\epsilon})$ so $\De = (g_{\vec{\epsilon}}\circ F)(M)$.
Moreover, if the marked semitoric polygon admits a semitoric corner chop of size $\lambda$ at the corner $F(p)$ then a blowup of size $\lambda$ may be performed at $p$.
\end{lm}

By this lemma, whether or not a blowup may be performed on the system can be read off of the marked semitoric polygon by examining the $\mathrm{SL}(2,\Z)$-length of the adjacent edges and the positions of the marked points. An example of corner chop is displayed in Figure \ref{fig:semitoricchop}.

\begin{figure}
\begin{center}
\begin{tikzpicture}
\draw [->] (3.5,-0.5) -- (4.5,-0.5);
\draw [->] (8.5,-0.5) -- (9.5,-0.5);

\draw[decoration={brace,raise=5pt, amplitude = 5pt},decorate] (0,-4.2) -- (0,3.2);
\draw[decoration={brace,raise=5pt, amplitude = 5pt},decorate] ( 5,-4.2) -- (5,3.2);
\draw[decoration={brace,raise=5pt, amplitude = 5pt},decorate] ( 10,-4.2) -- (10,3.2);

\filldraw[draw=black, fill=gray!60] (0,0) node[anchor=north,color=black]{}
  -- (0,3) node[anchor=south,color=black]{}
  -- (1,3) node[anchor=south,color=black]{}
  -- (2,2) node[anchor=north,color=black]{}
  -- (3,0) node[anchor=north,color=black]{}
  -- cycle;
\draw [dashed] (1,1.1) -- (1,3);
\draw (1,1.1) node[] {$\times$};
\draw [dashed] (2,0.6) -- (2,2);
\draw (2,0.6) node[] {$\times$};

\filldraw[draw=black, fill=gray!60] (5,0) node[anchor=north,color=black]{}
  -- (5,3) node[anchor=south,color=black]{}
  -- (6,3) node[anchor=south,color=black]{}
  -- (7,2) node[anchor=north,color=black]{}
  -- (8,0) node[anchor=north,color=black]{}
  -- cycle;
\draw [dashed] (6,1.1) -- (6,3);
\draw (6,1.1) node[] {$\times$};
\draw [dashed] (7,0.6) -- (7,2);
\draw (7,0.6) node[] {$\times$};
  
\filldraw[draw=black, fill=gray!60, pattern=north east lines] (5,0.5) node[anchor=north,color=black]{}
  -- (5,3) node[anchor=south,color=black]{}
  -- (6,3) node[anchor=south,color=black]{}
  -- (7,2) node[anchor=south,color=black]{}
  -- (7.5,1) node[anchor=north,color=black]{}
  -- (7,1.5) node[anchor=north,color=black]{}
  -- (6,1.5) node[anchor=north,color=black]{}
  -- cycle;
  
\filldraw[draw=black, fill=gray!60] (10,0) node[anchor=north,color=black]{}
  -- (10,0.5) node[anchor=south,color=black]{}
  -- (11,1.5) node[anchor=south,color=black]{}
  -- (12,1.5) node[anchor=south,color=black]{}
  -- (12.5,1) node[anchor=north,color=black]{}
  -- (13,0) node[anchor=north,color=black]{}
  -- cycle;
\draw [dashed] (11,1.1) -- (11,1.5);
\draw (11,1.1) node[] {$\times$};
\draw [dashed] (12,0.6) -- (12,1.5);
\draw (12,0.6) node[] {$\times$};  
  
\filldraw[draw=black, fill=gray!60] (0,-4) node[anchor=north,color=black]{}
  -- (0,-1) node[anchor=south,color=black]{}
  -- (3,-1) node[anchor=south,color=black]{}
  -- (2,-3) node[anchor=north,color=black]{}
  -- (1,-4) node[anchor=north,color=black]{}
  -- cycle;
\draw [dashed] (1,-2.9) -- (1,-4);
\draw (1,-2.9) node[] {$\times$};
\draw [dashed] (2,-2.4) -- (2,-3);
\draw (2,-2.4) node[] {$\times$};

\filldraw[draw=black, fill=gray!60] (5,-4) node[anchor=north,color=black]{}
  -- (5,-1) node[anchor=south,color=black]{}
  -- (8,-1) node[anchor=south,color=black]{}
  -- (7,-3) node[anchor=north,color=black]{}
  -- (6,-4) node[anchor=north,color=black]{}
  -- cycle;
\draw [dashed] (6,-2.9) -- (6,-4);
\draw (6,-2.9) node[] {$\times$};
\draw [dashed] (7,-2.4) -- (7,-3);
\draw (7,-2.4) node[] {$\times$};
  
\filldraw[draw=black, fill=gray!60, pattern=north east lines] (5,-3.5) node[anchor=north,color=black]{}
  -- (5,-1) node[anchor=south,color=black]{}
  -- (7.5,-1) node[anchor=north,color=black]{}
  -- cycle;
  
\filldraw[draw=black, fill=gray!60] (10,-4) node[anchor=north,color=black]{}
  -- (10,-3.5) node[anchor=south,color=black]{}
  -- (12.5,-1) node[anchor=south,color=black]{}
  -- (13,-1) node[anchor=north,color=black]{}
  -- (12,-3) node[anchor=north,color=black]{}
  -- (11,-4) node[anchor=north,color=black]{}
  -- cycle;
\draw [dashed] (11,-2.9) -- (11,-4);
\draw (11,-2.9) node[] {$\times$};
\draw [dashed] (12,-2.4) -- (12,-3);
\draw (12,-2.4) node[] {$\times$};
\end{tikzpicture}
\end{center}
\caption{An example of a corner chop. Note that the region that is removed only has to be a simplex for one representative of the polygon. The simplex $\mathrm{Simp}_q(\lambda v_1, \lambda v_2)$ is shaded in the lower middle figure.}
\label{fig:semitoricchop}
\end{figure}
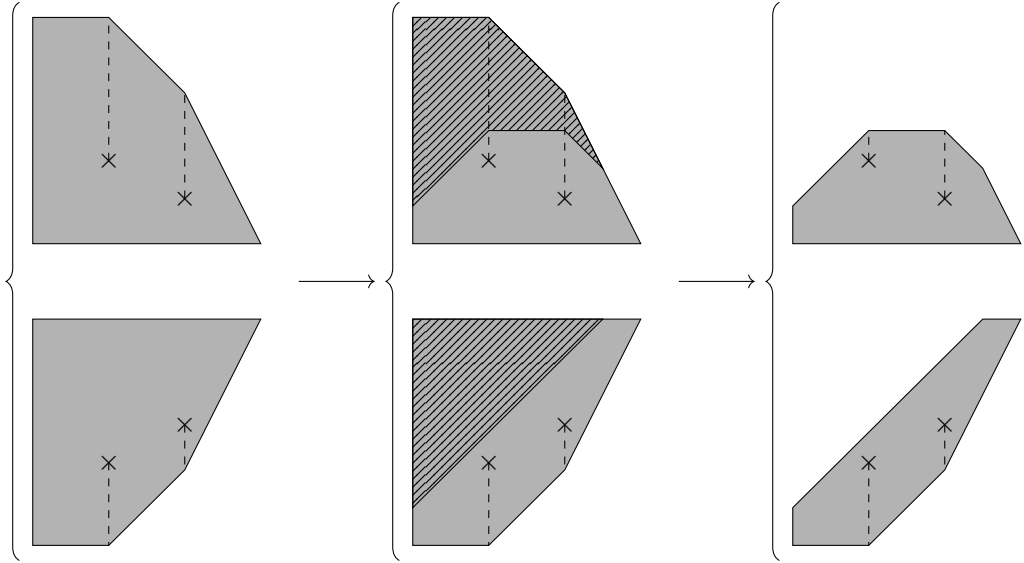

\subsubsection{Semitoric type blowups}
\label{subsec:semi_blow}

There is another operation on semitoric systems which induces a $S^1$-equivariant blowup on the underlying $S^1$-space: the semitoric type blowup.
This operation is a special case of almost toric blowups introduced by Symington~\cite[Section 5.4]{Sym},
and Auroux explicitly describes a specific example of a semitoric type blowup on $\C^2$ in~\cite[Example 3.1.2]{Auroux2009} when studying wall-crossing and instanton corrections.
This type of blowup for general integrable systems was first discussed by Zung~\cite[Example 4.16]{Zung03}.
Furthermore, semitoric type blowups for general four-dimensional integrable systems $(M,\om,F)$ will be studied in the upcoming~\cite{HSSS2}, but for the present paper it is sufficient to only consider the case of semitoric systems by looking at the effect of the semitoric type blowdown on the marked semitoric polygon, following~\cite[Section 2.10.2]{HohPal21}, and then apply the Pelayo-V\~{u} Ng\d{o}c classification~\cite{PVNinventiones,PVNacta} to conclude that there exists a system with that marked semitoric polygon.

Roughly, the idea of a semitoric type blowup on a semitoric polygon can be understood as a method to add a marked point to the polygon (i.e.~add a focus-focus point to the system). Introducing a new marked point with an upwards cut causes the top boundary of the polygon to bend downwards, and the polygon can admit this downwards bend if the original polygon has a vertical wall of sufficient size. In order to make this idea more precise, we first need to introduce some notation.

\paragraph{Wall chops and wall unchops.} We start by defining an operation on marked weighted polygons that will correspond to a semitoric type blowup. Let $P \subset \R^2$ be a polygon. We define the top boundary of $P$ as 
\[ \partial^+ P = \left\{ (x,y) \in P \ \Bigg| \ y = \max_{(x,v) \in P} v \right\} \]
and its bottom boundary $\partial^- P$ in a similar fashion.
Let $(\Delta,\vec{c} = (c_1, \ldots, c_s),\vec{\epsilon} = (\epsilon_1, \ldots, \epsilon_s))$ be a marked weighted polygon which has a vertical wall (here $s = 0$ is allowed, in which case we consider $(\Delta,\emptyset,\emptyset)$). Without loss of generality (see Remark \ref{rmk:left_right_blow})  we may assume that the vertical wall is 
\begin{equation} W = \left\{ (x,y) \in \Delta \ \Bigg| \ x =  \min_{(u,y) \in \Delta} u \right\}. \label{eq:left_wall}\end{equation}
Let 
\[ \ell(W) = \max_{c \in W} \pi_2(c) - \min_{c \in W} \pi_2(c), \quad L(\Delta) = \max_{c \in \Delta} \pi_1(c) - \min_{c \in \Delta} \pi_1(c), \]
so that $\ell(W)$ is the length of $W$ and $L(\De)$ is the width of $\De$.
Let $\lambda$ be a real number such that
\[ 0 < \lambda < \min(\ell(W),L(\Delta)), \]
and, for $\mu \in \R$, let $\tilde{t}_{\mu}: \R^2 \to \R^2$ be defined as the identity on $\{x \geq \mu\}$ and as $T$, relative to a given origin on the vertical line $\pi_1^{-1}(\mu)$, on $\{x \leq \mu\}$. Explicitly,
\begin{equation} \tilde{t}_{\mu}(x,y) = \begin{cases} (x,y) & \text{ if } x \geq \mu, \\ (x, y + x - \mu) & \text{ if } x \leq \mu. \end{cases} \label{eq:t_tilde_mu_eq} \end{equation}
Let $\Delta'$ be the unique convex polygon such that 
\[ \partial^- \Delta' = \partial^- \Delta, \quad \partial^+ \Delta' = \tilde{t}_{\pi_1(W)+\lambda}(\partial^+ \Delta). \]
Choose any $c_b' \in \Delta' \cap (\pi_1)^{-1}(\pi_1(W)+\lambda)$. Let $c_1' \leq_{\text{lex}} \ldots \leq_{\text{lex}} c_s'$ be such that
\[ \forall \ell \in \{1, \ldots, s\} \qquad c_{\ell}' \in \text{int}(\Delta') \cap (\pi_1)^{-1}(\pi_1(c_{\ell})). \]
Let $\ell \in \{1, \ldots, s\}$ be such that $c_{\ell}' \leq_{\text{lex}} c'_b  \leq_{\text{lex}} c_{\ell+1}'$.

\begin{dfn}
\label{dfn:wall_chop}
We call the marked weighted polygon $(\Delta',\vec{c'}=(c_1', \ldots, c_{\ell}', c_b', c_{\ell+1}', \ldots, c_s'),\vec{\epsilon'} = (\epsilon_1, \ldots, \epsilon_{\ell},1,\epsilon_{\ell+1}, \ldots, \epsilon_s))$ a (left) \emph{\blpoly} of size $\lambda$ of $(\Delta,\vec{c},\vec{\epsilon})$. 
\end{dfn}

We illustrate this operation in Figure \ref{fig:wall_chop}. Roughly, it amounts to removing from it a triangle with one side on the vertical wall, effectively producing the downwards bend of the top boundary of the polygon discussed earlier. 

\begin{rmk}
Since the operation of a semitoric type blowup includes reducing the size of the polygon, we may have to allow certain marked points to move vertically in order to keep them in the interior of the polygon. For this reason, we allow all marked points to move vertically, and afterwards it may be necessary to relabel them to keep them in lexicographic order.
\end{rmk}

\begin{rmk}
\label{rmk:left_right_blow}
If instead the vertical wall is 
\[ W = \left\{ (x,y) \in \Delta, \ x =  \max_{(u,y) \in \Delta} u \right\}, \]
we define the (right) wall chop of size $\lambda$ in a similar way, replacing $\tilde{t}_{\mu}$ with $t_{\mu}$ defined in Equation \eqref{eq:t_lamb_eq}. Alternatively, we can use the notion of $J$-reflection defined in the next section (Definition \ref{def:reflections_semitoric_poly}), and define the right wall chop of a marked semitoric polygon $(\Delta,\vec{c},\vec{\epsilon})$ as the $J$-reflection of the left wall chop of the $J$-reflection of $(\Delta,\vec{c},\vec{\epsilon})$. If there are two vertical walls, we may use either of them to perform a wall chop.
\end{rmk}

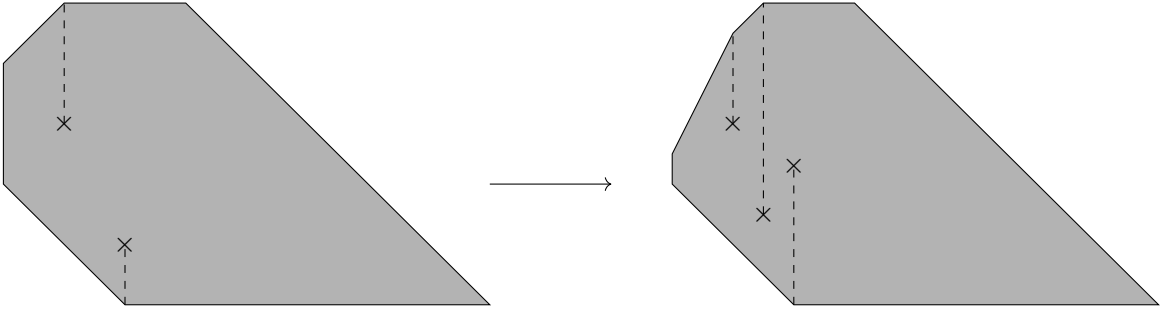
\begin{figure}
\begin{center}

\begin{tikzpicture}[scale=0.8]
\filldraw[draw=black, fill=gray!60] (0,2) 
  -- (0,4)
  -- (1,5)
  -- (3,5)
  -- (8,0) 
  -- (2,0)
  -- cycle;
\draw [dashed] (1,3) -- (1,5);
\draw (1,3) node[] {$\times$};

\draw [dashed] (2,0) -- (2,1);
\draw (2,1) node[] {$\times$};

\draw [->] (8,2) -- (10,2);

\begin{scope}[xshift=11cm]
\filldraw[draw=black, fill=gray!60] (0,2) 
  -- (0,2.5)
  -- (1,4.5)
  -- (1.5,5)
  -- (3,5)
  -- (8,0) 
  -- (2,0)
  -- cycle;
\draw [dashed] (1,3) -- (1,4.5);
\draw (1,3) node[] {$\times$};
\draw [dashed] (1.5,1.5) -- (1.5,5);
\draw (1.5,1.5) node[] {$\times$};

\draw [dashed] (2,0) -- (2,2.3);
\draw (2,2.3) node[] {$\times$};

\end{scope}
\end{tikzpicture}

\caption{Performing a (left) \blpoly on a marked weighted polygon; the resulting polygon has one additional marked point and a smaller vertical wall, and the height of each of the original marked points is allowed to change.}
\label{fig:wall_chop}
\end{center}
\end{figure}

Now we can define another operation that will correspond to a semitoric type blowdown. Let $(\Delta,\vec{c} = (c_1, \ldots, c_s),\vec{\epsilon} = (\epsilon_1, \ldots, \epsilon_s))$ be a marked weighted polygon having a vertical wall (here $s \geq 1$, since during this operation we will remove one marked point). As before, let $W$ be this vertical wall, and assume that it corresponds to the leftmost part of $\Delta$, which means it is as in Equation \eqref{eq:left_wall}. Let $b \in \{1, \ldots, s\}$.

Now there will be cases depending on the cut direction at the point $c_b$: let $\Delta'$ be the unique convex polygon such that: 
\[\left\{\begin{aligned} &\partial^- \Delta' = \partial^- \Delta,  && \partial^+ \Delta' = (\tilde{t}_{\pi_1(c_b)})^{-1}(\partial^+ \Delta), && \text{ if }\epsilon_b=1,\\[2mm]
&\partial^- \Delta' = \tilde{t}_{\pi_1(c_b)}(\partial^- \Delta),  && \partial^+ \Delta' = \partial^+ \Delta, && \text{ if }\epsilon_b=-1,\end{aligned}\right.\]
with $\tilde{t}_{\mu}$ as in Equation \eqref{eq:t_tilde_mu_eq}. Let $c_1' \leq_{\text{lex}} \ldots \leq_{\text{lex}} c_{b-1}' \leq_{\text{lex}} c_{b+1}' \leq_{\text{lex}} \ldots \leq_{\text{lex}} c_s'$ be such that
\[ \forall \ell \in \{1, \ldots, s\} \setminus \{b\} \qquad c_{\ell}' \in \text{int}(\Delta') \cap (\pi_1)^{-1}(c_{\ell}). \]

\begin{dfn}
\label{dfn:wall_unchop}
We call the marked weighted polygon $(\Delta',\vec{c'}=(c_1', \ldots, c_{b-1}', c_{b+1}', \ldots, c_s'),\vec{\epsilon'}= (\epsilon_1, \ldots, \epsilon_{b-1},\epsilon_{b+1}, \ldots, \epsilon_s))$ a (left) \emph{\bldpoly} at $c_b$ of size $\lambda = \pi_1(c_b) - \pi_1(W)$ of $(\Delta,\vec{c},\vec{\epsilon})$.
\end{dfn}

As in Remark \ref{rmk:left_right_blow}, one can define a wall unchop when the vertical wall constitutes the rightmost part of the polygon.

Note that when defining \blpolys and \bldpolys, we made several choices; indeed, the vertical coordinates of the original marked points were allowed to change. However, for \bldpolys there is one convenient choice, which will be explained when we apply it to marked Delzant semitoric polygons.

\begin{dfn}
\label{dfn:pref_unchop}
Let $(\Delta,\vec{c} = (c_1, \ldots, c_s),\vec{\epsilon} = (\epsilon_1, \ldots, \epsilon_s))$ be a marked weighted polygon having a left (respectively right) vertical wall, and let $b \in \{1, \ldots, s\}$. The \bldpoly of $(\Delta,\vec{c},\vec{\epsilon})$ at $c_b$ obtained by choosing $c_{\ell}' = c_{\ell}$ for every $\ell \in \{1,\ldots,s\} \setminus \{b\}$ in Definition \ref{dfn:wall_unchop} is called  left (respectively right) \emph{preferred} \bldpoly of $(\Delta,\vec{c},\vec{\epsilon})$ of size $\lambda$ at $c_b$.
\end{dfn}

This definition is illustrated in Figure \ref{fig:preferred_wall_unchop}.

\begin{figure}
\begin{center}

\begin{tikzpicture}[scale=0.8]

\filldraw[draw=black, fill=gray!60] (0,2) 
  -- (0,2.5)
  -- (1,4.5)
  -- (1.5,5)
  -- (3,5)
  -- (8,0) 
  -- (2,0)
  -- cycle;
\draw [dashed] (1,3) -- (1,4.5);
\draw (1,3) node[] {$\times$};
\draw [dashed] (1.5,1.5) -- (1.5,5);
\draw (1.5,1.5) node[] {$\times$};

\draw [dashed] (2,0) -- (2,2.3);
\draw (2,2.3) node[] {$\times$};

\draw [->] (8,2) -- (10,2);

\begin{scope}[xshift=11cm]

\filldraw[draw=black, fill=gray!60] (0,2) 
  -- (0,4)
  -- (1,5)
  -- (3,5)
  -- (8,0) 
  -- (2,0)
  -- cycle;
\draw [dashed] (1,3) -- (1,5);
\draw (1,3) node[] {$\times$};

\draw [dashed] (2,0) -- (2,2.3);
\draw (2,2.3) node[] {$\times$};

\end{scope}
\end{tikzpicture}

\caption{Performing the preferred \bldpoly of a marked weighted semitoric polygon (see Definition \ref{dfn:pref_unchop}); the heights of the remaining marked points is not allowed to change. Compare with Figure \ref{fig:wall_chop}.}
\label{fig:preferred_wall_unchop}
\end{center}
\end{figure}
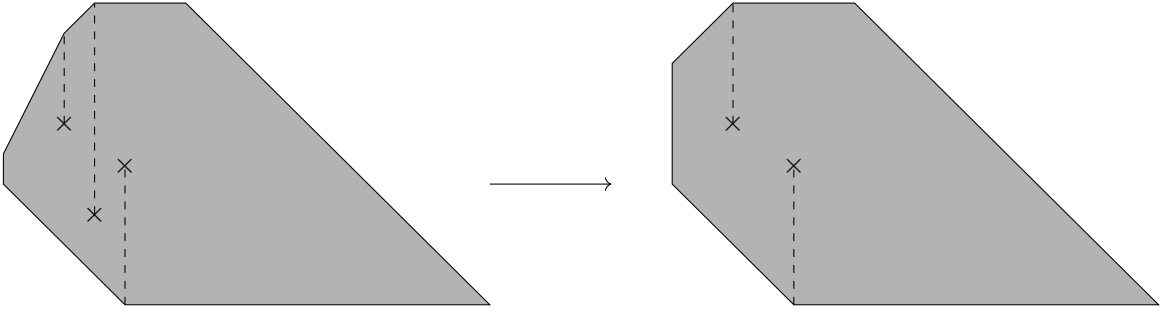

Note that if it is possible to perform a \bldpoly on a marked weighted polygon (i.e. if this polygon has a vertical wall and at least one marked point), then it is also possible to do so on the resulting marked weighted polygon, unless it has no marked points. In what follows (especially in Section \ref{sec:strictlymin}) we will want to make all possible consecutive \bldpolys on a given marked weighted polygon. It turns out that this does not depend on any choice.

\begin{dfn-prop}
\label{dfn-prop:complete_unchop}
Let $(\Delta,\vec{c} = (c_1, \ldots, c_s),\vec{\epsilon} = (\epsilon_1, \ldots, \epsilon_s))$ be a marked weighted polygon with a left vertical wall. We define the \emph{complete left \bldpoly} of $(\Delta,\vec{c},\vec{\epsilon})$ as the marked weighted polygon $(\Delta_0, \emptyset, \emptyset)$ obtained by performing $s$ consecutive left \bldpolys on $(\Delta,\vec{c},\vec{\epsilon})$. This does not depend on the choice of the order of the \bldpolys, nor on the choice of the values $c_{\ell}'$ at each stage.
\end{dfn-prop}

This definition is illustrated in Figure \ref{fig:complete_unchop}. Again, one can similarly define the \emph{complete right \bldpoly} if the vertical wall lies on the rightmost boundary of the polygon.

\begin{proof}[Proof of Definition-Proposition~\ref{dfn-prop:complete_unchop}]
Observe that for $\mu, \mu' \in \R$, the maps $\tilde{t}_{\mu}$ and $\tilde{t}_{\mu'}$ defined in Equation \eqref{eq:t_tilde_mu_eq} commute. Indeed, one readily checks that if $\mu' > \mu$, 
\[ \left(\tilde{t}_{\mu} \circ \tilde{t}_{\mu'} \right)(x,y) = \left(\tilde{t}_{\mu'} \circ \tilde{t}_{\mu} \right)(x,y) = \begin{cases} (x,y) & \text{ if } x \geq \mu',\\
(x, y + x - \mu') & \text{ if } \mu \leq x \leq \mu', \\
(x, y + 2 x - \mu - \mu') & \text{ if } x \leq \mu . \end{cases}\]
It follows that $\Delta_0$ is uniquely defined. In particular, if $\vec{\epsilon}=(1, \ldots, 1)$ then $\Delta_0$ is the unique convex polygon such that 
\[ \partial^- \Delta_0 = \partial^- \Delta, \quad \partial^+ \Delta_0 = \left( (\tilde{t}_{\pi_1(c_1)})^{-1} \circ \ldots \circ (\tilde{t}_{\pi_1(c_s)})^{-1} \right)(\partial^+ \Delta). \qedhere \]
\end{proof}

\begin{rmk}
If there are two vertical walls, we can perform either the complete left \bldpoly or complete right \bldpoly of the polygon. Another possibility would be to choose the wall we want to use for each \bldpoly. This is perfectly acceptable but we do not consider this here because it would lead to be multiple possible final polygons, depending on the choice of which \bldpolys are performed using the left wall and which ones are performed using the right wall.
\end{rmk}

\begin{figure}
\begin{center}

\begin{tikzpicture}[scale=0.8]

\filldraw[draw=black, fill=gray!60] (0,2) 
  -- (0,2.5)
  -- (1,4.5)
  -- (1.5,5)
  -- (3,5)
  -- (8,0) 
  -- (2,0)
  -- cycle;
\draw [dashed] (1,3) -- (1,4.5);
\draw (1,3) node[] {$\times$};
\draw [dashed] (1.5,1.5) -- (1.5,5);
\draw (1.5,1.5) node[] {$\times$};

\draw [dashed] (2,0) -- (2,2.3);
\draw (2,2.3) node[] {$\times$};

\draw [->] (8,2) -- (10,2);

\begin{scope}[xshift=11cm]

\filldraw[draw=black, fill=gray!60] (0,0) 
  -- (0,5)
  -- (3,5)
  -- (8,0) 
  -- cycle;

\end{scope}
\end{tikzpicture}

\caption{Performing the complete \bldpoly of a marked weighted polygon (see Definition \ref{dfn-prop:complete_unchop}).}
\label{fig:complete_unchop}
\end{center}
\end{figure}
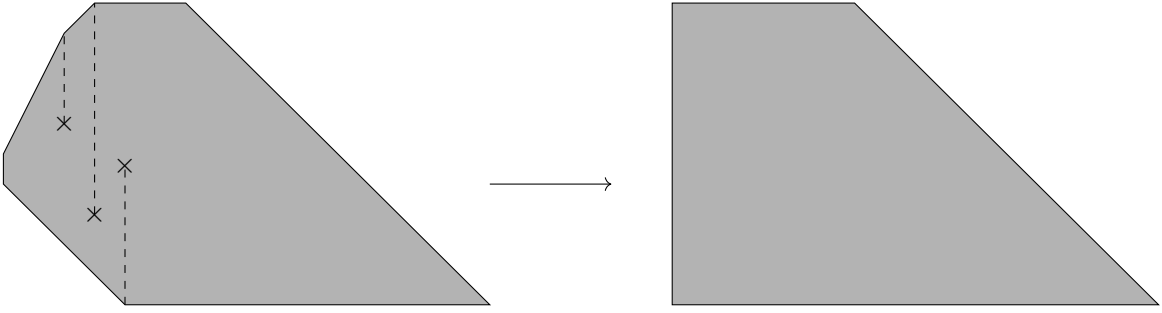

So far we have defined these operations on marked weighted polygons, and we need to argue that they induce operations on marked semitoric polygons.

\begin{lm}[{\cite[Remark 2.31]{HohPal21}}]
\label{lem:st-blowup-polygon}
If a marked weighted polygon $(\Delta,\vec{c},\vec{\epsilon})$ satisfies the conditions in Definition \ref{def:marked_poly}, then any \blpoly or \bldpoly of $(\Delta,\vec{c},\vec{\epsilon})$ also satisfies these conditions.
\end{lm}

The proof of Lemma~\ref{lem:st-blowup-polygon} proceeds by checking that the semitoric type blowup still satisfies the appropriate Delzant, $k$-hidden, and $k$-fake corner conditions.

Since for $\lambda, \mu \in \R$, $T$, $t_{\lambda}$ and $\tilde{t}_{\mu}$ pairwise commute, it can easily be checked that wall chops and wall unchops send an orbit for the $G_s \times \mathcal{T}$-action given in Equations \eqref{eq:T_action} and \eqref{eq:Gs_action} to another orbit for this action.
Combining this observation with Lemma~\ref{lem:st-blowup-polygon} we conclude that the operations of wall chop and wall unchop from Definitions \ref{dfn:wall_chop} and \ref{dfn:wall_unchop} lift to operations on marked semitoric polygons, which we will refer to by the same names (with their qualificatives left and right, and similarly for preferred and complete wall unchops from Definitions \ref{dfn:pref_unchop} and \ref{dfn-prop:complete_unchop}). 
By singling out the orbit of a marked point $c_b$, we can talk about the wall unchop at $[c_b]$.

\begin{rmk}
 There are two more symplectic invariants of semitoric systems, that we did not discuss since they are not necessary for the present paper (see for instance \cite{PVNinventiones,PPT}): the Taylor series invariant and the twisting index invariant.
 Roughly speaking, in a simple semitoric system, each focus-focus point in the system is labeled by a Taylor series and twisting index. Therefore, since a \bldpoly removes a focus-focus point, there is a preferred choice of these invariants for the resulting system, namely leaving them unchanged.
 On the other hand, since a \blpoly adds a new marked point to the polygon (and thus a new focus-focus point to the system) the Taylor series and twisting index of this new point are not determined from the original system. 
 In the non-simple case, the situation is similar but the labels are more complicated, and the labels of several focus-focus points in the same fiber are intertwined, so if a \bldpoly is performed on a focus-focus point in a multipinched fiber, defining preferred invariants of the resulting system is more delicate.
\end{rmk}

In view of the above discussion, we define semitoric type blowups and blowdowns as follows.

\begin{dfn}
\label{def:semitoricblowupdown}
We say that a semitoric system $(M',\omega',F'=(J',H'))$ is:
\begin{itemize}
    \item a (left/right) \emph{semitoric type blowup} of size $\lambda$ of the semitoric system $(M,\omega,F=(J,H))$ if its marked semitoric polygon $\Delta_{(M',\omega',F')}$ is a (left/right) wall chop of size $\lambda$ of the marked semitoric polygon $\Delta_{(M,\omega,F)}$;
    \item a (left/right) \emph{semitoric type blowdown} of $(M,\omega,F)$ at the focus-focus point $m$ if its marked semitoric polygon $\Delta_{(M',\omega',F')}$ is a (left/right) wall unchop of $\Delta_{(M,\omega,F)}$ at the orbit $[c_b]$ of marked points corresponding to $m$;
    \item a (left/right) \emph{complete (semitoric type) blowdown} of $(M,\omega,F)$ if its marked semitoric polygon $\Delta_{(M',\omega',F')}$ is the (left/right) complete wall unchop of $\Delta_{(M,\omega,F)}$. 
\end{itemize}
\end{dfn}

The next result is about reading when a semitoric type blowdown is possible from the invariants.

\begin{lm}\label{lem:st-blowdown}
 Suppose that $(M,\om,(J,H))$ is a semitoric system with at least one focus-focus point $p$ and such that the $S^1$-action generated by $J$ has a fixed surface $\Sigma$. Then $(M,\om,(J,H))$ admits a semitoric type blowdown at $p$.
\end{lm}

\begin{proof}
 Let $(\De,\vec{c},\vec{\epsilon})$ be a representative of the marked semitoric polygon of the system, and let $b$ be such that $c_b$ is the marked point corresponding to $p$. Since $(M,\omega,(J,H))$ has a fixed surface, $\Delta$ has a vertical wall. So we can perform a wall unchop to $(\De,\vec{c},\vec{\epsilon})$ at $c_b$ (see Definition \ref{dfn:pref_unchop}) and obtain a marked weighted polygon $(\De',\vec{c'}, \vec{\epsilon'})$ satisfying the conditions of Definition \ref{def:marked_poly}. By~\cite{PVNacta,PPT}, there exists a semitoric system with the $G_{m_f} \times \mathcal{T}$-orbit of this marked weighted polygon as its marked semitoric polygon, and this system is therefore a semitoric type blowdown of $(M,\om,(J,H))$.
 \end{proof}

Both toric and semitoric type blowups and blowdowns on semitoric systems can also be performed on and read off of semitoric helices. In Section \ref{sec:strictlymin}, we will investigate the relationship between semitoric helices and semitoric type blowups and blowdowns, and now we recall some relations between toric type blowups and blowdowns and helices and some applications.

\subsection{Semitoric helices}
\label{sec:helix}

In~\cite{KPP_min}, Kane-Palmer-Pelayo associate a combinatorial object called a semitoric helix to any compact simple semitoric system, which is the analogue of the fan of a toric variety, in order to study toric type blowups of semitoric systems. As we will explain later (see Section~\ref{sec:strict-min-helix}), these helices are also useful to understand semitoric type blowups.

In fact, the idea behind the construction of the helix is quite similar to the one of a toric fan. In~\cite{KPP_min} they construct the helix only for simple semitoric systems, but in fact it can be constructed in exactly the same way for possibly non-simple systems, which is what we do here.

 We define an equivalence relation on the set of sequences of elements of $\Z^2$ as follows. If $v = (v_n)_{n \in \Z}$ and $w = (w_n)_{n \in \Z}$ are two such sequences, they are equivalent if and only if there exists $k, \ell \in \Z$ such that 
\begin{equation} \forall  n \in \Z \qquad v_n =  (T^*)^k w_{n+\ell}, \label{eq:equi_helix} \end{equation}
where $T^*$ is the transpose of the matrix $T$ given in Equation~\eqref{eqn:T}.

\begin{dfn}
\label{def:helix}
A \emph{semitoric helix} is a triple $(d,s,[v])$ where $d \in \Z_{>0}$ and $s \in \Z_{\geq 0}$ are called the \emph{length} and \emph{complexity} of the helix, and $[v] = [(v_n)_{n \in \Z}]$ is an equivalence class of sequences of elements of $\Z^2$ for the above relation, such that 
\begin{enumerate}
\item $\forall n \in \Z$, $\det(v_n,v_{n+1}) = 1$;
\item the vectors $v_0, \ldots, v_{d-1}$ are arranged in counter-clockwise order;
\item $\forall n \in \Z$, $v_{n+d} = (T^*)^s v_n$. 
\end{enumerate} 
\end{dfn}

For practical purposes, a helix is represented by the vectors $v_0, \ldots, v_{d-1}$ of one representative of $[v]$ together with the value of $s$ (the value of $d$ being implicitly given from the number of vectors), see for instance Figure \ref{fig:example_helix}. The third property above implies that all other elements $v_n$ can be recovered from those.

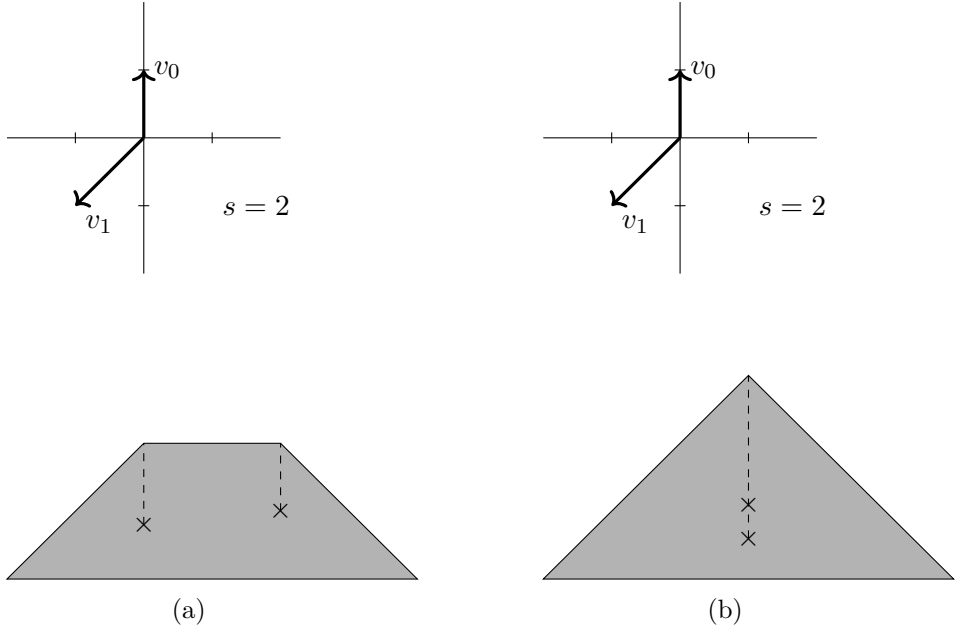
\begin{figure}
\begin{center} 
\def\dashlength{0.08}

\begin{subfigure}[b]{.3\linewidth} 
\centering 
\begin{tikzpicture}[scale=.9]
\draw (0,-2)--(0,2); 
\draw (-2,0)--(2,0); 
\draw [very thick,->] (0,0)--(0,1); 
\draw [very thick,->] (0,0)--(-1,-1); 
\draw (0,1) node [right] {$v_0$}; 
\draw (-1,-1) node [below right] {$v_1$}; 
\draw (-1,-\dashlength) -- (-1,\dashlength);
\draw (-\dashlength,1) -- (\dashlength,1);
\draw (1,-\dashlength) -- (1,\dashlength);
\draw (-\dashlength,-1) -- (\dashlength,-1);
\draw (1,-1) node[right] {$s = 2$};

\begin{scope}[xshift=-2cm,yshift = -6.5cm] 
\filldraw[draw=black, fill=gray!60] (0,0) -- (2,2) -- (4,2) -- (6,0) -- cycle; 
\draw (2,0.8) node {$\times$}; 
\draw [dashed] (2,0.8) -- (2,2); 
\draw (4,1) node {$\times$}; 
\draw [dashed] (4,1) -- (4,2); 
\end{scope} 
\end{tikzpicture} 
\caption{} 
\end{subfigure}  
\hspace{2cm} 
\begin{subfigure}[b]{.3\linewidth} 
\centering 
\begin{tikzpicture}[scale=.9]
\draw (0,-2)--(0,2); 
\draw (-2,0)--(2,0); 
\draw [very thick,->] (0,0)--(0,1); 
\draw [very thick,->] (0,0)--(-1,-1); 
\draw (0,1) node [right] {$v_0$}; 
\draw (-1,-1) node [below right] {$v_1$}; 
\draw (-1,-\dashlength) -- (-1,\dashlength);
\draw (-\dashlength,1) -- (\dashlength,1);
\draw (1,-\dashlength) -- (1,\dashlength);
\draw (-\dashlength,-1) -- (\dashlength,-1);
\draw (1,-1) node[right] {$s = 2$};

\begin{scope}[xshift=-2cm,yshift = -6.5cm] 
\filldraw[draw=black, fill=gray!60] (0,0) -- (3,3) -- (6,0) -- cycle; 
\draw (3,0.6) node {$\times$}; 
\draw [dashed] (3,0.6) -- (3,1.5); 
\draw (3,1.1) node {$\times$}; 
\draw [dashed] (3,1.5) -- (3,3); 
\end{scope} 
\end{tikzpicture} 
\caption{} 
\end{subfigure}
\end{center}
\caption{Two semitoric polygons giving the same helix. The system in (b) is not simple.}
\label{fig:example_helix}
\end{figure}

\paragraph{Constructing the helix from the polygon.}
The semitoric helix of a semitoric system can be obtained from the associated unmarked semitoric polygon by ``unwinding'' all of the cuts.
Now we will describe the method from~\cite[Section 5.3]{KPP_min}, slightly generalized to the case of a non-simple system.

Let $[(\Delta,\vec{\lambda},\vec{\epsilon})]$ be an unmarked Delzant semitoric polygon. The following algorithm, combined with Lemma \ref{lm:helix_GT}, produces its helix.

\begin{algo}
\label{algo:helix_from_poly}

Let $w_0,\ldots,w_{p-1}\in\Z^2$ denote the set of primitive inwards pointing normal vectors of $\De$,
ordered in counter-clockwise order, and for convenience denote $w_p := w_0$.
Each consecutive pair of vectors $(w_j, w_{j+1})$ for $j\in\{0,\ldots, p-1\}$ corresponds to a vertex of the polygon, which is either Delzant, $k$-hidden Delzant, or $k$-fake; we can (and will) choose the vectors so that $(w_{p-1},w_0)$ is Delzant.

\begin{enumerate}
\item For each $j$ such that $(w_j,w_{j+1})$ is either $k$-hidden Delzant or $k$-fake replace $w_{j+1},\ldots,w_{p-1}$ by
\[
 (T^*)^k w_{j+1},\ldots,(T^*)^k w_{p-1}.
\]
\item After dealing with each such $j$, we obtain a new set of vectors $w_0',\ldots,w_{p-1}'$.
\item By construction we now have that each pair $(w'_j,w'_{j+1})$ either satisfies $\det (w'_j,w'_{j+1})=1$ (if $(w_j, w_{j+1})$ was Delzant or hidden Delzant) or $w'_j = w'_{j+1}$ (if $(w_j, w_{j+1})$ was fake). If $w'_j=w'_{j+1}$ then remove $w'_{j+1}$ from the list.
\item After doing so for each such $j$, we denote the remaining vectors by $v_0,\ldots,v_{d-1}$.
\item Finally, we now extend this to form the helix. For $a \in \{0,\ldots, d-1\}$ and $\ell \in\Z\setminus\{0\}$, define
\[ v_{\ell d+a} := (T^*)^{s \ell} v_a. \]
This defines $v_n$ for all $n\in\Z$, and we set $[v] = [(v_n)_{n\in \Z}]$. The vectors $v_0, \ldots, v_{d-1}$ depend on the initial choice of $w_0, \ldots, w_{d-1}$ but the equivalence class $[v]$ does not. The triple $(d, s, [v])$ is the helix of $(\Delta,\vec{\lambda},\vec{\epsilon})$ and, in view of Lemma \ref{lm:helix_GT}, defines the helix of $[(\Delta,\vec{\lambda},\vec{\epsilon})]$. 
\end{enumerate}

\end{algo}

\begin{lm}
\label{lm:helix_GT}
Let $(\Delta,\vec{\lambda},\vec{\epsilon})$ and $\left(\tilde{\Delta},\vec{\tilde{\lambda}},\vec{\tilde{\epsilon}}\right)$ be two admissible unmarked weighted polygons in the same $G_s \times \mathcal{T}$-orbit. Then their helices coincide.
\end{lm}

This lemma is morally similar to the proof of \cite[Lemma 5.2]{KPP_min}, but there are two caveats. First, in the present paper we allow for non-simple systems, which was not the case in \cite{KPP_min}. Second, the point of view of \cite[Lemma 5.2]{KPP_min} is not exactly the same as here, since it deals with the construction of the helix from a semitoric system directly, and not from its unmarked semitoric polygon, see Remark \ref{rmk:helix_from_system}. Hence for the sake of completeness, we give a proof of Lemma \ref{lm:helix_GT}.

\begin{proof}[Proof of Lemma~\ref{lm:helix_GT}]
Let $(d, s, [v])$ be the helix of $(\Delta,\vec{\lambda},\vec{\epsilon})$. Let $w_0, \ldots, w_{p-1}$ be counter-clockwise ordered primitive inwards pointing normal vectors of $\Delta$ as in the initialization of Algorithm \ref{algo:helix_from_poly}, with $(w_{p-1},w_0)$ Delzant. In what follows $(\tilde{w}_0, \ldots, \tilde{w}_{\tilde{p}-1})$ will be a choice of such primitive inwards pointing normal vectors for $\tilde{\Delta}$ (note that $\tilde{p}$ might differ from $p$), $(w_0', \ldots, w'_{p-1})$ (respectively $(\tilde{w}'_0, \ldots, \tilde{w}'_{\tilde{p}-1})$) will be the corresponding outcome of Step 2 of Algorithm \ref{algo:helix_from_poly}, and $(v_0, \ldots, v_{d-1})$ (respectively $(\tilde{v}_0, \ldots, \tilde{v}_{\tilde{d}-1})$) will be the elements of the helix obtained in Step 4 of this algorithm.

Assume that $\left(\tilde{\Delta},\vec{\tilde{\lambda}},\vec{\tilde{\epsilon}}\right) = (\tau(\Delta),\tau(\vec{\lambda}),\vec{\epsilon})$ with $\tau = \left(T^k,\begin{pmatrix} 0 \\ y\end{pmatrix} \in \mathcal{T}\right)$. Then we can choose $\tilde{w}_0 = (T^*)^{-k} w_0, \ldots, \tilde{w}_{p-1} = (T^*)^{-k} w_{p-1}$. The resulting helix obtained from Algorithm \ref{algo:helix_from_poly} is $(d, s, [\tilde{v}])$ with $\tilde{v}_j = (T^*)^{-k} v_j$ for every $j \in \{ 0, \ldots, d-1 \}$, which means that $[\tilde{v}] = [v]$.

So it only remains to understand what happens when acting by an element of $G_s$. In order to do so, it is sufficient to restrict our attention to an element $\vec{\epsilon'} \in G_s$ with $\epsilon'_{\ell} = 1$ if $\ell \neq \ell_0$ and $\epsilon'_{\ell_0} = -1$, and assume that 
\[ \left(\tilde{\Delta},\vec{\tilde{\lambda}},\vec{\tilde{\epsilon}}\right) = \vec{\epsilon'} \cdot (\Delta,\vec{\lambda},\vec{\epsilon}) = \left(t_{\lambda_{\ell_0}}(\Delta),\vec{\lambda},\vec{\epsilon'} * \vec{\epsilon}\right) .\]
Moreover, we may assume without loss of generality that $\epsilon_{\ell_0} = 1$. So $P = (\lambda_{\ell_0},\max\{ y, (\lambda_{\ell_0},y) \in \Delta \})$ is a vertex $(w_t,w_{t + 1})$ of $\Delta$ which is $k$-fake or $k$-hidden Delzant for some $k \geq 1$. Now there are two cases.
\begin{enumerate}
    \item If $Q = (\lambda_{\ell_0},\min\{y,  (\lambda_{\ell_0},y) \in \Delta \})$ is a vertex $(w_b,w_{b + 1})$ of $\Delta$, then it is necessarily Delzant since $\epsilon_{\ell_0} = 1$. There are two subcases.
    \begin{enumerate}
        \item If $p$ is not $1$-fake, we can choose
    \[ (\tilde{w}_0, \ldots, \tilde{w}_{p-1}) = \left( w_0, \ldots, w_b, (T^*)^{-1} w_{b+1}, \ldots, (T^*)^{-1} w_t, w_{t+1}, \ldots, w_{p-1} \right), \]
    $(\tilde{w}_b, \tilde{w}_{b+1})$ is $1$-hidden, $(\tilde{w}_t, \tilde{w}_{t+1})$ is $(k-1)$-hidden or fake, and for every $j \notin \{b,t\}$, $(\tilde{w}_j,\tilde{w}_{j+1})$ is of the same type as $(w_j,w_{j+1})$. So we obtain that $(\tilde{w}'_0, \ldots, \tilde{w}'_{p-1}) = (w'_0, \ldots, w'_{p-1})$ and therefore $(\tilde{v}_0, \ldots, \tilde{v}_{d-1}) = (v_0, \ldots, v_{d-1})$. 
        \item If $p$ is $1$-fake, we can choose
    \[ (\tilde{w}_0, \ldots, \tilde{w}_{p-2}) = \left( w_0, \ldots, w_b, (T^*)^{-1} w_{b+1}, \ldots, (T^*)^{-1} w_t = w_{t+1}, w_{t+2}, \ldots, w_{p-1} \right), \]
    $(\tilde{w}_b, \tilde{w}_{b+1})$ is $1$-hidden, and for every $j \neq b$, $(\tilde{w}_j,\tilde{w}_{j+1})$ is of the same type as the corresponding vertex in $\Delta$. In this case after Step 2 of Algorithm \ref{algo:helix_from_poly} we will not obtain the same set of vectors for $\Delta$ and $\Delta'$, but we will obtain again $(\tilde{v}_0, \ldots, \tilde{v}_{d-1}) = (v_0, \ldots, v_{d-1})$ after Step 4 because for $\Delta$ we remove $w'_{t+1}$ since the vertex $P = (w_t,w_{t+1})$ is fake, and $P$ is not a vertex in $\tilde{\Delta}$.        
    \end{enumerate}
    \item Otherwise, $Q$ lies on an edge of $\Delta$ whose primitive inwards pointing normal vector is $w_b$ for some $b \in \{0, \ldots, p-1\}$, and there are again two subcases. 
    \begin{enumerate}
        \item If $P$ is $1$-fake, we can choose
    \[ (\tilde{w}_0, \ldots, \tilde{w}_{p-1}) = \left( w_0, \ldots, w_b, (T^*)^{-1} w_{b}, \ldots, (T^*)^{-1} w_t = w_{t+1}, w_{t+2}, \ldots, w_{p-1} \right), \]
    the vertex $(w_b, (T^*)^{-1} w_{b})$ is $1$-fake and for every $j \neq b$, $(\tilde{w}_j,\tilde{w}_{j+1})$ is of the same type as the corresponding vertex in $\Delta$. Then
    \[ (\tilde{w}'_0, \ldots, \tilde{w}'_{p-1}) = (w'_0, \ldots, w'_b, w'_b, \ldots, w'_t, w'_{t+2}, \ldots, w'_{p-1}) \]
    and then again $(\tilde{v}_0, \ldots, \tilde{v}_{d-1}) = (v_0, \ldots, v_{d-1})$ because for $\tilde{\Delta}$ we will remove the second copy of $w_b'$ and for $\Delta$ we remove $w'_{t+1}$. 
        \item Otherwise, we can choose
    \[ (\tilde{w}_0, \ldots, \tilde{w}_{p}) = \left( w_0, \ldots, w_b, (T^*)^{-1} w_b, (T^*)^{-1} w_{b+1}, \ldots, (T^*)^{-1} w_t, w_{t+1}, \ldots, w_{p-1} \right), \]
    the vertex $(w_b, (T^*)^{-1} w_b)$ is $1$-fake, $((T^*)^{-1} w_t, w_{t+1})$ is $(k-1)$-hidden/fake, and for every $j \notin \{b,t\}$, $(\tilde{w}_j,\tilde{w}_{j+1})$ is of the same type as the corresponding vertex in $\Delta$. One readily checks that, again, we obtain $(\tilde{v}_0, \ldots, \tilde{v}_{d-1}) = (v_0, \ldots, v_{d-1})$. \qedhere
    \end{enumerate}
\end{enumerate}
\end{proof}

Now if $(M,\omega,F=(J,H))$ is a semitoric system, we define its helix to be the helix of its unmarked semitoric polygon. The map from the set of unmarked semitoric polygons to the set of helices is not injective (see for instance Figure \ref{fig:example_helix}); in fact, in general, there are many unmarked semitoric polygons with a given helix. Moreover, constructing the possible unmarked semitoric polygons from a single semitoric helix is a non-trivial problem, see for instance Section \ref{subsec:min_poly}, and in particular Theorem \ref{thm:min_poly}, where we do so in simple cases.

\begin{rmk} 
\label{rmk:helix_from_system}
Following~\cite{KPP_min}, we have now shown how to construct the helix of a semitoric system $(M,\om,F=(J,H))$ from its unmarked semitoric polygon, but in~\cite{KPP_min} the authors also show how to construct the helix directly from the semitoric system. The idea is to let $V$ be a neighbourhood of the boundary of $F(M)$, small enough so that it does not contain any focus-focus value and let $U$ be $V$ minus a small line segment, chosen so that $U$ is simply connected. Then there exists a toric momentum map $G: F^{-1}(U) \to \R^2$, so that $\partial G(F^{-1}(U))\cap G(F^{-1}(U))$ is a piecewise linear curve. The inwards pointing normal vectors of this curve are the vectors $v_0,\ldots,v_{d-1}$ which generate the helix.
\end{rmk}

\subsection{Minimal helices}
\label{sec:min_helices}

In this section, we investigate semitoric systems that are minimal with respect to toric type blowups and blowdowns, in the following sense.

\begin{dfn}
\label{def:min_sys}
A semitoric system is said to be \emph{minimal} if it does not admit any toric type blowdown.
\end{dfn}

Recall from \cite[Lemma 4.6]{LFPfamilies} that performing a toric type blowup on a semitoric system amounts to performing a corner chop on the corresponding marked semitoric polygon. In turn, such a corner chop has the effect of introducing a new primitive inwards pointing normal vector of one representative of the marked semitoric polygon, and this new vector is the sum of the primitive inwards pointing normal vectors to the two edges forming the original vertex. 

Therefore, the helix corresponding to a system includes a vector which is the sum of the adjacent two vectors if and only if the system admits a toric type blowdown. This naturally leads to the following definition.

\begin{dfn}
\label{dfn:min_hel_cond}
A semitoric helix $(d,s,[v])$ is said to be \emph{minimal} if 
\[ \forall j \in \Z \qquad v_j \neq v_{j-1} + v_{j+1}. \]
\end{dfn}

In other words, a semitoric helix is minimal if and only if all semitoric systems associated with this helix are minimal. Minimal semitoric helices were classified in \cite{KPP_min}. Unfortunately, the main result from \cite{KPP_min} missed some degrees of freedom. So in order to give a precise statement we will need to define two operations on helices. Recall that a helix $(d,s,[v])$ is completely determined by any set of $d$ consecutive vectors $v_{\ell}, \ldots, v_{\ell+d-1}$.

\begin{dfn}
Let $(d,s,[v])$ be a semitoric helix. We define
\begin{itemize}
    \item the \emph{$J$-reflection} of $(d,s,[v])$ is the helix $(d,s,[\tilde{v}])$ defined by 
    \[ \forall j \in \{ 0, \ldots, d-1 \}, \qquad \tilde{v}_j = \begin{pmatrix} -1 & 0 \\ 0 & 1 \end{pmatrix} v_{d-1-j}; \]
    \item the \emph{$H$-reflection} of $(d,s,[v])$ is the helix $(d,s,[\check{v}])$ defined by
    \[ \forall j \in \{ 0, \ldots, d-1 \}, \qquad \check{v}_j = \begin{pmatrix} 1 & 0 \\ 0 & -1 \end{pmatrix} v_{d-1-j}. \] 
\end{itemize}
\end{dfn}

These operations are illustrated in Figure \ref{fig:reflections_helices}.

\begin{figure}
\begin{center} 
\def\dashlength{0.08}

\begin{subfigure}[b]{.25\linewidth} 
\centering 
\begin{tikzpicture}[scale=.5]
\draw (0,-2.2)--(0,2.2); 
\draw (-3.2,0)--(5.2,0); 
\draw [very thick,->] (0,0)--(0,1); 
\draw [very thick,->] (0,0)--(-3,-1);
\draw [very thick,->] (0,0)--(1,0); 
\draw [very thick,->] (0,0)--(5,1); 
\draw (0,1) node [above left] {$v_0$}; 
\draw (-3,-1) node [below] {$v_1$}; 
\draw (1,0) node [below right] {$v_2$}; 
\draw (5,1) node [above left] {$v_3$}; 
\draw (1,-1.5) node[right] {$s = 4$};
\draw (-1,-\dashlength) -- (-1,\dashlength);
\draw (-\dashlength,1) -- (\dashlength,1);
\draw (1,-\dashlength) -- (1,\dashlength);
\draw (-\dashlength,-1) -- (\dashlength,-1);
\end{tikzpicture} 
\caption{$(d,s,[v])$.} 
\end{subfigure}  
\hspace{1cm} 
\begin{subfigure}[b]{.25\linewidth} 
\centering 
\begin{tikzpicture}[scale=.5]
\draw (0,-2.2)--(0,2.2); 
\draw (-5.2,0)--(3.2,0); 
\draw [very thick,->] (0,0)--(-5,1); 
\draw [very thick,->] (0,0)--(-1,0);
\draw [very thick,->] (0,0)--(3,-1); 
\draw [very thick,->] (0,0)--(0,1); 
\draw (-5,1) node [above right] {$\tilde{v}_0$}; 
\draw (-1,0) node [below] {$\tilde{v}_1$}; 
\draw (3,-1) node [below left] {$\tilde{v}_2$}; 
\draw (0,1) node [above right] {$\tilde{v}_3$}; 
\draw (-1,-1.5) node[left] {$s = 4$};
\draw (-1,-\dashlength) -- (-1,\dashlength);
\draw (-\dashlength,1) -- (\dashlength,1);
\draw (1,-\dashlength) -- (1,\dashlength);
\draw (-\dashlength,-1) -- (\dashlength,-1);
\end{tikzpicture} 
\caption{$(d,s,[\tilde{v}])$.} 
\end{subfigure} 
\hspace{1cm} 
\begin{subfigure}[b]{.25\linewidth} 
\centering 
\begin{tikzpicture}[scale=.5]
\draw (0,-2.2)--(0,2.2); 
\draw (-3.2,0)--(5.2,0); 
\draw [very thick,->] (0,0)--(5,-1); 
\draw [very thick,->] (0,0)--(1,0);
\draw [very thick,->] (0,0)--(-3,1); 
\draw [very thick,->] (0,0)--(0,-1); 
\draw (5,-1) node [below left] {$\check{v}_0$}; 
\draw (1,0) node [above] {$\check{v}_1$}; 
\draw (-3,1) node [above right] {$\check{v}_2$}; 
\draw (0,-1) node [below left] {$\check{v}_3$}; 
\draw (1,1.5) node[right] {$s = 4$};
\draw (-1,-\dashlength) -- (-1,\dashlength);
\draw (-\dashlength,1) -- (\dashlength,1);
\draw (1,-\dashlength) -- (1,\dashlength);
\draw (-\dashlength,-1) -- (\dashlength,-1);
\end{tikzpicture} 
\caption{$(d,s,[\check{v}])$.} 
\end{subfigure} 

\end{center} 
\caption{A semitoric helix $(d,s,[v])$ and its $J$-reflection $(d,s,[\tilde{v}])$ and $H$-reflection $(d,s,[\check{v}])$.} 
\label{fig:reflections_helices} 
\end{figure}
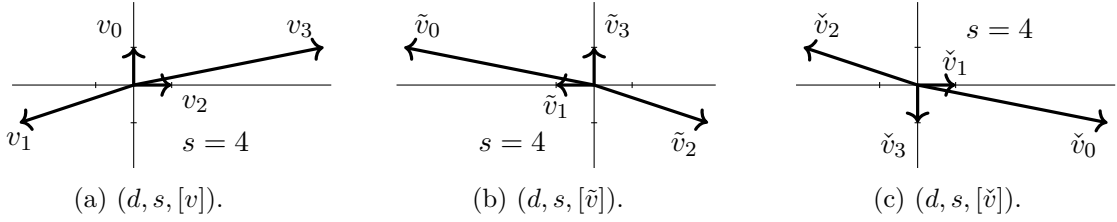

As we will see below, these operations on semitoric helices correspond to very simple operations on semitoric polygons, that we define now. Let $r_1, r_2: \R^2 \to \R^2$ be the reflections across the $x$-axis and the $y$-axis, respectively. For $\vec{u} = (u_1, \ldots, u_d)$ in $\R^{2d}$, we set $r{_\ell}(\vec{u}) = (r_{\ell}(u_1), \ldots, r_{\ell}(u_d))$ for $\ell=1,2$, and for any permutation $\sigma \in \mathfrak{S}_d$, we set $\sigma(\vec{u}) = (u_{\sigma(1)}, \ldots, u_{\sigma(d)})$.

\begin{dfn}
\label{def:reflections_semitoric_poly}
\begin{enumerate}
    \item Let $[(\Delta,\vec{\lambda} = (\lambda_1, \ldots, \lambda_s),\vec{\epsilon} = (\epsilon_1, \ldots, \epsilon_s))]$ be an unmarked Delzant semitoric polygon. We define
\begin{itemize}
    \item the \emph{$J$-reflection} of $[(\Delta,\vec{\lambda},\vec{\epsilon})]$ as the unmarked Delzant semitoric polygon \\ $[(r_2(\Delta),(\lambda_s, \ldots, \lambda_1),(\epsilon_s, \ldots, \epsilon_1))]$;
    \item the \emph{$H$-reflection} of $[(\Delta,\vec{\lambda},\vec{\epsilon})]$ as the unmarked Delzant semitoric polygon $[(r_1(\Delta),\vec{\lambda},-\vec{\epsilon})]$.
\end{itemize}
    \item Let $[(\Delta,\vec{c} = (c_1, \ldots, c_s),\vec{\epsilon} = (\epsilon_1, \ldots, \epsilon_s))]$ be a marked Delzant semitoric polygon. We define
\begin{itemize}
    \item the \emph{$J$-reflection} of $[(\Delta,\vec{c},\vec{\epsilon})]$ as the marked Delzant semitoric polygon \\
    $[(r_2(\Delta),\sigma(r_2(\vec{c})),\sigma(\vec{\epsilon}))]$ where $\sigma \in \mathfrak{S}_s$ is such that the components of $\sigma(r_2(\vec{c}))$ are ordered lexicographically;
    \item the \emph{$H$-reflection} of $[(\Delta,\vec{c},\vec{\epsilon})]$ as the marked Delzant semitoric polygon \\ 
    $[(r_1(\Delta),\sigma(r_1(\vec{c})),\sigma(-\vec{\epsilon}))]$ where $\sigma \in \mathfrak{S}_s$ is such that the components of $\sigma(r_1(\vec{c}))$ are ordered lexicographically.
\end{itemize}
\end{enumerate}
\end{dfn}

Examples of a $J$-reflection and a $H$-reflection for a marked semitoric polygon are displayed in Figure \ref{fig:reflections_poly}.

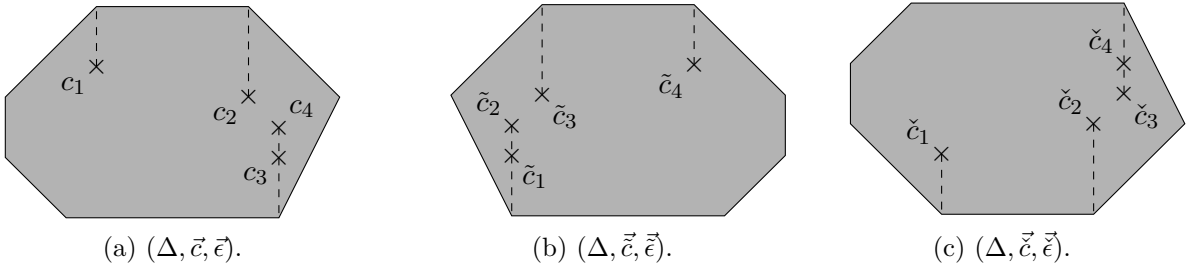
\begin{figure}
\begin{center} 
\def\dashlength{0.08}

\begin{subfigure}[b]{.3\linewidth} 
\centering 
\begin{tikzpicture}[scale=0.4]
\filldraw[draw=black, fill=gray!60] (0,2) -- (0,4) -- (3,7) -- (8,7) -- (11,4) -- (9,0) -- (2,0) -- cycle;

\draw [dashed] (3,5) -- (3,7);
\draw (3,5) node[] {$\times$};
\draw (3,5) node[below left] {$c_1$};

\draw [dashed] (8,4) -- (8,7);
\draw (8,4) node[] {$\times$};
\draw (8,4) node[below left] {$c_2$};

\draw [dashed] (9,2) -- (9,0);
\draw (9,2) node[] {$\times$};
\draw (9,2) node[below left] {$c_3$};

\draw [dashed] (9,3) -- (9,2);
\draw (9,3) node[] {$\times$};
\draw (9,3) node[above right] {$c_4$};

\end{tikzpicture}
\caption{$(\Delta,\vec{c},\vec{\epsilon})$.} 
\end{subfigure}  
\hspace{1cm} 
\begin{subfigure}[b]{.25\linewidth} 
\centering 
\begin{tikzpicture}[scale=0.4]
\begin{scope}[xscale=-1,yscale=1]
\filldraw[draw=black, fill=gray!60] (0,2) -- (0,4) -- (3,7) -- (8,7) -- (11,4) -- (9,0) -- (2,0) -- cycle;

\draw [dashed] (3,5) -- (3,7);
\draw (3,5) node[] {$\times$};
\draw (3,5) node[below left] {$\tilde{c}_4$};

\draw [dashed] (8,4) -- (8,7);
\draw (8,4) node[] {$\times$};
\draw (8,4) node[below right] {$\tilde{c}_3$};

\draw [dashed] (9,2) -- (9,0);
\draw (9,2) node[] {$\times$};
\draw (9,2) node[below right] {$\tilde{c}_1$};

\draw [dashed] (9,3) -- (9,2);
\draw (9,3) node[] {$\times$};
\draw (9,3) node[above left] {$\tilde{c}_2$};
\end{scope}
\end{tikzpicture}
\caption{$(\Delta,\vec{\tilde{c}},\vec{\tilde{\epsilon}})$.} 
\end{subfigure} 
\hspace{1cm} 
\begin{subfigure}[b]{.25\linewidth} 
\centering 
\begin{tikzpicture}[scale=0.4]
\begin{scope}[xscale=1,yscale=-1]
\filldraw[draw=black, fill=gray!60] (0,2) -- (0,4) -- (3,7) -- (8,7) -- (11,4) -- (9,0) -- (2,0) -- cycle;

\draw [dashed] (3,5) -- (3,7);
\draw (3,5) node[] {$\times$};
\draw (3,5) node[above left] {$\check{c}_1$};

\draw [dashed] (8,4) -- (8,7);
\draw (8,4) node[] {$\times$};
\draw (8,4) node[above left] {$\check{c}_2$};

\draw [dashed] (9,2) -- (9,0);
\draw (9,2) node[] {$\times$};
\draw (9,2) node[above left] {$\check{c}_4$};

\draw [dashed] (9,3) -- (9,2);
\draw (9,3) node[] {$\times$};
\draw (9,3) node[below right] {$\check{c}_3$};
\end{scope}
\end{tikzpicture}
\caption{$(\Delta,\vec{\check{c}},\vec{\check{\epsilon}})$.} 
\end{subfigure} 

\end{center} 
\caption{A representative of a marked semitoric polygon $[(\Delta,\vec{c},\vec{\epsilon})]$ and of its $J$-reflection $[(\Delta,\vec{\tilde{c}},\vec{\tilde{\epsilon}})]$ and $H$-reflection $[(\Delta,\vec{\check{c}},\vec{\check{\epsilon}})]$.} 
\label{fig:reflections_poly} 
\end{figure}

\begin{rmk}
The above definition is more involved in the case of a marked semitoric polygon because in Definition \ref{def:marked_poly}, we have assumed that the marked points are in lexicographic order. This requirement is important when dealing with non-simple semitoric systems, especially when incorporating the other invariants in the symplectic classification, see \cite{PPT}.    
\end{rmk}

\begin{prop}
\label{prop:reflections}
Let $[(\Delta,\vec{\lambda},\vec{\epsilon})]$ be an unmarked semitoric polygon and let $(d,s,[v])$ be its semitoric helix. Then the semitoric helix associated with the $J$-reflection (respectively $H$-reflection) of   $[(\Delta,\vec{\lambda},\vec{\epsilon})]$ is the $J$-reflection (respectively $H$-reflection) of $(d,s,[v])$.
\end{prop}

\begin{example}
The helix $(d,s,[v])$ in Figure \ref{fig:reflections_helices} is the helix associated with the semitoric polygon $[(\Delta,\vec{c},\vec{\epsilon})]$ of Figure \ref{fig:reflections_poly}. One readily checks that the helices of the $J$ and $H$-reflections of this polygon are indeed the $J$ and $H$-reflections of that helix.
\end{example}

\begin{proof}[Proof of Proposition \ref{prop:reflections}]
We only prove the first item, since the reasoning is similar for the second one. The idea is to construct simultaneously the helices of $[(\Delta,\vec{\lambda},\vec{\epsilon})]$ and of its $J$-reflection by applying Algorithm \ref{algo:helix_from_poly} to well-chosen initial normal vectors, and to compare the outcomes.

Let $A = \begin{pmatrix} -1 & 0 \\ 0 & 1 \end{pmatrix}$. Observe that $A^{-1} = A$ and that for any $k \in \Z$,
\begin{equation} A^{-1} (T^*)^k A = \begin{pmatrix} -1 & 0 \\ 0 & 1 \end{pmatrix} \begin{pmatrix} 1 & k \\ 0 & 1 \end{pmatrix} \begin{pmatrix} -1 & 0 \\ 0 & 1 \end{pmatrix} = \begin{pmatrix} 1 & -k \\ 0 & 1 \end{pmatrix} = (T^*)^{-k}. \label{eq:AT} \end{equation}

As in Algorithm \ref{algo:helix_from_poly}, let $w_0, \ldots, w_{p-1}$ be the primitive inwards pointing normal vectors of $\Delta$, ordered in counter-clockwise order and chosen such that the vertex $(w_{p-1},w_0)$ is Delzant. For $j \in \{0, \ldots, p-1\}$, let 
\begin{equation} \hat{w}_j = A w_{p-1-j}, \label{eq:what}\end{equation}
and observe that $\hat{w}_0, \ldots, \hat{w}_{p-1}$ are primitive inwards pointing normal vectors of $r_2(\Delta)$ in counter-clockwise order. Observe also that the vertex $(\hat{w}_{p-1}, \hat{w}_0)$ is Delzant since
\[ \det(\hat{w}_0, \hat{w}_{p-1}) = \det(A w_0, A w_{p-1}) = - \det(w_0,  w_{p-1}) = 1. \]
So we may, and will, take $\hat{w}_0, \ldots, \hat{w}_{p-1}$ as the initial vectors for $r_2(\Delta)$ in Algorithm \ref{algo:helix_from_poly}. 

Let $(v_0, \ldots, v_{d-1})$ and $(\hat{v}_0, \ldots, \hat{v}_{d-1})$ be the outcomes of Algorithm \ref{algo:helix_from_poly} for $\Delta$ and $r_2(\Delta)$ respectively, using the initial vectors $w_0,\ldots, w_{p-1}$ and $\hat{w}_0, \ldots, \hat{w}_{p-1}$ discussed above. We claim that 
\[ \forall \ell \in \{0, \ldots, d-1\}, \qquad \hat{v}_{\ell} = A (T^*)^{-s} v_{d-1-\ell} = (T^*)^s A v_{d-1-\ell} \]
(the last equality being implied from the first one and Equation \eqref{eq:AT}). This implies that the helices resulting from $(\hat{v}_0, \ldots, \hat{v}_{d-1})$ and $(A v_{d-1}, \ldots, A v_0)$ coincide (recall the equivalence relation in Definition \ref{def:helix} and the discussion before it, in particular Equation \eqref{eq:equi_helix}).  

To prove the claim, we first check that a vertex $(w_j,w_{j+1})$ is $k$-hidden Delzant (respectively $k$-fake) if and only if $(\hat{w}_{p-2-j},\hat{w}_{p-1-j})$ is $k$-hidden Delzant (respectively $k$-fake). Indeed,
\[ \begin{split} \det(\hat{w}_{p-2-j},(T^*)^k\hat{w}_{p-1-j}) & = \det(A w_{j+1}, (T^*)^k A w_j) \\
& = - \det(w_{j+1}, A^{-1} (T^*)^k A w_j) \\
& = - \det(w_{j+1}, (T^*)^{-k} w_j) \\
& = - \det((T^*)^k w_{j+1}, w_j) \\
& = \det(w_j, (T^*)^k w_{j+1}) \end{split} \]
(here we have used Equation \eqref{eq:AT} and the fact that $\det((T^*)^{-k}) = 1$). In particular the number of vectors obtained at the end of Algorithm \ref{algo:helix_from_poly} for $r_2(\Delta)$ is also $d$.

Therefore, in Steps 1 and 2 of Algorithm \ref{algo:helix_from_poly}, for each $j$, if $(w_j,w_{j+1})$ is $k$-hidden Delzant or $k$-fake, we replace $(w_{j+1}, \ldots, w_{p-1})$ with $((T^*)^k w_{j+1}, \ldots, (T^*)^k w_{p-1})$, and replace $(\hat{w}_{p-1-j}, \ldots, \hat{w}_{p-1})$ with 
    \[ ((T^*)^k \hat{w}_{p-1-j}, \ldots, (T^*)^k \hat{w}_{p-1}) = ((T^*)^k A w_j, \ldots, (T^*)^k A w_0) \]
    (here we have used Equation \eqref{eq:what}).
For Step 3 of Algorithm~\ref{algo:helix_from_poly}, it suffices to observe that if $(w_j,w_{j+1})$ is $k$-fake, then by Equations \eqref{eq:what} and \eqref{eq:AT}
\[ \hat{w}_{p-2-j} = A w_{j+1} = A (T^*)^{-k} w_j = (T^*)^k A w_j.  \]
Consequently, after removing any repeated vectors in the lists we obtain helix representatives of the same length $d$ which satisfy $\hat{v}_{\ell} = (T^*)^s A v_{d-1-\ell}$ for each $\ell \in \{0, \ldots, d-1\}$ and are therefore equivalent, as desired.
\end{proof}

Using that $J$ and $H$-reflections are involutions, Proposition \ref{prop:reflections} immediately yields the following corollary. Given a semitoric helix $(d,s,[v])$, let $\mathcal{P}(d,s,[v])$ be the set of unmarked semitoric polygons whose helix is $(d,s,[v])$.

\begin{cor}
\label{cor:reflections_preimage}
Let $(d,s,[v])$ be a semitoric helix and let $(d',s',[v'])$ be its $J$-reflection (respectively $H$-reflection). Then $\mathcal{P}(d',s',[v'])$ coincides with the set of $J$-reflections (respectively $H$-reflections) of elements of $\mathcal{P}(d,s,[v])$. 
\end{cor}

We now define several types of helices that were introduced in \cite{KPP_min} and will be important in the next result.

\begin{dfn}
\label{def:types_helices}  
A semitoric helix $(d,s,[v])$ is 
\begin{itemize}
    \item of type (1) if $d=2$, $s=1$, $v_0 = \begin{pmatrix} 0 \\ 1 \end{pmatrix}$ and $v_1 = \begin{pmatrix} -1 \\ -2 \end{pmatrix}$;
    \item of type (2) if $d=2$, $s=2$, $v_0 = \begin{pmatrix} 0 \\ 1 \end{pmatrix}$ and $v_1 = \begin{pmatrix} -1 \\ -1 \end{pmatrix}$;
    \item of type (3) if $d=3$, $s=1$, $v_0 = \begin{pmatrix} 0 \\ 1 \end{pmatrix}$, $v_1 = \begin{pmatrix} -1 \\ 1-n \end{pmatrix}$ and $v_2 = \begin{pmatrix} 0 \\ -1 \end{pmatrix}$ for some $n \geq 1$;
    \item of type (4) if $d=3$, $s \neq 2$, $v_0 = \begin{pmatrix} 1 \\ 0 \end{pmatrix}$, $v_1 = \begin{pmatrix} 0 \\ 1 \end{pmatrix}$ and $v_2 = \begin{pmatrix} -1 \\ -1 \end{pmatrix}$;
    \item of type (5) if $d=4$, $s \neq 1$, $v_0 = \begin{pmatrix} 1 \\ 0 \end{pmatrix}$, $v_1 = \begin{pmatrix} 0 \\ 1 \end{pmatrix}$, $v_2 = \begin{pmatrix} -1 \\ 1-n \end{pmatrix}$ and $v_3 = \begin{pmatrix} 0 \\ -1 \end{pmatrix}$ for some $n \geq 1$;
    \item of type (6) if $d=4$, $s \geq 1$, $v_0 = \begin{pmatrix} 1 \\ 0 \end{pmatrix}$, $v_1 = \begin{pmatrix} 0 \\ 1 \end{pmatrix}$, $v_2 = \begin{pmatrix} -1 \\ 0 \end{pmatrix}$ and $v_3 = \begin{pmatrix} 1-n \\ -1 \end{pmatrix}$ for some $n \in \Z$.
\end{itemize}
\end{dfn}

These six types are displayed in Figure \ref{fig:minimal_helices}.

\begin{figure}
\begin{center} 
\def\dashlength{0.08}

\begin{subfigure}[t]{.25\linewidth} 
\centering 
\begin{tikzpicture}[scale=.9]
\draw (0,-2.2)--(0,2.2); 
\draw (-2.2,0)--(2.2,0); 
\draw [very thick,->] (0,0)--(0,1); 
\draw [very thick,->] (0,0)--(-1,-2);
\draw (0,1) node [above right] {$v_0$}; 
\draw (-1,-2) node [left] {$v_1$}; 
\draw (0.5,-1) node[right] {$s = 1$};
\draw (-1,-\dashlength) -- (-1,\dashlength);
\draw (-\dashlength,1) -- (\dashlength,1);
\draw (1,-\dashlength) -- (1,\dashlength);
\draw (-\dashlength,-1) -- (\dashlength,-1);
\end{tikzpicture} 
\caption{Type (1).} 
\end{subfigure}  
\hspace{1cm} 
\begin{subfigure}[t]{.25\linewidth} 
\centering 
\begin{tikzpicture}[scale=.9]
\draw (0,-2.2)--(0,2.2); 
\draw (-2.2,0)--(2.2,0); 
\draw [very thick,->] (0,0)--(0,1); 
\draw [very thick,->] (0,0)--(-1,-1);
\draw (0,1) node [above right] {$v_0$}; 
\draw (-1,-1) node [left] {$v_1$}; 
\draw (0.5,-1) node[right] {$s = 2$};
\draw (-1,-\dashlength) -- (-1,\dashlength);
\draw (-\dashlength,1) -- (\dashlength,1);
\draw (1,-\dashlength) -- (1,\dashlength);
\draw (-\dashlength,-1) -- (\dashlength,-1);
\end{tikzpicture} 
\caption{Type (2).} 
\end{subfigure} 
\hspace{1cm} 
\begin{subfigure}[t]{.25\linewidth} 
\centering 
\begin{tikzpicture}[scale=.9]
\draw (0,-2.2)--(0,2.2); 
\draw (-2.2,0)--(2.2,0); 
\draw [very thick,->] (0,0)--(0,1); 
\draw [very thick,->] (0,0)--(-1,-1);
\draw [very thick,->] (0,0)--(0,-1);
\draw (0,1) node [right] {$v_0$}; 
\draw (-1,-1) node [above left] {$v_1$}; 
\draw (0,-1) node [below right] {$v_2$}; 
\draw (0.5,-0.5) node[right] {$s = 1$};
\draw (-1,-\dashlength) -- (-1,\dashlength);
\draw (-\dashlength,1) -- (\dashlength,1);
\draw (1,-\dashlength) -- (1,\dashlength);
\draw (-\dashlength,-1) -- (\dashlength,-1);

\end{tikzpicture} 
\caption{Type (3) (here with $n = 2$).} 
\end{subfigure} 
\hspace{1cm} 
\begin{subfigure}[t]{.25\linewidth} 
\centering 
\begin{tikzpicture}[scale=.9]
\draw (0,-2.2)--(0,2.2); 
\draw (-2.2,0)--(2.2,0); 
\draw [very thick,->] (0,0)--(1,0); 
\draw [very thick,->] (0,0)--(0,1);
\draw [very thick,->] (0,0)--(-1,-1);
\draw (1,0) node [above right] {$v_0$}; 
\draw (0,1) node [above left] {$v_1$}; 
\draw (-1,-1) node [above left] {$v_2$}; 
\draw (0.5,-1) node[right] {$s \neq 2$};
\draw (-1,-\dashlength) -- (-1,\dashlength);
\draw (-\dashlength,1) -- (\dashlength,1);
\draw (1,-\dashlength) -- (1,\dashlength);
\draw (-\dashlength,-1) -- (\dashlength,-1);

\end{tikzpicture} 
\caption{Type (4).} 
\end{subfigure}
\hspace{1cm} 
\begin{subfigure}[t]{.25\linewidth} 
\centering 
\begin{tikzpicture}[scale=.9]
\draw (0,-2.2)--(0,2.2); 
\draw (-2.2,0)--(2.2,0); 
\draw [very thick,->] (0,0)--(1,0); 
\draw [very thick,->] (0,0)--(0,1);
\draw [very thick,->] (0,0)--(-1,-2);
\draw [very thick,->] (0,0)--(0,-1);
\draw (1,0) node [above right] {$v_0$}; 
\draw (0,1) node [above left] {$v_1$}; 
\draw (-1,-2) node [left] {$v_2$};
\draw (0,-1) node [below right] {$v_3$};
\draw (0.5,-0.5) node[right] {$s \neq 1$};
\draw (-1,-\dashlength) -- (-1,\dashlength);
\draw (-\dashlength,1) -- (\dashlength,1);
\draw (1,-\dashlength) -- (1,\dashlength);
\draw (-\dashlength,-1) -- (\dashlength,-1);

\end{tikzpicture} 
\caption{Type (5) (here with $n = 3$).} 
\end{subfigure}
\hspace{1cm} 
\begin{subfigure}[t]{.25\linewidth} 
\centering 
\begin{tikzpicture}[scale=.9]
\draw (0,-2.2)--(0,2.2); 
\draw (-2.2,0)--(2.2,0); 
\draw [very thick,->] (0,0)--(1,0); 
\draw [very thick,->] (0,0)--(0,1);
\draw [very thick,->] (0,0)--(-1,0);
\draw [very thick,->] (0,0)--(-2,-1);
\draw (1,0) node [above right] {$v_0$}; 
\draw (0,1) node [above left] {$v_1$}; 
\draw (-1,0) node [above left] {$v_2$};
\draw (-2,-1) node [below right] {$v_3$};
\draw (0.5,-1) node[right] {$s \geq 1$};
\draw (-1,-\dashlength) -- (-1,\dashlength);
\draw (-\dashlength,1) -- (\dashlength,1);
\draw (1,-\dashlength) -- (1,\dashlength);
\draw (-\dashlength,-1) -- (\dashlength,-1);

\end{tikzpicture} 
\caption{Type (6) (here with $n = 3$).} 
\end{subfigure}

\end{center} 
\caption{The six types of helices from Definition \ref{def:types_helices}.} 
\label{fig:minimal_helices} 
\end{figure}
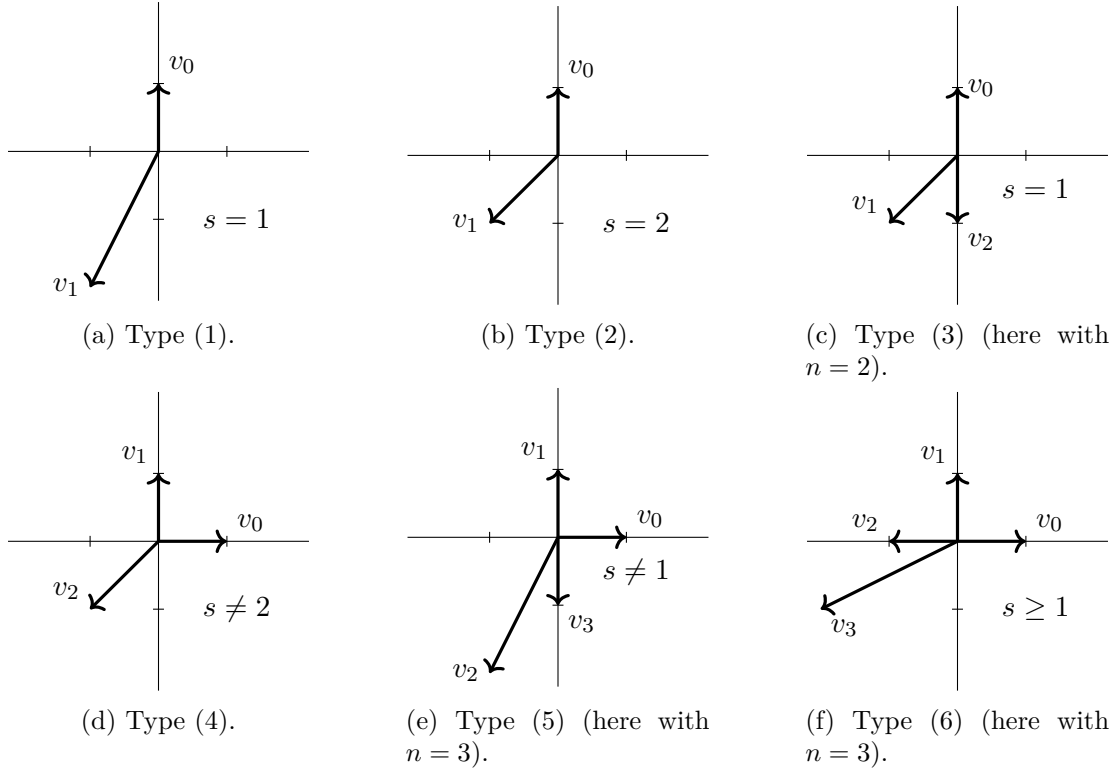

\begin{rmk}
There are three differences between Definition \ref{def:types_helices} and the way these six types are defined in \cite{KPP_min}. The first one is only notational: we performed the change of parameter $k = 1-n$, as we find more natural to work with $n$. The second is that in \cite{KPP_min}, the definition of the types (3), (5) and (6) includes constraints on the parameter $n$, which we did not keep here because we still want to say that the helices with the excluded $n$ (and later, the associated polygons) are of one of these types; instead we directly included these constraints in the statement of Theorem \ref{thm:min_helices}. The last difference is that here, in the definition of type (3) helices, we only need to consider $n \geq 1$ in view of the statement of Theorem \ref{thm:min_helices} since the helix 
\[ \left(3, 1, \left[ \begin{pmatrix} 0 \\ 1 \end{pmatrix}, \begin{pmatrix} -1 \\ 1-n \end{pmatrix}, \begin{pmatrix} 0 \\ -1 \end{pmatrix} \right] \right) \]
is clearly the $H$-reflection of the helix 
\[ \left(3, 1, \left[ \begin{pmatrix} 0 \\ 1 \end{pmatrix}, \begin{pmatrix} -1 \\ 1-(2-n) \end{pmatrix}, \begin{pmatrix} 0 \\ -1 \end{pmatrix} \right] \right). \]     
The same comment applies to type (5), although it is a bit more subtle since one has to recall the equivalence \eqref{eq:equi_helix} of helices. Indeed, the $H$-reflection of the helix  
\[ (d,s,[v]) = \left(4, s, \left[ \begin{pmatrix} 1 \\ 0 \end{pmatrix}, \begin{pmatrix} 0 \\ 1 \end{pmatrix}, \begin{pmatrix} -1 \\ 1-n \end{pmatrix}, \begin{pmatrix} 0 \\ -1 \end{pmatrix} \right] \right) \] 
is the helix
\[ (d',s',[v']) = \left(4, s, \left[ \begin{pmatrix} 0 \\ 1 \end{pmatrix}, \begin{pmatrix} -1 \\ 1 - (2-n) \end{pmatrix}, \begin{pmatrix} 0 \\ -1 \end{pmatrix}, \begin{pmatrix} 1 \\ 0 \end{pmatrix} \right] \right). \] 
Using the third item in Definition \ref{def:helix}, we see that 
\[ v'_{-1} = (T^*)^{-s} v'_3 = \begin{pmatrix} 1 & -s \\ 0 & 1 \end{pmatrix} \begin{pmatrix} 1 \\ 0 \end{pmatrix} = \begin{pmatrix} 1 \\ 0 \end{pmatrix}, \]
so by Equation \eqref{eq:equi_helix} this helix is the same as 
\[ \left(4, s, \left[ \begin{pmatrix} 1 \\ 0 \end{pmatrix}, \begin{pmatrix} 0 \\ 1 \end{pmatrix}, \begin{pmatrix} -1 \\ 1-(2-n) \end{pmatrix}, \begin{pmatrix} 0 \\ -1 \end{pmatrix} \right] \right). \] 
\end{rmk}

The following theorem is the main result from \cite{KPP_min}; however, as explained above, Theorem 4.15 in \cite{KPP_min} was missing some of the minimal helices. Moreover, in \cite{KPP_min} the authors also give an explicit description of minimal helices with $d\geq 6$, that they call of type (7). We chose to avoid the (tedious) description of helices of type (7) since we only need a certain property of such helices for the present paper, included in the following statement which combines the corrected version of \cite[Theorem 4.15]{KPP_min} with \cite[Corollary 4.16]{KPP_min}.

\begin{thm}[{\cite[Theorem 4.15, Corollary 4.16]{KPP_min}}]
\label{thm:min_helices}
Let $(M,\omega,F=(J,H))$ be a minimal semitoric system (see Definition \ref{def:min_sys}) and let $(d,s,[v])$ be its helix. Then $d \neq 5$ and:
\begin{itemize}
\item if $d \leq 4$, $(d,s,[v])$ is of one of the types (1) to (6) introduced in Definition \ref{def:types_helices}, or the $J$-reflection or $H$-reflection of one of these, with the additional restrictions:
\begin{itemize}
\item for type (3), $n\neq3$;
\item for type (5), $n \geq 3$; 
\item for type (6), $n \notin \{2,s\}$;
\end{itemize}
\item if $d \geq 6$, then there exists a representative $v = (v_n)_{n \in \Z}$ such that $v_2 = -v_0$ and $v_k = \pm \begin{pmatrix} 1 \\ 0 \end{pmatrix}$ for some $k \in \{3, \ldots, d-1\}$.
\end{itemize}
\end{thm}

Let $(M,\omega,F=(J,H))$ be a semitoric system and let $\Delta_{(M,\omega,F)}$ be its marked semitoric polygon. Then the marked semitoric polygon of the semitoric system $(M,\omega,(-J,H))$ (respectively $(M,\omega,(J,-H))$) is the $J$-reflection (respectively the $H$-reflection) of $\Delta_{(M,\omega,F)}$. Consequently, when trying to find explicit systems for every strictly minimal polygon, it suffices to come up with one system for each equivalence class of the action obtained by combining these symmetries $J \mapsto -J$ and $H \mapsto -H$. This is what we do in the examples from Sections \ref{sec:CP2} and \ref{sec:type_3}.

\section{Strictly minimal systems}
\label{sec:strictlymin}

In this section we prove that the helices of types (1), (2), and (3) are the only strictly minimal ones
and obtain the list of all possible marked semitoric polygons of systems with these helices. In particular, systems of types (4), (5), or (6) are not strictly minimal because a  semitoric blowdown can be performed on them. Continuing to perform blowdowns on those systems yields a system of toric type, and we investigate which toric type systems arise in this way from each of these types (4), (5), and (6). Moreover, we describe the relationship between two semitoric systems having the same marked semitoric polygon.

\subsection{Strictly minimal helices}
\label{sec:strict-min-helix}

Recall from Section \ref{sec:min_helices} (in particular Theorem \ref{thm:min_helices}) that in~\cite{KPP_min} the authors explore which semitoric systems do not admit a toric type blowdown
by examining the associated helix; they call such systems minimal. Now that we have introduced the operations of semitoric type blowups and blowdowns, there is an even stronger notion of minimality.

\begin{dfn}
\label{def:strictly_min_system}
We say that a semitoric integrable system is \emph{strictly minimal}
if it does not admit any blowdowns of either toric or semitoric type.
\end{dfn}

We will see now that the helix is also useful for determining if a system is strictly minimal.
Due to the construction of the helix in Section \ref{sec:helix}, the fact that a vertical wall of the semitoric polygon corresponds to a fixed surface of the underlying $S^1$-action, and the fact that
\[ T^* \begin{pmatrix}1\\0\end{pmatrix} = 
 \begin{pmatrix} 1 & 1 \\ 0 & 1\end{pmatrix} \begin{pmatrix}1\\0\end{pmatrix} = \begin{pmatrix}1\\0\end{pmatrix},
\]
we immediately have the following.

\begin{lm}\label{lem:horzvec-fixedsurf}
 Let $(M,\om,(J,H))$ be a semitoric system with associated helix $(d,s,[v])$. If $(v_n)_{n\in \Z}$ includes a horizontal vector then the $S^1$-action generated by $J$ on $M$ has a fixed surface.
\end{lm}

With this lemma, we can prove the following.

\begin{prop}
\label{prop:blowdown-helix}
Let $(M,\om,F)$ be a semitoric system with semitoric helix $(d, s, [(v_{\ell})_{\ell\in\Z}])$ where $s>0$. Then $(M,\om,F)$ admits a semitoric type blowdown (see Definition \ref{def:semitoricblowupdown}) if and only if there exists $h \in \{0,\ldots,d-1\}$ such that
 \[  v_h = \pm \begin{pmatrix} 1 \\ 0 \end{pmatrix} \]
 (i.e.~$v_h$ is horizontal). In this case, $(M,\om,F)$ admits a semitoric type blowdown $(M',\omega',F')$  with helix $(d, s-1, [(\hat{v}_{\ell})_{\ell \in \Z}])$ where for $\ell \in \{0,\ldots,d-1\}$,
 \[ \hat{v}_{\ell} = \begin{cases} T^* v_{\ell} & \text{ if } 0 \leq \ell \leq h-1,\\
  v_{\ell} & \text{ if } h \leq \ell \leq d-1. \end{cases}  \]
  In particular, if $h=0$ then $\hat{v}_{\ell} = v_{\ell}$ for every $\ell \in\{0,\ldots,d-1\}$.
\end{prop}

This statement is illustrated in Figure \ref{fig:blowdown}.

\begin{rmk}
Concerning the case of $h=0$ in Proposition~\ref{prop:blowdown-helix}, notice that
it is \textit{not} true that the helix does not change at
all when a semitoric type blowdown is performed. 
By choosing a representative of the helix which starts on
the horizontal vector on which the semitoric type blowdown will be performed, that vector and the next $d-1$ vectors in the helix will remain unchanged, but since the number of focus-focus points changes, the value of $s$ changes, and thus the recurrence relation which uses $\hat{v}_0,\ldots, \hat{v}_{d-1}$ to determine the rest of the helix changes.
In summary, though the first $d$ vectors are unchanged, the rest of the helix is changed.
\end{rmk}

\begin{proof}[Proof of Proposition \ref{prop:blowdown-helix}]
 Suppose that $(M,\om,(J,H))$ is a semitoric system which has a helix $(d,s,[(v_{\ell})_{\ell\in\Z}])$
 as described in the statement of this lemma.
 Since the vector $v_h$ is horizontal, by Lemma~\ref{lem:horzvec-fixedsurf} the $S^1$-action generated by $J$ has a fixed surface,
 and thus, by Lemma~\ref{lem:st-blowdown}, the system admits a semitoric type blowdown on that fixed surface. All that remains is to check that any resulting system has the desired helix.

For the rest of the proof, we will assume for convenience that this fixed surface corresponds to the minimum of $J$; the case where it corresponds to the maximum of $J$ is similar. 

Using the equivalences of helices (see Equation \eqref{eq:equi_helix}), we may choose a representative of $(d,s,[(v_{\ell})_{\ell\in\Z}])$ for which the vector indexed by zero is horizontal, namely replacing $v_{\ell}$ by $u_{\ell} = v_{\ell+h}$ for every $\ell \in \Z$. Then once the statement for this case is proved, we go to the general case as follows. Define, for $\ell \in \Z$, $\hat{v}_{\ell} = \hat{u}_{\ell-h}$, so that $(d,s-1,[\hat{v}])$ is the helix of $(M',\omega',F')$. Let $\ell \in \{0, \ldots, d-1\}$ and consider the following two cases. If $h \leq \ell \leq d-1$, then $0 \leq \ell - h \leq d-1$ so $\hat{v}_{\ell} = \hat{u}_{\ell-h} = u_{\ell-h} = v_{\ell}$. If $0 \leq \ell \leq h-1$, then by using the recurrence relation on helices (item 3 in Definition \ref{def:helix}) twice and remembering that the helix represented by $\hat{u}$ is of complexity $s-1$, we obtain
\[ \hat{v}_{\ell} = \hat{u}_{\ell-h} = (T^*)^{1-s} \hat{u}_{d+\ell-h} = (T^*)^{1-s} u_{d+\ell-h} = (T^*)^{1-s} v_{d+\ell} = (T^*)^{1-s} (T^*)^s v_{\ell} = T^* v_{\ell}. \]

So for the rest of the proof we assume that $v_0$ is horizontal. Choose any focus-focus point $m$ where the blowdown will be performed, and take a representative $(\Delta, \vec{c}, \vec{\epsilon})$ of the marked semitoric polygon of $(M,\omega,(J,H))$ such that $\epsilon_{\ell} = 1$ and for every $i \neq \ell$, $\epsilon_i = -1$, where $c_{\ell}$ is the marked point corresponding to $m$. 

Choose a starting set of counter-clockwise ordered primitive inwards pointing normal vectors $w_0, \ldots, w_{p-1}$ for $\Delta$ as in Algorithm \ref{algo:helix_from_poly} such that $w_0 = v_0$, and let $(w_j, w_{j+1})$ be the vertex of $\Delta$ lying above $c_{\ell}$. Because of our choice of $\vec{\epsilon}$, this vertex is either $1$-fake or $1$-hidden Delzant.

Let $(\Delta',\vec{c'},\vec{\epsilon'})$ be the (left) wall unchop of $(\Delta,\vec{c},\vec{\epsilon})$ at $c_{\ell}$, as in Definition \ref{dfn:wall_unchop}; by definition, it is a representative of the marked semitoric polygon of $(M',\omega',F')$. In what follows, we will choose a set $(\hat{w}_0, \ldots, \hat{w}_{q-1})$ of counter-clockwise ordered primitive inwards pointing normal vectors for $\Delta'$ and apply Algorithm \ref{algo:helix_from_poly}, distinguishing the two cases corresponding to the possible types of $(w_j, w_{j+1})$.

\begin{itemize}
    \item If $(w_j, w_{j+1})$ is $1$-hidden Delzant, then $q = p$ and we may, and will, choose 
    \[ (\hat{w}_0, \ldots, \hat{w}_{p-1}) = (w_0, \ldots, w_j, T^* w_{j+1}, \ldots, T^*w_{p-1}) \]
    (recall the construction of $\Delta'$ from $\Delta$ above Definition \ref{dfn:wall_unchop}). Then $(\hat{w}_j, \hat{w}_{j+1})$ is Delzant and for every $i \neq j$, $(\hat{w}_i, \hat{w}_{i+1})$ is of the same type as $(w_i, w_{i+1})$. So when performing Step 1 of Algorithm \ref{algo:helix_from_poly}, the only difference between $\Delta$ and $\Delta'$ is that when dealing with $(w_j, w_{j+1})$, we will replace $w_{j+1}, \ldots, w_{p-1}$ with $T^* w_{j+1}, \ldots T^* w_{p-1}$, while for $(\hat{w}_j, \hat{w}_{j+1})$, we will not change $\hat{w}_{j+1}, \ldots, \hat{w}_{p-1}$. So after Step 2 we obtain $(\hat{w}'_0, \ldots, \hat{w}'_{p-1}) = (w'_0, \ldots, w'_{p-1})$, and we remove the same vectors in Step 3, hence at the end of the algorithm we obtain $(\hat{v}_0, \ldots, \hat{v}_{d-1}) = (v_0, \ldots, v_{d-1})$.
    \item If $(w_j, w_{j+1})$ is $1$-fake, then $q = p-1$ and we may, and will, choose 
    \[ (\hat{w}_0, \ldots, \hat{w}_{p-2}) = (w_0, \ldots, w_j = T^* w_{j+1}, T^* w_{j+2}, \ldots, T^*w_{p-1}). \]
    Then $(\hat{w}_i, \hat{w}_{i+1})$ is of the same type as $(w_i, w_{i+1})$ if $0 \leq i \leq j-1$ and of the same type as $(w_{i+1}, w_{i+2})$ if $j \leq i \leq p-2$ (with the convention $w_p = w_0$). This time after Steps 1 and 2 of Algorithm \ref{algo:helix_from_poly} we obtain $(\hat{w}'_0, \ldots, \hat{w}'_{p-2}) = (w'_0, \ldots, w'_j, w'_{j+2}, \ldots, w'_{p-1})$ but in Step 3, the difference is that we remove the vector $w'_{j+1}$ for $\Delta$. Therefore we again obtain  $(\hat{v}_0, \ldots, \hat{v}_{d-1}) = (v_0, \ldots, v_{d-1})$ at the end of the algorithm.\qedhere
\end{itemize}
\end{proof}

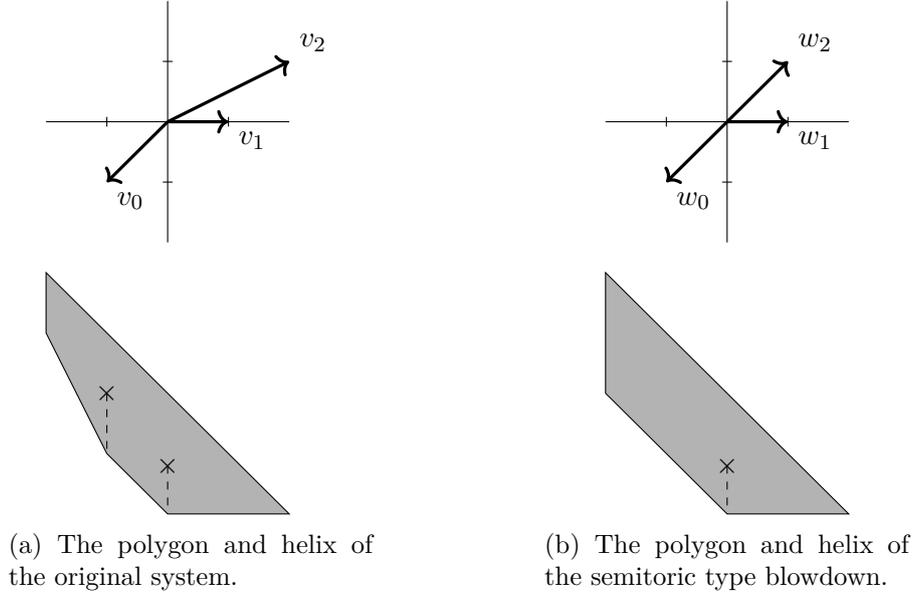
\begin{figure}
\begin{center} 
\def\dashlength{0.08}

\begin{subfigure}[b]{.3\linewidth} 
\centering 
\begin{tikzpicture}[scale=.8]
\draw (0,-2)--(0,2); 
\draw (-2,0)--(2,0); 
\draw [very thick,->] (0,0)--(-1,-1); 
\draw [very thick,->] (0,0)--(1,0); 
\draw [very thick,->] (0,0)--(2,1); 
\draw (-1,-1) node [below right] {$v_0$}; 
\draw (1,0) node [below right] {$v_1$}; 
\draw (2,1) node [above right] {$v_2$}; 
\draw (-1,-\dashlength) -- (-1,\dashlength);
\draw (-\dashlength,1) -- (\dashlength,1);
\draw (1,-\dashlength) -- (1,\dashlength);
\draw (-\dashlength,-1) -- (\dashlength,-1);
\begin{scope}[xshift=-2cm,yshift = -6.5cm] 
\filldraw[draw=black, fill=gray!60] (0,3) -- (0,4) -- (4,0) -- (2,0) -- (1,1) -- cycle; 
\draw (2,0.8) node {$\times$}; 
\draw [dashed] (2,0.8) -- (2,0); 
\draw (1,2) node {$\times$}; 
\draw [dashed] (1,2) -- (1,1); 

\end{scope} 
\end{tikzpicture} 
\caption{The polygon and helix of the original system.} 
\label{fig:blowdown-a}
\end{subfigure}  
\hspace{2cm} 
\begin{subfigure}[b]{.3\linewidth} 
\centering 
\begin{tikzpicture}[scale=.8]
\draw (0,-2)--(0,2); 
\draw (-2,0)--(2,0); 
\draw [very thick,->] (0,0)--(-1,-1); 
\draw [very thick,->] (0,0)--(1,0); 
\draw [very thick,->] (0,0)--(1,1); 
\draw (-1,-1) node [below right] {$w_0$}; 
\draw (1,0) node [below right] {$w_1$}; 
\draw (1,1) node [above right] {$w_2$}; 
\draw (-1,-\dashlength) -- (-1,\dashlength);
\draw (-\dashlength,1) -- (\dashlength,1);
\draw (1,-\dashlength) -- (1,\dashlength);
\draw (-\dashlength,-1) -- (\dashlength,-1);
\begin{scope}[xshift=-2cm,yshift = -6.5cm] 
\filldraw[draw=black, fill=gray!60] (0,2) -- (0,4) -- (4,0) -- (2,0) -- cycle; 
\draw (2,0.8) node {$\times$}; 
\draw [dashed] (2,0.8) -- (2,0); 

\end{scope} 
\end{tikzpicture} 
\caption{The polygon and helix of the semitoric type blowdown.} 
\label{fig:blowdown-b}
\end{subfigure}

\end{center} 
\caption{Figure~\ref{fig:blowdown-a} shows the helix and polygon of a semitoric system. Since the helix includes a horizontal vector it admits a semitoric type blowdown, and Proposition~\ref{prop:blowdown-helix} can be used to calculate the resulting helix, shown in Figure~\ref{fig:blowdown-b}.} 
\label{fig:blowdown} 
\end{figure}

\begin{prop}
\label{prop:strict-min-helix}
 A semitoric system with at least one focus-focus point is strictly minimal if and only if the associated helix is of type (1), (2), (3) with $n = 2$ or $n \geq 4$, or the $J$-reflection of $H$-reflection of one of these.
\end{prop}

\begin{proof}
Note that the helices of types (4), (5) and (6) and (3) with $n=1$ each include a horizontal vector while the helices of types (1), (2), and (3) with $n \geq 2$ do not. By Proposition~\ref{prop:blowdown-helix}, a helix admits a semitoric type blowdown if and only if it includes a horizontal vector. Thus, applying Theorem \ref{thm:min_helices}, we conclude that a semitoric system does not admit a toric or semitoric type blowdown if and only if its helix is of one of the types discussed in the statement, or the $J$-reflection or $H$-reflection of one of these.
\end{proof}

\begin{rmk}
 Notice that the special case of Proposition~\ref{prop:blowdown-helix} with $h=0$ will be very helpful when doing blowdowns on a system which
 is minimal of type (4), (5), or (6), see Section \ref{subsect:min_456}. If the system has $s$ focus-focus points then $s$ semitoric type blowdowns can
 be performed and, assuming that $v_0$ was chosen to be horizontal, the vectors $v_0,\ldots, v_{d-1}$ are the fan of the
 resulting toric type system.
 Thus, it is easy to read off of the helix which toric minimal model we will end up with afterwards.
\end{rmk}

\subsection{Strictly minimal polygons}
\label{subsec:min_poly}

Proposition \ref{prop:strict-min-helix} tells us which helices correspond to strictly minimal semitoric systems, but we want a result about the marked semitoric polygons rather than the helices. 

In general, listing the polygons that are associated to a given helix can be very hard. In \cite[Lemma 5.6]{KPP_min}, a construction yielding one unmarked semitoric polygon producing a given helix is provided. However, given a helix, there are in general multiple such semitoric polygons, and obtaining all of them is very complicated (essentially because there exists some freedom in the relative positions of the marked points which leaves the helix unchanged). Nevertheless, for the helices of types (1), (2) and (3), one can obtain in a relatively straightforward way all these polygons.

\begin{dfn}
\label{def:min-poly-123}
A marked semitoric polygon is said to be of type (1), (2a), (2b), (3a), (3b) or (3c) if and only if one of its representatives is as in Figure \ref{fig:strictly_min_poly}. In this figure the parameters $\alpha, \beta$ are positive real numbers, the parameter $n$ is an integer, and the height invariants are $h$ or $(h_1,h_2)$; the constraints on these parameters are made precise in the captions.
\end{dfn}

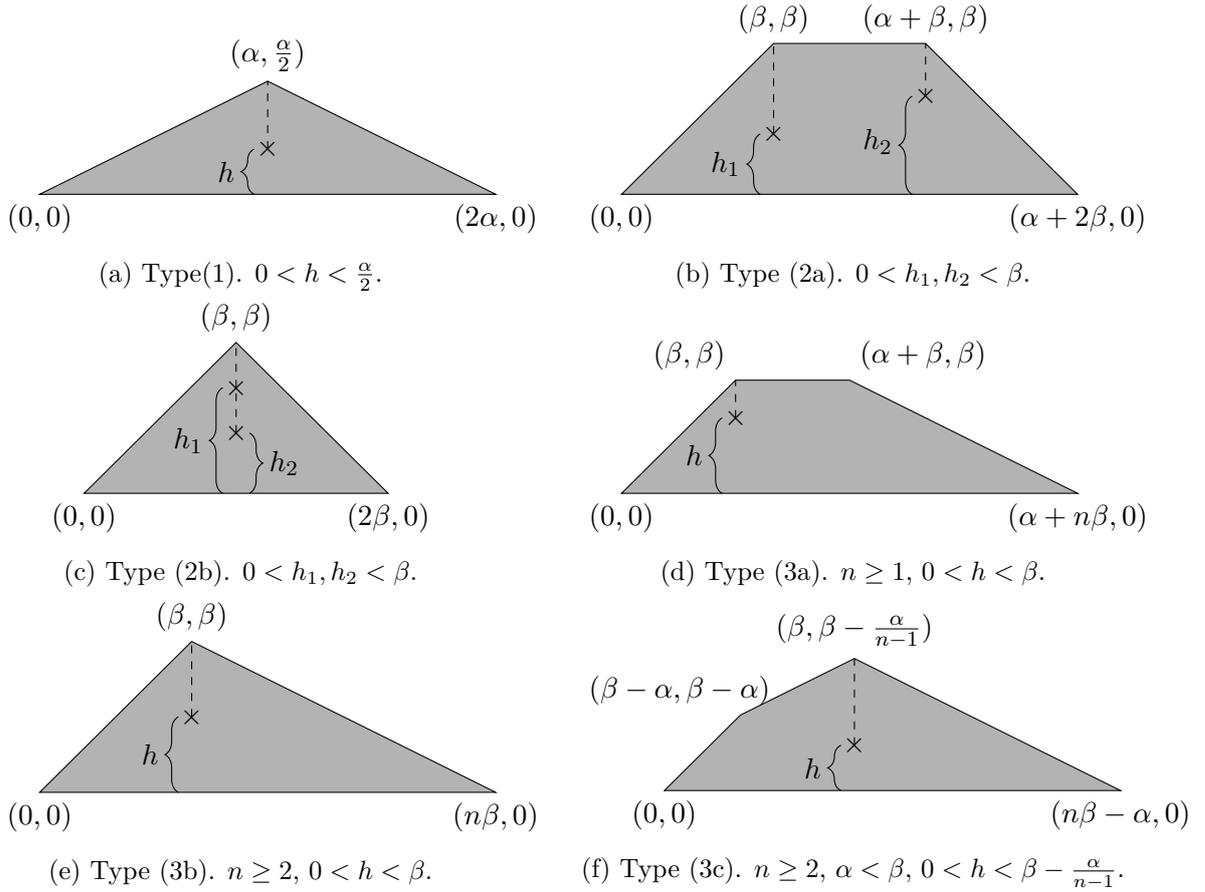
\begin{figure}

\begin{center}
\begin{subfigure}[b]{.4\linewidth} 
\centering 
\begin{tikzpicture}
\filldraw[draw=black, fill=gray!60] (-3,0) -- (3,0) -- (0,1.5) -- cycle; 
\draw (0,0.6) node {$\times$}; 
\draw [dashed] (0,0.6) -- (0,1.5); 

\draw (-3,0) node[below] {$(0,0)$};
\draw (3,0) node[below] {$(2\alpha,0)$};
\draw (0,1.5) node[above] {$(\alpha,\frac{\alpha}{2})$};

\draw[decoration={brace,raise=5pt, amplitude = 5pt},decorate] (0,0) -- (0,0.6);
\draw (-0.3,0.3) node[left] {$h$}; 
\end{tikzpicture} 
\caption{Type(1). $0 < h < \frac{\alpha}{2}$. } 
\label{fig:min_type1}
\end{subfigure}  
\hspace{1cm}
\begin{subfigure}[b]{.45\linewidth} 
\centering 
\begin{tikzpicture}
\filldraw[draw=black, fill=gray!60] (-3,0) -- (-1,2) -- (1,2) -- (3,0) -- cycle; 

\draw (-1,0.8) node {$\times$}; 
\draw [dashed] (-1,0.8) -- (-1,2); 

\draw (1,1.3) node {$\times$}; 
\draw [dashed] (1,1.3) -- (1,2);

\draw (-3,0) node[below] {$(0,0)$};
\draw (3,0) node[below] {$(\alpha + 2 \beta,0)$};
\draw (-1,2) node[above] {$(\beta,\beta)$};
\draw (1,2) node[above] {$(\alpha + \beta,\beta)$};

\draw[decoration={brace,raise=5pt, amplitude = 5pt},decorate] (-1,0) -- (-1,0.8);
\draw (-1.3,0.4) node[left] {$h_1$}; 
\draw[decoration={brace,raise=5pt, amplitude = 5pt},decorate] (1,0) -- (1,1.3);
\draw (0.7,0.7) node[left] {$h_2$}; 
\end{tikzpicture} 
\caption{Type (2a). $0 < h_1, h_2 < \beta$.} 
\label{fig:min_type2a}
\end{subfigure}  
\hspace{1cm} 
\begin{subfigure}[b]{.4\linewidth} 
\centering 
\begin{tikzpicture}
\filldraw[draw=black, fill=gray!60] (-2,0) -- (2,0) -- (0,2) -- cycle; 
\draw (0,1.4) node {$\times$}; 
\draw [dashed] (0,1.4) -- (0,2); 

\draw (0,0.8) node {$\times$}; 
\draw [dashed] (0,0.8) -- (0,1.4);

\draw (-2,0) node[below] {$(0,0)$};
\draw (2,0) node[below] {$(2\beta,0)$};
\draw (0,2) node[above] {$(\beta,\beta)$};

\draw[decoration={brace,raise=5pt, amplitude = 5pt},decorate] (0,0) -- (0,1.4);
\draw (-0.3,0.7) node[left] {$h_1$};
\draw[decoration={brace, mirror, raise=5pt, amplitude = 5pt},decorate] (0,0) -- (0,0.8);
\draw (0.3,0.4) node[right] {$h_2$}; 
\end{tikzpicture} 
\caption{Type (2b). $0 < h_1, h_2 < \beta$.} 
\label{fig:min_type2b}
\end{subfigure}  
\hspace{1cm} 
\begin{subfigure}[b]{.45\linewidth} 
\centering 
\begin{tikzpicture}
\filldraw[draw=black, fill=gray!60] (-3,0) -- (3,0) -- (0,1.5) -- (-1.5,1.5) -- cycle; 

\draw (-1.5,1) node {$\times$}; 
\draw [dashed] (-1.5,1) -- (-1.5,1.5);

\draw (-3,0) node[below] {$(0,0)$};
\draw (3,0) node[below] {$(\alpha + n \beta,0)$};
\draw (0,1.5) node[above right] {$(\alpha + \beta,\beta)$};
\draw (-1.5,1.5) node[above left] {$(\beta, \beta)$};

\draw[decoration={brace,raise=5pt, amplitude = 5pt},decorate] (-1.5,0) -- (-1.5,1);
\draw (-1.8,0.5) node[left] {$h$};
\end{tikzpicture} 
\caption{Type (3a). $n \geq 1$, $0 < h < \beta$.} 
\label{fig:min_type3a}
\end{subfigure}  
\hspace{1cm} 
\begin{subfigure}[b]{.4\linewidth} 
\centering 
\begin{tikzpicture}
\filldraw[draw=black, fill=gray!60] (-3,0) -- (3,0) -- (-1,2) -- cycle; 
\draw (-1,1) node {$\times$}; 
\draw [dashed] (-1,1) -- (-1,2); 

\draw (-3,0) node[below] {$(0,0)$};
\draw (3,0) node[below] {$(n\beta,0)$};
\draw (-1,2) node[above] {$(\beta,\beta)$};

\draw[decoration={brace,raise=5pt, amplitude = 5pt},decorate] (-1,0) -- (-1,1);
\draw (-1.3,0.5) node[left] {$h$};
\end{tikzpicture} 
\caption{Type (3b). $n \geq 2$, $0 < h < \beta$.} 
\label{fig:min_type3b}
\end{subfigure} 
\hspace{1cm} 
\begin{subfigure}[b]{.45\linewidth} 
\centering 
\begin{tikzpicture}
\filldraw[draw=black, fill=gray!60] (-3,0) -- (3,0) -- (-0.5,1.75) -- (-2,1) -- cycle; 

\draw (-0.5,0.6) node {$\times$}; 
\draw [dashed] (-0.5,0.6) -- (-0.5,1.75); 

\draw (-3,0) node[below] {$(0,0)$};
\draw (3,0) node[below] {$(n\beta - \alpha,0)$};
\draw (-0.5,1.75) node[above] {$(\beta, \beta - \frac{\alpha}{n-1})$};
\draw (-2.8,1) node[above] {$(\beta - \alpha, \beta - \alpha)$};

\draw[decoration={brace,raise=5pt, amplitude = 5pt},decorate] (-0.5,0) -- (-0.5,0.6);
\draw (-0.8,0.3) node[left] {$h$};
\end{tikzpicture} 
\caption{Type (3c). $n \geq 2$, $\alpha < \beta$, $0 < h < \beta - \frac{\alpha}{n-1}$.} 
\label{fig:min_type3c}
\end{subfigure} 

\end{center} 
\caption{A representative of the marked semitoric polygon of each type among (1), (2a), (2b), (3a), (3b), (3c).} 
\label{fig:strictly_min_poly}
\end{figure}

\begin{thm}
\label{thm:min_poly}
 Let $(M,\om,F)$ be a semitoric system with at least one focus-focus singularity. This system does not admit a toric or semitoric type blowdown
 if and only if its marked semitoric polygon is of type (1), (2a), (2b), (3a) with $n = 2$ or $n \geq 4$, (3b) or (3c) (with $n \neq 3$ for these last two), see Definition \ref{def:min-poly-123}, or the $J$-reflection or $H$-reflection of one of these, see Definition \ref{def:reflections_semitoric_poly}.
\end{thm}

We illustrate the fact that a polygon of type (3a) with $n=3$ is not minimal, even though this is not obvious from its representative in Figure \ref{fig:min_type3a}, since it admits a corner unchop, in Figure \ref{fig:n3_not_minimal}. This can be seen from its helix $(d,s,[v])$ since in this case (using the representative given in Definition \ref{def:types_helices})
\[ v_1 + v_3 = v_1 + T^* v_0 = \begin{pmatrix} -1 \\ -2 \end{pmatrix} + \begin{pmatrix} 1 \\ 1 \end{pmatrix}  = \begin{pmatrix} 0 \\ -1 \end{pmatrix} = v_2, \]
so this helix is not minimal, see Definition \ref{dfn:min_hel_cond}, but it is more visual with the polygons. The cases of types (3b) and (3c) are similar.

\begin{figure}
    \begin{center}
      \begin{tikzpicture}[scale=0.65]
        \filldraw[draw=black, fill=gray!60] (0,0) node[below left]{$(0,0)$}
  -- (2,2) node[above left]{$(\beta,\beta)$}
  -- (3,2) node[above right]{$(\alpha + \beta,\beta)$}
  -- (7,0) node[below]{$(\alpha + 3 \beta,0)$}
  -- cycle;
\draw [dashed] (2,2) -- (2,1.1);
\draw (2,1.1) node[] {$\times$};
\draw[ultra thick] (0,0) -- (2,2);
\draw[ultra thick] (2,2) -- (3,2);

\draw[->] (8,1) -- (11,1);

\draw[->] (3,-1) -- (3,-4);

        \begin{scope}[xshift=14cm]
         \filldraw[draw=black, fill=gray!60] (-3,0) node[below]{$(-(\alpha+\beta),0)$}
  -- (2,2.5) node[above right]{$(\beta,\beta + \frac{\alpha}{2})$}
  -- (7,0) node[below]{$(\alpha + 3 \beta,0)$}
  -- cycle;
\draw [dashed] (2,2.5) -- (2,1.1);
\draw (2,1.1) node[] {$\times$};

\draw[->] (3,-1) -- (3,-4);

\filldraw[black, fill=gray!60, pattern=north east lines] (0,0) -- (-3,0) -- (2,2.5) -- (3,2) -- (2,2) -- cycle;

        \end{scope}

        \begin{scope}[yshift=-8cm]

        \filldraw[draw=black, fill=gray!60] (0,2) node[above left]{$(0,\beta)$}
  -- (3,2) node[above right]{$(\alpha + \beta,\beta)$}
  -- (7,0) node[below]{$(\alpha + 3 \beta,0)$}
  -- (2,0) node[below left]{$(\beta,0)$}
  -- cycle;
\draw [dashed] (2,0) -- (2,1.1);
\draw (2,1.1) node[] {$\times$};
\draw[ultra thick] (0,2) -- (3,2);

\draw[->] (8,1) -- (11,1);

        \end{scope}

         \begin{scope}[xshift=14cm,yshift=-8cm]
         
 \filldraw[draw=black, fill=gray!60] (-3,5) node[above right]{$(-(\alpha+\beta),\alpha + 2 \beta)$}
  -- (7,0) node[below]{$(\alpha + 3 \beta,0)$}
  -- (2,0) node[below left]{$(\beta,0)$}
  -- cycle;
\draw [dashed] (2,0) -- (2,1.1);
\draw (2,1.1) node[] {$\times$};

\filldraw[black, fill=gray!60, pattern=north east lines] (0,2) -- (-3,5) --  (3,2) -- cycle;

        \end{scope}
        
      \end{tikzpicture}
    \end{center}
    \caption{Illustration of the fact that semitoric systems with polygons of type (3a) with $n=3$ are not minimal. The left to right arrows indicate corner unchops (corresponding to toric type blowdowns), while the top to bottom arrows indicate changes of representative.}
    \label{fig:n3_not_minimal}
  \end{figure}
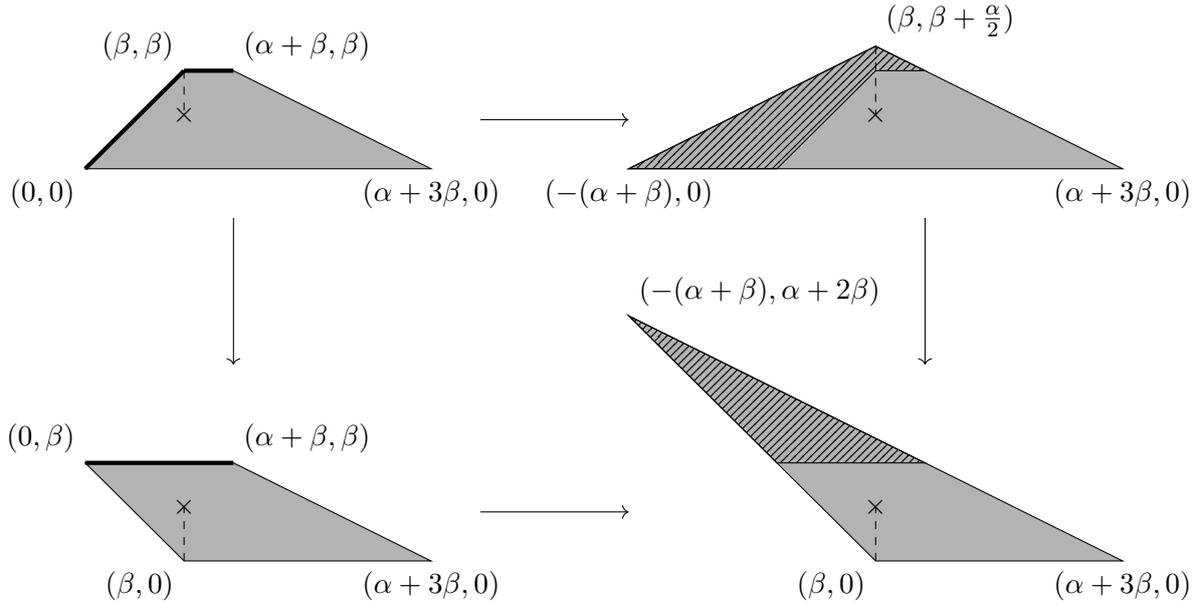

\begin{rmk}
In Definition \ref{def:min-poly-123} and the associated Figure \ref{fig:strictly_min_poly} we chose to include the $n=3$ case for types (3a), (3b) and (3c), because it is convenient to do so even though they are not minimal. For instance, it is still interesting to come up with explicit examples of semitoric systems with such polygons, since they exhibit a $\Z_2$-sphere; this is done in Sections \ref{subsubsec:3c_Z2}, \ref{subsubsec:3a_Z2} and \ref{subsec:explicit_3b}.  
\end{rmk}

\begin{proof}[Proof of Theorem \ref{thm:min_poly}]
 From Proposition~\ref{prop:strict-min-helix} we know that a semitoric system does not admit a toric or semitoric type blowdown if and only if the associated helix is of type (1), (2), (3) with $n=2$ or $n\geq4$, or the $J$-reflection or $H$-reflection of one of these. Thanks to Corollary \ref{cor:reflections_preimage}, it then suffices to compute the set of all polygons with helices of type (1), (2) or (3). One readily computes the helices of the semitoric polygons from Definition~\ref{def:min-poly-123} and checks that they are of type (1), (2), and (3) respectively. To complete the proof, we will show that these are the only polygons whose helices are of type (1), (2) or (3). Note that once we know the first coordinate of each marked point, we automatically know the allowed range for its height invariant.

Let $[(\Delta,\vec{c},\vec{\epsilon})]$ be a semitoric polygon giving a helix $(d,s,[v])$ of type (1), (2) or (3). Let $w_0, \ldots, w_{p-1}$ be the primitive inwards pointing normal vectors of $\Delta$ arranged in counter-clockwise order and such that $(w_{p-1},w_0)$ is Delzant, and let $e_0, \ldots, e_{p-1}$ be the corresponding edges of $\Delta$. Note that necessarily $p \geq 3$ since $\Delta$ is a polygon. Note also that, in view of Algorithm \ref{algo:helix_from_poly}, necessarily $p \leq d+s$.

The idea of the proof is that for the helices of types (1), (2) and (3), $p$ can only take one or two values, and for each of those, we can list all the possibilities for going from $w_0, \ldots, w_{p-1}$ to $v_0, \ldots, v_{d-1}$ when constructing the helix using Algorithm \ref{algo:helix_from_poly}. We will again use the notation $w_0', \ldots, w_{p-1}'$ for the vectors obtained in Step 2 of this algorithm. Note that if $r \geq 1$ is the number of vertices of $\Delta$ that are either $k$-fake or $k$-hidden Delzant, and $k_1, \ldots , k_r \geq 1$ are the corresponding such $k$, then $k_1 + \ldots + k_r = s$ (and in particular $r \leq s$). Note also that since removing one vector from $w_0', \ldots, w_{p-1}'$ occurs precisely when one vertex is fake (see Step 3 of the aforementioned algorithm), the number of fake vertices is $f = p-d$. Finally, recall that by construction, $v_0 = w_0$.

\paragraph{Type (1).} In this case $d=2$ and $s=1$, hence $p \leq 3$, so necessarily $p=3$: $\Delta$ is a triangle. Since $w_0 = v_0 = \begin{pmatrix} 0 \\ 1 \end{pmatrix}$ is a primitive inwards pointing normal vector of $\Delta$, necessarily the edge $e_0$ is horizontal; let $\lambda$ be the length of this edge, and denote its vertices by $(0,0)$ and $(\lambda,0)$. Since $s=1$, there is exactly one vertex of $\Delta$ that is either $1$-fake or $1$-hidden Delzant. Since $f = p-d = 1$, this vertex is necessarily fake. So there are only two possibilites.

\begin{enumerate}
    \item Either $w_1' \neq w_0'$ and $w_2' = w_1'$, which means that the $1$-fake vertex is $(w_1,w_2)$. Hence $w_1 = w_1' = v_1 = \begin{pmatrix} -1 \\ -2 \end{pmatrix}$ is the next primitive inwards pointing normal vector to $\Delta$, in which case the edge $e_1$ is directed by the vector $\begin{pmatrix} -2 \\ 1 \end{pmatrix}$. In this case, by construction
    \[ w_2 = (T^*)^{-1} w_2' = (T^*)^{-1} w_1' = (T^*)^{-1} v_1 = \begin{pmatrix} 1 & -1 \\ 0 & 1 \end{pmatrix} \begin{pmatrix} -1 \\ -2 \end{pmatrix} = \begin{pmatrix} 1 \\ -2 \end{pmatrix}, \]
    so $e_2$ is directed by $\begin{pmatrix} -2 \\ -1 \end{pmatrix}$. So the vertex $(w_1,w_2)$ of $\Delta$ is the intersection of the two lines starting at $(0,0)$ (respectively $(\lambda,0)$) and directed by $\begin{pmatrix} -2 \\ -1 \end{pmatrix}$ (respectively $\begin{pmatrix} -2 \\ 1 \end{pmatrix}$), namely $(\frac{\lambda}{2},\frac{\lambda}{4})$. The marked point is attached to this vertex. This gives the polygon in Figure \ref{fig:type1-a}; when $\lambda = 2\alpha$, this is the polygon in Figure \ref{fig:min_type1}.
    \item Or $w_1' = w_0'$ and $w_2' \neq w_1'$, which means that the $1$-fake vertex is $(w_0,w_1)$. Then $w_1 = (T^*)^{-1} w_1'$ with $w_1' = w_0$, which means that $w_1 = \begin{pmatrix} -1 \\ 1 \end{pmatrix}$, so that $e_1$ is directed by $\begin{pmatrix} 1 \\ 1 \end{pmatrix}$. Then necessarily (since $d=1$), $w_2 = (T^*)^{-1} v_1 = \begin{pmatrix} 1 \\ -2 \end{pmatrix}$ and $e_2$ is directed by the vector $\begin{pmatrix} -2 \\ -1 \end{pmatrix}$. The third vertex of $\Delta$ is $(3\lambda,2\lambda)$. This gives the polygon in Figure \ref{fig:type1-b}. After changing the cut direction and setting $\lambda' = 2\lambda$, this gives the same polygon as in Figure \ref{fig:type1-a} (with parameter $\lambda'$).
\end{enumerate}

\begin{figure}
\begin{center} 

\begin{subfigure}[b]{.4\linewidth} 
\centering 
\begin{tikzpicture}
\filldraw[draw=black, fill=gray!60] (-3,0) -- (3,0) -- (0,1.5) -- cycle; 
\draw (0,0.6) node {$\times$}; 
\draw [dashed] (0,0.6) -- (0,1.5); 

\draw (-3,0) node[below] {$(0,0)$};
\draw (3,0) node[below] {$(\lambda,0)$};
\draw (0,1.5) node[above] {$(\frac{\lambda}{2},\frac{\lambda}{4})$};
\end{tikzpicture} 
\caption{Subcase 1.} 
\label{fig:type1-a}
\end{subfigure}  
\hspace{1cm} 
\begin{subfigure}[b]{.4\linewidth} 
\centering 
\begin{tikzpicture}[scale=.9]
\filldraw[draw=black, fill=gray!60] (0,0) -- (3,0) -- (6,3) -- cycle; 
\draw (3,0.8) node {$\times$}; 
\draw [dashed] (3,0.8) -- (3,0); 

\draw (0,0) node[below] {$(0,0)$};
\draw (3,0) node[below] {$(\lambda,0)$};
\draw (6,3) node[above] {$(2\lambda,\lambda)$};
\end{tikzpicture} 
\caption{Subcase 2.} 
\label{fig:type1-b}
\end{subfigure}

\end{center} 
\caption{The polygons in the case of type $(1)$ in the proof of Theorem \ref{thm:min_poly}.} 
\end{figure}

\paragraph{Type (2).} In this case $d=2$ and $s=2$, hence $p \leq 4$, so $p = 3$ or $p = 4$. Since $w_0 = v_0 = \begin{pmatrix} 0 \\ 1 \end{pmatrix}$, the edge $e_0$ is horizontal. We denote by $\lambda$ the length of this edge, and by $(0,0)$ and $(\lambda,0)$ its vertices. Since $s=2$, $\Delta$ has either two $1$-fake or $1$-hidden Delzant vertices, or one $2$-fake or $2$-hidden Delzant vertex.

\begin{enumerate}
    \item In the case $p=3$, there is only one vertex (the interior vertex) that can be either fake or hidden Delzant, and $f=p-d=1$, so necessarily this vertex is $2$-fake. So there are two subcases.
    \begin{enumerate}
        \item Either $(w_1,w_2)$ is $2$-fake, in which case $w_1 = v_1$, so the edge $e_1$ is directed by $\begin{pmatrix} -1 \\ 1 \end{pmatrix}$, and $w_2 = (T^*)^{-2} v_1 = \begin{pmatrix} 1 \\ -1 \end{pmatrix}$, so $e_2$ is directed by $\begin{pmatrix} -1 \\ -1 \end{pmatrix}$. The vertex $(w_1,w_2)$ is then $(\frac{\lambda}{2},\frac{\lambda}{2})$. This gives the polygon in Figure \ref{fig:type2-1a}, which corresponds to the polygon in Figure \ref{fig:min_type2b} after choosing $\lambda = 2\beta$.
        \item Or $(w_0,w_1)$ is $2$-fake and $w_1 = (T^*)^{-2} v_0 = \begin{pmatrix} -2 \\ 1 \end{pmatrix}$. Then $e_1$ is directed by the vector $\begin{pmatrix} 1 \\ 2 \end{pmatrix}$, and $w_2 = (T^*)^{-2} v_1 = \begin{pmatrix} 1 \\ -1 \end{pmatrix}$ so $e_2$ is directed by the vector $\begin{pmatrix} -1 \\ -1 \end{pmatrix}$. The vertex $(w_1,w_2)$ is then computed to be $(2\lambda,2\lambda)$, see Figure \ref{fig:type2-1b}. After changing both cut directions and setting $\lambda' = 2\lambda$, this gives the same polygon as in Figure \ref{fig:type2-1a} (with parameter $\lambda'$).
    \end{enumerate}
    
\begin{figure}
\begin{center} 

\begin{subfigure}[b]{.4\linewidth} 
\centering 
\begin{tikzpicture}
\filldraw[draw=black, fill=gray!60] (-3,0) -- (3,0) -- (0,3) -- cycle; 
\draw (0,1.4) node {$\times$}; 
\draw [dashed] (0,1.4) -- (0,3); 

\draw (0,0.8) node {$\times$}; 
\draw [dashed] (0,0.8) -- (0,1.4);

\draw (-3,0) node[below] {$(0,0)$};
\draw (3,0) node[below] {$(\lambda,0)$};
\draw (0,3) node[above] {$(\frac{\lambda}{2},\frac{\lambda}{2})$};
\end{tikzpicture} 
\caption{Subcase 1.(a).} 
\label{fig:type2-1a}
\end{subfigure}  
\hspace{1cm} 
\begin{subfigure}[b]{.4\linewidth} 
\centering 
\begin{tikzpicture}
\filldraw[draw=black, fill=gray!60] (0,0) -- (2.5,0) -- (5,5) -- cycle; 
\draw (2.5,1) node {$\times$}; 
\draw [dashed] (2.5,1) -- (2.5,0); 
\draw (2.5,1.4) node {$\times$}; 
\draw [dashed] (2.5,1.4) -- (2.5,1); 

\draw (0,0) node[below] {$(0,0)$};
\draw (2.5,0) node[below] {$(\lambda,0)$};
\draw (5,5) node[above] {$(2\lambda,2\lambda)$};
\end{tikzpicture} 
\caption{Subcase 1.(b).} 
\label{fig:type2-1b}
\end{subfigure}

\end{center} 
\caption{The polygons in the case $p=3$ of type $(2)$ in the proof of Theorem \ref{thm:min_poly}.} 
\end{figure}
    
    \item In the case $p=4$, we have $f=p-d=2$, so necessarily there are two vertices that are both $1$-fake. This leaves us with three subcases.
    \begin{enumerate}
        \item Either $(w_1,w_2)$ and $(w_2,w_3)$ are $1$-fake. Then $w_1 = v_1$, so $e_1$ is directed by $\begin{pmatrix} -1 \\ 1 \end{pmatrix}$, $w_2 = (T^*)^{-1} v_1 = \begin{pmatrix} 0 \\ -1 \end{pmatrix}$ so $e_2$ is horizontal, and $w_3 = (T^*)^{-2} v_1 = \begin{pmatrix} 1 \\ -1 \end{pmatrix}$ so $e_3$ is directed by $\begin{pmatrix} -1 \\ -1 \end{pmatrix}$. In this case the vertex $(w_1,w_2)$ is $(\lambda - \mu, \mu)$ for some $0 < \mu < \frac{\lambda}{2}$, and then necessarily the vertex $(w_2,w_3)$ is $(\mu, \mu)$. This is the polygon in Figure \ref{fig:type2-2a}. After setting $\lambda = \alpha + 2\beta$ and $\mu = \beta$, this yields the polygon in Figure \ref{fig:min_type2a}.
        \item Or $(w_0,w_1)$ and $(w_1,w_2)$ are $1$-fake. Then $w_1 = (T^*)^{-1} v_0 = \begin{pmatrix} -1 \\ 1 \end{pmatrix}$ so $e_1$ is directed by $\begin{pmatrix} 1 \\ 1 \end{pmatrix}$, $w_2 = (T^*)^{-2} v_0 = \begin{pmatrix} -2 \\ 1 \end{pmatrix}$ so $e_2$ is directed by $\begin{pmatrix} 1 \\ 2 \end{pmatrix}$, and $w_3 = (T^*)^{-2} v_1 = \begin{pmatrix} 1 \\ -1 \end{pmatrix}$ so $e_3$ is directed by $\begin{pmatrix} -1 \\ 1 \end{pmatrix}$. The vertices $(w_1,w_2)$ and $(w_2,w_3)$ are $(\lambda + \mu, \mu)$ and $(2\lambda + \mu, 2 \lambda + \mu)$ for some $\mu > 0$, respectively. This is the polygon in Figure \ref{fig:type2-2b}; after changing both cut directions and setting $\mu' = \lambda$ and $\lambda' = 2\lambda + \mu$, this gives the polygon in Figure \ref{fig:type2-2a} (with parameters $\lambda', \mu'$).
        \item Or $(w_0,w_1)$ and $(w_2,w_3)$ are $1$-fake. Then $w_1 = (T^*)^{-1} v_0 = \begin{pmatrix} -1 \\ 1 \end{pmatrix}$, so $e_1$ is directed by $\begin{pmatrix} 1 \\ 1 \end{pmatrix}$, $w_2 = (T^*)^{-1} v_1 = \begin{pmatrix} 0 \\ -1 \end{pmatrix}$ so $e_2$ is horizontal, and $w_3 = (T^*)^{-2} v_1 = \begin{pmatrix} 1 \\ -1 \end{pmatrix}$ so $e_3$ is directed by $\begin{pmatrix} -1 \\ -1 \end{pmatrix}$. The vertices $(w_1, w_2)$ and $(w_2, w_3)$ are $(\lambda + \mu, \mu)$ and $(\mu , \mu)$ for some $\mu > 0$, respectively. There are three cases $\mu < \lambda$, $\mu = \lambda$ and $\mu > \lambda$, and only the first one is represented in Figure \ref{fig:type2-2c}. By changing the cut direction of the rightmost marked point we see that the case $\mu < \lambda$ yields the polygon in Figure \ref{fig:min_type2a} by setting $\lambda = \alpha + \beta$ and $\mu = \beta$, and to treat the case $\mu > \lambda$ it suffices to exchange the roles of $\mu$ and $\lambda$; we also see that the case $\mu = \lambda$ leads to the polygon in Figure \ref{fig:min_type2b} by choosing $\lambda = \beta$.
    \end{enumerate}

\begin{figure}
\begin{center} 

\begin{subfigure}[b]{.4\linewidth} 
\centering 
\begin{tikzpicture}
\filldraw[draw=black, fill=gray!60] (-3,0) -- (-1,2) -- (1,2) -- (3,0) -- cycle; 

\draw (-1,0.8) node {$\times$}; 
\draw [dashed] (-1,0.8) -- (-1,2); 

\draw (1,1.3) node {$\times$}; 
\draw [dashed] (1,1.3) -- (1,2);

\draw (-3,0) node[below] {$(0,0)$};
\draw (3,0) node[below] {$(\lambda,0)$};
\draw (-1,2) node[above] {$(\mu,\mu)$};
\draw (1,2) node[above] {$(\lambda-\mu,\mu)$};
\end{tikzpicture} 
\caption{Subcase 2.(a).} 
\label{fig:type2-2a}
\end{subfigure}  
\hspace{1cm} 
\begin{subfigure}[b]{.4\linewidth} 
\centering 
\begin{tikzpicture}
\filldraw[draw=black, fill=gray!60] (-1,0) -- (1,0) -- (2,1) -- (4,5) -- cycle; 
\draw (1,0.7) node {$\times$}; 
\draw [dashed] (1,0.7) -- (1,0); 

\draw (2,1.8) node {$\times$}; 
\draw [dashed] (2,1.8) -- (2,1); 

\draw (-1,0) node[below] {$(0,0)$};
\draw (1,0) node[below] {$(\lambda,0)$};
\draw (2,1) node[below right] {$(\lambda + \mu,\mu)$};
\draw (4,5) node[above] {$(2\lambda + \mu,2\lambda + \mu)$};
\end{tikzpicture} 
\caption{Subcase 2.(b).} 
\label{fig:type2-2b}
\end{subfigure} 
\hspace{1cm} 
\begin{subfigure}[b]{.4\linewidth} 
\centering 
\begin{tikzpicture}
\filldraw[draw=black, fill=gray!60] (-3,0) -- (3,0) -- (5,2) -- (-1,2) -- cycle; 
\draw (-1,1) node {$\times$}; 
\draw [dashed] (-1,1) -- (-1,2); 

\draw (3,1) node {$\times$}; 
\draw [dashed] (3,1) -- (3,0); 

\draw (-3,0) node[below] {$(0,0)$};
\draw (3,0) node[below] {$(\lambda,0)$};
\draw (5,2) node[above] {$(\lambda + \mu,\mu)$};
\draw (-1,2) node[above] {$(\mu,\mu)$};
\end{tikzpicture} 
\caption{Subcase 2.(c).} 
\label{fig:type2-2c}
\end{subfigure} 

\end{center} 
\caption{The polygons in the case $p=4$ of type $(2)$ in the proof of Theorem \ref{thm:min_poly}.} 
\end{figure}
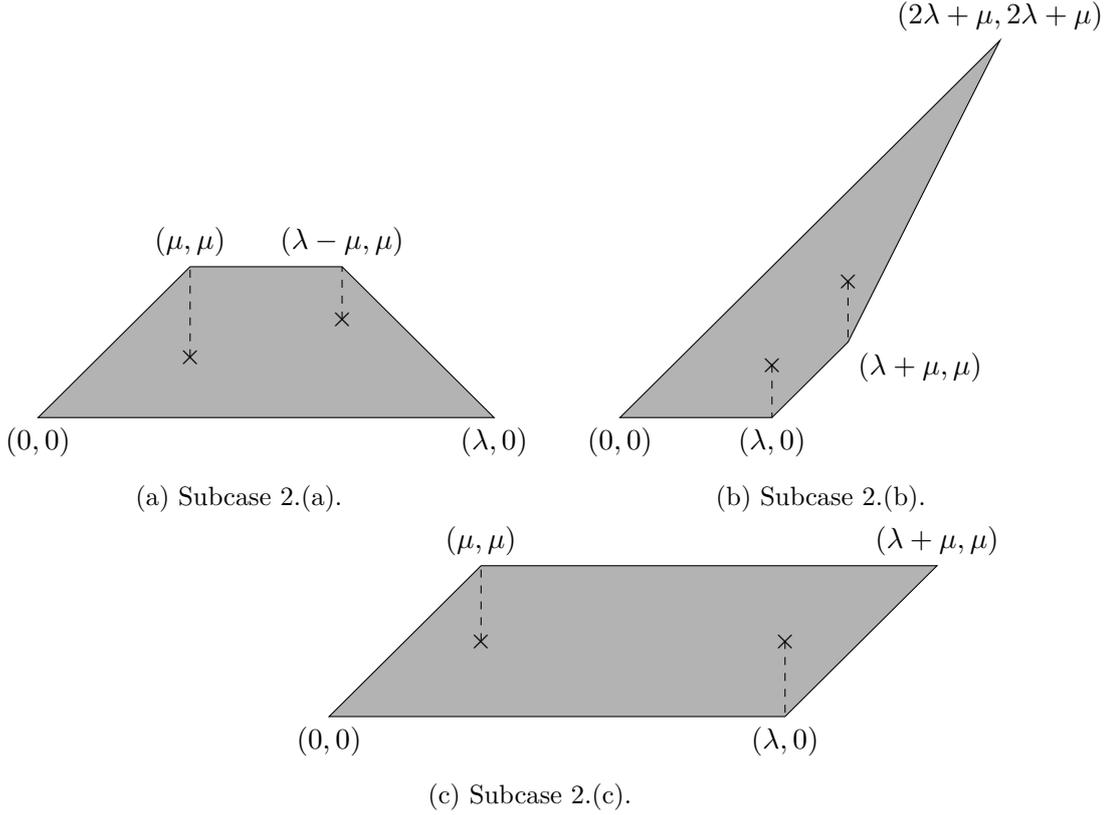    
\end{enumerate}

\paragraph{Type (3).} This case is slightly more involved, since it depends on the parameter $n \geq 1$. In this case $d=3$ and $s=1$, so again $p=3$ or $p=4$. And again, $e_0$ is horizontal, we denote by $(0,0)$ and $(\lambda,0)$ its vertices. Since $s=1$, there is exactly one vertex that is either $1$-fake or $1$-hidden Delzant. 
\begin{enumerate}
    \item In the case $p=3$, we have $f=p-d = 0$, so this vertex is $1$-hidden Delzant. So there are two cases.
    \begin{enumerate}
        \item Either $(w_0, w_1)$ is $1$-hidden Delzant. Then $w_1 = (T^*)^{-1} v_1 = \begin{pmatrix} n-2 \\ 1-n \end{pmatrix}$; since $n \geq 1$, this is impossible. Indeed, if $n=1$ this would give $w_1 = \begin{pmatrix} -1 \\ 0 \end{pmatrix}$, so $e_1$ would be vertical, which contradicts the fact that the marked point above $(w_0, w_1)$ lies in the interior of $\Delta$. And if $n \geq 2$, $w_1$ would point towards the exterior of $\Delta$, which contradicts its definition.
        \item Or $(w_1, w_2)$ is $1$-hidden Delzant. In this case, $w_1 = v_1 = \begin{pmatrix} -1 \\ 1-n \end{pmatrix}$. Necessarily $n \geq 2$, since if $n=1$, $e_1$ would be vertical, which contradicts the fact that the marked point lies in the interior of $\Delta$. Then $e_1$ is directed by $\vect{1-n}{1}$, and $w_2 = (T^*)^{-1} v_2 = \vect{1}{-1}$ so $e_2$ is directed by $\vect{-1}{-1}$. The vertex $(w_1, w_2)$ is $(\frac{\lambda}{n},\frac{\lambda}{n})$. This gives the polygon in Figure \ref{fig:type3-1b}; by setting $\lambda = n\beta$, we obtain the polygon in Figure \ref{fig:min_type3b}.
    \end{enumerate}

\begin{figure}
\begin{center} 

\begin{tikzpicture}
\filldraw[draw=black, fill=gray!60] (-3,0) -- (3,0) -- (-1,2) -- cycle; 
\draw (-1,1) node {$\times$}; 
\draw [dashed] (-1,1) -- (-1,2); 

\draw (-3,0) node[below] {$(0,0)$};
\draw (3,0) node[below] {$(\lambda,0)$};
\draw (-1,2) node[above] {$(\frac{\lambda}{n},\frac{\lambda}{n})$};
\end{tikzpicture}  
\end{center} 
\caption{The polygon in the case $p=3$ of type $(3)$ in the proof of Theorem \ref{thm:min_poly}.} 
\label{fig:type3-1b}
\end{figure}
    
    \item In the case $p=4$, we have $f = p - d = 1$ so there is exactly one $1$-fake vertex. There are three subcases.
    \begin{enumerate}
        \item Either $(w_0, w_1)$ is $1$-fake. Then $w_1 = (T^*)^{-1} v_0 = \vect{-1}{1}$, so $e_1$ is directed by $\vect{1}{1}$. Moreover $w_2 = (T^*)^{-1} v_1 = \vect{n-2}{1-n}$, so $e_2$ is directed by $\vect{1-n}{2-n}$. Additionally, $w_3 = (T^*)^{-1} v_2 = \vect{1}{-1}$, so $e_3$ is directed by $\vect{-1}{-1}$. The vertex $(w_1, w_2)$ is $(\lambda + \mu, \mu)$ for some $\mu > 0$. If $n=2$, the edge $e_2$ is horizontal and the vertex $(w_2, w_3)$ equals $(\mu, \mu)$. If $n \neq 2$, the vertex $(w_2, w_3)$ is $((2-n)\lambda + \mu, (2-n)\lambda + \mu)$ (so this formula still works for $n=2$). If $n = 1$, $e_2$ is vertical, see Figure \ref{fig:type3-2a-ninf2}; after changing the cut direction and choosing $\lambda = \beta$ and $\mu = \alpha$, we obtain the polygon in Figure \ref{fig:min_type3a} with $n=1$. When $n \geq 2$, for $\Delta$ to be a polygon, necessarily $\mu > (n-2) \lambda$, and then there are three cases according to the value of $\mu$:
        \begin{enumerate}
            \item if $(n-2) \lambda < \mu < (n-1) \lambda$, then $(w_2, w_3)$ lies to the left of the fake vertex $(w_0, w_1)$, see Figure \ref{fig:type3-2a-nsup2-muinf2lamb}. After changing the cut direction and setting $\lambda = \beta$ and $\mu = (n-1)\beta - \alpha$, we obtain the polygon in Figure \ref{fig:min_type3c};
            \item if $\mu = (n-1) \lambda$, then $(w_2, w_3)$ lies on top of the fake vertex $(w_0, w_1)$, see Figure \ref{fig:type3-2a-nsup2-mueq2lamb}. By changing the cut direction and choosing $\lambda = \beta$, we obtain the polygon in Figure \ref{fig:min_type3b}; 
            \item if $\mu > (n-1) \lambda$, then $(w_2, w_3)$ lies to the right of the fake vertex $(w_0, w_1)$, see Figure \ref{fig:type3-2a-nsup2-musup2lamb}. After changing the cut direction and setting $\lambda = \beta$ and $\mu = \alpha + (n-1) \beta$, we obtain the polygon in Figure \ref{fig:min_type3a}.
        \end{enumerate}
        \item Or $(w_1, w_2)$ is $1$-fake. Then $w_1 = v_1 = \vect{-1}{1-n}$. But $w_2 = (T^*)^{-1} v_1 = \vect{n-2}{1-n}$ so if $n = 1$, then $w_1 = \vect{-1}{0} = w_2$; hence necessarily $n \geq 2$. So $e_1$ is directed by $\vect{1-n}{1}$, $e_2$ is directed by $\vect{1-n}{2-n}$, and $w_3 = (T^*)^{-1} v_2 = \vect{1}{-1}$ so $e_3$ is directed by $\vect{-1}{-1}$. The vertex $(w_1, w_2)$ is $(\lambda - \mu, \frac{\mu}{n-1})$ for some $\mu > 0$, and the vertex $(w_2, w_3)$ is $((n-1) \mu - (n - 2)\lambda, (n-1) \mu - (n - 2)\lambda)$. For $\Delta$ to be a polygon, it is necessary that $\frac{n-2}{n-1} \lambda < \mu < \frac{n-1}{n} \lambda$. See Figure \ref{fig:type3-2b-ngeq2}. By setting $\lambda = n\beta - \alpha$ and $\mu = (n-1) \beta - \alpha$, we obtain the polygon in Figure \ref{fig:min_type3c}. 
        \item Or $(w_2, w_3)$ is $1$-fake. In this case $w_1 = v_1 = \vect{-1}{1-n}$, $w_2 = v_2 = \vect{0}{-1}$ and $w_3 = (T^*)^{-1} v_2 = \vect{1}{-1}$, so $e_1$ is directed by $\vect{1-n}{1}$, $e_2$ is horizontal and $e_3$ is directed by $\vect{-1}{-1}$. The case $n = 1$ is special because the edge $e_1$ is vertical; in this case $(w_1, w_2)$ equals $(\lambda, \mu)$ for some $0 < \mu < \lambda$, and $(w_2, w_3)$ equals $(\mu, \mu)$, see Figure \ref{fig:type3-2c-n1}; by choosing $\mu = \beta$ and $\lambda = \alpha + \beta$, we obtain the polygon in Figure \ref{fig:min_type3a} with $n=1$. If $n \geq 2$, then $(w_1, w_2)$ equals $(\lambda - \mu, \frac{\mu}{n-1})$ for some $\mu > 0$, and $(w_2, w_3)$ equals $(\frac{\mu}{n-1}, \frac{\mu}{n-1})$; since $\Delta$ is a polygon we have that $\mu < \frac{n-1}{n}\lambda$, see Figure \ref{fig:type3-2c-ngeq2}. By choosing $\lambda = \alpha + n\beta$ and $\mu = (n-1)\beta$, we get the polygon in Figure \ref{fig:min_type3a}.\qedhere
    \end{enumerate}
\begin{figure}
\begin{center} 

\begin{subfigure}[b]{.4\linewidth} 
\centering 
\begin{tikzpicture}
\filldraw[draw=black, fill=gray!60] (-1,0) -- (1,0) -- (2.5,1.5) -- (2.5,3.5) -- cycle; 

\draw (1,0.8) node {$\times$}; 
\draw [dashed] (1,0.8) -- (1,0);

\draw (-1,0) node[below] {$(0,0)$};
\draw (1,0) node[below] {$(\lambda,0)$};
\draw (2.5,1.5) node[right] {$(\lambda + \mu,\mu)$};
\draw (2.5,3.5) node[above] {$(\lambda + \mu, \lambda + \mu)$};
\end{tikzpicture} 
\caption{Subcase $n=1$.} 
\label{fig:type3-2a-ninf2}
\end{subfigure}  
\hspace{1cm}
\begin{subfigure}[b]{.4\linewidth} 
\centering 
\begin{tikzpicture}
\filldraw[draw=black, fill=gray!60] (-1,0) -- (1,0) -- (4,3) -- (0,1) -- cycle; 

\draw (1,1.1) node {$\times$}; 
\draw [dashed] (1,1.1) -- (1,0); 

\draw (-1,0) node[below] {$(0,0)$};
\draw (1,0) node[below] {$(\lambda,0)$};
\draw (4,3) node[above] {$(\lambda + \mu,\mu)$};
\draw (0,1) node[above left] {$(a_n, a_n)$};
\end{tikzpicture} 
\caption{Subcase 2.(a).i.} 
\label{fig:type3-2a-nsup2-muinf2lamb}
\end{subfigure}  
\hspace{1cm} 
\begin{subfigure}[b]{.4\linewidth} 
\centering 
\begin{tikzpicture}
\filldraw[draw=black, fill=gray!60] (-1,0) -- (1,0) -- (5,4) -- (1,2) -- cycle; 

\draw (1,1.3) node {$\times$}; 
\draw [dashed] (1,1.3) -- (1,0); 

\draw (-1,0) node[below] {$(0,0)$};
\draw (1,0) node[below] {$(\lambda,0)$};
\draw (5,4) node[above] {$(n\lambda, n\lambda)$};
\draw (1,2) node[above left] {$(\lambda, \lambda)$};
\end{tikzpicture} 
\caption{Subcase 2.(a).ii.} 
\label{fig:type3-2a-nsup2-mueq2lamb}
\end{subfigure}  
\hspace{1cm} 
\begin{subfigure}[b]{.4\linewidth} 
\centering 
\begin{tikzpicture}
\filldraw[draw=black, fill=gray!60] (-1,0) -- (1,0) -- (6,5) -- (2,3) -- cycle; 

\draw (1,0.6) node {$\times$}; 
\draw [dashed] (1,0.6) -- (1,0); 

\draw (-1,0) node[below] {$(0,0)$};
\draw (1,0) node[below] {$(\lambda,0)$};
\draw (6,5) node[above] {$(\lambda + \mu,\mu)$};
\draw (2,3) node[above left] {$(a_n, a_n)$};
\end{tikzpicture} 
\caption{Subcase 2.(a).iii.} 
\label{fig:type3-2a-nsup2-musup2lamb}
\end{subfigure}

\end{center} 
\caption{The polygons in the case $p=4$, $(w_0, w_1)$ fake of type $(3)$ in the proof of Theorem \ref{thm:min_poly}. Here $a_n = (2-n)\lambda + \mu$.} 
\end{figure}
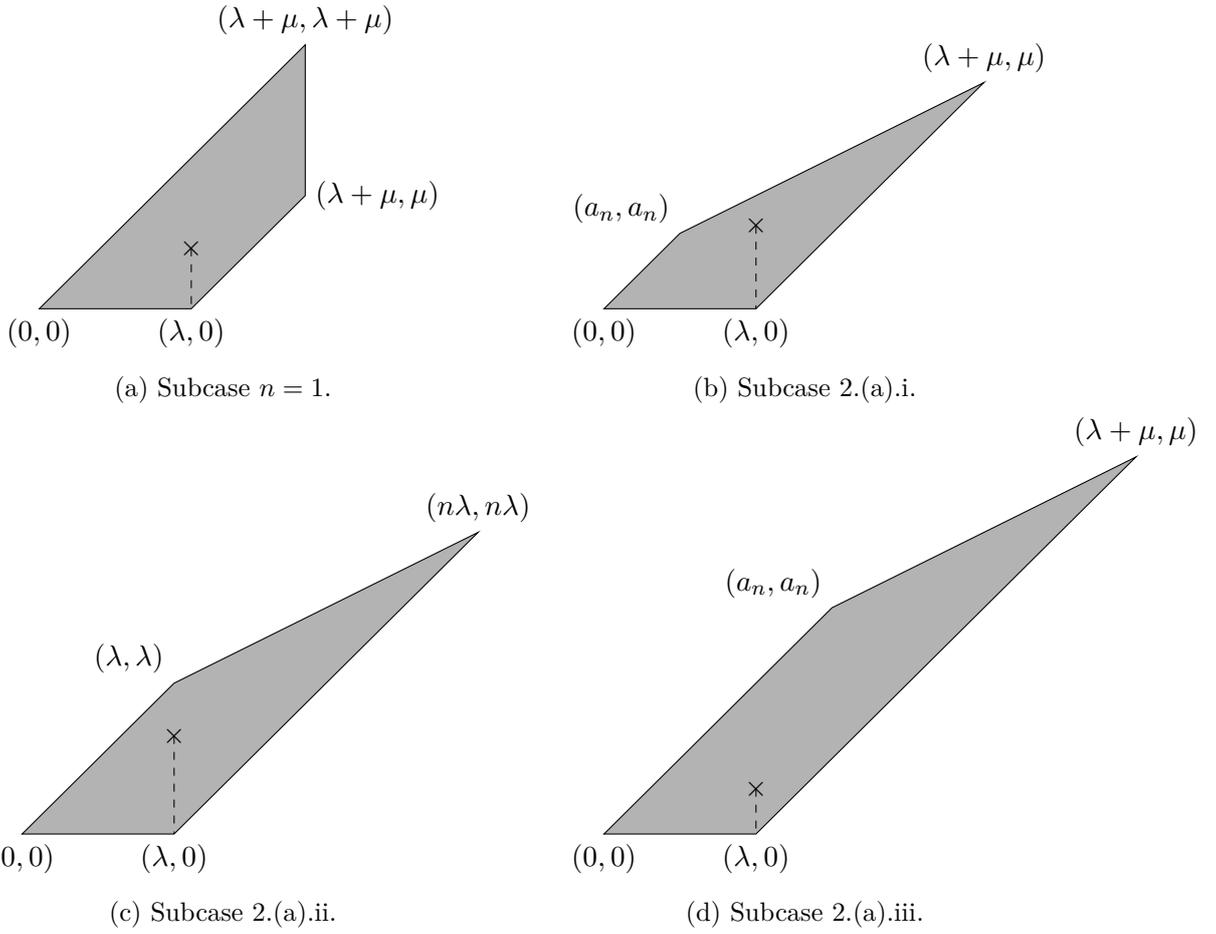   

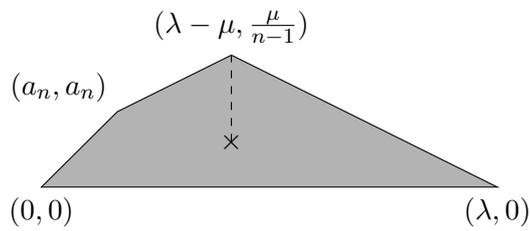
\begin{figure}
\begin{center}

\begin{tikzpicture}
\filldraw[draw=black, fill=gray!60] (-3,0) -- (3,0) -- (-0.5,1.75) -- (-2,1) -- cycle; 

\draw (-0.5,0.6) node {$\times$}; 
\draw [dashed] (-0.5,0.6) -- (-0.5,1.75); 

\draw (-3,0) node[below] {$(0,0)$};
\draw (3,0) node[below] {$(\lambda,0)$};
\draw (-0.5,1.75) node[above] {$(\lambda - \mu, \frac{\mu}{n-1})$};
\draw (-2,1) node[above left] {$(a_n,a_n)$};
\end{tikzpicture} 

\end{center}
\caption{The polygon in the case $p=4$, $(w_1, w_2)$ fake of type $(3)$ in the proof of Theorem \ref{thm:min_poly}. Here $a_n = (n-1) \mu - (n - 2)\lambda$. Recall that $n \geq 2$. }
\label{fig:type3-2b-ngeq2}
\end{figure}
    
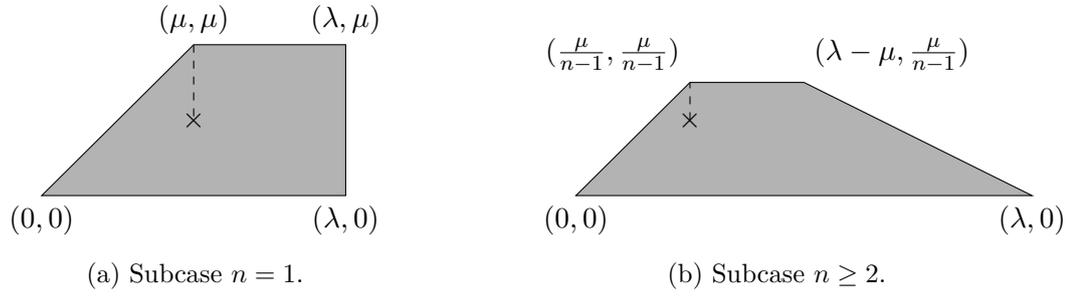
\begin{figure}
\begin{center} 

\begin{subfigure}[b]{.4\linewidth} 
\centering 
\begin{tikzpicture}
\filldraw[draw=black, fill=gray!60] (-2,0) -- (2,0) -- (2,2) -- (0,2) -- cycle; 

\draw (0,1) node {$\times$}; 
\draw [dashed] (0,2) -- (0,1);

\draw (-2,0) node[below] {$(0,0)$};
\draw (2,0) node[below] {$(\lambda,0)$};
\draw (2,2) node[above] {$(\lambda,\mu)$};
\draw (0,2) node[above] {$(\mu, \mu)$};
\end{tikzpicture} 
\caption{Subcase $n=1$.} 
\label{fig:type3-2c-n1}
\end{subfigure}  
\hspace{1cm}
\begin{subfigure}[b]{.4\linewidth} 
\centering 
\begin{tikzpicture}
\filldraw[draw=black, fill=gray!60] (-3,0) -- (3,0) -- (0,1.5) -- (-1.5,1.5) -- cycle; 

\draw (-1.5,1) node {$\times$}; 
\draw [dashed] (-1.5,1) -- (-1.5,1.5);

\draw (-3,0) node[below] {$(0,0)$};
\draw (3,0) node[below] {$(\lambda,0)$};
\draw (0,1.5) node[above right] {$(\lambda - \mu,\frac{\mu}{n-1})$};
\draw (-1.5,1.5) node[above left] {$(\frac{\mu}{n-1}, \frac{\mu}{n-1})$};
\end{tikzpicture} 
\caption{Subcase $n \geq 2$.} 
\label{fig:type3-2c-ngeq2}
\end{subfigure}  

\end{center} 
\caption{The polygons in the case $p=4$, $(w_2, w_3)$ fake of type $(3)$ in the proof of Theorem \ref{thm:min_poly}.} 
\end{figure}   
\end{enumerate}
\end{proof}

\subsection{Minimal toric systems}
\label{subsect:min_toric}

Theorem \ref{thm:min_poly} tells us that if one starts with any semitoric system and performs all possible toric and semitoric type blowdowns, the resulting system is either semitoric of type (1), (2a), (2b), (3a), (3b) or (3c), or of toric type. In the latter case, this system will necessarily be semitoric isomorphic to a toric system which does not admit any blowdown (i.e. which is minimal). The classification of these minimal toric systems is a classical results, which for instance appears in \cite{Ful}. We describe them now; in order to do so, we display their Delzant polygons in Figure \ref{fig:minimal_toric} and use the Delzant construction to obtain a corresponding toric system.

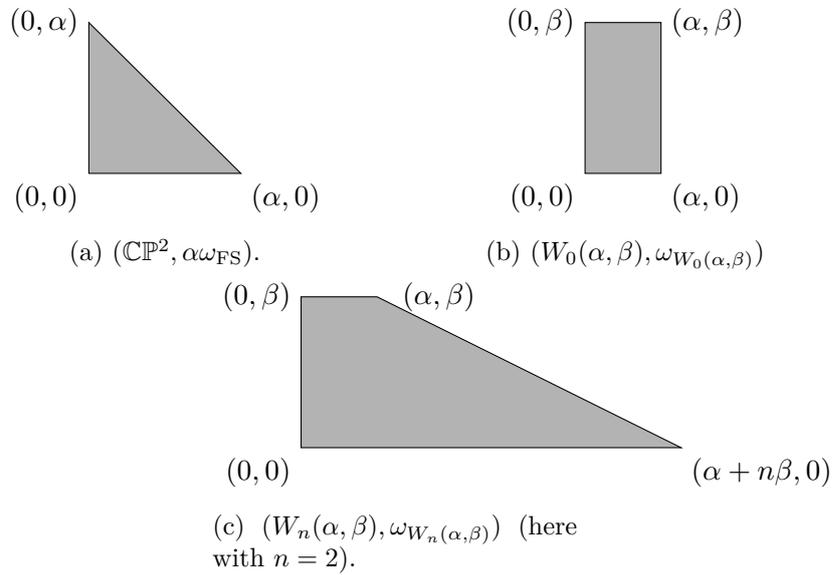
\begin{figure}
\begin{center} 
\def\dashlength{0.08}

\begin{subfigure}[b]{.3\linewidth} 
\centering 
\begin{tikzpicture}
\filldraw[draw=black, fill=gray!60] (0,0) -- (0,2) -- (2,0) -- cycle;
\draw (0,0) node[below left] {$(0,0)$};
\draw (0,2) node[left] {$(0,\alpha)$};
\draw (2,0) node[below right] {$(\alpha,0)$};
\end{tikzpicture}
\caption{$(\C\P^2, \alpha \omega_{\text{FS}})$.} 
\label{fig:min_toric_CP2}
\end{subfigure}  
\hspace{1cm} 
\begin{subfigure}[b]{.3\linewidth} 
\centering 
\begin{tikzpicture}
\filldraw[draw=black, fill=gray!60] (0,0) -- (0,2) -- (1,2) -- (1,0) -- cycle;
\draw (0,0) node[below left] {$(0,0)$};
\draw (0,2) node[left] {$(0,\beta)$};
\draw (1,2) node[right] {$(\alpha,\beta)$};
\draw (1,0) node[below right] {$(\alpha,0)$};
\end{tikzpicture}
\caption{$(W_0(\alpha,\beta),\omega_{W_0(\alpha,\beta)})$} 
\label{fig:min_toric_W0}
\end{subfigure} 
\hspace{1cm} 
\begin{subfigure}[b]{.3\linewidth} 
\centering 
\begin{tikzpicture}
\filldraw[draw=black, fill=gray!60] (0,0) -- (0,2) -- (1,2) -- (5,0) -- cycle;
\draw (0,0) node[below left] {$(0,0)$};
\draw (0,2) node[left] {$(0,\beta)$};
\draw (1.2,2) node[right] {$(\alpha,\beta)$};
\draw (5,0) node[below right] {$(\alpha + n \beta,0)$};
\end{tikzpicture}
\caption{$(W_n(\alpha,\beta),\omega_{W_n(\alpha,\beta)})$ (here with $n=2$).}
\label{fig:min_toric_Wn}
\end{subfigure} 

\end{center} 
\caption{Delzant polygons for the toric systems of Theorem \ref{thm:min_toric_models}.} 
\label{fig:minimal_toric} 
\end{figure}

We need to introduce some notation to describe these toric systems. Let $\omega_{\text{FS}}$ be the Fubini-Study form on $\C\P^2$, normalized such that the Liouville volume of $\C\P^2$ equals $2 \pi^2$. Then for $\alpha > 0$, $(\C\P^2, \alpha \omega_{\text{FS}})$ can be obtained as the symplectic reduction of $\C^3$ with its standard symplectic form with respect to the $\T^2$-action generated by
\[ N = \frac{1}{2} \left( |z_1|^2 + |z_2|^2 + |z_3|^2 \right) \]
at level $\alpha$.

Moreover, for $n \geq 0$ and $\alpha, \beta > 0$, we define the $n$-th Hirzebruch surface $(W_n(\alpha,\beta),\omega_{W_n(\alpha,\beta)})$ as the symplectic reduction of $\C^4$ with its standard symplectic form with respect to the $\T^2$-action generated by 
\[ N(z_1,z_2,z_3,z_4) = \frac{1}{2} \left( |z_1|^2 + |z_2|^2 + n |z_3|^2, |z_3|^2 + |z_4|^2 \right) \]
at level $(\alpha + n \beta, \beta)$ (there are of course other equivalent definitions of Hirzebruch surfaces, but this is the most suited to our needs). Note that for $n=0$, $(W_0(\alpha,\beta),\omega_{W_0(\alpha,\beta)})$ identifies with the product $S^2 \times S^2$ of two spheres, each equipped with its standard symplectic form rescaled in such a way that the volume of the first (respectively second) factor is $\alpha$ (respectively $\beta$). Since this case is a bit special, we single it out in Figure \ref{fig:minimal_toric} and in the following statement.

Recall that two (four-dimensional) toric systems $(M,\om,F)$ and $(M',\om',F')$ are \emph{isomorphic} if there exists a symplectomorphism $\phi:M \to M'$ such that $\phi^* F' =F$. 
The group $\mathrm{GL}(2,\Z)$ acts on toric integrable systems by $A\cdot (M,\om,F) = (M,\om,A\circ F)$ for $A\in\mathrm{GL}(2,\Z)$ and $(M,\om,F)$ toric; this induces an action on the Delzant polygons which coincides with the standard action of $\mathrm{GL}(2,\Z)$.

\begin{thm}[\cite{Ful}]
\label{thm:min_toric_models}
Any minimal four-dimensional toric system is isomorphic, up to the action of $\mathrm{GL}(2,\Z)$, to one of the following: 
\begin{itemize}
    \item $(\C\P^2, \alpha \omega_{\text{FS}})$ with the toric momentum map 
    \begin{equation} \left(J_{\C\P^2},H_{\C\P^2}\right) = \left( \frac{1}{2} |z_1|^2, \frac{1}{2} |z_2|^2 \right), \label{eq:CP2_std_toric} \end{equation}
    whose Delzant polygon is displayed in Figure \ref{fig:min_toric_CP2};
    \item $(W_0(\alpha,\beta),\omega_{W_0(\alpha,\beta)})$ with the toric momentum map
    \begin{equation} \left(J_{W_0(\alpha,\beta)},H_{W_0(\alpha,\beta)}\right) = \left( \frac{1}{2} |z_2|^2, \frac{1}{2} |z_3|^2 \right), \label{eq:W0_std_toric} \end{equation}
    whose Delzant polygon is displayed in Figure \ref{fig:min_toric_W0};
    \item $(W_n(\alpha,\beta),\omega_{W_n(\alpha,\beta)})$, for some $n \geq 2$, with the toric momentum map
    \begin{equation} \left(J_{W_n(\alpha,\beta)},H_{W_n(\alpha,\beta)}\right) = \left( \frac{1}{2} |z_2|^2, \frac{1}{2} |z_3|^2 \right) \label{eq:Wn_std_toric}, \end{equation}
    whose Delzant polygon is displayed in Figure \ref{fig:min_toric_Wn}.
\end{itemize} 
\end{thm}

\subsection{Minimal systems of types (4), (5), and (6)}
\label{subsect:min_456}

In the course of the proof of Proposition \ref{prop:strict-min-helix}, we saw that a semitoric system with semitoric helix of type (4), (5) or (6) always admits a semitoric type blowdown. In fact, we can perform a complete blowdown (see Definition \ref{def:semitoricblowupdown}) to such a system to remove all the focus-focus points and thus obtain a system of toric type.

In this section, we show that via a sequence of semitoric type blowdowns, each of the systems of type (4), (5) or (6) can be reduced to a system which is isomorphic (as a semitoric system) to a familiar toric minimal model; this is of course not surprising but we also determine which of these models appears for each type. We do this by using the helix to find the semitoric polygon of the toric type system that each of these systems can be reduced to, and then use the fact that any system of toric type is isomorphic (as a semitoric system) to a toric system, which can be read off of the semitoric polygon. In particular the Hamiltonian $S^1$-spaces underlying these two systems are isomorphic.

\begin{prop}
\label{prop:min_toric_456}
Let $(M,\omega,F)$ be a semitoric system with semitoric helix $(d,s,[v])$ of type (4), (5) with $n \geq 3$ or (6) with $n \notin \{2,s\}$. Then after performing the complete left semitoric type blowdown of $(M,\omega,F)$, we obtain a system which is semitoric isomorphic to the toric system $(M',\omega',F')$ where:
\begin{itemize}
    \item for type (4), $(M',\omega',F')$ is $(\CP^2,\alpha \omega_{\text{FS}})$, for some $\alpha > 0$, equipped with the Hamiltonian $\T^2$-action with toric momentum map $\left(J_{\C\P^2},H_{\C\P^2}\right)$ given in Equation \eqref{eq:CP2_std_toric};
    \item for type (5), $(M',\omega',F')$ is the $(n-1)$-th Hirzebruch surface $(W_{n-1}(\alpha,\beta),\omega_{W_{n-1}(\alpha,\beta)})$, for some $\alpha, \beta > 0$, equipped with the Hamiltonian $\T^2$-action with toric momentum map $\left(J_{W_{n-1}(\alpha,\beta)},H_{W_{n-1}(\alpha,\beta)}\right)$ given in Equation \eqref{eq:Wn_std_toric};
    \item for type (6), there exists $\alpha, \beta > 0$ such that $(M',\omega',F')$ is
    \begin{itemize}
        \item $(W_0(\alpha,\beta),\omega_{W_0(\alpha,\beta)})$ equipped with the Hamiltonian $\T^2$-action with toric momentum map $\left(J_{W_0(\alpha,\beta)},H_{W_0(\alpha,\beta)}\right)$ given in Equation \eqref{eq:W0_std_toric} if $n=1$;
        \item $(W_{n-1}(\alpha,\beta),\omega_{W_{n-1}(\alpha,\beta)})$ equipped with the Hamiltonian $\T^2$-action with toric momentum map $\left(H_{W_{n-1}(\alpha,\beta)},J_{W_{n-1}(\alpha,\beta)}\right)$, with $J_{W_{n-1}(\alpha,\beta)}$ and $H_{W_{n-1}(\alpha,\beta)}$ given in Equation \eqref{eq:Wn_std_toric}, if $n \geq 3$;
        \item $(W_{1-n}(\alpha,\beta),\omega_{W_{1-n}(\alpha,\beta)})$ equipped with the Hamiltonian $\T^2$-action with toric momentum map $\left(-H_{W_{1-n}(\alpha,\beta)},J_{W_{1-n}(\alpha,\beta)}\right)$, with $J_{W_{1-n}(\alpha,\beta)}$ and $H_{W_{1-n}(\alpha,\beta)}$ given in Equation \eqref{eq:Wn_std_toric}, if $n \leq 0$.
    \end{itemize}
\end{itemize}
\end{prop}

\begin{rmk}
While semitoric systems $(M,\omega,(J,H))$ with helices of type (4) and (5) have only one fixed surface for $J$, corresponding to the minimum of $J$, systems with helices of type (6) have two such fixed surfaces, and thus one can choose to perform either left or right semitoric type blowdowns to these systems. Note that the above statement only deals with the outcome of the complete left semitoric type blowdown.
\end{rmk}

\begin{proof}[Proof of Proposition \ref{prop:min_toric_456}]
In each of these cases, there exists a representative of the helix with $v_0 = \begin{pmatrix} 1 \\ 0 \end{pmatrix}$. Thus, by Proposition~\ref{prop:blowdown-helix} we can perform $s$ semitoric type blowdowns to obtain a system whose helix still has the same $d$, has the same $v_0, \ldots, v_{d-1}$, but now has $s=0$.

Additionally, in each of the following sections, we use the fact that for a semitoric helix of the form $(d,0,[v])$, the vectors $v_0, \ldots, v_{d-1}$ are actually the counter-clockwise ordered inwards pointing normal vectors of a representative of the corresponding semitoric polygon. Therefore, it is easy to obtain the semitoric polygon from the helix in this case. In each case, this polygon is the Delzant polygon of a toric system that we identify.

\paragraph{Case 1: Type (4).}
Recall that a representative for the helix of type (4) is given by
\[ v_0 = \vect{1}{0}, v_1 = \vect{0}{1}, v_2 = \vect{-1}{-1}, \]
$d=3$, and $s\in\Z_{\geq 1}\setminus \{2\}$. So the helix that we obtain after performing $s$ semitoric type blowdowns is of the form $(3,0,[w])$ with $w_0 = v_0$, $w_1 = v_1$ and $w_2 = v_2$. The associated marked semitoric polygon thus has no marked points and has a representative with vertices at $(0,0)$, $(\alpha,0)$, $(0,\alpha)$ for some $\alpha>0$, in other words the polygon in Figure \ref{fig:min_toric_CP2}.

\paragraph{Case 2: Type (5).}
A representative for the helix of type (5) is given by
\[  v_0 = \vect{1}{0}, v_1 = \vect{0}{1}, v_2 = \vect{-1}{1-n}, v_3 = \vect{0}{-1}, \]
$d=4$, $s \in\Z_{\geq 2}$, and parameter $n \geq 3$. The resulting helix after performing $s$ blowdowns is of the form $(4,0,[w])$ with $w_0 = v_0, \ldots, w_3 = v_3$. Therefore, the corresponding semitoric polygon is a trapezoid with vertices $(0,0)$, $(\alpha + (n-1) \beta,0)$, $(\alpha,\beta)$ and $(0,\beta)$ for some $\alpha, \beta > 0$. This is the polygon in Figure \ref{fig:min_toric_Wn} where $n$ is replaced by $n-1$.

\paragraph{Case 3: Type (6).}
The helix for type (6) is given by
\[
 v_0 = \vect{1}{0}, v_1 = \vect{0}{1}, v_2 = \vect{-1}{0}, v_3 = \vect{1-n}{-1},
\]
$d=4$, $s\in\Z_{>0}$, and parameter $n\in\Z\setminus\{2,s\}$. The helix obtained after performing the complete left semitoric type blowdown is $(4,0,[w])$ with $w_0 = v_0, \ldots, w_3 = v_3$.  Observe that this helix is the $J$-reflection of the same helix but with parameter $2-n$ instead of $n$ (this is only true because $s=0$ and the same reasoning could not apply to helices of type (6), for which $s \geq 1$), so by Corollary \ref{cor:reflections_preimage} we only need to investigate the case $n \geq 1$. Since $n\neq 2$, two cases remain:
\begin{itemize}
    \item if $n=1$, the corresponding semitoric polygon is a rectangle with vertices $(0,0)$, $(\alpha,0)$, $(\alpha,\beta)$ and $(0,\beta)$ for some $\alpha, \beta > 0$, i.e. the polygon in Figure \ref{fig:min_toric_W0}; 
    \item if $n \geq 3$, this semitoric polygon is a trapezoid with vertices $(0,0)$, $(\beta,0)$, $(\beta,\alpha)$ and $(0,\alpha + (n-1)\beta)$ for some $\alpha, \beta > 0$. In other words, it is the image of the polygon in Figure \ref{fig:min_toric_Wn}, where $n$ is replaced by $n-1$, under the action of $\begin{pmatrix} 0 & 1 \\ 1 & 0 \end{pmatrix} \in \text{GL}(2,\Z)$. \qedhere
\end{itemize}
\end{proof}

\subsection{The marked polygon isomorphism}
\label{subsec:marked_poly_iso}

Recall that the general question that we investigate in this paper is if one can find a procedure which, given a marked semitoric polygon, produces a semitoric system with this polygon as its invariant. Since the marked semitoric polygon contains only partial information about the isomorphism class of the system (because it does not include the Taylor series or twisting index invariant), it is natural to ask the following question: how are two semitoric systems with the same marked semitoric polygon related? This is particularly relevant for systems of type (1), (2) and (3), since we obtain an explcit example of each of these types by combining the already existing systems and the ones that we introduced in this paper.

In this section, we investigate this question. Its resolution will naturally involve results about the Hamiltonian $S^1$-space $(M,\omega,J)$ underlying a semitoric system $(M,\omega,(J,H))$. Recall from Section \ref{sec:S1-actions} that the Karshon graph of this Hamiltonian $S^1$-space can be read from the marked semitoric polygon of the semitoric system, but of course, there are distinct marked semitoric polygons which produce the same Karshon graph. Hence we want to know which data is lost when passing from the marked semitoric polygon to the Karshon graph.
The first step, which is the object of the next result, is to understand which information needs to be added to the Karshon graph of $(M,\om,J)$ to obtain the unmarked semitoric polygon of $(M,\om,(J,H))$.

\begin{lm}
\label{lm:graph_to_poly}
Let $(M,\omega,(J,H))$ be a semitoric system. Then the unmarked semitoric polygon of $(M,\omega,(J,H))$ is completely determined by the Karshon graph of $(M,\omega,J)$, together with the following extra data: each interior regular vertex is labeled by whether the corresponding fixed point lies in the top boundary $\partial^+(F(M))$, the bottom boundary $\partial^-(F(M))$, or the interior $\mathrm{int}(F(M))$ of $F(M)$.
\end{lm}

\begin{proof}
Recall from Section~\ref{sec:S1-actions} that the edges of the Karshon graph correspond to certain chains of edges of the polygon, and the vertices of the Karshon graph correspond to either vertices of the polygon or marked points of the polygon. 

We now have three observations. First, the lengths of any vertical walls are determined by the area labels on the corresponding fat vertices. Second, suppose that $e_1$ and $e_2$ are two consecutive edges on the top or bottom boundary of a representative $\Delta$ of the marked semitoric polygon, and are connected by a Delzant or hidden Delzant vertex, and suppose that $e_2$ corresponds to a $\mathbb{Z}_k$-sphere. Then the slope of $e_1$ and the type of the vertex between these edges determines the slope of $e_2$. Third, notice that the Karshon graph determines the edges emanating from the minimum (or maximum) of $J$ up to an overall action of $T$. Indeed:
\begin{itemize}
 \item if this minimum corresponds to a single point $p$, then the slopes of these edges are (up to $T$) determined by the weights of the $S^1$-action at $p$, which are included in the Karshon graph;
 \item if this minimum corresponds to a fixed surface (inducing a vertical wall in $\Delta$) then the slopes of the edges emanating from this wall must be integers, and the difference of these two integers is the derivative of the Duistermaat-Heckman function of the $S^1$-action. Since the Karshon graph determines the $S^1$-space up to isomorphism, this derivative is determined by the Karshon graph, and therefore the edges are again determined up to an action of $T$.
\end{itemize}

Therefore, given a Karshon graph coming from a semitoric system, the information needed to determine the polygon of the semitoric system is which of these features of the Karshon graph correspond to the top boundary, bottom boundary, or interior of the polygon, by starting at the minimum of $J$ and adding the edges corresponding to these features one at a time from left to right. Which of these features correspond to the top boundary, bottom boundary, or interior, in turn, is completely determined by which fixed points in $J^{-1}(\mathrm{int}(J(M)))$ are in $\partial^+(F(M))$, $\partial^-(F(M))$, or $\mathrm{int}(F(M))$. 
\end{proof}

\begin{rmk}
Note that not all possible labels discussed in Lemma \ref{lm:graph_to_poly} can actually occur in the Karshon graph $\Gamma$ of an $S^1$-space underlying a semitoric system. The restrictions are that any two interior vertices connected by an edge must have the same label, the only marked points which may appear in the interior of $F(M)$ must be connected to no edges, and (letting each edge be labeled by the labels of its endpoints) there is at most one component of the graph with either the top or bottom boundary label in each level set of $J$. Any such labeling corresponds to an unmarked semitoric polygon from a semitoric system which has $\Gamma$ as the Karshon graph of its underlying $S^1$-space.
\end{rmk}

 \begin{thm}[Unmarked polygon isomorphism] 
 \label{thm:poly-iso} 
 Two semitoric systems $(M,\om,(J,H))$ and $(M',\om',(J',H'))$ have the same unmarked semitoric polygon if and only if there exists a symplectomorphism $\Phi\colon M\to M'$ such that $\Phi^*J'=J$ and for all fixed points $p\in M^{S^1}$ that are not a global maximum or minimum of $J$, $p$ is a maximum (respectively minimum) of $H$ on $J^{-1}(J(p))$ if and only if $\Phi(p)$ is a maximum (respectively minimum) of $H'$ on $J'^{-1}(J'(p))$.
\end{thm}
 
 \begin{proof}
  Suppose the systems have the same unmarked semitoric polygon. Since the Karshon graph of the underlying $S^1$-space can be obtained from the polygon, as explained in~\cite{HSS} and recalled in Section \ref{sec:S1-actions}, we conclude that there is an isomorphism between the Karshon graphs of $(M,\om,J)$ and $(M',\om',J')$. Note that there exists such an isomorphism which respects the labels introduced in the statement of Lemma \ref{lm:graph_to_poly}. By~\cite[Theorem 4.1]{karshon} this isomorphism of graphs lifts to a symplectomorphism $\Phi$ between
  $(M,\om)$ and $(M',\om')$ such that $\Phi^*J'=J$. Since the isomorphism of Karshon graphs preserved the labeling of each vertex of the graph, the symplectomorphism $\Phi$ also sends fixed points $p\in M^{S^1}$ which are maximums (respectively minimums) of $H$ on $J^{-1} (J(p))$ onto the set of points $p'\in (M')^{S^1}$ which are maximums (respectively minimums) of $H'$ on $(J')^{-1} (J'(p'))$, as in the statement of the theorem.
  
  Conversely, suppose such a $\Phi$ exists. Then $(M,\om,J)$ and $(M',\om',J')$ are isomorphic as Hamiltonian $S^1$-spaces, and thus have isomorphic Karshon graphs. Thus, the semitoric polygons of $(M,\om,(J,H))$ and $(M',\om',(J',H'))$ both induce the same Karshon graph using the process described in Section \ref{sec:S1-actions}. We know from Lemma \ref{lm:graph_to_poly} that the additional information needed to determine the unmarked semitoric polygon is which fixed points are in the upper or lower boundary of the momentum map image, and which are in the interior. This, in turn, is determined by which fixed points are the maximum, the minimum, or neither of $H$ (and $H'$)
  restricted to level sets of $J$ (and $J'$). 
 \end{proof}

\begin{rmk}
If the minimum or maximum of $J$ is attained along a fixed surface $\Sigma$, then the map $\Phi$ defined in the statement of Theorem \ref{thm:poly-iso} does not need
to preserve which points are the maximum or minimum of $H$ restricted to $\Sigma$.
\end{rmk}

The same argument gives the following result for marked semitoric polygons.

 \begin{cor}[Marked polygon isomorphism]
 \label{cor:marked_iso}
 Two semitoric systems $(M,\om,(J,H))$ and \\
 $(M',\om',(J',H'))$ have the same marked semitoric polygon if and only if there exists a symplectomorphism $\Phi\colon M\to M'$ such that $\Phi^*J'=J$ and  for all fixed points $p\in M^{S^1}$ that are not a global maximum or minimum of $J$, $p$ is a maximum (respectively minimum) of $H$ on $J^{-1}(J(p))$ if and only if $\Phi(p)$ is a maximum (respectively minimum) of $H'$ on $J'^{-1}(J'(p))$ and for each such $p$ which is not a maximum or minimum of $H$ on $J^{-1}(J(p))$ the volumes of $J^{-1}(J(p))\cap \{H<H(p)\}$ and $(J')^{-1}(J'(p))\cap \{H'<H'(p)\}$ are equal.
\end{cor}

\begin{proof}
 Theorem~\ref{thm:poly-iso} deals with the case of unmarked semitoric polygons, and the only extra information encoded in marked polygons is the height of each marked point, which is given by the volumes described in the statement.
\end{proof}

Suppose that $\Phi\colon (M,\om)\to (M',\om')$ is any symplectomorphism between symplectic manifolds, and suppose that $(M',\om',F')$ is an integrable system. Then $(M,\om, F = \Phi^*F')$ is also integrable, and moreover $p\in M$ is a singular point of $F$ if and only if $\Phi(p)$ is a singular point of $F'$, and they have the same type. This immediately implies the following.

\begin{lm}
\label{lm:symp_fam}
Let $(M,\omega)$ be a compact, four-dimensional symplectic manifold and let $(M',\omega',F'_t=(J,H'_t))$ be a semitoric family (see Definition \ref{def:semitoric-family}). If $\Phi: (M,\omega) \to (M',\omega')$ is a symplectomorphism, then $(M,\omega,F_t = \Phi^* F'_t)$ is a semitoric family with the same degenerate times. Moreover, if $(M',\omega',F'_t)$ is a (half-)semitoric transition family, then so is $(M,\omega,F_t)$.
\end{lm}

It is natural to wonder if there is an analogue to Corollary \ref{cor:marked_iso} for semitoric helices. Indeed, as already explained, it is in general difficult to find all marked semitoric polygons producing a given helix. However, it could be useful to understand how two semitoric systems with the same helix are related; for instance this could help restating Theorem \ref{thm:min_helices} in terms of minimal systems instead of minimal helices.

It would be interesting to investigate this ``helix isomorphism'' problem, but on the one hand, it is more complicated than the marked polygon isomorphism problem, and on the other hand, it is not needed for the purpose of this paper. Indeed, here we only needed to understand the semitoric systems with helices of type (1), (2) and (3) (see Definition \ref{def:types_helices}), but these are sufficiently simple that we were able to list all the corresponding polygons, see Definition \ref{def:min-poly-123} and Theorem \ref{thm:min_poly}.

We conjecture that the helix of a semitoric system determines, in particular, the $S^1$-equivariant diffeomorphism type of the system.
The idea of a potential proof of this conjecture is that the helix determines the Karshon graph with the $J$-labels and area labels removed.

\section{Obstructions to the existence of semitoric and half-semitoric transition families}
\label{sec:obstructions}

In~\cite{LFPfamilies} we constructed explicit examples of semitoric systems with specified semitoric polygons
by forming a semitoric transition family (see Definition~\ref{def:semitoric-transition-family}).
This strategy relies on being able to perform a Hamiltonian-Hopf bifurcation (in this situation also called a nodal trade)
when the focus-focus point in question meets the top and bottom boundary of the polygon. In this section we show that there are certain limitations to this technique; in Section \ref{sec:strategy} below we propose an adaptation of this general strategy that ultimately leads us to overcoming these limitations and to finding new explicit systems in Sections \ref{sec:CP2} and \ref{sec:type_3}.

\subsection{Semitoric transition families}

Here we extend~\cite[Lemma 3.13]{LFPfamilies} to the non-simple case:

\begin{lm}\label{lem:stfam-EE-nonsimple}
 Let $(M,\om,(J,H_t))$ be a semitoric transition family with transition point
 $p$. Then for all $t\in [0,1]$ all fixed points of the $S^1$-action
 in the set $J^{-1}(J(p))\setminus\{p\}$ are of focus-focus type.
\end{lm}

\begin{proof}
 The proof is very similar to the proof of~\cite[Lemma 3.13]{LFPfamilies}.
 Let $(M,\om,F_t = (J,H_t))$ be a semitoric transition family with transition point $p$ and
 transition times $t^-$, $t^+$.
 Suppose that $q\in J^{-1}(J(p))\setminus\{p\}$ is a fixed point of the $S^1$-action, and let $t_0 \in [0,1]$ be such that $(M,\omega,F_{t_0})$ is semitoric. By \cite[Lemma 2.15]{LFPfamilies}, $q$ is a rank zero singular point of $F_{t_0}$; suppose that it is not focus-focus, which means that it must be elliptic-elliptic.
 Since $q$ can never be a degenerate singular point of $(J,H_t)$, since it is not the transition point, we
 conclude that $q$ cannot change type. Thus, $q$ is a singular point of elliptic-elliptic type for all
 $t \in [0,1]$.
 Thus, $q$ is either the maximum or minimum of $(H_t)|_{J^{-1}(J(p))}$ for all $t$.
 Then either for $t>t^+$ or $t<t_-$ we have that $p$ and $q$ are both at the maximum or
 minimum of $(H_t)|_{J^{-1}(J(p))}$.
 This contradicts the fact that an elliptic-elliptic point is the only point in its fiber in a semitoric system.
\end{proof}

Now, we will use essentially the same strategy to show that the presence of certain
$\Z_k$-spheres can obstruct the existence of semitoric transition families.

\begin{prop}
\label{prop:non-stfam}
 Let $(M,\om,F)$ be a semitoric system and let $p\in M$ be a focus-focus point.
 If $J^{-1}(J(p))$ contains any elliptic-elliptic singular points of $F$ or any points
 with non-trivial isotropy of the $S^1$-action (i.e.~elements of $\Z_k$-spheres for $k \geq 2$) then
 $(M,\om,F)$ can not be obtained as the $t=t_0$ system, $t_0\in (t^-,t^+)$, in a semitoric transition family
 with $p$ as the transition point.
\end{prop}

\begin{proof}
The case that $J^{-1}(J(p))$ contains an elliptic-elliptic singular point 
follows from Lemma~\ref{lem:stfam-EE-nonsimple}.

Now we deal with the case that $J^{-1}(J(p))$ intersects a $\Z_k$-sphere $\Sigma$ with $k \geq 2$
at some non-fixed point (recall that the fixed points in a $\Z_k$-sphere are elliptic-elliptic).
Assume that there exists a semitoric transition family $(M,\omega,F_t=(J,H_t))$, with 
transition times $t^- < t^+$, and $t_0 \in (t^-,t^+)$ such that $H = H_{t_0}$. 
Recall that this means that 
\[ H_0(p) = \min \left\{ H_0(m) \ | \ m \in J^{-1}(J(p)) \right\}, \quad H_1(p) = \max \left\{ H_1(m) \ | \ m \in J^{-1}(J(p)) \right\}.  \]
We know $p \notin \Sigma$ since $p$ is a focus-focus point of $F_{t_0}$ and the only rank zero singular points in $\Sigma$ are elliptic-elliptic points.
For any $t \in [0,1]\setminus\{t^-, t^+\}$ by Lemma~\ref{lem:Zk-boundary} we have that $F_t(\Sigma) \subset \partial F_t(M)$.

Recall that for $t=0$ and $t=1$ the point $p$ is an elliptic-elliptic point, and further recall
that $F_0(p) \in \partial^- F_0(M)$ and $F_1(p)\in\partial^+ F_1(M)$. Since any elliptic-elliptic point in a semitoric system is the only point in its fiber and since $F_t(\Sigma)\subset \partial F_t(M)$, we conclude that
\[
 F_0(\Sigma) \subset \partial^+ F_0(M) \textrm{ and } F_1(\Sigma) \subset \partial^- F_1(M).
\]
We will complete the proof by showing that this is impossible.

For $t = t_0$ we either have that $F_{t_0}(\Sigma) \subset \partial^+ F_{t_0}(M)$ or $F_{t_0}(\Sigma) \subset \partial^- F_{t_0}(M)$; without loss of generality suppose that $F_{t_0}(\Sigma) \subset \partial^+ F_{t_0}(M)$.
Let $j\in \mathrm{int}(J(\Sigma))$, $m\in J^{-1}(j)\cap\Sigma$, and $x\in J^{-1}(j)\setminus\Sigma$, and consider the continuous function $f_{j,x}\colon t \mapsto H_t(m)- H_t(x)$. Recall that $p \notin \Sigma$ and that when $t>t^+$:
\begin{itemize}
    \item $p$ is the only point in its $F_t$-fiber;
    \item $F_t(p) \in \partial^- F_t(M)$;
    \item $F_t(\Sigma) \subset \partial F_t(M)$.
\end{itemize}
So necessarily $F_t(\Sigma) \subset \partial^+ F_t(M)$ for $t > t^+$, and thus $f_{j,x}(t) < 0$ for $t > t^+$. Since moreover $f_{j,x}(t_0) > 0$ because $F_{t_0}(\Sigma) \subset \partial^+ F_{t_0}(M)$, there exists $t_x \in [t_0,t^+]$ such that $f_{j,x}(t_x) = 0$, which means that $x$ is in the same $F_{t_x}$-fiber $S$ as $m$. Necessarily $t_x = t^+$ for all $x$, since for $t \in [t_0,t^+)$ the system $(M,\omega,F_t)$ is semitoric so by \cite[Theorem 1]{VNpoly} its fibers are connected and $x$ cannot be on the same connected component of $S$ as $m$.

Thus, for any $j \in \mathrm{int}(J(\Sigma))$ and any $x \in J^{-1}(j)\setminus\Sigma$ we have that $f_{j,x}(t^+)=0$; this means that
\[  \max \left\{ H_{t^+}(q) \ | \ q \in J^{-1}(j) \right\} = \min \left\{ H_{t^+}(q) \ | \ q \in J^{-1}(j) \right\}. \]
In other words, given any regular value of $j$ in $\mathrm{int}(J(\Sigma))$, $H_t^+$ is constant on the three-dimensional submanifold $J^{-1}(j)$, and thus all points of $J^{-1}(j)$ are singular points of $F_{t^+}$.
Therefore, if $\mathcal{J} \subset \mathrm{int}(J(\Sigma))$ is an open interval of regular values of $J$, then $J^{-1}(\mathcal{J})$ is a non-empty open set in $M$ consisting entirely of singular points of $F_{t^+}$, which contradicts the fact that $(J,H_{t^+})$ is integrable.
\end{proof}

It is convenient to translate the above statement into a statement about marked semitoric polygons. 

\begin{cor}
\label{cor:non-stfam-poly}
Let $(M,\om,F)$ be a semitoric system and let $p\in M$ be a focus-focus point. Let $(\De,\vec{c},\vec{\epsilon})$ be any representative of the marked semitoric polygon of $(M,\om,F)$, let $c_j$ be the marked point corresponding to $p$, and let $\ell$ be the vertical line passing through $c_j$. If $\ell$ contains a Delzant or hidden Delzant vertex of $\Delta$, or intersects an edge of $\Delta$ of non-integer slope (or one of its vertices), then $(M,\om,F)$ can not be obtained as the $t=t_0$ system, $t_0\in (t^-,t^+)$, in a semitoric transition family with $p$ as the transition point.  
\end{cor}

Note that whether or not a representative of the marked semitoric polygon of $(M,\om,F)$ satisfies the conditions in this statement does not depend on the choice of such a representative.

\begin{proof}[Proof of Corollary \ref{cor:non-stfam-poly}]
Recall from the construction of the marked semitoric polygon in Section \ref{subsect:msp} that a Delzant or hidden Delzant vertex of $\De$ corresponds to an elliptic-elliptic point of $(M,\om,F)$.
Furthermore, it follows from the results of \cite{HSS} that we recall in Section \ref{sec:S1-actions} that an edge of slope $\frac{a}{k}$, with $a,k\in \Z$, $k \geq 2$, and $\mathrm{gcd}(a,k)=1$ corresponds to a $\Z_k$-sphere of the underlying $S^1$-space $(M,\om, J)$. Using these correspondences, the result follows immediately from Proposition~\ref{prop:non-stfam}.
\end{proof}

In particular, if we inspect the Karshon graph of the underlying Hamiltonian $S^1$-space $(M,\omega,J)$ and find that the vertex corresponding to the fixed point $p$ of $J$ lies below or above an edge with label $k \geq 2$ (corresponding to a $\Z_k$-sphere), then $(M,\om,F)$ can not be obtained as the $t=t_0$ system, $t_0 \in (t^-,t^+)$, in a semitoric transition family with $p$ as the transition point.

We can apply Corollary \ref{cor:non-stfam-poly} to systems which are minimal of type (1).

\begin{cor}
\label{cor:type_(1)_notfam}
Let $(M, \omega, F = (J,H))$ be a semitoric system which is minimal of type (1) (see Definition \ref{def:min-poly-123} and Figure \ref{fig:min_type1}). Then 
$(M,\om,F)$ can not be obtained as the $t=t_0$ system, $t_0\in (t^-,t^+)$, in a semitoric transition family.
\end{cor}

\begin{proof}
This is a direct application of Corollary \ref{cor:non-stfam-poly} by looking at the polygon in Figure \ref{fig:min_type1} and noticing that the only marked point lies below an edge of slope $\frac{1}{2}$. Alternatively, one can look at the Karshon graph of the underlying Hamiltonian $S^1$-space, displayed in Figure \ref{fig:graph_CP2}.
\end{proof}

\begin{cor}
\label{cor:type_(3bc)_notfam}
Let $(M, \omega, F = (J,H))$ be a semitoric system whose polygon is of type (3b) with $n \geq 2$, or (3c) with $n \geq 3$ (see Definition \ref{def:min-poly-123} and Figures \ref{fig:min_type3b} and \ref{fig:min_type3c}). Then $(M,\om,F)$ can not be obtained as the $t=t_0$ system, $t_0\in (t^-,t^+)$, in a semitoric transition family.
\end{cor}

\begin{proof}
This is again a direct application of Corollary \ref{cor:non-stfam-poly}: in this case the only marked point in the polygon from Figure \ref{fig:min_type3b} or \ref{fig:min_type3c} lies below the vertex of an edge of slope $\frac{1}{n-1}$. This corresponds to a $\Z_{n-1}$-sphere, so a non-trivial isotropy sphere as soon as $n \geq 3$. When $n = 2$, in type (3b) the vertex over the marked point is hidden Delzant. Again, except for the $n=2$ case in type (3b) (see Remark \ref{rmk:3b_n2_noKarshon}) one can also inspect the Karshon graph of the underlying Hamiltonian $S^1$-space, see Figures \ref{fig:graph_3b} and \ref{fig:graph_3c}.
\end{proof}

\begin{rmk}
\label{rmk:3b_n2_noKarshon}
In the case that $n=2$ in type (3b) notice that the underlying Karshon graph (see Figure \ref{fig:graph_3b}) could be the Karshon graph of a semitoric transition family by taking both isolated vertices to correspond to focus-focus points, which is for instance the case for type (2b), see Figure \ref{fig:min_type2b}. Still, type (3b) with $n=2$ cannot come from a semitoric transition family since in such systems one of the isolated fixed points corresponds to an elliptic-elliptic point of the system.
\end{rmk}

\begin{figure}
\begin{center}

\begin{tikzpicture}[scale=1]

\draw (-1,0) -- (1,0);
\fill[black] (-1,0) circle (1/16); 
\fill[black] (1,0) circle (1/16); 
\draw (0,0) node[above] {$2$};

\fill[black] (0,-0.5) circle (1/16); 

\end{tikzpicture}

\end{center}
\caption{The Karshon graph of the Hamiltonian $S^1$-space underlying a semitoric system of type (1), which cannot be obtained as part of a semitoric transition family, see Corollary \ref{cor:type_(1)_notfam}. Indeed, this is prevented by the interior fixed point lying below an edge (here, coming from a $\Z_2$-sphere).}
\label{fig:graph_CP2}
\end{figure}
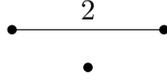

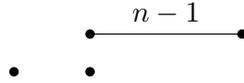
\begin{figure}
\begin{center}

\begin{tikzpicture}[scale=1]

\draw (-1,0) -- (1,0);
\fill[black] (-1,0) circle (1/16); 
\fill[black] (1,0) circle (1/16); 
\draw (0,0) node[above] {$n-1$};

\fill[black] (-2,-0.5) circle (1/16); 
\fill[black] (-1,-0.5) circle (1/16);

\end{tikzpicture}

\end{center}
\caption{The Karshon graph of the Hamiltonian $S^1$-space underlying a semitoric system of type (3b), which cannot be obtained as part of a semitoric transition family, see Corollary \ref{cor:type_(3bc)_notfam}.}
\label{fig:graph_3b}
\end{figure}

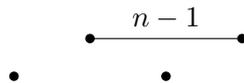
\begin{figure}
\begin{center}

\begin{tikzpicture}[scale=1]

\draw (-1,0) -- (1,0);
\fill[black] (-1,0) circle (1/16); 
\fill[black] (1,0) circle (1/16); 
\draw (0,0) node[above] {$n-1$};

\fill[black] (-2,-0.5) circle (1/16); 
\fill[black] (0,-0.5) circle (1/16);

\end{tikzpicture}

\end{center}
\caption{The Karshon graph of the Hamiltonian $S^1$-space underlying a semitoric system of type (3c), which cannot be obtained as part of a semitoric transition family if $n \geq 3$, see Corollary \ref{cor:type_(3bc)_notfam}. Indeed, this is prevented by the interior fixed point lying below an edge (here, coming from a $\Z_{n-1}$-sphere).}
\label{fig:graph_3c}
\end{figure}

\subsection{Half-semitoric transition families}

Because of the above obstructions, when looking for explicit systems of type (1), (3b) or (3c), we cannot use semitoric transition families as we did in \cite{LFPfamilies}. This is why we introduced half-semitoric transition families, see Definition \ref{def:half-semitoric-family}, which allow us to construct examples of these systems in Sections~\ref{sec:CP2} and~\ref{sec:type_3}. However, even with this new tool we cannot obtain every possible semitoric system, as we will see from the next statement.

\begin{prop}
\label{prop:non-half-stfam}
 Let $(M,\om,F)$ be a semitoric system and let $p\in M$ be a focus-focus point.
 If $J^{-1}(J(p))$:
 \begin{itemize}
  \item either contains two elliptic-elliptic singular points of $F$;
  \item or intersects a $\Z_k$-sphere and a $\Z_{k'}$-sphere of $J$ which are distinct, with $k, k'\geq 2$; 
  \item or contains an elliptic-elliptic point of $F$ and intersects a $\Z_k$-sphere of $J$ with $k\geq 2$,
  \end{itemize}
then $(M,\om,F)$ can not be obtained as the $t=t_0$ system, $t_0\in (t^-,t^+)$, in a half-semitoric transition family with $p$ as the transition point.
\end{prop}

\begin{proof}
Suppose that $(M,\om,F_t)$ is such a half-semitoric transition family with $(M,\om,F) = (M,\om,F_{t_0})$. Then at $t=0$ the point $p$ is an elliptic-elliptic point. Recall that the images of both elliptic-elliptic points and $\Z_k$-spheres (for $k\geq 2$) can only occur in the boundary of $F(M)$. Since an elliptic-elliptic point is a connected component of its $F_0$-fiber, and since components of two distinct $\Z_k$-spheres cannot be in the same fiber, we obtain a contradiction.
\end{proof}

\begin{rmk}
It would be interesting to investigate the converse statements of Propositions~\ref{prop:non-stfam} and~\ref{prop:non-half-stfam}. One possible approach would be to use half-semitoric families containing the strictly minimal semitoric systems, such as the families constructed in this paper, to produce half-semitoric families containing a given semitoric system not satisfying the conditions of the statement of Proposition~\ref{prop:non-half-stfam}.
\end{rmk}

\section{Semitoric systems of type (3) from toric type blowups and blowdowns}
\label{sec:blowup_down_type3}

In this section, we show how every semitoric system of type (3) can be obtained as part of a semitoric family with one degenerate time, by applying a sequence of alternative blowups and blowdowns to a fully explicit system. For systems of type (3a), this was proved in \cite{LFPfamilies}, but for systems of type (3b) and (3c), this is a new result, that we obtain by adapting the tools introduced in \cite{LFPfamilies}.

\subsection{Type (3a)}

We start by recalling what happens for systems of type (3a).

\begin{thm}[{\cite[Theorem 5.2]{LFPfamilies}}]
\label{thm:blowups_SF1}
For every $n \geq 1$ and every $\alpha,\beta > 0$, there exists a semitoric transition family $(W_n(\alpha,\beta),\omega_{W_n(\alpha,\beta)},F_t=(J,H_t))_{0 \leq t \leq 1}$ with transition times $t^- < t^+$ such that for $t \in (t^-,t^+)$, the marked semitoric polygon of $(W_n(\alpha,\beta),\omega_{W_n(\alpha,\beta)},F_t)$ is the one displayed in Figure \ref{fig:min_type3a} for some $h = h(t) \in (0,\beta)$. 
\end{thm}

This family is obtained by performing a sequence of toric type blowups and blowdowns on the coupled angular momenta system from \cite{SZ}, see Theorem 5.2 and its proof in \cite{LFPfamilies} for more details. Since it is a semitoric transition family, $t \mapsto h(t)$ describes the whole interval $(0,\beta)$ when $t$ varies in $(t^-,t^+)$, and hence we obtain a semitoric system for each marked semitoric polygon of type (3a). 

For types (3c) and (3b), we cannot apply the exact same reasoning as in the proof of \cite[Theorem 5.2]{LFPfamilies} because, as we already know from Corollary \ref{cor:type_(3bc)_notfam}, semitoric systems of these two types cannot be part of semitoric transition families. However, we will adapt our techniques to prove that they can be obtained as part of a nice semitoric family.

\subsection{Type (3c)}

Recall that there already exist explicit systems for every possible marked semitoric polygon of type (3c) with $n=2$ in the literature (see the discussion at the beginning of Section \ref{subsec:explicit_intro}). One example is the system from \cite[Theorem 7.5]{LFPfamilies} for $t \in (t^-,t^+)$ with $t^{\pm}$ as defined in the statement of that theorem. Another example is obtained by taking the parameters $s_1 = 0$ and $s_2 = t$ in the family of systems on $S^2 \times S^2$ described in \cite[Theorem 1.2]{HohPal}, and by choosing some $t$ in the interval $(t^-,t^+) \subset (0,1)$, where $t^-$ and $t^+$ will be described below. This is the starting system that we will use in this section because by taking $t=0$, we obtain a system of toric type with the desired polygon (see Figure \ref{fig:polygon_W0_3c_half}).
Note that taking a representative with a downwards cut is convenient since we will perform the blowups and blowdowns on the top boundary. We will now recall some of the features of this family of systems.

Endow $S^2 \times S^2$ with the Cartesian coordinates $(u_1,v_1,w_1,u_2,v_2,w_2)$ and with the symplectic form $-\left( R_1 \omega_{S^2} \oplus R_2 \omega_{S^2} \right)$, with $0 < R_1 < R_2$, where $\omega_{S^2}$ is the standard symplectic form on $S^2$. Then the system studied in \cite[Theorem 1.2]{HohPal}, for $s_1 = 0$ and $s_2 = t$, reads
\[ J = R_1 w_1 + R_2 w_2, \quad H_t = (1-t) w_1 + t \left( u_1 u_2 + v_1 v_2 - w_1 w_2 \right). \]
For our purpose it is convenient to translate this into a system on $(W_0(\alpha',\beta'),\omega_{W_0(\alpha',\beta')})$ for some well-chosen $\alpha', \beta' > 0$. Recall from Section \ref{subsect:min_toric} that $(W_0(\alpha',\beta'),\omega_{W_0(\alpha',\beta')})$ is the symplectic reduction of $\C^4$ by $N = \frac{1}{2} \left( |z_1|^2 + |z_2|^2, |z_3|^2 + |z_4|^2 \right)$ at level $(\alpha',\beta')$. From this one easily sees that $(W_0(\alpha',\beta'),\omega_{W_0(\alpha',\beta')})$ identifies to $(\C\P^1 \times \C\P^1, \alpha' \omega_{\text{FS}} \oplus \beta' \omega_{\text{FS}})$, after rescaling $(z_1,z_2)$ to $\sqrt{2\alpha'}(z_1,z_2)$ and $(z_3,z_4)$ to $\sqrt{2\beta'}(z_3,z_4)$. Then one can translate the above system by means of the stereographic projection $\pi_N$ (from the north pole to the equatorial plane) identifying $\C\P^1$ with $S^2$. One readily checks that $\pi_N^* \omega_{\text{FS}} = -\frac{1}{2} \omega_{S^2}$, so that $\omega_{W_0(\alpha',\beta')} = -\frac{1}{2} \left( \alpha' \omega_{S^2} \oplus \beta' \omega_{S^2} \right)$; hence we want to take $\alpha' = 2 R_1$ and $\beta' = 2 R_2$. In what follows we slightly abuse notation and still use $J$ and $H_t$.

\begin{lm}
\label{lm:system_W0_HP}
Let $0 < \alpha' < \beta'$. The family of systems $(W_0(\alpha',\beta'),\omega_{W_0(\alpha',\beta')},(J,H_t))$ where
\[ J = \frac{1}{4} \left( |z_2|^2 - |z_1|^2 \right) + \frac{1}{4} \left( |z_4|^2 - |z_3|^2 \right) = \frac{1}{2} \left( |z_2|^2 + |z_4|^2 \right) - \frac{1}{2} (\alpha' + \beta') \]
and 
\[ H_t = \frac{(1-t)}{2 \alpha'} \left( |z_2|^2 - |z_1|^2 \right) + \frac{t}{\alpha' \beta'} \left(\Re(z_1 \bar{z}_2 \bar{z}_3 z_4) - \frac{1}{4} \left( |z_2|^2 - |z_1|^2 \right) \left( |z_4|^2 - |z_3|^2 \right)  \right) \]
is a semitoric transition family with transition point $m = [\sqrt{2\alpha'},0,0,\sqrt{2\beta'}]$ and transition times 
\[ t^- = \frac{\beta'}{2\beta' + \alpha' + 2 \sqrt{\alpha' \beta'}}, \qquad t^+ = \frac{\beta'}{2\beta' + \alpha' - 2 \sqrt{\alpha' \beta'}}.  \]
More precisely, it is 
\begin{itemize}
    \item of toric type with (unmarked) semitoric polygon $[(\Delta_2(\alpha',\beta'),\varnothing,\varnothing)]$ as in Figure \ref{fig:polygon_W0_3c_0} when $0 < t < t^-$;
    \item semitoric with one focus-focus point $m$ and marked semitoric polygon $\left[\left( \Delta_2(\alpha',\beta'), (\beta',h(t)), -1 \right)\right]$ as in Figure \ref{fig:polygon_W0_3c_half} when $t^- < t < t^+$, where $h(t)$ is the height invariant of $m$ for the system $(W_0(\alpha',\beta'),\omega_{W_0(\alpha',\beta')},(J,H_t))$;
    \item of toric type with (unmarked) semitoric polygon $[(\tilde{\Delta}_2(\alpha',\beta'),\varnothing,\varnothing)]$ as in Figure \ref{fig:polygon_W0_3c_1} when $t^+ < t \leq 1$.
\end{itemize}
Furthermore, the system $(J,\frac{\alpha'}{2} H_0)$ is actually toric, with Delzant polygon $\Delta_2(\alpha',\beta')$.
\end{lm}

\begin{rmk}
It is tempting to work with $\tilde{H}_t = \frac{\alpha'}{2} H_t$ instead of $H_t$ so that $(J,\tilde{H}_0)$ is toric, but we prefer to keep the original system from \cite{HohPal} for the sake of clarity.    
\end{rmk}

\begin{figure}
\begin{center}
\begin{subfigure}[b]{.8\linewidth}
\begin{center}
\begin{tikzpicture}[scale=.9]
    \filldraw[draw=black,fill=gray!60] (0,0) -- (2,2) -- (5,2) -- (3,0) -- cycle;
    \draw (0,0) node[left] {$(0,0)$};
    \draw (2,2) node[left] {$(\alpha',\alpha')$};
    \draw (5,2) node[right] {$(\alpha' + \beta',\alpha')$};
    \draw (3,0) node[right] {$(\beta',0)$};
\draw (3,1.1) node[] {$\times$};
\draw[dashed] (3,1.1) -- (3,0);
\end{tikzpicture}
\end{center}
\caption{$\left( \Delta_2(\alpha',\beta'), (\beta',h(t)), -1 \right)$.}
\label{fig:polygon_W0_3c_half}
\end{subfigure}

\vspace{15pt}

 \begin{subfigure}[b]{.45\linewidth}
  \begin{center}
   \begin{tikzpicture}[scale=.9]
    \filldraw[draw=black,fill=gray!60] (0,0) -- (2,2) -- (5,2) -- (3,0) -- cycle;
    \draw (0,0) node[left] {$(0,0)$};
    \draw (2,2) node[left] {$(\alpha',\alpha')$};
    \draw (5,2) node[right] {$(\alpha' + \beta',\alpha')$};
    \draw (3,0) node[right] {$(\beta',0)$};
   \end{tikzpicture}
  \end{center}
 \caption{$\Delta_2(\alpha',\beta')$.}
 \label{fig:polygon_W0_3c_0}
 \end{subfigure}\,\,\,
 \begin{subfigure}[b]{.45\linewidth}
  \begin{center}
   \begin{tikzpicture}[scale=.9]
    \filldraw[draw=black,fill=gray!60] (0,0) -- (2,2) -- (3,2) -- (5,0) -- cycle;
    \draw (0,0) node[left] {$(0,0)$};
    \draw (2,2) node[left] {$(\alpha',\alpha')$};
    \draw (3,2) node[right] {$(\beta',\alpha')$};
    \draw (5,0) node[right] {$(\alpha' + \beta',0)$};
   \end{tikzpicture}
  \end{center}
 \caption{$\tilde{\Delta}_2(\alpha',\beta')$.}
 \label{fig:polygon_W0_3c_1}
 \end{subfigure}
\end{center}
\caption{Representatives of the marked semitoric polygons associated with the semitoric transition family of Lemma \ref{lm:system_W0_HP} for (a) $t^-<t<t^+$, (b)  $t<t^-$, and (c) $t>t^+$.}
\label{fig:polygons_W0_3c}
\end{figure}
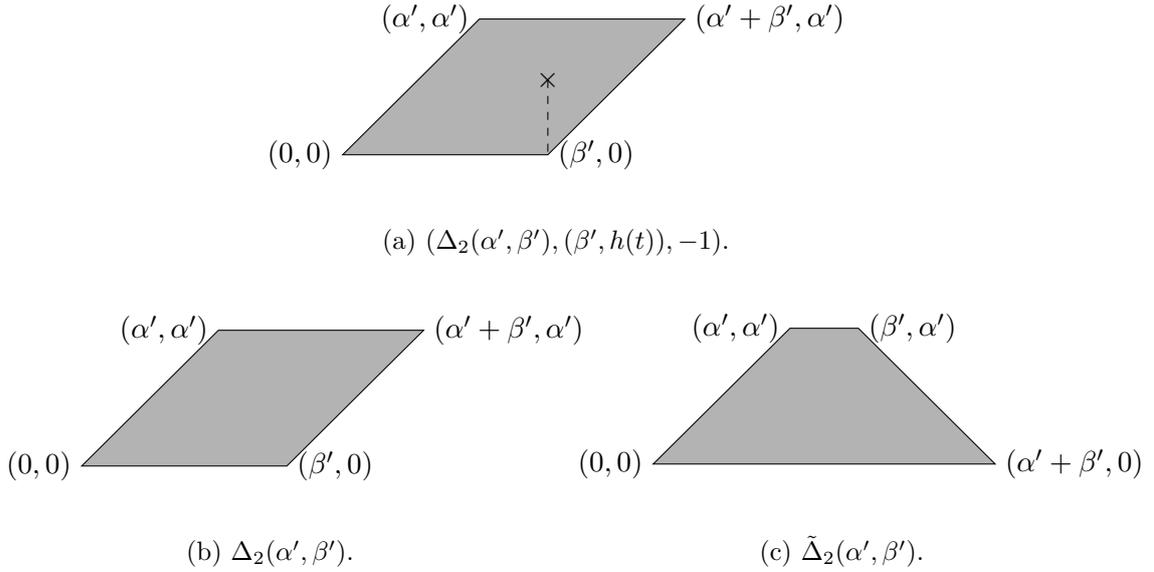

We check that this system is indeed of type (3c) when $t^- < t < t^+$ by inspecting more closely its marked semitoric polygon. Indeed, it suffices to notice that by changing the cut for the representative of Figure \ref{fig:min_type3c} with $n=2$, we obtain the polygon $\Delta_2(\beta-\alpha,\alpha)$. This yields the following result.

\begin{lm}
\label{lm:W0_HP_3c}
Let $\alpha, \beta > 0$ with $\alpha < \beta$. Then the system from Lemma \ref{lm:system_W0_HP} with $\alpha' = \beta-\alpha$ and $\beta' = \beta$ is semitoric of type (3c) with $n=2$ and parameters $\alpha$ and $\beta$.   
\end{lm}

We now obtain the following result by alternatively performing toric type blowups and blowdowns to this system.

\begin{thm}
\label{thm:(3c)_blowups_downs}
For any $n \geq 2$, any $\alpha, \beta > 0$ with $\alpha < \beta$ and any $h \in (0,\beta-\frac{\alpha}{n-1})$, there exists $b \in (0,1)$ and a semitoric family
\[ \left(W_{n-2}(\beta - \alpha,\beta), \omega_{W_{n-2}(\beta - \alpha,\beta)}, F_t^n = (J,H_t^n) \right)_{0 \leq t \leq b}\]
with one degenerate time $t_n^-$ satisfying
\[ \frac{\beta}{3\beta - \alpha + 2 \sqrt{\beta(\beta-\alpha)}} \leq t_n^- < \frac{1}{2} \]
with equality on the left if and only if $n=2$, such that there exists one point $p_n \in W_{n-2}(\beta - \alpha,\beta)$ such that
\begin{enumerate}
    \item $F_{t_n^-}$ has no degenerate singular point in $M \setminus \{p_n\}$;
    \item if $t \in (t_n^-, b)$, the system is semitoric with one focus-focus point $p_n$ and its marked semitoric polygon is $\left[\left( \Delta_n(\beta-\alpha,\beta), (\beta,h_n(t)), -1 \right)\right]$ as in Figure \ref{fig:poly_3c_blowupsdown} where the image of $t \mapsto h_n(t)$ is the interval $(0,h)$ (in particular, it is of type (3c) with parameters $\alpha, \beta, n, h_n(t)$);
    \item if $t \in [0,t_n^-)$, the system is of toric type ($p_n$ being elliptic-elliptic) and its (unmarked) semitoric polygon is $[(\Delta_n(\beta-\alpha,\beta),\varnothing,\varnothing)]$. Furthermore, for $t=0$ it is toric up to multiplication of $H_0$ by a constant, with Delzant polygon $\Delta_n(\beta-\alpha,\beta)$.
\end{enumerate}
Moreover, such a family can be obtained by alternatively performing toric type blowups and blowdowns, each $n-2$ times, to the system of Lemma \ref{lm:system_W0_HP} for some appropriate choice of $\alpha'$ and $\beta'$.
\end{thm}

The fact that the system is of type (3c) for $t \in (t_n^-,b)$ can easily be seen from the polygon $\Delta_n(\beta-\alpha,\beta)$; indeed, by changing the cut direction in this polygon, we obtain the polygon of Figure \ref{fig:min_type3c}. Note that the maximal height invariant of a system of type (3c) with this polygon is $\beta - \frac{\alpha}{n-1}$. Hence the theorem allows one to obtain every marked semitoric polygon of type (3c).

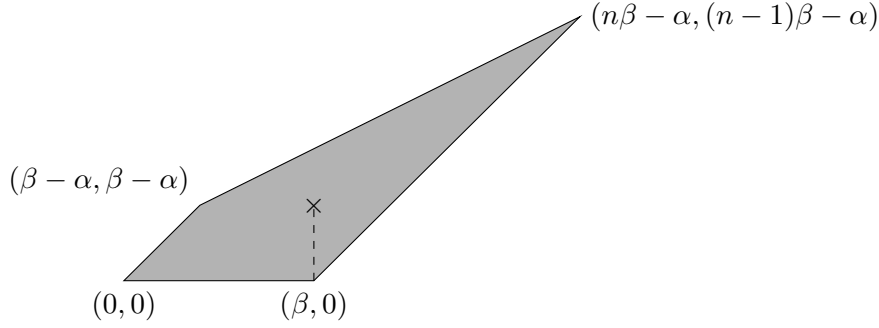
\begin{figure}

\begin{center}
\begin{tikzpicture}
\filldraw[draw=black, fill=gray!60] (0,0) -- (1,1) -- (6,3.5) -- (2.5,0) -- cycle; 

\draw (2.5,1) node {$\times$}; 
\draw [dashed] (2.5,0) -- (2.5,1); 

\draw (0,0) node[below] {$(0,0)$};
\draw (6,3.5) node[right] {$(n\beta - \alpha,(n-1)\beta - \alpha)$};
\draw (2.5,0) node[below] {$(\beta, 0)$};
\draw (1,1) node[above left] {$(\beta - \alpha, \beta - \alpha)$};
\end{tikzpicture} 

\end{center} 
\caption{$\left( \Delta_n(\beta-\alpha,\beta), (\beta,h_n(t)), -1 \right)$.} 
\label{fig:poly_3c_blowupsdown}
\end{figure}  

In the course of the proof of Theorem \ref{thm:(3c)_blowups_downs}, we will need the following statement, whose proof can easily be adapted from the proof of \cite[Lemma 4.12]{LFPfamilies}.

\begin{lm}
\label{lm:SF1_412_general}
 Suppose that $(M,\om,F_t)$ is a semitoric family which has exactly one degenerate point at each degenerate time, and that $q\in M$ is an elliptic-elliptic point of $F_t$ for each $t \in [0,1]$. If the system $(M,\om,F_t)$ admits a toric type blowup of size $\lambda>0$ at $q$ for all values of $t$ for which it is semitoric, then for any positive $\tilde{\lambda}<\lambda$ the system $(M,\om,F_t)$ admits a blowup of size $\tilde{\lambda}$ at $q$ for all $t\in[0,1]$.
\end{lm}

\begin{proof}[Proof of Theorem \ref{thm:(3c)_blowups_downs}]
The proof is similar to the proof of \cite[Theorem 5.2]{LFPfamilies}, but we need to be a bit more careful here in order to obtain every possible marked semitoric polygon. In particular, instead of performing an induction on $n$ it is easier to fix $n$ and to perform a finite induction on $\ell \in \{2, \ldots, n\}$. The case $n=2$ is already contained in Lemmas \ref{lm:system_W0_HP} and \ref{lm:W0_HP_3c}.

Hence, let $n \geq 3$, let $\alpha, \beta > 0$ with $\alpha < \beta$, and let $h \in (0,\beta-\frac{\alpha}{n-1})$. Define $0 < \alpha_2 < \ldots < \alpha_n$ as
\[ \forall \ell \in \{ 2, \ldots, n \} \qquad \alpha_{\ell} = \frac{\ell - 1}{n-1} \alpha. \]
Moreover, for every $\ell \in \{ 2, \ldots, n-1 \}$, let $\lambda_{\ell} = \alpha_{\ell + 1} - \alpha_{\ell} = \frac{\alpha}{n-1}$.

Let $(W_0(\beta-\alpha_2,\beta), \omega_{W_0(\beta-\alpha_2,\beta)},F_t)$ be the system from Lemma \ref{lm:system_W0_HP}, with parameters $\alpha' = \beta - \alpha_2$ and $\beta' = \beta$. This is a semitoric transition family with transition times $t_n^-$, $t_n^+$ where
\[ t_n^{\pm} = \frac{\beta}{3\beta - \alpha_2 \mp 2 \sqrt{\beta(\beta-\alpha_2)}} = \frac{(n-1)\beta}{3(n-1)\beta - \alpha \mp 2 \sqrt{(n-1)\beta((n-1)\beta-\alpha)}},  \]
and whose range of height invariant $t \mapsto h_2(t)$ is $(0,\beta - \alpha_2) = (0,\beta - \frac{\alpha}{n-1})$ by our choice of $\alpha_2$. In particular, there exists $b \in (t_n^-,t_n^+)$ such that when $t$ varies in $(t_n^-,b]$, $h(t)$ describes the interval $(0,h] \subset (0,\beta-\frac{\alpha}{n-1})$. 

Hence, let $\ell \geq 2$ and assume that we have constructed a semitoric family 
\[ \left(W_{\ell-2}(\beta-\alpha_{\ell},\beta), \omega_{W_{\ell-2}(\beta-\alpha_{\ell},\beta)}, F_t^{\ell} \right)_{0 \leq t \leq b}\]
as in the statement; we will perform a blowup and a blowdown to this family to obtain such a semitoric transition family 
\[ \left(W_{\ell-1}(\beta - \alpha_{\ell+1},\beta), \omega_{W_{\ell-1}(\alpha_{\ell+1},\beta)}, F_t^{\ell+1} \right).\]
We proceed as follows, see also Figure \ref{fig:3c_blow} for a visual aid.

\begin{enumerate}
    \item Let $p \in W_{\ell-2}(\beta - \alpha_{\ell-2},\beta)$ be the point whose image in any representative of the marked semitoric polygon $[\left( \Delta_{\ell}(\beta-\alpha_{\ell},\beta), (\beta,h_{\ell}(t)), -1 \right)]$ of $\left(W_{\ell-2}(\alpha_{\ell-2},\beta), \omega_{W_{\ell}(\beta-\alpha_{\ell-2},\beta)}, F_t^{\ell} \right)$ with $\epsilon = -1$ is the upper left vertex (marked by the black dot in the leftmost column of Figure \ref{fig:3c_blow}), for any $t \in (t_n^-,b]$. The edges adjacent to $p$ in $\left( \Delta_{\ell}(\beta-\alpha_{\ell},\beta), (\beta,h_{\ell}(t)), -1 \right)$ have respective $\mathrm{SL}(2,\Z)$-lengths (see Section \ref{subsubsec:tori_blowups}) $\beta - \alpha_{\ell}$ and $\beta$. Since 
    \[ \beta - \alpha_{\ell} - \lambda_{\ell} = \beta - \frac{\ell}{n-1} \geq \beta - \alpha > 0, \]
    we have $\lambda_{\ell} < \beta - \alpha_{\ell} < \beta$, so the triangle with vertices $(\beta - \alpha_{\ell} - \lambda_{\ell},\beta - \alpha_{\ell} - \lambda_{\ell})$, $(\beta-\alpha_{\ell},\beta-\alpha_{\ell})$ and $(\beta-\alpha_{\ell} + (\ell-1)\lambda_{\ell}, \beta-\alpha_{\ell} + (\ell-2)\lambda_{\ell})$ is strictly contained in $\Delta_{\ell}(\beta-\alpha_{\ell},\beta)$. Since moreover $\beta - \alpha_{\ell} + (\ell-1) \lambda_{\ell} = \beta$, and since for every $t \in (t_n^-,b]$, $h_{\ell}(t) \leq h < \beta - \frac{\alpha_{\ell}}{\ell-1} = \beta - \frac{\alpha}{n-1}$, the marked point $(\beta,h_{\ell}(t))$ never touches this triangle, and hence by Lemma \ref{lem:blowups_semitoric_chops} we can perform a corner chop of size $\lambda_{\ell}$ at the vertex $(\beta - \alpha_{\ell}, \beta-\alpha_{\ell})$ of $\Delta_{\ell}(\beta-\alpha_{\ell},\beta)$. In fact, we can do the same reasoning with a $\lambda$ slightly bigger than $\lambda_{\ell}$, and conclude that for any $t \in [0,b] \setminus \{t_n^-\}$, we can perform a toric type blowup of size $\lambda$ at $p$ to $\left(W_{\ell-2}(\beta-\alpha_{\ell},\beta), \omega_{W_{\ell-2}(\beta-\alpha_{\ell},\beta)}, F_t^{\ell} \right)$. By Lemma \ref{lm:SF1_412_general}, this implies that we can perform a blowup of size $\lambda_{\ell}$ for every $t \in [0,b]$. This in turn implies, by \cite[Theorem 4.8]{LFPfamilies}, that we can construct a family $(M_{\ell},\omega_{\ell},\tilde{F}_t^{\ell} = (J_{\ell},H_t^{\ell}))_{t \in [0,b]}$ of integrable systems with $H_t^{\ell}$ depending smoothly on $t$, where $M_{\ell} = \mathrm{Bl}_p\left( W_{\ell-2}(\beta - \alpha_{\ell-2},\beta) \right)$, which is symplectomorphic, for every $t$, to the blowup of size $\lambda_{\ell}$ at $p$ of $\left(W_{\ell-2}(\alpha_{\ell-2},\beta), \omega_{W_{\ell}(\beta-\alpha_{\ell-2},\beta)}, F_t^{\ell} \right)$. By \cite[Corollary 4.11]{LFPfamilies}, this new family is also a semitoric family, with one degenerate time $t_n^-$ and its marked semitoric polygon for $t \in (t_n^-,b]$ is displayed in the central column of Figure \ref{fig:3c_blow}.
    \item Next, we show that we can perform an appropriate toric type blowdown to this new family. For $t \in [0,b]$, let $\Sigma_t \subset M_{\ell}$ be the sphere which is the preimage of the bold segment in the polygon in the central column of Figure \ref{fig:3c_blow}. If $t \neq t_n^-$, this polygon admits a corner unchop of size $\beta - \lambda_{\ell}$ along this segment; indeed, it can be seen as the polygon obtained after performing a corner chop of size $\beta - \lambda_{\ell}$ at the upper right vertex of $\Delta_{\ell+1}(\beta-\alpha_{\ell+1},\beta)$ (marked by a blue dot in the rightmost column of Figure \ref{fig:3c_blow}). Hence $(M_{\ell},\omega_{\ell},\tilde{F}_t^{\ell})$ admits a blowdown of size $\lambda_{\ell}$ along $\Sigma_t$. Introducing a $\lambda$ slightly smaller than $\lambda_{\ell}$ if needed, we apply Lemma \ref{lm:SF1_412_general} again to obtain that this blowdown of size $\lambda_{\ell}$ along $\Sigma_t$ is also possible when $t = t_n^-$. By \cite[Proposition 4.9]{LFPfamilies} and \cite[Corollary 4.11]{LFPfamilies}, this blowdown produces a semitoric family $(\hat{M}_{\ell},\hat{\omega}_{\ell},\hat{F}_t^{\ell})_{t \in [0,b]}$ with one degenerate time $t_n^-$, where $\hat{M}_{\ell}$ is the blowdown of $M_{\ell}$. This family is of toric type for $t \in [0,t_n^-)$ and semitoric with one focus-focus point for $t \in (t_n^-,b]$. To conclude, it suffices to notice that since the system for $t=0$ is obtained after performing a blowup and a blowdown on the system $\left(W_{\ell-2}(\beta-\alpha_{\ell},\beta), \omega_{W_{\ell-2}(\beta-\alpha_{\ell},\beta)}, F_0^{\ell} \right)$, which is toric up to scaling of the second component of $F_0^{\ell}$ with Delzant polygon $\Delta_{\ell}(\beta-\alpha_{\ell},\beta)$, it is also toric up to scaling of the second component of the momentum map, and its Delzant polygon is $\Delta_{\ell+1}(\beta-\alpha_{\ell+1},\beta)$. Hence $(\hat{M}_{\ell},\hat{\omega}_{\ell}) \simeq (W_{\ell-1}(\beta-\alpha_{\ell+1}, \beta),\omega_{W_{\ell-1}(\beta-\alpha_{\ell+1}, \beta)})$. 
\end{enumerate} 
This takes care of the induction step. To conclude, it suffices to notice that $\alpha_n = \alpha$.
\end{proof}

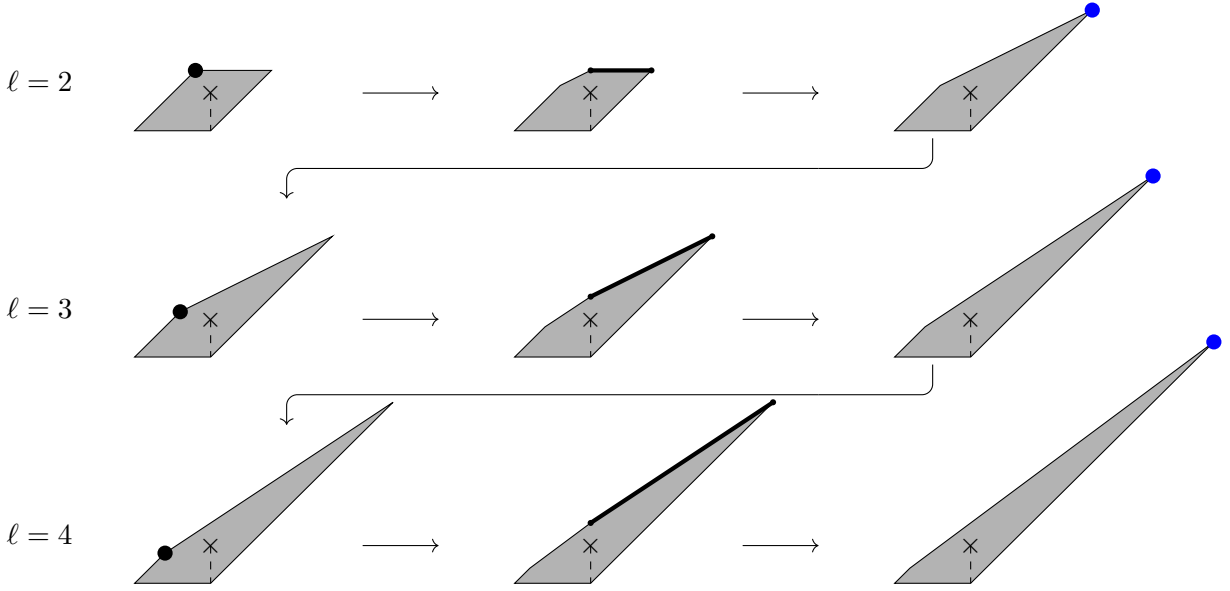
\begin{figure}
\begin{center}
\begin{tikzpicture}[scale=1]

\node[label={$\ell=2$}] at (-1.25,0.25){};
\node[label={$\ell=3$}] at (-1.25,-2.75){};
\node[label={$\ell=4$}] at (-1.25,-5.75){};

\draw [->] (3,0.5) -- (4,0.5);
\draw [->] (8,0.5) -- (9,0.5);
\draw [-, rounded corners] (10.5, -0.1) |- (9, -0.5);
\draw [->, rounded corners] (9, -0.5)  -| (2,-0.9);
\draw [->] (3,-2.5) -- (4,-2.5);
\draw [->] (8,-2.5) -- (9,-2.5);
\draw [-, rounded corners] (10.5, -3.1) |- (9, -3.5);
\draw [->, rounded corners] (9, -3.5)  -| (2,-3.9);
\draw [->] (3,-5.5) -- (4,-5.5);
\draw [->] (8,-5.5) -- (9,-5.5);

\filldraw[draw=black,fill=gray!60] (0,0) node[anchor=north,color=black]{}
  -- (4/5,4/5) node[anchor=south,color=black]{}
  -- (9/5,4/5) node[anchor=south,color=black]{}
  -- (1,0) node[anchor=south,color=black]{}
  -- cycle;  
\draw [dashed] (1,0.5) -- (1,0);
\draw (1,0.5) node[] {$\times$};
\fill[black] (4/5,4/5) circle (.1);	

\begin{scope}[xshift=5cm]

\filldraw[draw=black,fill=gray!60] (0,0) node[anchor=north,color=black]{}
  -- (3/5,3/5) node[anchor=south,color=black]{}
  -- (1,4/5) node[anchor=south,color=black]{}
  -- (9/5,4/5) node[anchor=south,color=black]{}
  -- (1,0) node[anchor=south,color=black]{}  
  -- cycle;  
\draw [dashed] (1,0.5) -- (1,0);
\draw (1,0.5) node[] {$\times$};
\draw[ultra thick] (1,4/5) to (9/5,4/5);
\fill[black] (1,4/5) circle (1/25);	
\fill[black] (9/5,4/5) circle (1/25);

\end{scope}

\begin{scope}[xshift=10cm]

\filldraw[draw=black,fill=gray!60] (0,0) node[anchor=north,color=black]{}
  -- (3/5,3/5) node[anchor=south,color=black]{}
  -- (13/5,8/5) node[anchor=south,color=black]{}
  -- (1,0) node[anchor=south,color=black]{}  
  -- cycle;  
\draw [dashed] (1,0.5) -- (1,0);
\draw (1,0.5) node[] {$\times$};
\fill[blue] (13/5,8/5) circle (.1);	

\end{scope}

\begin{scope}[yshift=-3cm]

\filldraw[draw=black,fill=gray!60] (0,0) node[anchor=north,color=black]{}
  -- (3/5,3/5) node[anchor=south,color=black]{}
  -- (13/5,8/5) node[anchor=south,color=black]{}
  -- (1,0) node[anchor=south,color=black]{}  
  -- cycle;  
\draw [dashed] (1,0.5) -- (1,0);
\draw (1,0.5) node[] {$\times$};
\fill[black] (3/5,3/5) circle (.1);

\begin{scope}[xshift=5cm]
  
\filldraw[draw=black,fill=gray!60] (0,0) node[anchor=north,color=black]{}
  -- (2/5,2/5) node[anchor=south,color=black]{}
  -- (1,4/5) node[anchor=south,color=black]{}
  -- (13/5,8/5) node[anchor=south,color=black]{}
  -- (1,0) node[anchor=south,color=black]{}  
  -- cycle;  
\draw [dashed] (1,0.5) -- (1,0);
\draw (1,0.5) node[] {$\times$};
\draw[ultra thick] (1,4/5) to (13/5,8/5);
\fill[black] (1,4/5) circle (1/25);	
\fill[black] (13/5,8/5) circle (1/25);

\end{scope}

\begin{scope}[xshift=10cm]

\filldraw[draw=black,fill=gray!60] (0,0) node[anchor=north,color=black]{}
  -- (2/5,2/5) node[anchor=south,color=black]{}
  -- (17/5,12/5) node[anchor=south,color=black]{}
  -- (1,0) node[anchor=south,color=black]{}
  -- cycle;
\draw [dashed] (1,0.5) -- (1,0);
\draw (1,0.5) node[] {$\times$};
\fill[blue] (17/5,12/5) circle (.1);	

\end{scope}

\end{scope}

\begin{scope}[yshift=-6cm]
 
\filldraw[draw=black,fill=gray!60] (0,0) node[anchor=north,color=black]{}
  -- (2/5,2/5) node[anchor=south,color=black]{}
  -- (17/5,12/5) node[anchor=south,color=black]{}
  -- (1,0) node[anchor=south,color=black]{}
  -- cycle;
\draw [dashed] (1,0.5) -- (1,0);
\draw (1,0.5) node[] {$\times$};
\fill[black] (2/5,2/5) circle (.1);	  

\begin{scope}[xshift=5cm]
  
\filldraw[draw=black,fill=gray!60] (0,0) node[anchor=north,color=black]{}
  -- (1/5,1/5) node[anchor=south,color=black]{}
  -- (1,4/5) node[anchor=south,color=black]{}
  -- (17/5,12/5) node[anchor=south,color=black]{}
  -- (1,0) node[anchor=south,color=black]{}  
  -- cycle;  
\draw [dashed] (1,0.5) -- (1,0);
\draw (1,0.5) node[] {$\times$};
\draw[ultra thick] (1,4/5) to (17/5,12/5);
\fill[black] (1,4/5) circle (1/25);	
\fill[black] (17/5,12/5) circle (1/25);

\end{scope}

\begin{scope}[xshift=10cm]

\filldraw[draw=black,fill=gray!60] (0,0) node[anchor=north,color=black]{}
  -- (1/5,1/5) node[anchor=south,color=black]{}
  -- (21/5,16/5) node[anchor=south,color=black]{}
  -- (1,0) node[anchor=south,color=black]{}
  -- cycle;
\draw [dashed] (1,0.5) -- (1,0);
\draw (1,0.5) node[] {$\times$};
\fill[blue] (21/5,16/5) circle (.1);	

\end{scope}

\end{scope}
  
\end{tikzpicture}
\end{center}
\caption{Illustration of the proof of Theorem \ref{thm:(3c)_blowups_downs} with $n=5$. Here we have $\alpha = \frac{4 \beta}{5}$ and we choose, in accordance with the proof, $\alpha_2 = \frac{\beta}{5}$, $\alpha_3 = \frac{2\beta}{5}$ and $\alpha_4 = \frac{3\beta}{5}$ so that $\lambda_2 = \lambda_3 = \lambda_4 = \frac{\beta}{5}$, and show the corresponding sequence of marked semitoric polygons (the vertices are not labeled for the sake of clarity). From left to center, we perform a corner chop at the point indicated by a big black dot. From center to right, we perform a corner unchop on the edge indicated by a thick black line, and the new vertex thus created is indicated by a big blue dot.}
\label{fig:3c_blow}
\end{figure}

\subsection{Type (3b)}

One can mimic the strategy of the previous section to obtain semitoric systems of type (3b) as part of semitoric families. At first glance this seems quite similar to Theorem \ref{thm:(3c)_blowups_downs}, but we must be more careful here because performing a toric type blowup, then a toric type blowdown, as in the case of type (3c), to a system of type (3b) does not yield another system of type (3b), see Figure \ref{fig:3b_blow_no}.

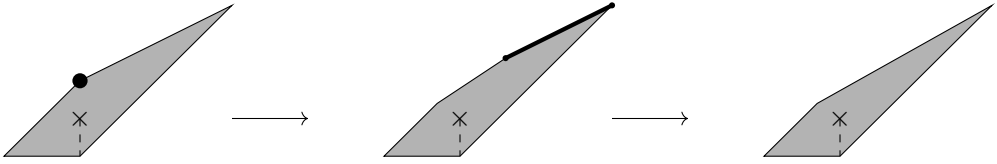
\begin{figure}
\begin{center}
\begin{tikzpicture}[scale=1]

\draw [->] (3,0.5) -- (4,0.5);
\draw [->] (8,0.5) -- (9,0.5);

\filldraw[draw=black,fill=gray!60] (0,0) -- (1,1) -- (3,2) -- (1,0) -- cycle;  
\draw [dashed] (1,0.5) -- (1,0);
\draw (1,0.5) node[] {$\times$};
\fill[black] (1,1) circle (.1);	

\begin{scope}[xshift=5cm]

\filldraw[draw=black,fill=gray!60] (0,0) -- (0.7,0.7) -- (1.6,1.3) -- (3,2) -- (1,0) -- cycle;  
\draw [dashed] (1,0.5) -- (1,0);
\draw (1,0.5) node[] {$\times$};
\draw[ultra thick] (1.6,1.3) to (3,2);
\fill[black] (1.6,1.3) circle (1/25);	
\fill[black] (3,2) circle (1/25);

\end{scope}

\begin{scope}[xshift=10cm]

\filldraw[draw=black,fill=gray!60] (0,0) -- (0.7,0.7) -- (3,2) -- (1,0) -- cycle;  
\draw [dashed] (1,0.5) -- (1,0);
\draw (1,0.5) node[] {$\times$};

\end{scope}

\end{tikzpicture}
\end{center}
\caption{Performing a corner chop at the vertex indicated by the black dot, then a corner unchop along the bold segment, to a marked semitoric polygon of type (3b) does not yield a marked semitoric polygon of type (3b).}
\label{fig:3b_blow_no}
\end{figure}

Another feature of this case is that for $n=2$, we will need to consider the family of systems from Lemma \ref{lm:system_W0_HP} with $\alpha' \geq \beta'$. When $\alpha' > \beta'$ the statement of this lemma remains valid, but when $\alpha' = \beta'$ this family is not a semitoric transition family anymore, since for any $t \geq t^+$, the image of the transition point coincides with the image of another point, which is of elliptic-elliptic type. However, we will only need to consider the restriction of this family to parameters in $[0,a)$ for $a < t^+$, and this restricted family is a semitoric family with one point transitioning from elliptic-elliptic to focus-focus, as desired.

\begin{thm}
\label{thm:(3b)_blowups_downs}
For any $n \geq 2$, any $\beta > 0$ and any $h \in (0,\beta)$, there exists $b \in (0,1)$ and a semitoric family
\[ \left(W_{n-2}(\beta,\beta), \omega_{W_{n-2}(\beta,\beta)}, F_t^n = (J,H_t^n) \right)_{0 \leq t \leq b}\]
with one degenerate time $t_n^- < \frac{1}{2}$, such that there exists one point $p_n \in W_{n-2}(\beta,\beta)$ such that
\begin{enumerate}
    \item $F_{t_n^-}$ has no degenerate singular point in $M \setminus \{p_n\}$;
    \item if $t \in (t_n^-, b)$, the system is semitoric with one focus-focus point $p_n$ and its marked semitoric polygon is $\left[\left( \Delta_n^0(\beta), (\beta,h_n(t)), -1 \right)\right]$ as in Figure \ref{fig:poly_3b_blowupsdown} where the image of $t \mapsto h_n(t)$ is the interval $(0,h)$ (in particular, it is of type (3b) with parameters $\beta, n, h_n(t)$);
    \item if $t \in [0,t_n^-)$, the system is of toric type ($p_n$ being elliptic-elliptic) and its (unmarked) semitoric polygon is $[(\Delta_n^0(\beta),\varnothing,\varnothing)]$, and for $t=0$ it is even toric up to multiplication of $H_0$ by a constant, with Delzant polygon $\Delta_n^0(\beta)$.
\end{enumerate}
Moreover, such a family can be obtained by alternatively performing blowups and blowdowns, each $n-2$ times, to the system of Lemma \ref{lm:system_W0_HP} for some appropriate choice of $\alpha'$ and $\beta'$.   
\end{thm}

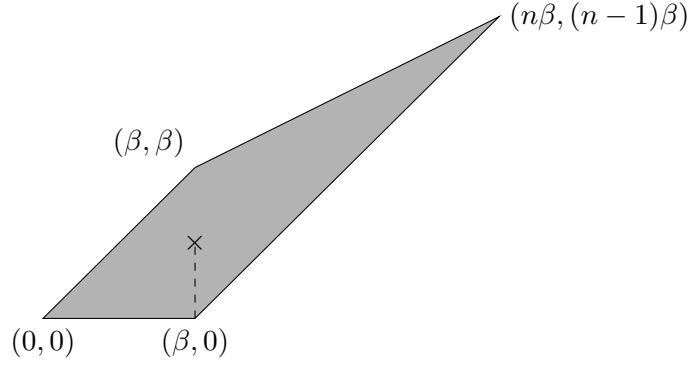
\begin{figure}

\begin{center}
\begin{tikzpicture}
\filldraw[draw=black, fill=gray!60] (0,0) -- (2,2) -- (6,4) -- (2,0) -- cycle; 

\draw (2,1) node {$\times$}; 
\draw [dashed] (2,0) -- (2,1); 

\draw (0,0) node[below] {$(0,0)$};
\draw (6,4) node[right] {$(n\beta,(n-1)\beta)$};
\draw (2,0) node[below] {$(\beta, 0)$};
\draw (2,2) node[above left] {$(\beta, \beta)$};
\end{tikzpicture} 

\end{center} 
\caption{$\left( \Delta_n^0(\beta), (\beta,h_n(t)), -1 \right)$.} 
\label{fig:poly_3b_blowupsdown}
\end{figure}  

\begin{proof}[Proof of Theorem \ref{thm:(3b)_blowups_downs}]   
This proof follows the general reasoning of the proofs of Theorem \ref{thm:(3c)_blowups_downs} and of \cite[Theorem 5.2]{LFPfamilies}, so we will only sketch it and highlight the key differences. This proof is illustrated in Figure \ref{fig:3b_blow}.

First, we observe that the statement is true for $n=2$ because the system of Lemma \ref{lm:system_W0_HP} with $\alpha' = \beta'$ satisfies all the requirements (recall the discussion just before the statement of Theorem~\ref{thm:(3b)_blowups_downs}).

So we fix $n \geq 3$ and perform a finite induction on $\ell \in \{2,\ldots,n\}$, where at each step we will perform a toric type blowup followed by a toric type blowdown on a system with marked semitoric polygon $[(\Delta_n(\alpha',\beta'),(\beta',h),-1]$ with $\alpha' > \beta$ (see Figure \ref{fig:poly_3a_cutdown}). For the case $\ell = 2$, we consider the system of Lemma \ref{lm:system_W0_HP} with $\alpha' = \beta + \alpha_2$ and $\beta' = \beta$ for some well-chosen $\alpha_2$. At each step $\ell$, we will then perform a toric type blowup of size $\lambda_{\ell} > 0$ well-chosen, and then a blowdown, to obtain a system with marked semitoric polygon 
\begin{itemize}
\item $[\left( \Delta_n(\beta + \alpha_{\ell+1},\beta), (\beta,h_n(t)), -1 \right)]$ where $\alpha_{\ell+1} = \alpha_{\ell} - \lambda_{\ell} > 0$ if $2 \leq \ell \leq n-2$;
\item $\left[\left( \Delta_n^0(\beta), (\beta,h_n(t)), -1 \right)\right]$ if $\ell = n-1$, which means that we want $\alpha_n = \alpha_{n-1} - \lambda_{n-1} = 0$.
\end{itemize}
Taking into account the conditions for the blowup to be possible, this means that the requirements on the constants $\alpha_{\ell}$ and $\lambda_{\ell}$ are:
\begin{itemize}
    \item $\alpha_2 > \alpha_3 > \ldots > \alpha_n = 0$;
    \item for every $\ell \in \{2, \ldots, n-1 \}$, $\lambda_{\ell} < \beta$.
\end{itemize}
One readily checks that the required blowups can be performed whenever these conditions are satisfied.
\end{proof}

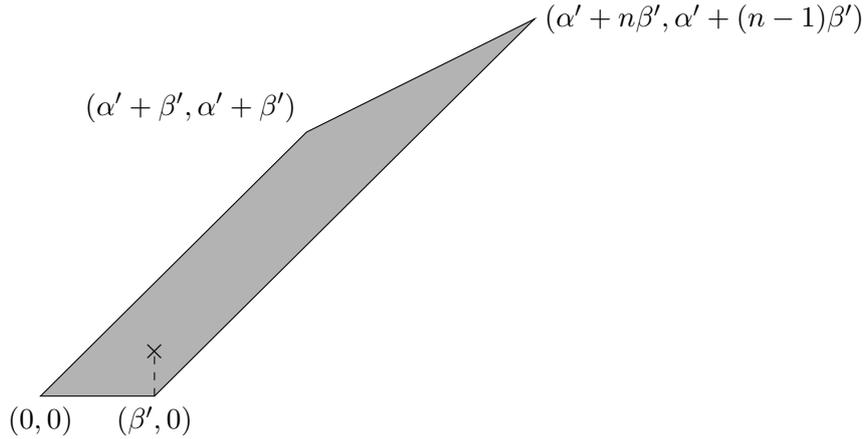
\begin{figure}

\begin{center}
\begin{tikzpicture}
\filldraw[draw=black, fill=gray!60] (0,0) -- (3.5,3.5) -- (6.5,5) -- (1.5,0) -- cycle; 

\draw (1.5,0.6) node {$\times$}; 
\draw [dashed] (1.5,0) -- (1.5,0.6); 

\draw (0,0) node[below] {$(0,0)$};
\draw (6.5,5) node[right] {$(\alpha' + n\beta', \alpha' + (n-1)\beta')$};
\draw (1.5,0) node[below] {$(\beta', 0)$};
\draw (3.5,3.5) node[above left] {$(\alpha' + \beta', \alpha' + \beta')$};
\end{tikzpicture} 

\end{center} 
\caption{$\left( \Delta_n(\alpha',\beta'), (\beta',h), -1 \right)$ for $\alpha' > \beta'$.} 
\label{fig:poly_3a_cutdown}
\end{figure}

\begin{figure}
\begin{center}
\begin{tikzpicture}[scale=.93]

\node[label={$\ell=2$}] at (-1.25,1.25){};
\node[label={$\ell=3$}] at (-1.25,-2.75){};
\node[label={$\ell=4$}] at (-1.25,-6.75){};

\draw [->] (3,1.5) -- (4,1.5);
\draw [->] (8,1.5) -- (9,1.5);
\draw [-, rounded corners] (10.5, -0.1) |- (9, -0.5);
\draw [->, rounded corners] (9, -0.5)  -| (2,-0.9);
\draw [->] (3,-2.5) -- (4,-2.5);
\draw [->] (8,-2.5) -- (9,-2.5);
\draw [-, rounded corners] (10.5, -4.1) |- (9, -4.5);
\draw [->, rounded corners] (9, -4.5)  -| (2,-4.9);
\draw [->] (3,-7) -- (4,-7);
\draw [->] (8,-7) -- (9,-7);

\filldraw[draw=black,fill=gray!60] (0,0) node[anchor=north,color=black]{}
  -- (3,3) node[anchor=south,color=black]{}
  -- (4,3) node[anchor=south,color=black]{}
  -- (1,0) node[anchor=south,color=black]{}
  -- cycle;  
\draw [dashed] (1,0.5) -- (1,0);
\draw (1,0.5) node[] {$\times$};
\fill[black] (3,3) circle (.1);	

\begin{scope}[xshift=5cm]

\filldraw[draw=black,fill=gray!60] (0,0) node[anchor=north,color=black]{}
  -- (2.3,2.3) node[anchor=south,color=black]{}
  -- (3.7,3) node[anchor=south,color=black]{}
  -- (4,3) node[anchor=south,color=black]{}
  -- (1,0) node[anchor=south,color=black]{}  
  -- cycle;  
\draw [dashed] (1,0.5) -- (1,0);
\draw (1,0.5) node[] {$\times$};
\draw[ultra thick] (3.7,3) to (4,3);
\fill[black] (3.7,3) circle (1/25);	
\fill[black] (4,3) circle (1/25);

\end{scope}

\begin{scope}[xshift=10cm]

\filldraw[draw=black,fill=gray!60] (0,0) node[anchor=north,color=black]{}
  -- (2.3,2.3) node[anchor=south,color=black]{}
  -- (4.3,3.3) node[anchor=south,color=black]{}
  -- (1,0) node[anchor=south,color=black]{}  
  -- cycle;  
\draw [dashed] (1,0.5) -- (1,0);
\draw (1,0.5) node[] {$\times$};
\fill[blue] (4.3,3.3) circle (.1);

\end{scope}

\begin{scope}[yshift=-4cm]

\filldraw[draw=black,fill=gray!60] (0,0) node[anchor=north,color=black]{}
  -- (2.3,2.3) node[anchor=south,color=black]{}
  -- (4.3,3.3) node[anchor=south,color=black]{}
  -- (1,0) node[anchor=south,color=black]{}  
  -- cycle;  
\draw [dashed] (1,0.5) -- (1,0);
\draw (1,0.5) node[] {$\times$};
\fill[black] (2.3,2.3) circle (.1);

\begin{scope}[xshift=5cm]
  
\filldraw[draw=black,fill=gray!60] (0,0) node[anchor=north,color=black]{}
  -- (1.5,1.5) node[anchor=south,color=black]{}
  -- (3.9,3.1) node[anchor=south,color=black]{}
  -- (4.3,3.3) node[anchor=south,color=black]{}
  -- (1,0) node[anchor=south,color=black]{}  
  -- cycle;  
\draw [dashed] (1,0.5) -- (1,0);
\draw (1,0.5) node[] {$\times$};
\draw[ultra thick] (3.9,3.1) to (4.3,3.3);
\fill[black] (3.9,3.1) circle (1/25);	
\fill[black] (4.3,3.3) circle (1/25);

\end{scope}

\begin{scope}[xshift=10cm]

\filldraw[draw=black,fill=gray!60] (0,0) node[anchor=north,color=black]{}
  -- (1.5,1.5) node[anchor=south,color=black]{}
  -- (4.5,3.5) node[anchor=south,color=black]{}
  -- (1,0) node[anchor=south,color=black]{}
  -- cycle;
\draw [dashed] (1,0.5) -- (1,0);
\draw (1,0.5) node[] {$\times$};
\fill[blue] (4.5,3.5) circle (.1);	  

\end{scope}

\end{scope}

\begin{scope}[yshift=-8.5cm]
 
\filldraw[draw=black,fill=gray!60] (0,0) node[anchor=north,color=black]{}
  -- (1.5,1.5) node[anchor=south,color=black]{}
  -- (4.5,3.5) node[anchor=south,color=black]{}
  -- (1,0) node[anchor=south,color=black]{}
  -- cycle;
\draw [dashed] (1,0.5) -- (1,0);
\draw (1,0.5) node[] {$\times$};
\fill[black] (1.5,1.5) circle (.1);	  

\begin{scope}[xshift=5cm]
  
\filldraw[draw=black,fill=gray!60] (0,0) node[anchor=north,color=black]{}
  -- (1,1) node[anchor=south,color=black]{}
  -- (3,2.5) node[anchor=south,color=black]{}
  -- (4.5,3.5) node[anchor=south,color=black]{}
  -- (1,0) node[anchor=south,color=black]{}  
  -- cycle;  
\draw [dashed] (1,0.5) -- (1,0);
\draw (1,0.5) node[] {$\times$};
\draw[ultra thick] (3,2.5) to (4.5,3.5);
\fill[black] (3,2.5) circle (1/25);	
\fill[black] (4.5,3.5) circle (1/25);

\end{scope}

\begin{scope}[xshift=10cm]

\filldraw[draw=black,fill=gray!60] (0,0) node[anchor=north,color=black]{}
  -- (1,1) node[anchor=south,color=black]{}
  -- (5,4) node[anchor=south,color=black]{}
  -- (1,0) node[anchor=south,color=black]{}
  -- cycle;
\draw [dashed] (1,0.5) -- (1,0);
\draw (1,0.5) node[] {$\times$};
\fill[blue] (5,4) circle (.1);	

\end{scope}

\end{scope}
  
\end{tikzpicture}
\end{center}
\caption{Illustration of the proof of Theorem \ref{thm:(3b)_blowups_downs} with $n=5$. Here we have $\beta=1$ and we choose $\alpha_2 = 2$, $\lambda_2 = 0.7$, $\lambda_3 = 0.8$ and $\lambda_4 = 0.5$. The corner chops and unchops are indicated in the same way as in Figure \ref{fig:3c_blow}.}
\label{fig:3b_blow}
\end{figure}
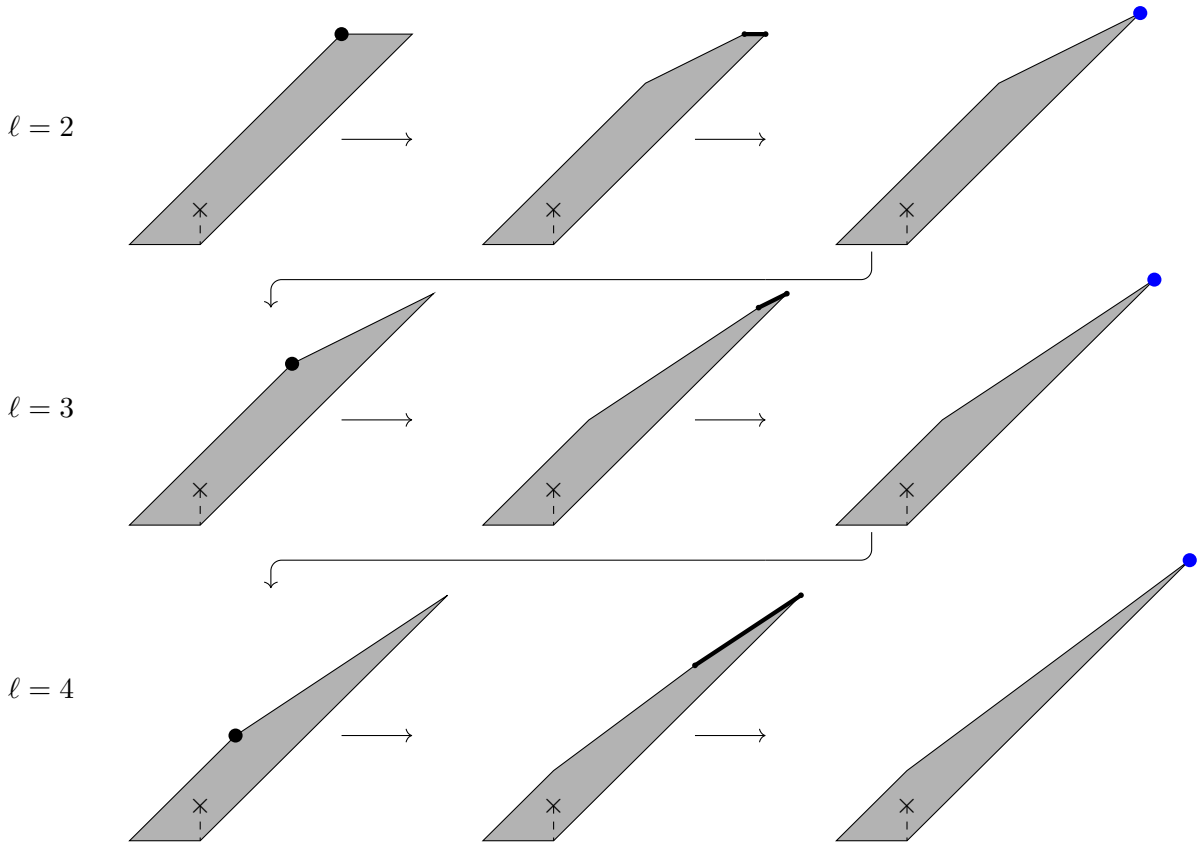

\section{General strategies for constructing explicit semitoric systems}
\label{sec:strategy}

The goal of this section is to collect some useful knowledge on how to construct an explicit semitoric system with a given marked semitoric polygon. We start by recalling the general ideas introduced in \cite{LFPfamilies} and briefly explain why the obstructions from Section~\ref{sec:obstructions} render them insufficient for the present paper and how we can adapt them to overcome these difficulties. Then we give more details on this general strategy by investigating both the local (near the focus-focus point) and global (from $\Z_k$-spheres) obstructions. Finally, under some mild assumptions that many known examples satisfy, we explain a simple way to determine if a rank one singular point is degenerate or non-degenerate, and we show how to determine if such a point is parabolic, by working in the reduced spaces with respect to the global $S^1$-action.

\subsection{Semitoric families and beyond}

In \cite{LFPfamilies} we introduced semitoric transition families (see Definition \ref{def:semitoric-transition-family}) and used them to construct new explicit examples of semitoric systems with given marked semitoric polygons by looking for them as the $t=t_0$ system in such families of the form $(M,\omega,(J,H_t = (1-t) H_0 + t H_1))_{0 \leq t \leq 1}$. The first idea was that, if one of the fake vertices in one of the representatives of the desired marked semitoric polygon satisfies the Delzant condition, then the corresponding marked point and cut can be removed to obtain a marked semitoric polygon; in fact, the same can be done if one $k$-fake vertex satisfies the $(k-1)$-hidden Delzant condition (see Remark~\ref{rmk:fake-nodal-trade}). Then, if we know an explicit semitoric system corresponding to this marked semitoric polygon, we take it as $(M,\omega,(J,H_0))$. We also gave some basic insights on how to choose $H_1$. In particular, if the desired marked semitoric polygon has exactly one marked point and one of its representatives is Delzant, we can choose $(M,\omega,(J,H_0))$ as a toric system with this Delzant polygon and we can construct this system explicitly making use of Delzant's algorithm. 

Using this strategy, we were able to come up with explicit examples of semitoric systems with polygons of types (2) and (3a) for $n=1$ and $n=2$ (see Definition \ref{def:min-poly-123}) as members of semitoric transition families $(M,\omega,F_t=(J,H_t))$. However, the underlying Hamiltonian $S^1$-spaces $(M,\omega,J)$ in these examples do not possess non-trivial isotropy spheres. This contrasts with the examples of strictly minimal semitoric systems of the other types that we are trying to construct explicitly here: in type (1) there is a $\Z_2$-sphere, and in types (3a), (3b) and (3c) with parameter $n \geq 3$ there is a $\Z_{n-1}$-sphere. 

Because of the presence of these non-trivial isotropy spheres, we know from Proposition \ref{prop:non-stfam}, and more precisely Corollaries \ref{cor:type_(1)_notfam} and \ref{cor:type_(3bc)_notfam}, that such systems cannot be obtained as part of semitoric transition families. However, as we will see in Sections \ref{sec:CP2} and \ref{sec:type_3}, they may be obtained as part of half-semitoric transition families (see Definition \ref{def:half-semitoric-family}). 

Hence our goal in this section is to add to and modify the general idea described in \cite{LFPfamilies} to construct such half-semitoric families explicitly. In order to do so, we start, in Section \ref{subsect:normal_form_H}, by investigating the form of the quadratic part of the Hamiltonian $H$ at a focus-focus point of an integrable system $(M,\omega,(J,H))$ lifting a Hamiltonian $S^1$-space $(M,\omega,J)$, by reconciling the local coordinates near $m$ provided by Eliasson's normal from (see Theorem \ref{thm:eliasson}) and the ones involved in the normal form for $J$. The main outcome is that in the coordinates in which $J$ is in normal form, the choice of this quadratic part is quite constrained, see Proposition \ref{prop:comm_q1_ff}. Since in the systems that we are trying to obtain there is only one focus-focus point, once such a suitable quadratic part is chosen, it remains to extend it to the whole of $M$ to obtain a good candidate for $H_{\frac{1}{2}}$ in the desired half-semitoric transition family. In fact, this also sheds a new light on the systems constructed in \cite{LFPfamilies}, as we will see in Examples \ref{ex:W1_SF1}, \ref{ex:W2_SF1_72} and \ref{ex:W2_SF1_73}. Note that most of the results in Section~\ref{subsect:normal_form_H} are not restricted to semitoric systems, and can even be stated for only local $S^1$-actions.

However, this local constraint does not allow us to deal with the potential problems caused by the presence of a $\Z_k$-sphere, which may prevent the system $(M,\omega,(J,H_{t_0}))$ to be semitoric. In Section \ref{subsect:general} we describe these problems and explain how one can try to overcome them by adding a suitable correction to $H_{t_0}$.

Once a good candidate is constructed, one must ensure that it is indeed a half-semitoric transition family. This requires computations that may prove tedious, but in Section \ref{subsect:strat_rank_one} we introduce some tools to simplify these.

\begin{rmk}
This strategy may also work for constructing semitoric systems with more than one focus-focus point by performing these explicit Hamiltonian-Hopf bifurcations sequentially, but in this paper we do not need it and hence do not investigate this question.   
\end{rmk}

\subsection{The form of \texorpdfstring{$H$}{H} at a focus-focus point}
\label{subsect:normal_form_H}

Let $(M,\omega,J)$ be a Hamiltonian $S^1$-space as in Definition~\ref{def:S1space} and let $p$ be a fixed point of $J$. Recall from Section \ref{sec:S1-actions} that locally near $p$, $J$ can be put in the normal form $J = q_{m,n}$ from Equation \eqref{eq:normal_J}, with $m,n \in \Z$ the weights of $J$ at $p$.

Suppose that we are given an additional Hamiltonian $H \in \mathcal{C}^{\infty}(M,\R)$ so that $(J,H)$ is an integrable system with rank zero singular point $p$. In this section, we investigate the form of the function $H$ near $p$. More specifically,
\begin{itemize}
    \item in Section \ref{subsubsec:normal_elia_weights}, assuming that $p$ is a focus-focus point of $(J,H)$, and starting from Eliasson's normal form from Theorem \ref{thm:eliasson}, we give a normal form for $H$ in the complex coordinates coming from the normal form for $J$ and at the same time we recover the well-known fact that the $S^1$-action musts have weights $\{-1,1\}$ at $p$;
    \item in Section \ref{subsubsect:locally_completing}, we find all possible quadratic parts of $H$ written in the complex coordinates of the normal form \eqref{eq:normal_J} for $J$, which in particular gives a simple criterion for determining if $p$ is a focus-focus, elliptic-elliptic, or degenerate point of $(J,H)$ and allows us to understand the Hamiltonian-Hopf bifurcation in these complex coordinates.
\end{itemize}
Note that the results in this section are not restricted to semitoric systems, and are useful whenever considering a rank zero point of an integrable system admitting an $S^1$-action.

\subsubsection{Normal forms at a focus-focus point}
\label{subsubsec:normal_elia_weights}

Let $p$ be a focus-focus point of an integrable system $(M,\omega,F=(J,H))$. Assuming that $F(p) = (0,0)$, Eliasson's theorem (Theorem \ref{thm:eliasson}) states that there exists Darboux coordinates $(x_1,\xi_1,x_2,\xi_2)$ near $p$ and a local diffeomorphism $g: \R^2 \to \R^2$ such that 
\begin{equation} F = g \circ (q_1, q_2), \quad q_1 = x_1 \xi_2 - x_2 \xi_1, \ q_2 = x_1 \xi_1 + x_2 \xi_2. \label{eq:elia}\end{equation}
Here the standard symplectic form on $\R^4$ is $\omega_{\R} = \dd \xi_1 \wedge \dd x_1 + \dd \xi_2 \wedge \dd x_2$, and the flow of $q_1$ is $2\pi$-periodic.

We now show how to reconcile this normal form with the normal form for $J$ in Equation \eqref{eq:normal_J}.

\begin{lm}
\label{lm:symp_nf}
Let $\varphi: (\R^4, \omega_{\R}) \to (\C^2, \omega_{\C})$ be a symplectomorphism such that $\varphi(0) = 0$ and $q_1 \circ \varphi^{-1} = q_{m,n}$ as in the normal form \eqref{eq:normal_J}. Then $\{ m, n \} = \{-1, 1\}$ and there exists $(\rho, \theta, \nu, \alpha) \in (0,+\infty) \times (0,2\pi) \times (0,2\pi) \times (0,\pi)$ such that 
\begin{itemize}
    \item if $m=1$ and $n=-1$, then 
    \[ \varphi(x_1,x_2,\xi_1,\xi_2) = \left( \rho e^{i\theta}(x_1 + i x_2) + \frac{e^{i(\theta - \alpha)}}{2 \rho \sin \alpha} (\xi_1 + i \xi_2), \rho e^{i\nu}(x_1 - i x_2) + \frac{e^{i(\nu - \alpha)}}{2 \rho \sin \alpha} (\xi_1 - i \xi_2)  \right) + \bigO{2}; \]
    \item if $m=-1$ and $n=1$, then 
    \[ \varphi(x_1,x_2,\xi_1,\xi_2) = \left( \rho e^{i\theta}(x_1 - i x_2) + \frac{e^{i(\theta - \alpha)}}{2 \rho \sin \alpha} (\xi_1 - i \xi_2), \rho e^{i\nu}(x_1 + i x_2) + \frac{e^{i(\nu - \alpha)}}{2 \rho \sin \alpha} (\xi_1 + i \xi_2)  \right) + \bigO{2}. \]
\end{itemize}
Conversely, any $\varphi$ of one of these two forms is a symplectomorphism sending $0$ to $0$ such that $q_1 \circ \varphi^{-1} = q_{m,n}$ with $\{m,n\} = \{-1,1\}$.
\end{lm}

\begin{proof}
Let $A = \text{Mat}_{\mathcal{B},\mathcal{B}'}(\dd \varphi(0))$ where $\mathcal{B} = (\partial_{x_1},\partial_{x_2},\partial_{\xi_1},\partial_{\xi_2})$ and $\mathcal{B}' = (\partial_{z_1},\partial_{z_2},\partial_{\bar{z}_1},\partial_{\bar{z}_2})$ are real bases of $\R^4$ and $\C^2$, respectively. Note that $A$ is of the form
\[ A = \begin{pmatrix} A_1 & A_2 \\ \bar{A}_1 & \bar{A}_2 \end{pmatrix} \]
where $A_1, A_2$ are $2 \times 2$ matrices. Moreover, let
\[ \Omega_{\R} = \text{Mat}_{\mathcal{B},\mathcal{B}}(\omega_{\R}) = \begin{pmatrix} 0 & -I \\ I & 0 \end{pmatrix}, \quad \Omega_{\C} = \text{Mat}_{\mathcal{B}',\mathcal{B}'}(\omega_{\C}) = \frac{i}{2} \begin{pmatrix} 0 & I \\ -I & 0 \end{pmatrix} \]
and 
\[ Q_{\R} = \text{Mat}_{\mathcal{B},\mathcal{B}}(q_1) = \begin{pmatrix} 0 & J \\ -J & 0 \end{pmatrix}, \quad Q_{\C} = \text{Mat}_{\mathcal{B}',\mathcal{B}'}(q_{m,n}) = \frac{1}{2} \begin{pmatrix} 0 & D \\ D & 0 \end{pmatrix} \]
where $J = \begin{pmatrix} 0 & 1 \\ -1 & 0 \end{pmatrix}$ and $D = \begin{pmatrix} m & 0 \\ 0 & n \end{pmatrix}$. Then $\varphi$ satisfies the requirements of the statement if and only if
\begin{equation} \Omega_{\R} = A^{\top} \Omega_{\C} A, \quad Q_{\R} = A^{\top} Q_{\C} A. \label{eq:omega_q}\end{equation} Note that by taking the determinant in these two equalities we obtain that
\[1 = \frac{1}{16} \det(A)^2 \quad \textrm{and} \quad 1 = \frac{m^2n^2}{16} \det(A)^2\]
so $m^2 n^2 =1$ and necessarily $m, n \in \{-1, 1 \}$. Moreover, these equalities amount to the two following sets of matrix equations:
\[ \begin{cases} 0 = A_1^{\top} \bar{A}_1 - \bar{A}_1^{\top} A_1 = 2 i \Im(A_1^{\top} \bar{A}_1), \\
0 = A_2^{\top} \bar{A}_2 - \bar{A}_2^{\top} A_2 = 2 i \Im(A_2^{\top} \bar{A}_2), \\
2 i = A_1^{\top} \bar{A}_2 - \bar{A}_1^{\top} A_2 = 2 i \Im(A_1^{\top} \bar{A}_2), \end{cases} \qquad
\begin{cases}0 = A_1^{\top} D \bar{A}_1 + \bar{A}_1^{\top} D A_1 = 2 \Re(A_1^{\top} D \bar{A}_1), \\
0 = A_2^{\top} D \bar{A}_2 + \bar{A}_2^{\top} D A_2 = 2 \Re(A_2^{\top} D \bar{A}_2), \\
2 J = A_1^{\top} D \bar{A}_2 + \bar{A}_1^{\top} D A_2 = 2 \Re(A_1^{\top} D \bar{A}_2). \\
\end{cases} \]
For the rest of the proof, we write $A_1 = \begin{pmatrix} v_1 & v_2 \\ v_3 & v_4 \end{pmatrix}$ and $A_2 = \begin{pmatrix} w_1 & w_2 \\ w_3 & w_4 \end{pmatrix}$. The top left coefficient of the first matrix equation in the second system above reads $m |v_1|^2 + n |v_3|^2 = 0$. So if $m$ and $n$ shared the same sign, then this would imply that $v_1 = 0 = v_3$. But the top left coefficient of the third matrix equation in the first system reads $\Im(v_1 \bar{w}_1) + \Im(v_3 \bar{w}_3) = 1$, which contradicts this. 

Hence $\{m, n\} = \{-1,1\}$, and writing the equations given by taking the matrix entries, the system becomes
\[ \begin{cases} \Re(v_1 \bar{v}_2) = \Re(v_3 \bar{v}_4), \quad \Im(v_1 \bar{v}_2) = \Im(\bar{v}_3 v_4), \\
\Re(w_1 \bar{w}_2) = \Re(w_3 \bar{w}_4), \quad \Im(w_1 \bar{w}_2) = \Im(\bar{w}_3 w_4), \\
\Re(v_1 \bar{w}_1) = \Re(v_3 \bar{w}_3), \quad \Im(v_1 \bar{w}_1) + \Im(v_3 \bar{w}_3) = 1, \\
\Re(v_2 \bar{w}_2) = \Re(v_4 \bar{w}_4), \quad \Im(v_2 \bar{w}_2) + \Im(v_4 \bar{w}_4) = 1, \\
|v_1|^2 = |v_3|^2, \quad |v_2|^2 = |v_4|^2, \quad |w_1|^2 = |w_3|^2, \quad |w_2|^2 = |w_4|^2,  \\
m \Re(v_1 \bar{w}_2) + n \Re(v_3 \bar{w}_4) = 1, \quad \Im(v_1 \bar{w}_2) = \Im(\bar{v}_3 w_4), \\
m \Re(v_2 \bar{w}_1) + n \Re(v_4 \bar{w}_3) = -1, \quad \Im(v_2 \bar{w}_1) = \Im(\bar{v}_4 w_3). \\
\end{cases} \]

This system can be rearranged as
\[ \begin{cases} v_1 \bar{v}_2 = \bar{v}_3 v_4, \\
w_1 \bar{w}_2 = \bar{w}_3 w_4, \\
v_1 \bar{w}_1 - \bar{v}_3 w_3 = i, \\
v_2 \bar{w}_2 - \bar{v}_4 w_4 = i, \\
|v_1|^2 = |v_3|^2, \quad |v_2|^2 = |v_4|^2, \quad |w_1|^2 = |w_3|^2, \quad |w_2|^2 = |w_4|^2,  \\
m \Re(v_1 \bar{w}_2) + n \Re(v_3 \bar{w}_4) = 1, \quad \Im(v_1 \bar{w}_2) = \Im(\bar{v}_3 w_4), \\
m \Re(v_2 \bar{w}_1) + n \Re(v_4 \bar{w}_3) = -1, \quad \Im(v_2 \bar{w}_1) = \Im(\bar{v}_4 w_3). \\
\end{cases} \]
Now we separate the two cases $(m,n) = (1,-1)$ and $(m,n) = (-1,1)$.

\paragraph{The case $m=1$, $n=-1$.} In this case the above system can be transformed to
\[ \begin{cases} v_1 \bar{v}_2 = \bar{v}_3 v_4, \\
w_1 \bar{w}_2 = \bar{w}_3 w_4, \\
v_1 \bar{w}_1 - \bar{v}_3 w_3 = i, \\
v_2 \bar{w}_2 - \bar{v}_4 w_4 = i, \\
v_1 \bar{w}_2 - \bar{v}_3 w_4 = 1, \\
v_2 \bar{w}_1 - \bar{v}_4 w_3 = -1, \\
|v_1|^2 = |v_3|^2,\\
|v_2|^2 = |v_4|^2,\\
|w_1|^2 = |w_3|^2,\\
|w_2|^2 = |w_4|^2.\\
\end{cases} \]
One readily checks that all the unknowns must be nonzero. Indeed, for instance, if $v_1 = 0$, then the seventh line would yield $v_3 = 0$, which would contradict the third line; the same reasoning can be applied for the other unknowns. Now, by multiplying the first line by $v_3$ and using the seventh line, we get that 
\[ v_1 \bar{v}_2 v_3 = |v_3|^2 v_4 = |v_1|^2 v_4,  \]
which gives $\bar{v}_2 v_3 = \bar{v}_1 v_4$. We use this by considering the sum of $v_2$ times the third line and $v_1$ times the sixth line:
\[ v_1 + i v_2 = ( v_1 \bar{v}_4 - v_2 \bar{v}_3 ) w_3 = 0 \]
so we obtain $v_2 = i v_1$. A similar reasoning yields $v_4 = -i v_3$, $w_2 = i w_1$ and $w_4 = - i w_3$. Writing $v_1 = \rho e^{i \theta}$, $v_3 = \rho e^{i \nu}$, $w_1 = \eta e^{i \mu}$ and $w_3 = \eta e^{i \tau}$, the third line in the above system reads
\[ \rho \eta \left( e^{i(\theta - \mu)} - e^{i(\tau - \nu)} \right) = i \]
which can be written as
\[ 2 i \rho \eta e^{i\left(\frac{\theta - \mu + \tau - \nu}{2}\right)} \sin\left( \frac{\theta - \mu - \tau + \nu}{2} \right) = i. \]
This implies that $\theta - \mu + \tau - \nu \in 2\pi \Z$. So if $\alpha = \theta - \mu$, then $\nu - \tau = \alpha$ modulo $2\pi$. Then necessarily $\alpha \in (0,\pi)$ modulo $2\pi$ and $\eta = \frac{1}{2 \rho \sin \alpha}$. So we finally find that 
\[ A_1 = \rho \begin{pmatrix} e^{i\theta} & i e^{i\theta} \\ e^{i\nu} & -i e^{i\nu} \end{pmatrix}, \qquad A_2 = \frac{e^{-i\alpha}}{2\rho \sin \alpha} \begin{pmatrix} e^{i\theta} & i e^{i\theta} \\ e^{i\nu} & -i e^{i\nu} \end{pmatrix} \]
(note that the function $\alpha \in \R \setminus \pi\Z \mapsto \frac{e^{-i\alpha}}{\sin \alpha}$ is $\pi$-periodic). Conversely, one readily checks that any such choice of $(\rho, \theta, \nu, \alpha) \in (0,+\infty) \times (0,2\pi) \times (0,2\pi) \times (0,\pi)$, the matrices $A_1, A_2$ displayed in the above equation provide a solution to \eqref{eq:omega_q}.

\paragraph{The case $m=-1$, $n=1$.} One readily checks that this time the solutions are 
\[ A_1 = \rho \begin{pmatrix} e^{i\theta} & -i e^{i\theta} \\ e^{i\nu} & i e^{i\nu} \end{pmatrix}, \qquad A_2 = \frac{e^{-i\alpha}}{2\rho \sin \alpha} \begin{pmatrix} e^{i\theta} & -i e^{i\theta} \\ e^{i\nu} & i e^{i\nu} \end{pmatrix} \]
for some $(\rho, \theta, \nu, \alpha) \in (0,+\infty) \times (0,2\pi) \times (0,2\pi) \times (0,\pi)$.
\end{proof}

Combining this with \cite[Proposition 4]{Zung97}, we immediately recover an alternative proof of the following result, originally obtained in \cite[Theorem 1.2]{Zung02} (see also \cite[Proposition 3.12]{HSS}), which we include for completeness. 

\begin{prop}
The weights of $J$ at the focus-focus point $p$ are $\{-1,1\}$.
\end{prop}

\begin{proof}
By Eliasson's theorem and the above considerations, there exists local complex coordinates $(z_1,z_2)$ near $p$ and a diffeomorphism $g: (\R^2,0) \to (\R^2,0)$ such that $(J,H) = g(q_{m,n},\tilde{q}_2)$ where $q_{m,n}$ is as in Equation \eqref{eq:normal_J} with $\{m,n\} = \{-1,1\}$ and $\tilde{q}_2 = q_2 \circ \varphi^{-1}$ with $\varphi$ a local symplectomorphism as in the previous lemma. By \cite[Proposition 4]{Zung97}, necessarily the first component of $g$ equals plus or minus the identity, so the weights of $J$ must also be $\{-1,1\}$.
\end{proof}

Lemma~\ref{lm:symp_nf} was useful in the above proof of determining the weights at a focus-focus point, but the primary aim of the computation in that lemma is to describe explicitly the set of all functions that $q_2$ can be sent to under a symplectomorphism bringing $q_1$ to normal form. 
Making use of Lemma~\ref{lm:symp_nf}, we accomplish this in the following proposition.

\begin{prop}
\label{prop:q2_nf}
Let $(z_1, z_2) = \varphi(x_1,x_2,\xi_1,\xi_2)$ with $\varphi: (\R^4, \omega_{\R}) \to (\C^2, \omega_{\C})$ a symplectomorphism such that $\varphi(0) = 0$ and $q_1 \circ \varphi^{-1} = q_{m,n}$ as in the normal form \eqref{eq:normal_J}. Let $q_2$ be the second Hamiltonian in Eliasson's normal form \eqref{eq:elia}. Then there exist $\lambda, \phi \in \R$ with $\lambda \geq 1$ and $\eta \in \{-1,1\}$ such that
\[ q_2 = \lambda \Re(e^{i \phi} z_1 z_2) + \eta \frac{\sqrt{\lambda^2 - 1}}{2} (|z_1|^2 + |z_2|^2)   + \bigO{3}. \]
Conversely, for any such $\lambda, \phi, \eta$, there exists a symplectomorphism $\varphi$ as above such that $q_1 \circ \varphi^{-1} = q_{m,n}$.
\end{prop}

\begin{proof}
Let $w_1 = x_1 + i x_2$, $w_2 = \xi_1 + i \xi_2$. 

\paragraph{The case $m=1$, $n=-1$.} In this case, by Lemma \ref{lm:symp_nf} there exists $(\rho, \theta, \nu, \alpha) \in (0,+\infty) \times (0,2\pi) \times (0,2\pi) \times (0,\pi)$ such that $(z_1,z_2) = \varphi(x_1,x_2,\xi_1,\xi_2)$ are given by
\[ z_1 = \rho e^{i\theta} w_1 + \frac{e^{i(\theta - \alpha)}}{2 \rho \sin \alpha} w_2 + \bigO{2}, \qquad z_2 = \rho e^{i\nu} \bar{w}_1 + \frac{e^{i(\nu - \alpha)}}{2 \rho \sin \alpha} \bar{w}_2 + \bigO{2},  \]
and a straightforward computation yields
\[ w_1 = -\frac{i}{2\rho \sin \alpha} \left( e^{-i(\theta - \alpha)} z_1 - e^{i(\nu - \alpha)} \bar{z}_2 \right) + \bigO{2}, \qquad  w_2 = i \rho \left( e^{-i\theta} z_1 - e^{i\nu} \bar{z}_2 \right) + \bigO{2}.  \]
Hence
\[ q_2 = x_1 \xi_1 + x_2 \xi_2 =\Re(w_1 \bar{w}_2) = -\frac{1}{2 \sin \alpha} \Re\left( e^{i\alpha} |z_1|^2 + e^{-i\alpha} |z_2|^2 - 2 \Re(e^{-i(\theta + \nu - \alpha)} z_1 z_2) \right) + \bigO{3}.  \]
So we finally obtain that 
\[ q_2 = \frac{1}{\sin \alpha} \Re(e^{-i(\theta + \nu - \alpha)} z_1 z_2) - \frac{1}{2} (|z_1|^2 + |z_2|^2) \cot \alpha  + \bigO{3}. \]
We obtain the result by setting $\lambda = \frac{1}{\sin \alpha}$, $\phi = \alpha - \theta - \nu$ and $\eta = -\mathrm{sgn}(\cos \alpha)$.  

\paragraph{The case $m=-1$, $n=1$.} In view of Lemma \ref{lm:symp_nf}, this case is obtained from the previous one by exchanging the roles of $z_1$ and $z_2$ and of $\theta$ and $\nu$. So we obtain the same formula for $q_2$.
\end{proof}

\subsubsection{Locally completing a periodic Hamiltonian to an integrable system}
\label{subsubsect:locally_completing}

In this section we classify all possible Hamiltonians which commute with $q_1 = \pm \left( \frac{1}{2}|z_1|^2 - \frac{1}{2}|z_2|^2 \right)$. This is useful when working directly in the complex coordinates in which $J$ is in normal form, that are obtained from Eliasson's coordinates as explained in Section \ref{subsubsec:normal_elia_weights}.

\begin{prop}
\label{prop:comm_q1_ff}
Let $q_1 = \varepsilon \left( \frac{1}{2}|z_1|^2 - \frac{1}{2}|z_2|^2 \right)$ with $\varepsilon \in \{-1,1\}$. Then $H \in \mathcal{C}^{\infty}(\C^2,\R)$ is a Hamiltonian such that $\{q_1, H\} = 0$, $H(0) = 0$ and $\dd H(0) = 0$ if and only if there exists $\mu_1,\mu_2,\mu_3,\psi \in \R$ and $R_3 \in \mathcal{C}^{\infty}(\C^2,\R)$ with $R_3(z_1,z_2) = O(3)$ and $\{ q_1, R_3 \} = 0$ such that 
\begin{equation} H(z_1,z_2) = \mu_1 \Re(e^{i\psi} z_1 z_2) + \mu_2 |z_1|^2 + \mu_3 |z_2|^2 + R_3(z_1,z_2). \label{eq:H_quad} \end{equation}
Moreover, if $H$ is of this form, then $(q_1,H)$ is integrable if and only if $(\mu_1,\mu_2 + \mu_3) \neq (0,0)$ and in this case the singular point $(0,0)$ of $(q_1,H)$ is 
\begin{itemize}
    \item of focus-focus type if $|\mu_2 + \mu_3| < |\mu_1|$;
    \item of elliptic-elliptic type if $|\mu_2 + \mu_3| > |\mu_1|$;
    \item degenerate if $|\mu_2 + \mu_3| = |\mu_1|$.
\end{itemize}
\end{prop}

In this statement $O(3)$ means $O(\|(z_1,z_2)\|^3)$; we will use this notation throughout the paper. 

\begin{rmk}
Notice that the type of the singular point in Proposition~\ref{prop:comm_q1_ff} does not depend on the choice of $\psi$. Therefore, when trying to choose such an $H$ to produce an integrable system it is possible to take $\psi=0$, and indeed this is the case in many common examples, see the discussion at the end of this section.
\end{rmk}

\begin{proof}[Proof of Proposition~\ref{prop:comm_q1_ff}]
By simply checking all possible quadratic terms and requiring that $H$ be real-valued, we conclude that any Hamiltonian commuting with $q_1$ and such that $H(0) = 0$ and $\dd H(0) = 0$ is necessarily of the form
\[ H = a \Re(z_1 z_2) + b \Im(z_1 z_2) + \mu_2 |z_1|^2 + \mu_3 |z_2|^2 + R_3(z_1,z_2) \]
with $a,b,\mu_2,\mu_3 \in \R$ and $R_3$ as in the statement. By putting $a + ib = \mu_1 e^{-i\psi}$, we see that $H$ must be of the form given in Equation \eqref{eq:H_quad}, and it remains to obtain the conditions on the scalars $\mu_1, \mu_2, \mu_3$. 

One readily computes
\[ \begin{vmatrix} \frac{\partial{J}}{\partial z_1}(z_1,z_2) & \frac{\partial{H}}{\partial z_1}(z_1,z_2) \\[2mm] \frac{\partial{J}}{\partial z_2}(z_1,z_2) & \frac{\partial{H}}{\partial z_2}(z_1,z_2) \end{vmatrix} = \frac{\varepsilon}{4} \left( \mu_1 e^{i\psi} (|z_1|^2 + |z_2|^2) + 2 (\mu_2 + \mu_3) \bar{z}_1 \bar{z}_2 \right). \]
When $(\mu_1,\mu_2 + \mu_3) = (0,0)$ this determinant always vanishes and $\dd J,\dd H$ are nowhere linearly independent. And when $(\mu_1,\mu_2 + \mu_3) \neq (0,0)$, it vanishes on a two-dimensional set.

Now, let $\Omega, A, B$ be the respective matrices of $\omega_{\C}$, $\dd^2 H$ and $\dd^2 q_1$ in the basis $(\partial_{z_1},\partial_{\bar{z}_1},\partial_{z_2},\partial_{\bar{z}_2})$, and let $\lambda \in \R$. One readily checks that
\[ C_{\lambda} := \Omega^{-1} A + \lambda \Omega^{-1} B = \begin{pmatrix} i(\varepsilon \lambda + 2 \mu_2) & 0 & 0 & i \mu_1 e^{-i\psi} \\ 0 & -i(\varepsilon \lambda + 2 \mu_2) & -i \mu_1 e^{i\psi} & 0 \\ 0 & i \mu_1 e^{-i\psi} & -i (\varepsilon \lambda - 2 \mu_3) & 0 \\ -i \mu_1 e^{i\psi} & 0 & 0 & i (\varepsilon \lambda - 2 \mu_3) \end{pmatrix}. \]
The reduced characteristic polynomial (see Section \ref{subsec:sing_semi}) of this matrix is 
\[ P_{\lambda} = X^2 + 2 \left( \left( \lambda + \varepsilon (\mu_2 - \mu_3) \right)^2 + (\mu_2 + \mu_3)^2 - \mu_1^2 \right) X + \left( \left( \lambda + \varepsilon (\mu_2 - \mu_3) \right)^2 + \mu_1^2 - (\mu_2 + \mu_3)^2 \right)^2. \]
and its discriminant reads
\[ \Delta_{\lambda} = 16 (\lambda + \varepsilon (\mu_2 - \mu_3))^2 \left( (\mu_2 + \mu_3)^2 - \mu_1^2 \right). \]
Let $\lambda \neq \varepsilon(\mu_3 - \mu_2)$ be such that $\left( \lambda + \varepsilon (\mu_2 - \mu_3) \right)^2 + \mu_1^2 - (\mu_2 + \mu_3)^2 \neq 0$. We investigate the three possible cases.
\begin{itemize}
    \item If $|\mu_2 + \mu_3| = |\mu_1|$, then $\Delta_{\lambda} = 0$ and $(0,0)$ is degenerate;
    \item if $|\mu_2 + \mu_3| < |\mu_1|$, then $\Delta_{\lambda} < 0$ so the eigenvalues of $C_{\lambda}$ are of the form $\pm \alpha \pm i \beta$ with $\alpha, \beta \in \R \setminus \{0\}$, hence $(0,0)$ is of focus-focus type;
    \item if $|\mu_2 + \mu_3| > |\mu_1|$, then $\Delta_{\lambda} > 0$, so $P_{\lambda}$ has two real roots $\sigma_1 \neq \sigma_2$. Since $\sigma_1 \sigma_2 = \left( \left( \lambda + \varepsilon (\mu_2 - \mu_3) \right)^2 + \mu_1^2 - (\mu_2 + \mu_3)^2 \right)^2 > 0$, $\sigma_1$ and $\sigma_2$ share the same sign. Since moreover $\sigma_1 + \sigma_2 = -2 \left( \left( \lambda + \varepsilon (\mu_2 - \mu_3) \right)^2 + (\mu_2 + \mu_3)^2 - \mu_1^2 \right) < 0$, we conclude that $\sigma_1, \sigma_2 < 0$. Therefore the eigenvalues of $C_{\lambda}$ are $\pm i \sqrt{-\sigma_1}, \pm i \sqrt{-\sigma_2}$ and $(0,0)$ is elliptic-elliptic. \qedhere
\end{itemize}
\end{proof}

\begin{rmk}
The form of the function $H$ in Proposition \ref{prop:q2_nf} is a particular case of the form in Proposition \ref{prop:comm_q1_ff}, and by using an appropriate diffeomorphism we can bring the former to the latter. Indeed, let $q_1 = \varepsilon \left( \frac{1}{2}|z_1|^2 - \frac{1}{2}|z_2|^2 \right)$ with $\varepsilon \in \{-1,1\}$ and let $H$ be as in Equation \eqref{eq:H_quad} with $|\mu_2 + \mu_3| < |\mu_1|$ and $R_3 = 0$. Let $b = \text{sgn}(\mu_1) \sqrt{ \mu_1^2 - (\mu_2 + \mu_3)^2 } \neq 0$. Let $g:\R^2 \to \R^2$ be the diffeomorphism given by
\[ \forall (x,y) \in \R^2 \qquad g(x,y) = (x, \varepsilon(\mu_2 - \mu_3) x + b y). \]
Then $(q_1,H) = g(q_1,q_2)$ where $q_2$ is of the form of Proposition \ref{prop:q2_nf}
with $\lambda = \frac{\mu_1}{b}$, $\eta = \text{sgn}\left(\frac{\mu_2 + \mu_3}{\mu_1}\right)$ and $\phi = \psi$.
\end{rmk}

Proposition \ref{prop:comm_q1_ff} gives the possible choices for $H$ when $J$ is in the normal form of Equation \eqref{eq:normal_J}. Of course there always exists local coordinates in which $J$ is in normal form, but these coordinates can be obtained in a relatively simple way in many known examples. Indeed, the Hamiltonian $S^1$-space underlying these examples is obtained through a symplectic reduction of $\C^d$ for some $d$ (via Delzant's algorithm) and the coordinates putting $J$ in normal form come from the coordinates on $\C^d$. Then in these natural coordinates we can apply Proposition \ref{prop:comm_q1_ff} to construct systems with the desired properties and check the type of the transition point. Below we illustrate this on examples from \cite{LFPfamilies}.

\begin{example}
\label{ex:W1_SF1}
The system on $\Hirzscaled{1}$ in \cite[Section 6.2]{LFPfamilies} is given by 
    \[ J = \frac{1}{2} |u_2|^2, \quad H_t = \frac{(1-2t)}{2} |u_3|^2 + t \gamma \Re(\bar{u}_1 u_3 \bar{u}_4) \]
    for $0 \leq t \leq 1$. The manifold is obtained as the reduction of $\C^4$ by $N: (u_1,u_2,u_3,u_4) \mapsto \frac{1}{2} \left( |u_1|^2 + |u_2|^2 + |u_3|^2, |u_3|^2 + |u_4|^2 \right)$ at level $(\alpha + \beta, \beta)$ and the transition point $C$ corresponds to $u_1 = 0 = u_4$. Using the local coordinates $u_1 = x_1 + i y_1$, $u_4 = x_4 + i y_4$ defined in \cite[Section 5.1]{LFPfamilies}, in which the symplectic form is $\frac{i}{2} (\dd u_1 \wedge \dd \bar{u}_1 + \dd u_4 \wedge \dd \bar{u}_4)$, we obtain that 
    \[ J = \alpha - \frac{1}{2} |u_1|^2  + \frac{1}{2} |u_4|^2; \quad H_t = (1-2t) \beta - \frac{(1-2t)}{2} |u_4|^2 + t \gamma \sqrt{2\beta} \Re(u_1 u_4) + \bigO{3}.  \]
    Hence $J$ is in normal form and the quadratic part of $H_t$ is of the form \eqref{eq:H_quad} with $\mu_1 = t \gamma \sqrt{2\beta}$, $\mu_2 = 0$ and $\mu_3 = - \frac{(1-2t)}{2}$. So by Proposition \ref{prop:comm_q1_ff}, the point $C$ is of focus-focus type if and only if $|1-2t| < 2 t \gamma \sqrt{2\beta}$, i.e. if and only if $t^- < t < t^+$ with $t^{\pm}$ as in \cite[Theorem 6.2]{LFPfamilies}, elliptic-elliptic if $0 \leq t < t^-$ or $t^+ < t \leq 1$ and degenerate if and only if $t \in \{t^-, t^+\}$.
\end{example}

\begin{example}
\label{ex:W2_SF1_72}
     The system on $\Hirzscaled{2}$ in \cite[Section 7.2]{LFPfamilies} is given by 
    \[ J = \frac{1}{2} \left( |u_2|^2 + |u_3|^2 \right), \quad H_t = (1-t) R + \frac{\beta t}{\alpha (\alpha + 2 \beta)} \left( \gamma \mathcal{X} + \left( 2 J - \alpha - 2 \beta \right) \left( R + \alpha + \beta \right) \right)  \]
    for $0 \leq t \leq 1$, with 
    \[ R = \frac{1}{2} (|u_3| - |u_4|^2) = \beta - |u_4|^2, \quad \mathcal{X} = \Re(\bar{u}_1\bar{u}_2 u_3 \bar{u}_4).\]
    The manifold is obtained as the reduction of $\C^4$ by 
    \[ N: (u_1,u_2,u_3,u_4) \mapsto \frac{1}{2} \left( |u_1|^2 + |u_2|^2 + 2 |u_3|^2, |u_3|^2 + |u_4|^2 \right)\] at level $(\alpha + 2 \beta, \beta)$ and the transition point $B$ corresponds to $u_2 = 0 = u_4$. Using the local coordinates $u_2 = x_2 + i y_2$, $u_4 = x_4 + i y_4$ defined in \cite[Section 5.1]{LFPfamilies}, in which the symplectic form is $\frac{i}{2} (\dd u_2 \wedge \dd \bar{u}_2 + \dd u_4 \wedge \dd \bar{u}_4)$, we obtain that
    \[ J = \beta + \frac{1}{2} |u_2|^2 - \frac{1}{2} |u_4|^2, \quad X = 2 \sqrt{\alpha \beta} \Re(u_2 u_4) + \bigO{3}  \]
    and a straightforward computation shows that  
    \[ H_t = (1 - 2t) \beta + \frac{2 \gamma t \beta \sqrt{\alpha \beta}}{\alpha (\alpha + 2 \beta)} \Re(u_2 u_4) + \frac{\beta t}{\alpha} |u_2|^2 + \left( \frac{(\alpha + 3 \beta) t}{\alpha + 2\beta} - \frac{\beta t}{\alpha} - 1 \right) |u_4|^2 + \bigO{3}.   \]
    So the quadratic part of $H_t$ is as in Equation \eqref{eq:H_quad} with 
    \[ \mu_1 = \frac{2 \gamma t \beta \sqrt{\alpha \beta}}{\alpha (\alpha + 2 \beta)} = \frac{c t \nu}{1 + 2 \nu}, \quad \mu_2 = \frac{\beta t}{\alpha} = \nu t, \quad \mu_3 = \frac{(\alpha + 3 \beta) t}{\alpha + 2\beta} - \frac{\beta t}{\alpha} - 1 = \frac{(1 + 3 \nu) t}{1 + 2\nu} - \nu t - 1 \]
    where $\nu = \frac{\beta}{\alpha}$ and $c = 2 \gamma \sqrt{\nu}$. So by Proposition \ref{prop:comm_q1_ff}, $B$ is of focus-focus type if $|\frac{(1 + 3 \nu) t}{1 + 2\nu} - 1| < \frac{c t \nu}{1 + 2 \nu}$, i.e. if $t^- < t < t^+$ with $t^{\pm}$ as in \cite[Theorem 7.2]{LFPfamilies}, elliptic-elliptic if $0 \leq t < t^-$ or $t^+ < t \leq 1$ and degenerate if $t \in \{t^-, t^+\}$.
\end{example}

\begin{example}  
\label{ex:W2_SF1_73}
     For the system on $\Hirzscaled{2}$ from \cite[Section 7.3]{LFPfamilies}, the general setting is the same as in the previous example, so we will omit some details. This system is given by 
    \[ J = \frac{1}{2} (|u_2|^2 + |u_3|^2), \quad H_t = (1-t) R + \frac{(\alpha + \beta) t}{\alpha (\alpha + 2 \beta)} \left( \gamma \mathcal{X} - \left( 2 J - \alpha - 2 \beta \right) \left( R + \frac{\beta^2}{\alpha + \beta} \right) \right), \]
    and the transition point $C$ is given by $u_1 = 0 = u_4$. In the corresponding complex coordinates, we have that
    \[ J = \alpha + \beta - \frac{1}{2} |u_1|^2 + \frac{1}{2} |u_4|^2, \quad H_t = (1-2t) \beta + \frac{\beta t}{\alpha} |u_1|^2 + \left( \frac{(2 \alpha + 3 \beta) t}{\alpha + 2\beta} - 1 - \frac{\beta t}{\alpha} \right) |u_4|^2 + \bigO{3}. \]
    So this time the quadratic part of $H_t$ is of the form \eqref{eq:H_quad} with
    \[ \mu_1 = \frac{c (1 + \nu) t}{1 + 2 \nu}, \quad \mu_2 = t \nu, \quad \mu_3 = \frac{(2 + 3 \nu) t}{1 + 2 \nu} - 1 - t \nu \]
    and another application of Proposition \ref{prop:comm_q1_ff} shows that $C$ is of focus-focus type if $t \in (t^-,t^+)$ with $t^{\pm}$ as in \cite[Theorem 7.5]{LFPfamilies}, elliptic-elliptic if $0 \leq t < t^-$ or $t^+ < t \leq 1$ and degenerate if $t \in \{t^-, t^+\}$.
\end{example}

In fact, the two systems from Examples \ref{ex:W2_SF1_72} and \ref{ex:W2_SF1_73} are part of a two-parameter family of systems studied in \cite[Section 7.4]{LFPfamilies} and in which two points undergo Hamiltonian-Hopf bifurcations. Originally the types of these points were only computed for a strict subset of the parameter space, but they can now easily be computed in all generality thanks to Proposition \ref{prop:comm_q1_ff}.

It would be interesting to perform the same analysis on other existing examples, such as the ones appearing in \cite{san-alvaro-spin,LFP,HohPal,HohMeu,AH-height}. In particular, it is easy to do so on the system described in Lemma \ref{lm:system_W0_HP}, obtained from the system studied in \cite{HohPal} for a particular choice of parameters.

Furthermore, when looking for additional examples, it is helpful to keep in mind that the Hamiltonian $H$ is constrained by Proposition \ref{prop:comm_q1_ff} once $J$ is in normal form, greatly limiting the possible forms of $H$.

\subsection{General strategy and the role of \texorpdfstring{$\Z_k$-spheres}{Zk-spheres}}
\label{subsect:general}

The previous section gives some insight about how to choose, once we have obtained a nice $(M,\omega,(J,H_0))$, the quadratic part of the second Hamiltonian $H_1$ (or $H_{\frac{1}{2}}$) so that $F_t = (J,H_t)$ with $H_t = (1-t) H_0 + t H_1$ (or $H_t = (1-2t) H_0 + 2t H_{\frac{1}{2}}$) is a semitoric or half-semitoric transition family (see Definitions \ref{def:semitoric-transition-family} and \ref{def:half-semitoric-family}). Let $m$ be the fixed point of $J$ that we want to be a transition point for $F_t$. First we obtain good local coordinates near $m$, in which $J$ is in normal form, and we can then choose the second Hamiltonian using Proposition \ref{prop:comm_q1_ff}. In practice, when trying to find a system with exactly one focus-focus point, in many cases it turns out that one of the representatives $\Delta$ of the unmarked semitoric polygon is Delzant (which is equivalent to the fact that the underlying Hamiltonian $S^1$-space can be lifted to a toric system, see \cite[Theorem 4.1]{HSS}), and our strategy consists in constructing $(M,\omega,(J,H_0))$ by applying Delzant's algorithm to $\Delta$; in particular $M$ is obtained as a symplectic reduction of some $\C^d$ by a Hamiltonian torus action, and the complex coordinates on $\C^d$ induce natural coordinates on $M$. Using such local coordinates near $m$, we may put $J$ in normal form; we then take $H_1$ (or $H_{\frac{1}{2}})$ of the form \eqref{eq:H_quad} locally, with a good choice of parameters, and try to extend it to a global Hamiltonian.

\begin{example}
Consider the setting of Example \ref{ex:W1_SF1} on $\Hirzscaled{1}$ coming from \cite[Section 6.2]{LFPfamilies}. Then $(M,\omega)$ is obtained as the symplectic reduction of $\C^4$ by $N: (u_1,u_2,u_3,u_4) \mapsto \frac{1}{2} \left( |u_1|^2 + |u_2|^2 + |u_3|^2, |u_3|^2 + |u_4|^2 \right)$ at level $(\alpha + \beta, \beta)$, and $(J,H_0)$ are given by the $N$-invariant functions $J = \alpha - \frac{1}{2} |u_1|^2  + \frac{1}{2} |u_4|^2$ and $H_0 = \frac{1}{2}|u_3|^2$. Slightly abusing notation, letting $u_1,u_4$ denote local coordinates near the transition point $m$, Proposition \ref{prop:comm_q1_ff} suggests that a natural choice for $H_{\frac{1}{2}}$ is $\Re(u_1 u_4)$. However, this is not well-defined globally, but $H_{\frac{1}{2}} = \Re(u_1 \bar{u}_3 u_4)$ is, and its quadratic part coincides (up to multiplication by a constant) with $\Re(u_1 u_4)$ in local coordinates near $m$. One can play the same game with the systems from Examples \ref{ex:W2_SF1_72} and \ref{ex:W2_SF1_73}.
\end{example}

However, one problem can arise with this strategy if $J$ possesses $\Z_k$-spheres. Indeed, it can happen that with this choice of $H_1$ (or $H_{\frac{1}{2}}$), the image by $F_t$ of one or several of these $\Z_k$-spheres intersects the interior of $F_t(M)$. By Lemma \ref{lem:Zk-boundary}, this means that $F_t$ cannot be semitoric.

\begin{example}
\label{ex:Z2_sphere_bad}
In Section \ref{sec:CP2}, $(M,\omega) = (\C\P^2, \alpha \omega_{\text{FS}})$ is obtained as the symplectic reduction of $\C^3$ by $N = \frac{1}{2}(|z_1|^2 + |z_2|^2 + |z_3|^2)$ at level $\alpha$, and $J = \frac{1}{2}(|z_1|-|z_2|^2)$ and $H_0 = \frac{1}{2} |z_3|^2$ are invariant under the action of $N$, hence descend to $M$. The transition point is $B = [0:0:1]$, and $(z_1,z_2)$ are good local coordinates near $B$. In view of Proposition \ref{prop:comm_q1_ff}, a natural choice for $H_{\frac{1}{2}}$ is $\Re(z_1 z_2)$ near $B$; this can be extended to the global Hamiltonian $H_{\frac{1}{2}} = \Re(z_1 z_2 \bar{z}_3^2)$. However, for $t \neq 0$, the $\Z_2$-sphere $\{z_3 = 0\}$ lies inside $F_t(M)$, so $F_t$ is not semitoric. 
\end{example}

Hence, we need to avoid this situation. The idea is to correct $H_1$ (or $H_{\frac{1}{2}}$) by adding a term that will ensure that the $\Z_k$-spheres will be sent to the boundary of $F_t(M)$, for all $t$. Of course, it is important to do so without modifying the quadratic part of this Hamiltonian near $m$, so that $m$ stays focus-focus. For instance, for the example on $\C\P^2$ from Section \ref{sec:CP2}, we correct the natural choice described in Example \ref{ex:Z2_sphere_bad} in such a way that the image of the $\Z_2$-sphere $\{z_3 = 0\}$ always lies on the bottom boundary of $F_t(M)$ (in other words, the points with $z_3 = 0$ are minima of $H$ in their $J$-fiber) by adding a multiple of $(|z_1|^2+|z_2|^2)^2$; see Remark \ref{rmk:boundary_CP2} for more details regarding this example. See also Remark \ref{rmk:boundary_3c} for an illustration of this correction on another explicit example.

In practice, as can be seen from the examples of Sections \ref{sec:CP2} and \ref{sec:type_3}, the final choice for $H_{\frac{1}{2}}$ will be slightly more complicated, since we will need to introduce some additional parameters. The main purpose of these parameters will be to control the range of height invariants that we can obtain. The idea is that there are two conflicting effects: roughly speaking, we want to push away the image of the $\Z_k$-sphere so that it stays on the boundary of $F_t(M)$, but at the same time we want the focus-focus point to come as close as possible to this $\Z_k$-sphere, so as to obtain the biggest possible range for the height invariant. The extra parameters help tuning the intensity of these two effects to reconcile them as much as possible.

\subsection{Rank one points in the presence of an effective Hamiltonian circle action}
\label{subsect:strat_rank_one}

Throughout the paper, we work with a Hamiltonian $S^1$-space $(M,\omega,J)$ (see Definition \ref{def:S1space}), and try to construct a one-parameter family $H_t \in \mathcal{C}^{\infty}(M)$ such that $(M,\omega,F=(J,H_t))$ is semitoric for some values of $t$. Checking this property for an explicit example can prove tedious, especially when dealing with rank one singular points and their types. This is why we propose in this section a few tools to simplify this task. 

Note that semitoric systems are not generic, and after making a choice of $H$ as above, $(M,\omega,F=(J,H))$ may only be hypersemitoric (see \cite{HohPal21}): it may exhibit hyperbolic-regular singularities, but also some special degenerate singular points which are called parabolic points. Our tools are also useful for studying these systems.

\subsubsection{Parabolic points}

Before specializing to our context, we recall the general definition of a parabolic point. For this purpose, we will need the following notation.

Let $S$ be a smooth manifold, let $m \in S$ and let $f: S \to \R$ be a smooth function such that $\dd f(m) = 0$. Let $v \in  \ker \dd^2 f(m)$. Then the quantity 
\[ \frac{\dd^3}{\dd t^3}\Big|_{t=0} f(\gamma(t)) \]
does not depend on the choice of the smooth curve $\gamma: I \to S$ ($I$ open interval containing $0$) such that $\gamma(0) = m$ and $\gamma'(0) = v$, and we denote it by $\dd^3_v f(m)$.

\begin{dfn}[{\cite[Definition 2.1]{BolGugKud}, see also \cite[Definition 2.37]{HohPal21}}]
\label{dfn:para}
Let $(M,\omega,F=(f_1,f_2))$ be an integrable system on a four-dimensional symplectic manifold and let $p \in M$ be such that $\dd f_1(p) \neq 0$ where $(f_1, f_2) = g \circ F$ for some local diffeomorphism $g: (\R^2, F(p)) \to \R^2$. Let $\tilde{f}_2$ be the restriction of $f_2$ to the level $f_1^{-1}(f_1(p))$. Then $p$ is a \emph{parabolic point} of $F$ if all the following assertions are satisfied:
\begin{enumerate}
\item $p$ is a critical point of $\tilde{f}_2$ (so there exists $k \in \R$ be such that $\dd f_2(p) = k \ \dd f_1(p)$);
\item $\mathrm{rank}(\dd \tilde{f}_2(p)) = 1$;
\item there exists $v \in \ker(\dd^2 \tilde{f}_2(p))$ such that $d^3_v \tilde{f}_2(p) \neq 0$;
\item $\mathrm{rank}(\dd^2(f_2 - k f_1)(p)) = 3$ with $k$ as in item 1.
\end{enumerate}
A \emph{parabolic value} is the image of a parabolic point by $F$.
\end{dfn}

It can be checked that this definition does not depend on the choice of $g$, see \cite{BolGugKud} for details. Roughly speaking, the four assertions in this definition mean that a parabolic point is a degenerate point that is ``as close to non-degenerate as possible''. We will see below that in the presence of a given Hamiltonian $S^1$-action, this definition becomes clearer. Note that a parabolic point always comes with a natural local Hamiltonian $S^1$-action (see \cite{KudMar}), but here we will work, as in the rest of the paper, with the global $S^1$-action coming from the structure of a Hamiltonian $S^1$-space. 

So consider $(M,\omega,F = (J,H))$ where $(M,\omega)$ is a four-dimensional compact symplectic manifold, $J,H$ are smooth functions from $M$ to $\R$, and $J$ generates an effective Hamiltonian $S^1$-action. 

As shown in \cite[Corollary 2.44]{HohPal21}, if $(M,\omega,F=(J,H))$ is a compact integrable system and $J$ generates a global effective $S^1$-action, then at a parabolic point $p$, $\dd J(p) \neq 0$; in fact one can show that the $S^1$-action is free in a neighborhood of $p$ (see for instance the discussion after Proposition 2.48 in \cite{HohPal21}). Moreover, in the explicit examples that we will study in Sections \ref{sec:CP2} and \ref{sec:type_3}, and as explained in Section \ref{subsect:general}, we will construct our systems so that the non-free points (elements of the $\Z_k$-spheres) for the $S^1$-action generated by $J$ are mapped to the boundary of the image of the momentum map, where there are no parabolic values. 

Let $p \in M$. Since we are trying to see if $p$ is parabolic, we can assume that the $S^1$-action generated by $J$ is free at $p$ (and hence in a neighborhood of $p$). Then we can perform symplectic reduction near $p$ and write conditions for $p$ to be parabolic using this reduction.

Let $\Sigma_j = J^{-1}(j)$ with $j = J(p)$, let $M_j^{\text{red}} = \Sigma_j \slash S^1$ be the reduced space with respect to the action generated by $J$, and let $\pi_j: M \to M_j^{\text{red}}, m \mapsto [m]$. Since the $S^1$-action is free at $p$, $M_j^{\text{red}}$ is locally near $[p]$ a smooth submanifold of $M/S^1$. 

Let $\mathcal{O}_p$ be the orbit of $p$ under the $S^1$-action, and recall that, using $\dd \pi_j(p)$, we have the natural identification
\[ T_{[p]} M_j^{\text{red}} \simeq T_p \Sigma_j \slash T_p \mathcal{O}_p = T_p \Sigma_j \slash \mathrm{Span}(X_J(p)). \] 
Recall also that if $W$ is a complementary subspace of $T_p \mathcal{O}_p$ in $T_p \Sigma_j$, then the restriction $\phi_W$ of $\dd \pi_j(p)$ to $W$ gives an isomorphism between $W$ and $T_{[p]} M_j^{\text{red}}$.

\begin{lm}
\label{lm:para_S1}
Let $J,H$ be smooth functions on a four-dimensional compact symplectic manifold $(M, \omega)$, such that $J$ generates an effective Hamiltonian $S^1$-action. Let $p \in M$ be such that the action of $J$ is free at $p$ (which implies that $\dd J(p) \neq 0$) and let $j = J(p)$. Then $p$ is a parabolic point for $F = (J,H)$ if and only if all the assertions below hold:
\begin{enumerate}
\item $[p]$ is a critical point of $H^{\mathrm{red},j}$ (in particular, $p$ is a critical point of $F$, so there exists $k \in \R$ such that $\dd H(p) = k \, \dd J(p)$);
\item $\mathrm{rank}(\dd^2 H^{\mathrm{red},j}([p])) = 1$;
\item for any nonzero $v \in \ker \dd^2 H^{\mathrm{red},j}([p])$, $\dd^3_v H^{\mathrm{red},j}([p]) \neq 0$;
\item $ 2 \gamma a b \neq \alpha b^2 + \beta a^2$ with $a = \dd^2 (H - k J)(p)\cdot(\phi_W^{-1}(w_1),\xi)$, $b = \dd^2 (H - kJ)(p)\cdot(\phi_W^{-1}(w_2),\xi)$, $\alpha = \dd^2 H^{\mathrm{red},j}([p])\cdot(w_1,w_1)$, $\beta = \dd^2 H^{\mathrm{red},j}([p])\cdot(w_2,w_2)$ and $\gamma =  \dd^2 H^{\mathrm{red},j}([p])\cdot(w_1,w_2)$, where $W$ is any complementary subspace of $T_p \mathcal{O}_p$ in $T_p \Sigma_j$, $(w_1,w_2)$ is any basis of $T_{[p]} M_j^{\mathrm{red}}$ and $\xi$ is any nonzero vector in a complementary subspace $L$ of $T_p \Sigma_j$ in $T_p M$. 
\end{enumerate}
\end{lm}

\begin{rmk}
The fourth condition may not seem very enlightening since it mixes conditions on the reduced Hamiltonian and on $H$ and involves the choice of the subspace $W$, but it will prove very useful when dealing with the special form of $H$ that we will study in the next section, see Lemma \ref{lm:para_polar}. 
\end{rmk}

\begin{rmk}
If $p\in M$ is a parabolic singular point and $[p]$ is its image in the reduced space by the $S^1$-action, then we also call $[p]$ parabolic.
\end{rmk}

\begin{proof}[Proof of Lemma~\ref{lm:para_S1}]
Note that since $\dd J(p) \neq 0$, we can take $f_1 = J$, $f_2 = H$ (i.e. $g = \text{id}$) in Definition \ref{dfn:para}. So we consider $\tilde{H} = H_{|\Sigma_j}$ and compare the conditions in Definition \ref{dfn:para} to the ones in the current statement.

Recall that $\tilde{H} = H^{\text{red},j} \circ \pi_j$ with $\pi_j$ the projection from $M$ to $M_j^{\text{red}}$. In particular $\dd \tilde{H}(p) = \dd H^{\text{red},j}([p]) \circ \dd\pi_j(p)$. Since $\dd\pi_j(p): T_p M \to T_{[p]} M_j^{\text{red}}$ is surjective, $[p]$ is a critical point of $H^{\text{red},j}$ if and only if $p$ is a critical point of $\tilde{H}$. For the remainder of this proof, we assume that these equivalent conditions are both satisfied.

In that case 
\[ \dd^2 \tilde{H}(p) = \dd^2 H^{\text{red},j}([p]) \circ (\dd\pi_j(p), \dd\pi_j(p)), \]
and again because of the surjectivity of $\dd\pi_j(p)$, this implies that the ranks of $\dd^2 \tilde{H}(p)$ and $\dd^2 H^{\text{red},j}(p)$ coincide. In particular, $\mathrm{rank}(\dd^2 \tilde{H}(p)) = 1$ if and only if $\mathrm{rank}(\dd^2 H^{\text{red},j}(p)) = 1$.

Moreover, if $u \in \ker \dd^2 \tilde{H}(p)$ and $\gamma: I \to M$ is a smooth curve with $\gamma(0) = p$ and $\gamma'(0) = u$, then $\nu = \pi_j \circ \gamma: I \to M_j^{\text{red}}$ is a smooth curve with $\nu(0) = [p]$ and $v := \nu'(0) = \dd\pi_j(p) \cdot u \in \ker \dd^2 H^{\text{red},j}([p])$. Since $\tilde{H} \circ \gamma = H^{\text{red},j} \circ \nu$, this implies that $\dd^3_u \tilde{H}(p) = \dd^3_v H^{\text{red},j}([p])$. So there exists $u \in \ker \dd^2 \tilde{H}(p)$ such that $\dd^3_u \tilde{H}(p) \neq 0$ if and only if there exists $v \in \ker \dd^2 H^{\text{red},j}([p])$ such that $\dd^3_v H^{\text{red},j}([p]) \neq 0$. Since $\dim \ker \dd^2H^{\text{red},j}([p]) = 1$ (since $\dim T_{[p]} M_j^{\text{red}} = 1$ and $\mathrm{rank}(\dd^2 \tilde{H}(p)) = 1$), the latter condition is equivalent to the fact that for any nonzero $v \in \ker \dd^2 H^{\text{red},j}([p])$, $\dd^3_v H^{\text{red},j}([p]) \neq 0$.

Finally, let $k \in \R$ be such that $\dd H(p) = k \dd J(p)$; then we claim that the matrix  of $\dd^2 (H - kJ)(p)$ in the basis $(X_J(p), \phi^{-1}(w_1), \phi^{-1}(w_2), \xi)$ of $T_p M$ is
\[ \begin{pmatrix} 0 & 0 \\ 0 & A \end{pmatrix}, \qquad A = \begin{pmatrix}  \alpha & \gamma & a \\
\gamma & \beta & b \\
a & b & \dd^2 (H - k J)(p)(\xi,\xi) \end{pmatrix} \]
with $a, b, \alpha, \beta$ and $\gamma$ as in the statement. To obtain this we have used once again that $\dd^2 \tilde{H}(p) = \dd^2 H^{\text{red},j}([p]) \circ (\dd\pi_j(p), \dd\pi_j(p))$, but also that for every $z \in T_p M$,
\[ \dd^2 (H - k J)(p) \cdot (X_J(p), z) = \left(\mathcal{L}_{Z} \mathcal{L}_{X_J} (H - k J)\right)(p) = 0 \]
(here $Z$ is any vector field extending $z$). This comes from the facts that $\mathcal{L}_{X_J} J = 0$ (by definition) and that $\mathcal{L}_{X_J} H = 0$ (since $J$ and $H$ Poisson commute). So $\mathrm{rank}(\dd^2(f_2 - k f_1)(p)) = 3$ if and only if 
\[ 0 \neq \det(A) = \dd^2 (H - k J)(p)(\xi,\xi) \left( \alpha \beta - \gamma^2 \right) - b (\alpha b - \gamma a) + a (\gamma b - \beta a).   \]
Now, assume that $\mathrm{rank}(\dd^2 H^{\text{red},j}(p)) = 1$ (equivalently $\mathrm{rank}(\dd^2 \tilde{H}(p)) = 1$). Then the determinant $\alpha \beta - \gamma^2$ of the matrix of $\dd^2 H^{\text{red},j}(p)$ in $(w_1,w_2)$ vanishes, so 
\[ \det(A) = 2 \gamma a b - \alpha b^2 - \beta a^2, \]
which concludes the proof.
\end{proof}

\subsubsection{Studying rank one points from the reduced space}
\label{subsect:rankone_reduced}

In this section, we explain how one can substantially reduce the amount of computations required to study rank one points of the integrable system $(M,\omega,F = (J,H))$ where $(M,\omega,J)$ is a Hamiltonian $S^1$-space, assuming that $H$ is of a suitable form which naturally arises in many examples.

Concretely, assume that we are working on a subset 
\[ U \subset \{ [p] \in M_j^{\text{red}}, \   \text{the } S^1\text{-action is free at } p \} \]
diffeomorphic to a cylinder and equipped with polar coordinates $(\rho,\theta) \in I \times [0,2\pi)$, $I = (\rho_-,\rho_+)$, $0 \leq \rho_- < \rho_+$, and that
\begin{equation}\label{eqn:Hred} H^{\text{red},j} = a(\rho^2)  - b(\rho^2) \cos \theta \sqrt{c(\rho^2)} 
\end{equation}
in these coordinates, where $a, b, c$ are smooth functions such that neither $a'$ nor $b$ vanishes on $K = (\rho_-^2, \rho_+^2)$ and $c_{|K} > 0$. We will explain in Remark \ref{rmk:Hred_polar} why these seemingly strong assumptions are extremely reasonable in our setting.

In this situation, 
\[ \frac{\partial H^{\text{red},j}}{\partial \theta}(\rho,\theta) = b(\rho^2) \sin \theta \sqrt{c(\rho^2)} \]
so if $(\rho_0,\theta_0)$ is a critical point for $H^{\text{red},j}$, then necessarily $\theta_0 \in \{0,\pi\}$, hence $\cos \theta_0 = \varepsilon \in \{-1,1\}$. Moreover,
\[ \frac{\partial H^{\text{red},j}}{\partial \rho}(\rho,\theta) =  \frac{\rho \left( 2 a'(\rho^2) \sqrt{c(\rho^2)} - \left( 2 b'(\rho^2) c(\rho^2) + b(\rho^2) c'(\rho^2) \right) \cos \theta \right)}{\sqrt{c(\rho^2)}} = \frac{\rho \left( f(\rho^2) - g(\rho^2) \cos \theta \right)}{\sqrt{c(\rho^2)}} \]
where
\begin{equation} f = 2 a' \sqrt{c} \qquad \textrm{ and } \qquad g = 2 b' c + b c'. \label{eq:fandg_general} \end{equation}
In view of Lemma \ref{lm:sing_red}, this immediately implies the following result.

\begin{lm}
\label{lm:crit_Hred_polar}
Assume that $H^{\text{red},j}$ is as in Equation \eqref{eqn:Hred}. Then the critical points of $H^{\text{red},j}$ in $U$ are the pairs $(\rho_{\varepsilon},\arccos \varepsilon) \in I \times [0,2\pi)$ where $\varepsilon \in \{-1,1\}$ and $\rho_{\varepsilon}^2$ is a solution of $f = \varepsilon g$.
\end{lm}

Furthermore, the Hessian matrix of $H^{\text{red},j}$ at such a critical point $(\rho_{\varepsilon},\arccos \varepsilon)$ reads
\begin{equation} \begin{pmatrix} \frac{\partial^2 H^{\text{red},j}}{\partial \rho^2} & \frac{\partial^2 H^{\text{red},j}}{\partial \rho \partial \theta} \\[2mm] \frac{\partial^2 H^{\text{red},j}}{\partial \rho \partial \theta} & \frac{\partial^2 H^{\text{red},j}}{\partial \theta^2} \end{pmatrix} = \begin{pmatrix} \frac{2 \rho_{\varepsilon}^2 (f'(\rho_{\varepsilon}^2) - \varepsilon g'(\rho_{\varepsilon}^2))}{\sqrt{c(\rho_{\varepsilon}^2)}}  & 0 \\[2mm] 0 & \varepsilon \ b(\rho_{\varepsilon}^2) \sqrt{c(\rho_{\varepsilon}^2)}  \end{pmatrix} \label{eq:hess_Hred_polar}\end{equation}
so we obtain the following result, invoking Lemma \ref{lm:sing_red} again.

\begin{lm}
\label{lm:type_rankone}
The critical point $(\rho_{\varepsilon},\arccos \varepsilon)$ is non-degenerate if and only if $f'(\rho_{\varepsilon}^2) \neq \varepsilon g'(\rho_{\varepsilon}^2)$, and in this case it corresponds to
\begin{itemize}
\item elliptic-regular points of $F$ if and only if $\varepsilon \ b(\rho_{\varepsilon}^2) (f'(\rho_{\varepsilon}^2) - \varepsilon g'(\rho_{\varepsilon}^2)) > 0$;
\item hyperbolic-regular points of $F$ if and only if $\varepsilon \ b(\rho_{\varepsilon}^2) (f'(\rho_{\varepsilon}^2) - \varepsilon g'(\rho_{\varepsilon}^2)) < 0$.
\end{itemize}
Moreover, in the elliptic-regular case, $(\rho_{\varepsilon},\arccos \varepsilon)$ is a local minimum (respectively maximum) of $H^{\text{red},j}$ if and only if $\varepsilon \ b(\rho_{\varepsilon}^2) > 0$ (respectively $\varepsilon \ b(\rho_{\varepsilon}^2) < 0$).
\end{lm}

Now we fix $j_0 \in J(M)$ and want to check, using the reduced Hamiltonian $H^{\text{red},j_0}$, whether a point $p$ with $J(p) = j_0$ is parabolic or not. We assume, as in the previous section, that the $S^1$-action generated by $J$ is free at $p$.

Then we want to translate Lemma \ref{lm:para_S1} in the current setting. Since the last condition in this lemma involves derivatives in a direction transverse to $T_p \Sigma_j$, it is necessary  to let $j$ vary in a neighborhood of $j_0$. This can be done as follows.

Since the $S^1$-action is free at $p$, it is free in a neighborhood of $p$ in $M$. As in \cite[Lemma A.14]{karshon} and the discussion before it, choose a compatible metric and consider the $S^1 \times \R$-action given by combining the Hamiltonian flow and the gradient flow of $J$. Let $C_p$ be the orbit of $p$ for this action, and note that the action of $(S^1 \times \R)$ is free on $C_p$ since it is free at $p$. By \cite[Lemma A.14]{karshon}, an invariant neighborhood $U$ of $C_p$ in $M$ is equivariantly diffeomorphic to $S^1 \times \mathcal{J} \times \mathcal{D}$ with coordinates $(t,j,\rho,\theta)$ where $\mathcal{J}$ is an open interval, $\mathcal{D} = \{ \rho_- < \rho < \rho_+, \theta_- < \theta < \theta_+ \}$ for some $0 \leq \rho_- < \rho_+$ and $\theta_-, \theta_+$ such that $0 < \theta_+ - \theta_- < 2\pi$, $J(t,j,x,y) = j$, $S^1$ acts on the first factor and the gradient flow acts on the second factor. We denote by $(t_0,j_0,\rho_0,\theta_0)$ the coordinates of the point $p$. 

Using these identifications, for any $j \in \mathcal{J}$, $U \cap M_j^{\text{red}} \simeq \mathcal{D}$ and for every $(t,j,\rho,\theta) \in S^1 \times \mathcal{J} \times \mathcal{D}$, $H(t,j,\rho,\theta) = H^{\text{red},j}(\rho,\theta)$. Recall also that $H$ is invariant with respect to the $S^1$-action generated by $J$, so $H(t,\cdot) = H(t',\cdot)$ for every $t,t' \in S^1$. In what follows, we further assume that $H$ is of the form
\begin{equation}\label{eqn:H_polar} H = a(j,\rho^2)  - b(j,\rho^2) \cos \theta \sqrt{c(j,\rho^2)}  \end{equation}
where $a, b, c$ are smooth functions of $(j, X) \in \mathcal{J} \times K$, where $K = (\rho_-^2,\rho_+^2)$, such that neither $\frac{\partial a}{\partial X}$ nor $b$ vanishes on $\mathcal{J} \times K$ and $c_{|\mathcal{J} \times K} > 0$. In this case, $H^{\text{red},j}$ is of the form \eqref{eqn:Hred} for any $j \in \mathcal{J}$. Similarly as above, we let
\begin{equation} f = 2 \frac{\partial a}{\partial X} \sqrt{c}, \qquad g = 2 \frac{\partial b}{\partial X} c + b \frac{\partial c}{\partial X}. \label{eq:fg_H_polar} \end{equation}

Using these coordinates, the statement of Lemma~\ref{lm:para_S1} greatly simplifies.

\begin{lm}
\label{lm:para_polar}
Assume that $H$ is of the form \eqref{eqn:H_polar}. The point $p = (t_0,j_0,\rho_0,\theta_0)$ is a parabolic point if and only if there exists $\varepsilon \in \{-1,1\}$ such that the following assertions are satisfied:
\begin{enumerate}
\item $(\rho_0,\theta_0) = (\rho_{\varepsilon},\arccos \varepsilon)$ where $\rho_{\varepsilon}$ is a solution of $f(j_0, \rho_{\varepsilon}^2) = \varepsilon g(j_0, \rho_{\varepsilon}^2)$;
\item $\frac{\partial f}{\partial X}(j_0,\rho_{\varepsilon}^2) = \varepsilon \frac{\partial g}{\partial X}(j_0,\rho_{\varepsilon}^2)$;
\item $\frac{\partial^2 f}{\partial X^2}(j_0,\rho_{\varepsilon}^2) \neq \varepsilon \frac{\partial^2 g}{\partial X^2}(j_0,\rho_{\varepsilon}^2)$;
\item $\frac{\partial f}{\partial j}(j_0,\rho_{\varepsilon}^2) \neq \varepsilon \frac{\partial g}{\partial j}(j_0,\rho_{\varepsilon}^2)$.
\end{enumerate}
Here $f$ and $g$ are the functions defined in Equation \eqref{eq:fg_H_polar}.
\end{lm}

\begin{proof}
By Lemma \ref{lm:crit_Hred_polar} and the above considerations, the first condition amounts to $[p]$ being a critical point of $H^{\text{red},j}$. 

In view of Equation \eqref{eq:hess_Hred_polar} (recall that $b$ does not vanish), the second assertion is equivalent to $\mathrm{rank}(\dd^2 H^{\text{red},j_0}([p])) = 1$, which is condition 2 in Lemma \ref{lm:para_S1}. 

In this case, the kernel of $\dd^2 H^{\text{red},j_0}([p])$ is spanned by $\partial_\rho$, so condition 3 in Lemma \ref{lm:para_S1} amounts to $\frac{\partial^3 H^{\text{red},j_0}}{\partial \rho^3}(\rho_{\varepsilon},\arccos \varepsilon) \neq 0$. But one readily checks that if $(\rho_{\varepsilon},\arccos \varepsilon)$ is a critical point of $H^{\text{red},j_0}$ and $\frac{\partial f}{\partial X}(j_0,\rho_{\varepsilon}^2) = \varepsilon \frac{\partial g}{\partial X}(j_0,\rho_{\varepsilon}^2)$, then
\[ \frac{\partial^3 H^{\text{red},j_0}}{\partial \rho^3}(\rho_{\varepsilon},\arccos \varepsilon) = \frac{4 \rho_{\varepsilon}^3}{\sqrt{c(j_0,\rho_{\varepsilon}^2)}} \left( \frac{\partial^2 f}{\partial X^2}(j_0,\rho_{\varepsilon}^2) - \varepsilon \frac{\partial^2 g}{\partial X^2}(j_0,\rho_{\varepsilon}^2) \right). \]

Moreover, we can take $w_1 = \partial_{\rho}$, $w_2 = \partial_{\theta}$ and $\xi = \partial_j$ in the statement of Lemma \ref{lm:para_S1}. Using the notation from that lemma and looking again at Equation \eqref{eq:hess_Hred_polar}, we obtain that 
\begin{itemize}
\item $\alpha = \frac{\partial^2 H^{\text{red},j_0}}{\partial \rho^2}(\rho_{\varepsilon},\arccos \varepsilon) = 0$;
\item $\gamma = \frac{\partial^2 H^{\text{red},j_0}}{\partial \rho \partial \theta}(\rho_{\varepsilon},\arccos \varepsilon) = 0$;
\item $\beta = \frac{\partial^2 H^{\text{red},j_0}}{\partial \theta^2}(\rho_{\varepsilon},\arccos \varepsilon) \neq 0$;
\end{itemize}
so the last condition in that lemma amounts to 
\[ 0 \neq a = \frac{\partial^2 (H - kj)}{\partial \rho \partial j}(p) = \frac{\partial^2 H}{\partial \rho \partial j}(p). \]
But one readily computes
\[ \frac{\partial^2 H}{\partial \rho \partial j} = \frac{\rho}{2 c(j,\rho^2)^{\frac{3}{2}}} \left( 2 \left( \frac{\partial f}{\partial j}(j,\rho^2) - \frac{\partial g}{\partial j}(j,\rho^2) \cos \theta  \right) c(j,\rho^2) - (f(j,\rho^2) - g(j,\rho^2) \cos \theta) \frac{\partial c}{\partial j}(j,\rho^2) \right) \]
so using that $f(j_0,\rho_{\varepsilon}^2) = \varepsilon g(j_0,\rho_{\varepsilon}^2)$, we obtain that 
\[ \frac{\partial^2 H}{\partial \rho \partial j}(p) = \frac{\rho}{\sqrt{c(j_0,\rho_{\varepsilon}^2)}} \left( \frac{\partial f}{\partial j}(j_0,\rho_{\varepsilon}^2) - \varepsilon \frac{\partial g}{\partial j}(j_0,\rho_{\varepsilon}^2) \right). \qedhere\]
\end{proof}

\begin{rmk}
\label{rmk:Hred_polar}
The above assumptions on the reduced space and Hamiltonian may seem restrictive at first glance but they are very natural in our context for several reasons that we detail here.

In the explicit examples from Sections \ref{sec:CP2} and \ref{sec:type_3} (but also in the examples from \cite{LFPfamilies}, \cite{HohPal} and \cite{HohMeu} for instance), the underlying Hamiltonian $S^1$-spaces $(M,\om,J)$ have the following properties: every non-extremal level of $J$ has at most two non-free orbits and the symplectic reduction at any non-extremal level of $J$ is homeomorphic to a sphere (note that these are exactly the conditions that the Hamiltonian $S^1$-space can be lifted to a Hamiltonian $T^2$-action~\cite{karshon}). Thus, the symplectic reduction of the level minus two orbits chosen to be the non-free ones, if any, is a smooth manifold diffeomorphic to a cylinder.

Moreover, after some natural choices, the quadratic part of the Hamiltonian $H$ of Equation \eqref{eq:H_quad} is of the form \eqref{eqn:Hred}; this means that, at least near a fixed point of $J$ with weights $\{-1,1\}$, the above assumptions are satisfied. Indeed, let $(J,H) = (q_1,H)$ be as in Proposition \ref{prop:comm_q1_ff}, with $q_1 = \left( \frac{1}{2} |z_1|^2 - \frac{1}{2} |z_2|^2 \right)$ (here we fix $\varepsilon = 1$ for the sake of simplicity) and $R_3 = 0$, and consider the symplectic reduction of the level $J^{-1}(j)$ minus the orbit $\{|z_1|^2 = 2j, z_2 = 0\}$ for $j \in \R$ (for $j = 0$, this orbit is the singular point $(0,0)$). Then by using the action of $q_1$, we find a representative $x_1 \in \R^+$ of $z_1$, and write $z_2 = \rho \exp(i\theta)$ with $\rho^2 > -2j$, so that $x_1 = \sqrt{2j + \rho^2}$. Then the reduced Hamiltonian reads
\[ H^{\text{red},j} = \mu_1 \rho \cos(\theta + \psi) \sqrt{2j + \rho^2} + 2 j \mu_2 + (\mu_2 + \mu_3) \rho^2, \]
so it is of the form \eqref{eqn:Hred} (after the simple change of variable $\theta \mapsto \theta + \psi$) with $a(X) = 2 j \mu_2 + (\mu_2 + \mu_3) X$, $b(X) = -\mu_1$ and $c(X) = X(2j+X)$.

Furthermore, as explained at the end of Section \ref{subsect:normal_form_H}, in many examples there are natural polar coordinates as above in which the reduced Hamiltonian $H^{\text{red},j}$ is of the form \eqref{eqn:Hred} for all non-extremal values of $j$. This is the case in the example of Section \ref{sec:CP2}, but also for the examples in \cite{LFPfamilies} (see in particular Sections 6.4 and 7.4 in that paper), in \cite[Section 7.1]{HohMeu} and in \cite[Section 3.3]{HohPal}.
\end{rmk}

\begin{rmk}
In the rest of the paper, we will omit the dependence on $j$ when looking for non-degenerate critical points, which, as seen above, does not require to differentiate with respect to $j$. In such cases we will simply write $u'$ for the differentiation of a function $u$ with respect to the variable $X$. 
\end{rmk}

\section{An explicit half-family for type (1)}
\label{sec:CP2}

As we saw in Section~\ref{sec:obstructions} (see Corollary \ref{cor:type_(1)_notfam}), type (1) minimal models cannot sit in a semitoric transition family, because they contain a $\Z_2$-sphere as one of their boundaries; in the example we will construct in this section, this $\Z_2$-sphere is the set $\{z_3 = 0\}$. Instead of looking for such a system in a semitoric transition family, in this section, using the general techniques developed in Section~\ref{sec:strategy}, we construct a new system which is the $H$-reflection of a minimal system of type (1) in a half-semitoric transition family (see Definition \ref{def:half-semitoric-family}). Note that there is another example of a type (1) system arising from a semitoric transition half-family in~\cite{CDEW}.

We view $M = \C\P^2$ as the symplectic reduction of $\C^3$ at level $\alpha > 0$ by the Hamiltonian
\[ N = \frac{1}{2} ( |z_1|^2 + |z_2|^2 + |z_3|^2 ) \]
equipped with the quotient symplectic form, which equals $\alpha \omega_{\text{FS}}$ where $\omega_{\text{FS}}$ is the Fubini-Study form on $\C\P^2$, normalized such that the Liouville volume of $\C\P^2$ equals $2 \pi^2$. Now, let 
\[ J = \frac{1}{2} ( |z_1|^2 - |z_2|^2 ), \quad R = \frac{1}{2} ( |z_1|^2 + |z_2|^2 ), \quad \mathcal{X} = \Re(z_1 z_2 \bar{z}_3^2), \quad Y = \Im(z_1 z_2 \bar{z}_3^2). \]
All of those are invariant under the action generated by $N$, so they all descend to $\C\P^2$. Note that the equation of the reduced space of $\C\P^2$ by the action of $J$ at level $j$ is 
\[ \mathcal{X}^2 + \mathcal{Y}^2 = 4 (R^2 - j^2) (\alpha - R)^2. \]

The image of the momentum map $(J,\alpha-R)$ is shown in Figure~\ref{fig:poly_CP2}.
Note that $(J,\alpha-R)$ is not toric, which can be seen for instance since its image is not Delzant, but it is of toric type, see Lemma~\ref{lm:CP2_zero}.

\begin{figure}
\begin{center}

\begin{tikzpicture}[scale=.8]
\filldraw[draw=black,fill=gray!60] (-2,0) node[left]{$(-\alpha,0)$}
  -- (2,0) node[right]{$(\alpha,0)$}
  -- (0,2) node[above]{$(0,\alpha)$}
  -- cycle;

\draw (0,0) node[below]{$z_3=0$};
\draw (-1.1,1) node[left]{$z_1=0$};
\draw (1.1,1) node[right]{$z_2=0$};

\end{tikzpicture}
\caption{Image of the momentum map for the system $F_0 = (J,H_0 = \alpha - R)$. Each of the edges equals $F_0(\{ z_{\ell} = 0 \})$ where $\ell \in \{1,2,3\}$ is indicated near the corresponding edge.}
\label{fig:poly_CP2}
\end{center}
\end{figure}
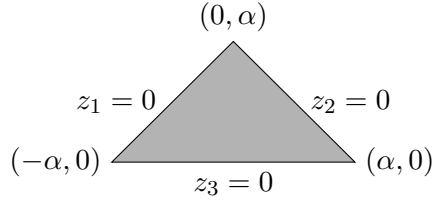

Consider the three points $A = [0,\sqrt{2\alpha},0]$, $B = [0,0,\sqrt{2\alpha}]$ and $C = [\sqrt{2\alpha},0,0]$ in $\C\P^2$.

\begin{thm}
\label{thm:CP2_sys}
Let $J = \frac{1}{2} ( |z_1|^2 - |z_2|^2 )$ and let $0 < \gamma < \frac{1}{4\alpha}$ and $\delta > \frac{1}{2 \gamma \alpha}$. Let 
\[ H_t = 2 \gamma \delta t \alpha^2 + (1 - 2t) \frac{|z_3|^2}{2} + 2 \gamma t (\mathcal{X} - \delta R^2).  \]
Then the family $(\C\P^2,\alpha \omega_{FS}, F_t=(J, H_t))_{0 \leq t \leq 1}$ is 
\begin{itemize}
    \item of toric type when $0 \leq t < t^-$;
    \item semitoric with one focus-focus point ($B = [0,0,\sqrt{2\alpha}]$) when $t^- < t < t^+$;
    \item hypersemitoric with one triangular flap with elliptic corner $F_t(B)$ when $t^+ < t \leq 1$
    \end{itemize}
where
\[ t^- = \frac{1}{2(1 + 2 \gamma \alpha)}, \quad t^+ = \frac{1}{2 (1 - 2 \gamma \alpha)}. \]
Moreover, when $t^- < t < t^+$, a representative of the polygonal invariant of this system is shown in Figure \ref{fig:polygon_CP2}. Furthermore, for any $h_0 \in (0,\frac{\alpha}{2})$, there exists a choice of $\gamma \in (0,\frac{1}{4\alpha})$, $\delta \in (\frac{1}{2\gamma \alpha},+\infty)$ and $t \in (t^-,t^+)$ such that the height invariant of the system is $h = h_0$.
\end{thm}
A proof of Theorem \ref{thm:CP2_sys} is given in Section~\ref{sec:CP2-proof}.
Additionally, we will give an upper bound on the size of the flap in Proposition \ref{prop:flap_CP2}. The image of the momentum map $F_t = (J,H_t)$ is displayed for some choices of the scaling $\alpha$ and of the parameters $\gamma$ and $\delta$ in Figures \ref{fig:image_CP2_alpha1_gamma1over8_delta5} and \ref{fig:image_CP2_alpha1_gamma1over5_delta3}.

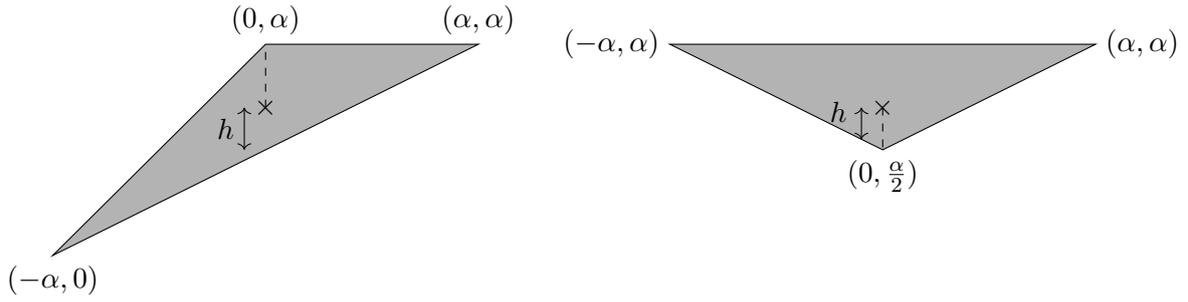
\begin{figure}
\begin{center}

\begin{tikzpicture}[scale=2.8]
\filldraw[draw=black, fill=gray!60] (-1,0) node[below]{$(-\alpha,0)$}
  -- (0,1) node[above]{$(0,\alpha)$}
  -- (1,1) node[above]{$(\alpha,\alpha)$}
  -- cycle;
\draw [dashed] (0,0.7) -- (0,1);
\draw (0,0.7) node[] {$\times$};

\draw [<->] (-0.1,0.5) -- (-0.1,0.7);
\draw (-0.1,0.6) node[left] {$h$};

\begin{scope}[xshift = 2.9cm]
    \filldraw[draw=black, fill=gray!60] (-1,1) node[left]{$(-\alpha,\alpha)$}
  -- (0,0.5)  node[below]{$(0,\frac{\alpha}{2})$}
  -- (1,1) node[right]{$(\alpha,\alpha)$}
  -- cycle;
\draw [dashed] (0,0.7) -- (0,0.5);
\draw (0,0.7) node[] {$\times$};

\draw [<->] (-0.1,0.55) -- (-0.1,0.7);
\draw (-0.12,0.68) node[left] {$h$};
\end{scope}

\end{tikzpicture}

\caption{Two representatives of the marked semitoric polygon for the system in Theorem \ref{thm:CP2_sys}. The polygon on the right is the $H$-reflection of the one shown in Figure \ref{fig:min_type1}.}
\label{fig:polygon_CP2}
\end{center}
\end{figure}

\begin{figure}
\begin{center}
\includegraphics[scale=0.4]{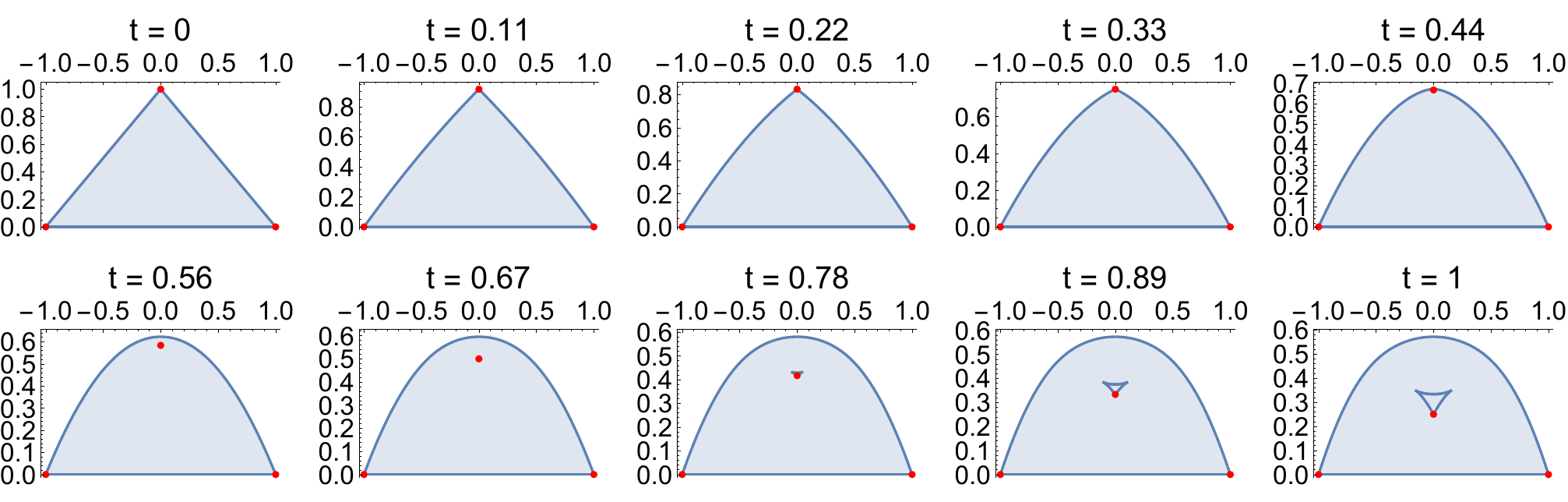}
\end{center}
\caption{Image of $(J,H_t)$ as in Theorem \ref{thm:CP2_sys} with $\alpha=1$, $\gamma = \frac{1}{8}$ and $\delta = 5$. Here $t^- = \frac{2}{5}$ and $t^+ = \frac{2}{3}$.}
\label{fig:image_CP2_alpha1_gamma1over8_delta5}
\end{figure}

\begin{figure}
\begin{center}
\includegraphics[scale=0.4]{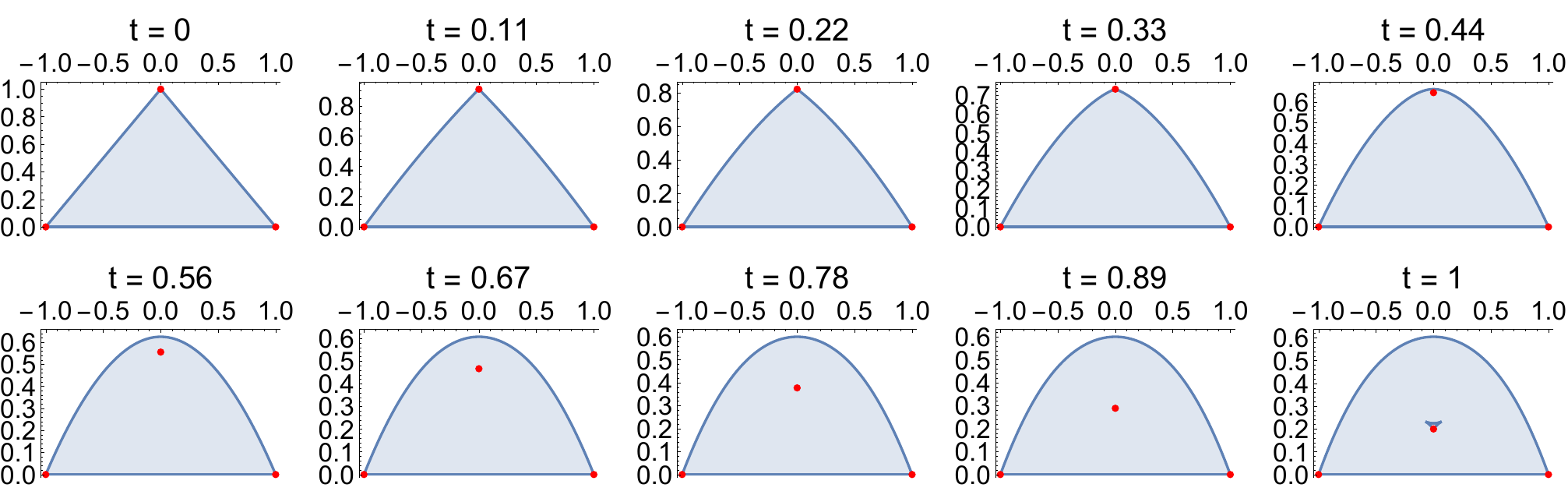}
\end{center}
\caption{Image of $(J,H_t)$ as in Theorem \ref{thm:CP2_sys} with $\alpha=1$, $\gamma = \frac{1}{5}$ and $\delta = 3$. Here $t^- = \frac{5}{14}$ and $t^+ = \frac{5}{6}$.}
\label{fig:image_CP2_alpha1_gamma1over5_delta3}
\end{figure}

\begin{rmk}
The freedom in the choice of $\alpha$, $t$, $\gamma$ and $\delta$ is necessary to recover all possible marked semitoric polygons associated with the minimal helix of type $(1)$. The parameter $\alpha$ determines the global scaling of the polygon and the parameters $t$, $\gamma$ and $\delta$ determine the height invariant of the system.  
\end{rmk}

\begin{rmk}
From the upper bound on $\gamma$ we see that $\delta > 2$, and we will intensively use this inequality in the proof of the theorem.
\end{rmk}

\begin{rmk}
\label{rmk:boundary_CP2}
Even though it would give the desired behavior near the transition point, taking $H_{\frac{1}{2}} = X$ would not produce a semitoric system, as explained in Section \ref{subsect:general} (in particular Example \ref{ex:Z2_sphere_bad}), because the line of critical values given by the image of the $\Z_2$-sphere $\{z_3 = 0\}$ would lie inside the interior of the image of $(J,H_{\frac{1}{2}} )$;
semitoric systems can only contain such lines of critical points in the boundary of their image, and thus the general idea is to modify the system to ``push'' the image of the $\Z_2$-sphere to the boundary. Now we will show how using $K = X - \delta R^2$ instead of $X$ sends $\{z_3=0\}$ to the bottom boundary of the image of the momentum map $F_t = (J,H_t)$ for all values of $t \in [0,1]$.

Let $z,w$ be representatives of points in $J^{-1}(j)$ with $w_3 = 0$, and let 
\[ \varepsilon = K(z_1,z_2,z_3) - K(w_1,w_2,0) = \Re(z_1 z_2 \bar{z}_3^2) + \delta (\alpha^2 - R(z_1,z_2,z_3)^2) \]
where the last equality comes from the fact that $R(w_1, w_2, 0) = N(w_1, w_2, 0) = \alpha$ and $X(w_1,w_2,0)=0$. 
We now bound $\varepsilon$ by below. Note that 
\[ \alpha^2 - R(z_1,z_2,z_3)^2 = \left( \alpha - R(z_1,z_2,z_3) \right) \left( \alpha + R(z_1,z_2,z_3) \right) = \frac{1}{2} |z_3|^2 \left( 2\alpha - \frac{1}{2} |z_3|^2 \right). \]
After expanding and writing $\alpha$ in terms of $z_1, z_2, z_3$, we obtain
\[ \varepsilon =  \Re(z_1 z_2 \bar{z}_3^2) + \delta\left( \frac{1}{2} |z_3|^2 \left( |z_1|^2 + |z_2|^2 \right) + \frac{1}{4} |z_3|^4 \right).   \]
But a well-known inequality yields
\[ \Re(z_1 z_2 \bar{z}_3^2) = \Re((z_1 \bar{z}_3)(z_2 \bar{z}_3)) \geq - \frac{1}{2}  \left( |z_1|^2 |z_3|^2 + |z_2|^2 |z_3|^2 \right). \]
Consequently, since $\delta \geq 1$
\[ \varepsilon \geq \frac{(\delta - 1)}{2} |z_3|^2 \left( |z_1|^2 + |z_2|^2 \right) + \frac{\delta}{4} |z_3|^4.  \]
Thus, for any $t \in [0,1]$,
\[ \begin{split} H_t(z_1, z_2, z_3) - H_t(w_1, w_2, 0) & = (1 - 2 t) \frac{1}{2}|z_3|^2 + 2 \gamma t \varepsilon  \\
& \geq (1 - 2 t) \frac{1}{2}|z_3|^2 + 2 \gamma t \left( \frac{(\delta - 1)}{2} |z_3|^2 \left( |z_1|^2 + |z_2|^2 \right) + \frac{\delta}{4} |z_3|^4 \right) \\
& = (1 - 2 t) \frac{1}{2}|z_3|^2 + 2 \gamma t \left( \frac{(\delta - 1)}{2} |z_3|^2 \left( 2 \alpha - |z_3|^2 \right) + \frac{\delta}{4} |z_3|^4 \right) \\
& = (1 - 2 t) \frac{1}{2}|z_3|^2 + \gamma t |z_3|^2 \left( 2 \alpha (\delta - 1) + \left( 1 -  \frac{\delta}{2} \right) |z_3|^2 \right).\end{split}\]
Since $\delta > 2$ and $|z_3|^2 \leq 2\alpha$, we obtain that 
\[ H_t(z_1, z_2, z_3) - H_t(w_1, w_2, 0)  \geq  \frac{1}{2}|z_3|^2 \left( 1 + 2 (\gamma \delta \alpha - 1) t \right). \]
Furthermore, since $\gamma \delta \alpha > 1/2$ we see that $1 + 2 (\gamma \delta \alpha - 1) t > 1 - t$, and thus we conclude that the above quantity is always non-negative and vanishes if and only if $z_3 = 0$. Hence $\{z_3 = 0\}$ consists of minima of $H_t$, and therefore its image lies in the boundary of the image of $F_t$.
\end{rmk} 

The rest of this section is devoted to the proof of Theorem \ref{thm:CP2_sys}. Since we will often need to divide by $t$, we treat the case $t=0$ separately.

\begin{lm}
\label{lm:CP2_zero}
For $t=0$, the system is of toric type.
\end{lm}

\begin{proof}
Since $\left( \frac{1}{2} |z_1|^2, \frac{1}{2} |z_3|^2 \right)$ is a toric momentum map, this comes from the fact that 
\[ (J,H_0) = \left( \frac{1}{2}(|z_1|^2 - |z_2|^2), \frac{1}{2} |z_3|^2 \right) = \left( |z_1|^2 + \frac{1}{2} |z_3|^2 - \alpha, \frac{1}{2} |z_3|^2 \right) = g\left( \frac{1}{2} |z_1|^2, \frac{1}{2} |z_3|^2 \right) \]
where 
\[ g: \R^2 \to \R^2, \quad (x,y) \mapsto (2x + y - \alpha, y) \]
is a diffeomorphism.
\end{proof}

So in what follows we always assume that $0 < t \leq 1$.

\subsection{Rank zero points}
\label{subsec:rank_zero_CP2}

The rank zero points of $F_t$ are in particular fixed points of $J$. One readily checks that all these fixed points, namely $A$, $B$, $C$, are indeed rank zero critical points of $F_t$. We now check that they are all non-degenerate and compute their types.

\paragraph{The transition point.} The transition point $B$ corresponds to $z_1 = 0 = z_2$, and $H_t(B) = 2 \gamma \delta t \alpha^2 + (1 - 2t) \alpha$. Near this point, we use the action of $N$ to assume that $z_3 = x_3 \in \R^+$ and obtain local complex coordinates $z_1, z_2$ such that $\omega = \frac{i}{2} \left( \dd z_1 \wedge \dd \bar{z}_1 + \dd z_2 \wedge \dd \bar{z}_2  \right)$. Then 
\[ x_3 = \sqrt{2\alpha - |z_1|^2 - |z_2|^2}. \]
Recall that $J = \frac{1}{2}(|z_1|^2 - |z_2|^2)$. A straightforward computation gives
\[ H_t = 2 \gamma \delta t \alpha^2 + (1 - 2t) \alpha + \frac{(2t-1)}{2}( |z_1|^2 + |z_2|^2) + 4 \gamma t \alpha \Re(z_1 z_2) + O(3). \]
Hence $H_t$ is as in Proposition \ref{prop:comm_q1_ff} with $\mu_1 = 4 \gamma t \alpha$ and $\mu_2 = \mu_3 = \frac{(2t-1)}{2}$. Therefore, $B$ is 
\begin{itemize}
    \item of focus-focus type if $|2t-1| < 4 \gamma t \alpha$, i.e. if $t^- < t < t^+$ where we recall that $t^{\pm} = \frac{1}{2(1 \mp 2 \gamma t \alpha)}$;
    \item of elliptic-elliptic type if $0 \leq t < t^-$ or $t^+ < t \leq 1$;
    \item degenerate if $t \in \{t^-,t^+\}$.
\end{itemize}

\paragraph{Other rank zero points.} Near the point $A$, we use the action generated by $N$ to assume that $z_2 = x_2 \in \R^+$ and obtain use local real symplectic coordinates $x_1, y_1, x_3, y_3$, so that $x_2 = \sqrt{2\alpha - (x_1^2 + y_1^2) - (x_3^2 + y_3^2)}$. We then compute in these coordinates
\[ J = -\alpha + x_1^2 + y_1^2 + \frac{1}{2}(x_3^2 + y_3^2), \qquad H_t = \frac{1}{2} \left( 1 - 2 t + 4\gamma\delta\alpha t \right) (x_3^2 + y_3^2) + O(3) \]
so that
\[ \Omega^{-1} \dd^2 J(A) = \begin{pmatrix} 0 & 2 & 0 & 0 \\ -2 & 0 & 0 & 0 \\  0 & 0 & 0 & 1 \\ 0 & 0 & -1 & 0 \end{pmatrix}, \Omega^{-1} \dd^2 H_t(A) = \begin{pmatrix} 0 & 0 & 0 & 0 \\ 0 & 0 & 0 & 0 \\ 0 & 0 & 0 & 1 -  2 (1 - 2 \gamma \delta \alpha) t \\ 0 & 0 & 2 (1 - 2 \gamma \delta \alpha) t - 1 & 0 \end{pmatrix} \]
which implies that
\[ \Omega^{-1} \dd^2 H_t(A) + \Omega^{-1} \dd^2 J(A) = \begin{pmatrix} 0 & 2 & 0 & 0 \\ -2 & 0 & 0 & 0 \\ 0 & 0 & 0 & 2 \left( 1 - (1 - 2 \gamma \delta \alpha) t \right) \\ 0 & 0 & -2 \left( 1 - (1 - 2 \gamma \delta \alpha) t \right) & 0 \end{pmatrix}. \]
The eigenvalues of this matrix are $\pm 2i, \pm 2 i \left( 1 - (1 - 2 \gamma \delta \alpha) t \right)$, hence $A$ is of elliptic-elliptic type for all $t\in(0,1]$. Indeed, these eigenvalues are always distinct since $1 - (1 - 2 \gamma \delta \alpha) t > 1$ for $t > 0$.

Similarly, near the point $C$, we use local coordinates $x_2, y_2, x_3, y_3$ to obtain
\[ \Omega^{-1} \dd^2 J(C) = \begin{pmatrix} 0 & -2 & 0 & 0 \\ 2 & 0 & 0 & 0 \\  0 & 0 & 0 & -1 \\ 0 & 0 & 1 & 0 \end{pmatrix}, \Omega^{-1} \dd^2 H_t(C) = \begin{pmatrix} 0 & 0 & 0 & 0 \\ 0 & 0 & 0 & 0 \\ 0 & 0 & 0 & 1 -  2 (1 - 2 \gamma \delta \alpha) t \\ 0 & 0 & 2 (1 - 2 \gamma \delta \alpha) t - 1 & 0 \end{pmatrix} \]
and conclude in a similar fashion that $C$ is elliptic-elliptic for all $t\in(0,1]$.

\subsection{Rank one points}
\label{subsec:rank_one_CP2}

Now we check the rank one critical points of the system. We will want to use the same coordinates $(\rho,\theta)$ as in Section \ref{subsect:strat_rank_one} on $M_j^{\text{red}}$ minus two points. So first we see that we can always discard two such points in our study.

We have already seen that the rank zero critical points of $F_t$ are exactly the points which satisfy $\dd J=0$. Thus, we may apply the technique of Lagrange multipliers and conclude that the rank one critical points are the solutions of the following system of equations obtained by writing $\nabla H_t(z) = \lambda \nabla J(z) + \mu \nabla N(z)$:
\[ \begin{cases} \gamma t z_2 \bar{z}_3^2 - \gamma \delta t \bar{z}_1 (|z_1|^2 + |z_2|^2) = \frac{\lambda}{2} \bar{z}_1 + \frac{\mu}{2} \bar{z}_1, \\[2mm]
 \gamma t z_1 \bar{z}_3^2 - \gamma \delta t \bar{z}_2 (|z_1|^2 + |z_2|^2) = -\frac{\lambda}{2} \bar{z}_2 + \frac{\mu}{2} \bar{z}_2,  \\[2mm]
 \frac{(1-2t)}{2} \bar{z}_3 + 2 \gamma t \bar{z}_1 \bar{z}_2 z_3 = \frac{\mu}{2} \bar{z}_3. \end{cases} \]
From the first line, we see that if $z_1=0$, then either $z_2=0$ (which gives the rank zero point $B$ of $F$) or $z_3=0$ (which gives the rank zero point $A$). The second line gives similar results if we assume that $z_2=0$. So there is no rank one point with $z_1 = 0$ or $z_2 = 0$ (note, though, that $z_3 = 0$ may occur, as expected).

\subsubsection{Rank one points on the $\Z_2$-sphere}
\label{subsec:rankone_Z2_CP2}

First we prove that the rank one critical points with $z_3 = 0$ are always non-degenerate of elliptic-regular type. Indeed, these points are special for the following reason. Recall that the $S^1$-action generated by $J$ is not free at the singular points which have $z_3=0$ (this set forms a $\Z_2$-sphere of the $S^1$-action generated by $J$), and thus the reduced space is not smooth at the points corresponding to these.
 Thus we cannot use the coordinates discussed in Section~\ref{subsect:strat_rank_one} to check the type and non-degeneracy of the singular points with $z_3=0$. Instead, we proceed by using the general definition of a non-degenerate rank one point as stated in Definition \ref{def:nondeg_rankone}.

Notice that a rank one point $m \in M$ with $z_3 = 0$ always satisfies $\dd H_t(m) = 0$. So we need to compute the eigenvalues of $\Omega^{-1} A_t $ where $A_t$ is the matrix of the restriction of $\dd^2 H_t(m)$ to $L^{\perp} \slash L$ where $L = \text{Span}(X_J(m))$, $L^{\perp}$ is the symplectic orthogonal of $L$ and $\Omega$ is the matrix of the symplectic form on $L^{\perp} \slash L$. We will work in the chart $z_1 \neq 0$ with local symplectic coordinates $(z_2,z_3) = (x_2,y_2,x_3,y_3)$; in this case, by using the action of $N$, we may choose a representative $x_1 \in \R^+$ of $z_1$ and obtain $x_1 = \sqrt{2\alpha - (x_2^2 + y_2^2 + x_3^2 + y_3^2)}$. First we note that $J = \alpha - |z_2|^2 - \frac{1}{2}|z_3|^2$ so in these local coordinates
\[ J = \alpha - (x_2^2 + y_2^2) - \frac{1}{2} (x_3^2 + y_3^2). \]
In particular $m = (x_2,y_2,0,0)$ for some $(x_2,y_2)$ and $X_J(m) = 2 y_2 \partial_{x_2} - 2 x_2 \partial_{y_2}$ so 
\[ L^{\perp} = \text{Span}\left\{ \partial_{x_3}, \partial_{y_3}, X_J(m) \right\}, \qquad  L^{\perp} \slash L \simeq \text{Span}\left\{ \partial_{x_3}, \partial_{y_3} \right\}.  \]
Moreover, in these coordinates,
\[ \begin{split} H_t & = \frac{(1-2t)}{2}(x_3^2 + y_3^2) \\
& + 2 \gamma t \left( \left(x_2 (x_3^2 - y_3^2) + 2 y_2 x_3 y_3\right) \sqrt{2\alpha - (x_2^2 + y_2^2 + x_3^2 + y_3^2)} - \delta \left( \alpha - \frac{1}{2}(x_3^2 + y_3^2) \right)^2  \right) \end{split} \]
so the matrix $A_t$ of the restriction of $\dd^2 H_t(m)$ to $L^{\perp} \slash L$ satisfies
\[ A_t =  \begin{pmatrix} 1 -  2 t + 4 \gamma t (\delta \alpha + x_2 \sqrt{2\alpha - (x_2^2 + y_2^2)}) & 4 \gamma t y_2  \sqrt{2\alpha - (x_2^2 + y_2^2)} \\ 4 \gamma t y_2  \sqrt{2\alpha - (x_2^2 + y_2^2)} & 1 -  2 t + 4 \gamma t (\delta \alpha - x_2  \sqrt{2\alpha - (x_2^2 + y_2^2)})  \end{pmatrix} \]
and consequently
\[ \Omega^{-1} A_t =  \begin{pmatrix} -4 \gamma t y_2  \sqrt{2\alpha - (x_2^2 + y_2^2)} & -(1 -  2 t + 4 \gamma t (\delta \alpha - x_2  \sqrt{2\alpha - (x_2^2 + y_2^2)})) \\ 1 -  2 t + 4 \gamma t (\delta \alpha + x_2  \sqrt{2\alpha - (x_2^2 + y_2^2)}) & 4 \gamma t y_2  \sqrt{2\alpha - (x_2^2 + y_2^2)} \end{pmatrix}.\]
Hence
\[ \det(\Omega^{-1} A_t) = (1-2t + 4\gamma t \delta \alpha)^2 - 16\gamma^2 t^2 (x_2^2 + y_2^2)(2\alpha - (x_2^2 + y_2^2))  \]
so if moreover $J(m) = j$ (with $-\alpha < j < \alpha$ since we exclude the rank zero points), we obtain, by using that $J(m) = \alpha - (x_2^2 + y_2^2)$, that
\[ \det(\Omega^{-1} A_t) = (1-2t + 4\gamma t \delta \alpha)^2 - 16\gamma^2 t^2 (\alpha^2 - j^2) = (1 - 2t)^2 + 8 \gamma t \delta \alpha (1 - 2t) + 16 \gamma^2 t^2 \left( (\delta^2 - 1) \alpha^2 + j^2 \right).  \]
We have that $\det(\Omega^{-1} A_t) \geq P(t)$, where $P$ is the degree two polynomial in $t$:
\[ P(t) = 1 + 4 \left(2 \gamma \delta n - 1\right) t + 4 \left( 1 - 4 \gamma \delta \alpha + 4 \gamma^2  (\delta^2 - 1) \alpha^2   \right) t^2. \]
One readily checks that 
\[ P(t) =  \left( 2 \left( 1 - 2 \gamma \delta \alpha - 2 \gamma \alpha \right) t - 1 \right) \left( 2 \left( 1 - 2 \gamma \delta \alpha + 2 \gamma \alpha \right) t - 1 \right).  \]
Since $\delta > \frac{1}{2\gamma \alpha}$, we have that
\[ 2 \left( 1 - 2 \gamma \delta \alpha - 2 \gamma \alpha \right) t - 1 < -(1 + 4 \gamma \alpha t), \quad 2 \left( 1 - 2 \gamma \delta \alpha + 2 \gamma \alpha \right) t - 1 < 4 \gamma \alpha t - 1. \]
Since moreover $4 \gamma \alpha < 1$, we have that $4 \gamma \alpha t - 1 < 0$ for every $t \in [0,1]$. Therefore, the above inequalities yield
\[ \forall t \in (0,1] \qquad P(t) > (1 + 4 \gamma \alpha t)(1 - 4 \gamma \alpha t) > 0. \]
Consequently, for every $t \in [0,1]$, $\det(\Omega^{-1} A_t) > 0$. So in view of Definition \ref{def:nondeg_rankone}, the rank one points with $z_3 = 0$ are non-degenerate of elliptic-regular type. 

\subsubsection{Other rank one points}
\label{subsect:other_rankone_CP2}

Now we investigate the rank one points with $z_3 \neq 0$. Note that $H_t$ is invariant with respect to the action of switching $z_1$ and $z_2$, which has the effect of changing the sign of $J$; in other words
\begin{equation} (J(z_2,z_1,z_3), H_t(z_2,z_1,z_3)) = (-J(z_1,z_2,z_3), H_t(z_1,z_2,z_3)). \label{eq:sym_CP2} \end{equation}
So it suffices to understand what happens for $0 \leq j < \alpha$. 

Following the previous results, we will now work on $M_j^{\text{red}} \setminus \left( \{z_1 = 0 \} \cup \{ z_3 = 0\} \right)$, which is a smooth manifold diffeomorphic to a cylinder. We want to prove the following:
\begin{prop}
\label{prop:summary_CP2}
We have the following cases on the types of the singularities of $H_{t}^{\mathrm{red},j}$ in $M_j^{\mathrm{red}} \setminus \left( \{z_1 = 0 \} \cup \{ z_3 = 0\} \right)$.
\begin{itemize}
    \item For $0 < t<t^-$:
    \begin{itemize}
        \item if $j=0$ then, $H_{t}^{\mathrm{red},j}$ has no critical points;
        \item if $j>0$ then $H_{t}^{\mathrm{red},j}$ has one critical point which is elliptic (local maximum);
    \end{itemize}
    \item for $t^-<t<t^+$, for all $j\geq 0$, $H_{t}^{\mathrm{red},j}$ has one critical point which is elliptic (local maximum);
    \item for $t^+ < t \leq 1$:
    \begin{itemize}
        \item if $j=0$ then $H_{t}^{\mathrm{red},j}$ has two critical points: one which is elliptic (local maximum) and one which is hyperbolic;
        \item if $j > 0$, then either
        \begin{itemize}
            \item $H_{t}^{\mathrm{red},j}$ has three critical points: one which is elliptic (local maximum), one which is elliptic (local minimum), and one which is hyperbolic;
            \item  $H_{t}^{\mathrm{red},j}$ has two critical points: one which is elliptic (local maximum) and one which is degenerate (parabolic); or
            \item  $H_{t}^{\mathrm{red},j}$ has one critical point, which is elliptic (local maximum).
        \end{itemize}
    \end{itemize}
\end{itemize} 
\end{prop}

The proof of this proposition is obtained by combining Lemmas \ref{lm:count_roots}, \ref{lm:sols_eps1}, \ref{lm:eps1_elliptic}, \ref{lm:elliptic_hyperbolic_points}, \ref{lm:fplusg} and \ref{lm:CP2-only-parabolic} together with Lemmas \ref{lm:crit_Hred_polar} and \ref{lm:type_rankone} from the previous section, and will be given at the end of this section.

\begin{rmk}
 The proposition above discusses the singularities of $H_{t}^{\text{red},j}$, but only in the region where $z_1$ and $z_3$ are nonzero, so some singularities on the reduced space are not included. 
 For instance, $z_3=0$ is not included in our coordinates and this always corresponds to an elliptic (local minimum) critical point of $H_{t}^{\text{red},j}$. If $j=0$, there is also the point $B$ (for which $z_1 = 0$) which corresponds to a non-smooth point in the reduced space.
\end{rmk}

Since we already saw that there is no rank one point with $z_1 = 0$, we may use cylindrical coordinates on $M_j^{\text{red}} \setminus \left( \{z_1 = 0 \} \cup \{ z_3 = 0\} \right)$ as in Section \ref{subsect:rankone_reduced}. Concretely, by using the actions of $J$ and $N$, we find representatives $x_2, x_3 \in \R^+$ of $z_2$ and $z_3$, and write $z_1 = \rho \exp(i \theta)$ with $\sqrt{2j} < \rho < \sqrt{\alpha+j}$. In these coordinates, 
\begin{equation} H_t^{\text{red},j} = 2\gamma \delta t \alpha^2 + (1-2t)(\alpha+j-\rho^2) - 2 \gamma \delta t (\rho^2 - j)^2 + 4 \gamma t \rho (\alpha + j - \rho^2) \cos \theta \sqrt{\rho^2 - 2j}. \label{eq:Htred_CP2}  \end{equation}
Note that $H_t^{\text{red},j}$ is as in Equation \eqref{eqn:Hred} with $a = 2\gamma \delta t \alpha^2 + (1-2t)(\alpha+j-X) - 2 \gamma \delta t (X - j)^2 $, $b = 4 \gamma t (X - \alpha - j)$ and $c = X(X - 2j)$. Using again the notation from Section~\ref{subsect:rankone_reduced}, this leads to 
\begin{equation}\label{eq:fg_CP2} f(X) = 2 (2t-1 - 4 \gamma \delta t (X-j)) \sqrt{X(X-2j)}, \quad g(X) = 8 \gamma t \left( 2 X^2 - (\alpha+4j) X + j(\alpha+j) \right).  \end{equation}
Recall from Lemma \ref{lm:crit_Hred_polar} that the critical points $H_t^{\text{red},j}$ are of the form $(\rho_{\varepsilon},\arccos \varepsilon)$ where $\varepsilon \in \{ -1, 1 \}$ and  $X_{\varepsilon} = \rho_{\varepsilon}^2$ is a solution of the equation $f = \varepsilon g$.

We treat the cases $j=0$ and $0 < j < \alpha$ separately; not only is it necessary because we will use the fact that $j \neq 0$ in the latter case, but also the $j=0$ case is both easier and enlightening. 

\paragraph{The case $j=0$.} In this case 
\[ f(X) = 2 X (2t-1 - 4 \gamma \delta t X) , \quad g(X) = 8 \gamma t X \left( 2 X - \alpha   \right).  \]
Hence the solutions of $f(X) = \varepsilon g(X)$ are $X = 0$ and  
\begin{equation} X = X_{\varepsilon} := \frac{2t - 1 + 4 \gamma t \alpha \varepsilon}{4\gamma t (\delta + 2 \varepsilon)} \label{eq:crit_j0} \end{equation}
(note that the denominator is nonzero since $\delta > 2$). Since we are looking for the solutions in $(0,\alpha)$, we want 
\[ 0 < 2t - 1 + 4 \gamma t \alpha \varepsilon < 4\gamma t \alpha (\delta + 2 \varepsilon). \]
The inequality on the right-hand side is automatically satisfied since the assumption $2 \gamma \delta \alpha > 1$ implies
\[ 4\gamma t \alpha (\delta + 2 \varepsilon) > 2t + 8 \gamma t \alpha \varepsilon > 2t - 1 + 4 \gamma t \alpha \varepsilon. \]
Moreover, $2t - 1 + 4 \gamma t \alpha \varepsilon > 0$ if and only if 
\[ t > \frac{1}{2(1 + 2 \gamma \alpha \varepsilon)} = \begin{cases} t^- \text{ if } \varepsilon = 1, \\ t^+ \text{ if } \varepsilon = -1. \end{cases} \]
Consequently:
\begin{itemize}
    \item when $t \leq t^-$, both equations $f = g$ and $f = -g$ have no solutions in $(0,\alpha)$;
    \item when $t^- < t \leq t^+$, the equation $f = g$ has a unique solution in $(0,\alpha)$, and the equation $f = -g$ has no solution in $(0,\alpha)$;
    \item when $t > t^+$, both equations $f = g$ and $f = -g$ have a unique solution in $(0,\alpha)$.
\end{itemize}
Note that $b(X_{\varepsilon}) < 0$ since $X_{\varepsilon} < \alpha$. Moreover one readily computes
\[ f'(X_{\varepsilon}) - \varepsilon g'(X_{\varepsilon}) = -2 \left( 2t - 1 + 4 \gamma t \alpha \varepsilon \right)  \]
and this quantity is negative when $t > t^-$, so that
\[ \varepsilon b(X_{\varepsilon}) \left( f'(X_{\varepsilon}) - \varepsilon g'(X_{\varepsilon}) \right) \begin{cases} < 0 \text{ if } \varepsilon = -1, \\ > 0 \text{ if }  \varepsilon = 1. \end{cases} \]
By using Lemma \ref{lm:type_rankone}, we deduce that $X_1$ corresponds to elliptic-regular points and $X_{-1}$ corresponds to hyperbolic-regular points. 
This settles  the $j=0$ case of Proposition~\ref{prop:summary_CP2}.

\begin{rmk}
With the statement of Proposition \ref{prop:summary_CP2} in mind, we now know what behavior we should expect to see: $\varepsilon=1$ should correspond to the upper boundary of the momentum map image (which should always have one solution which is elliptic) and $\varepsilon=-1$ should correspond to the flap, i.e.~the hyperbolic and the elliptic points on the flap should both be solutions for $\varepsilon=-1$. We will see that we always get one solution for $\varepsilon=+1$ and either zero or two solutions for $\varepsilon=-1$ depending on the values of $t$ and $j$, with one choice of $j$ yielding one solution (which has multiplicity two) in between.
\end{rmk}

\paragraph{The case $j > 0$.} Now we assume that $0 < j < \alpha$.

\begin{lm}
\label{lm:count_roots}
Counting with multiplicity, the equation $f = \varepsilon g$ has 
\begin{itemize}
\item one or three solutions in $(2j,\alpha+j)$ if $\varepsilon = +1$;
\item zero or two solutions in $(2j,\alpha+j)$ if $\varepsilon = -1$. 
\end{itemize} 
\end{lm}

\begin{proof}
On the one hand, $(f-g)(2j) = 8 \gamma t j (\alpha-j) > 0$ and $(f+g)(2j) = -8 \gamma t j (\alpha-j) < 0$. On the other hand,
\[ \begin{split} (f+g)(\alpha+j) & = 2 \sqrt{\alpha^2 - j^2} \left( 2t \left( 1 + 2 \gamma \sqrt{\alpha^2 - j^2} - 2 \gamma \delta \alpha \right) - 1 \right) \\
& \leq 2 \sqrt{\alpha^2 - j^2} \left( 2t \left( 1 + 2 \gamma \alpha - 2 \gamma \delta \alpha \right) - 1 \right) \\
& = 2 \sqrt{\alpha^2 - j^2} \left( 2t \left( 1 + 2 \gamma \alpha (1 - \delta) \right) - 1 \right) \\
& < 2 \sqrt{\alpha^2 - j^2} \left( 2t \left( 1 + \frac{1}{2} (1 - \delta) \right) - 1 \right) \quad \text{ since } \gamma < \frac{1}{4\alpha} \\
& < 2 \sqrt{\alpha^2 - j^2} \left( t - 1 \right) \quad \text{ since } \delta > 2 \\
& < 0 \quad \text{ since } t \leq 1. \end{split} \]
Since $g(\alpha+j) = \alpha^2 - j^2 > 0$, we also have $(f-g)(\alpha+j) < 0$. Hence, $f-g$ (respectively $f+g$) can only have an odd (respectively even) number of roots (counted with multiplicity) in $(2j,\alpha+j)$. 

Moreover, $(f-g)(f+g)$ is a degree $4$ polynomial, so it can have at most four real roots. Thus, the sum of the number of roots of $f - g$ and $f + g$ cannot exceed four. Therefore, $f-g$ can only have one or three roots in $(2j,\alpha+j)$, and $f+g$ can only have zero or two roots in $(2j,\alpha+j)$.
\end{proof}

\begin{lm}
\label{lm:sols_eps1}
The equation $f = g$ has exactly one solution.
\end{lm}

\begin{proof}
At $t=0$, one readily checks that $f^2-g^2 = 4X(X-2j)$ has exactly one zero at $X = 2j$. Since this zero has multiplicity one, we can apply the implicit function theorem in the interval $(2j-\varepsilon, \alpha+j+\varepsilon)$, for some sufficiently small $\varepsilon>0$, to conclude that $f^2-g^2$ still has exactly one zero in $(2j-\varepsilon, \alpha+j+\varepsilon)$ for sufficiently small $t>0$.
Thus, we conclude that $h = f-g$ has at most one zero in $(2j,\alpha+j)$ for sufficiently small $t$, but by Lemma \ref{lm:count_roots} we know that it has at least one zero.
Thus, $f=g$ has exactly one solution for sufficiently small $t$, and we will now show that the number of solutions of $f=g$ does not depend on $t$, which will prove the statement. 

Recall from the proof of the previous lemma that $h(2j)$ and $h(\alpha+j)$ are nonzero for all values of $t$. Moreover, below we show that for all values of $t$ the system $h = 0 = h'$ has no solution. Thus, by the implicit function theorem, $h$ will always have exactly one zero in $(2j,\alpha+j)$.

The system $h = 0 = h'$ amounts to
\[ 
\begin{cases}
(2t-1-4 \gamma \delta t (X-j))\sqrt{X(X-2j)} - 4 \gamma t (2X^2 - (\alpha+4j)X + j(\alpha+j)) = 0,\\[1em]  4 \gamma t (4X - \alpha - 4j) \sqrt{X(X-2j)} 
+ \left( 8 \gamma \delta t X^2 - (2t - 1 + 16 \gamma \delta t j) X + j (2t - 1 + 4\gamma \delta t j)\right) = 0.
\end{cases}
\]
Both equations are affine in $t$, and two affine functions $a t + b$, $c t + d$ vanish for the same $t$ if and only if $a d - b c = 0$. From identifying $a, b, c, d$ in the above equations and computing this quantity we find that this condition amounts to $k = 0$ where
\[ k: X \mapsto - \delta (X (X-2j))^{\frac{3}{2}}   - j^{3} - 3 \, j^{2} X + 6 \, j X^{2} - 2 \, X^{3} - j^{2} \alpha.  \]
But 
\[ k'(X) = -3 \left( \delta (X-j) \sqrt{X(X-2j)} + 2 X^2 - 4 j X + j^2 \right) \]
and we use the inequality $2 ab \geq - (a^2 + b^2)$ with $a = (X-j)$ and $b = \sqrt{X(X-2j)}$ to get, since $\delta > 2$,
 \[ \delta (X-j) \sqrt{X(X-2j)} > 2 (X-j) \sqrt{X(X-2j)} \geq -(2 X^2 - 4 j X + j^2) \]
 which means that $k'<0$, so $k$ is strictly decreasing. Since $k(2j) = j^2(j-\alpha) < 0$, this implies that $k$ never vanishes.
\end{proof}

Recall that each solution of $f=g$ and $f=-g$ corresponds to a critical point of $H_t^{\text{red}, j}$, which in turn corresponds to a singular fiber of the integrable system. 
In the next two lemmas we determine the types of these singular points.

\begin{lm}\label{lm:eps1_elliptic}
The solution to $f=g$ in $(2j,\alpha+j)$ corresponds to elliptic-regular points of $F_t$.
\end{lm}

\begin{proof}
Let $X_0$ be this solution. Recall from Lemma \ref{lm:type_rankone} that we need to compute the sign of $b(X_0) (f'(X_0) - g'(X_0))$. In view of Lemma \ref{lm:count_roots} and its proof, necessarily $f'(X_0) - g'(X_0) < 0$, since $(f-g)(2j) > 0$ and $(f-g)(\alpha+j) < 0$. Moreover 
\[ b(X_0) = 4 \gamma t (X_0 - \alpha - j) < 0. \]
Hence $b(X_0) (f'(X_0) - g'(X_0)) > 0$, which concludes the proof.
\end{proof}

\begin{lm}
\label{lm:elliptic_hyperbolic_points}
Whenever $f=-g$ has two distinct solutions $X_1 < X_2$ in $(2j,\alpha+j)$, then $X_1$ corresponds to elliptic-regular points of $F_t$ and $X_2$ corresponds to hyperbolic-regular points of $F_t$. Moreover, $H_t^{\mathrm{red},j}(X_2,\pi) > H_t^{\mathrm{red},j}(X_1,\pi)$.
\end{lm}

\begin{proof}
Let $X_1 < X_2$ be these solutions. In view of Lemma \ref{lm:count_roots} and its proof, necessarily $f'(X_1) + g'(X_1) > 0$ and $f'(X_2) + g'(X_2) < 0$, since $(f+g)(2j) < 0$ and $(f+g)(\alpha+j) < 0$. Since 
\[ b(X_1) = 4 \gamma t (X_1 - \alpha - j) < 0, \quad b(X_2) = 4 \gamma t (X_2 - \alpha - j) < 0, \]
we obtain that $- b(X_1) (f'(X_1) + g'(X_1)) > 0$, which means (again by Lemma \ref{lm:type_rankone}) that $X_1$ corresponds to elliptic-regular points, and that $- b(X_2) (f'(X_2) + g'(X_2)) < 0$, which means that $X_2$ corresponds to hyperbolic-regular points. Furthermore, again in view of the proof of Lemma \ref{lm:count_roots}, we know that $H_t^{\mathrm{red},j}(\cdot,\pi)$ is strictly increasing on $(X_1,X_2)$, since $f+g$ is positive in this interval.
\end{proof}

In view of the two previous lemmas, the system can only be semitoric if $f+g$ does not vanish. Recall that we are still in the case $0 < j < \alpha$ and that when $t = t^+$ and $j=0$, $f + g$ vanishes. In the next lemma we will prove that $f + g$ cannot vanish whenever $j > 0$ and $t \leq t^+$.

\begin{lm}
\label{lm:fplusg}
If $t \leq t^+$, the function $f+g$ does not vanish on $(2j,\alpha+j)$.
\end{lm}

\begin{proof}
For $X \in (2j,\alpha+j)$, write
\[ \begin{split} (f+g)(X) & = 2 \left( (2t-1-4 \gamma \delta t (X-j))\sqrt{X(X-2j)} + 4 \gamma t (2X^2 - (\alpha+4j)X + j(\alpha+j)) \right) \\
& = 2 \sqrt{X(X-2j)} \left( 2 \left( 1 - 2 \gamma \left( \delta (X-j) - \frac{2X^2 - (\alpha+4j)X + j(\alpha+j)}{\sqrt{X(X-2j)}}\right)  \right) t - 1 \right) \\
& < 2 \sqrt{X(X-2j)} \biggl( 2 \biggl( 1 - 2 \gamma \biggl( \underbrace{2 (X-j) - \frac{2X^2 - (\alpha+4j)X + j(\alpha+j)}{\sqrt{X(X-2j)}}}_{:=h(X)} \biggr)  \biggr) t - 1 \biggr)  \end{split} \]
because $\delta > 2$. We claim that for every $X$ in $(2j,\alpha+j)$, $h(X) > \alpha$; this implies that 
\[ (f+g)(X) < 2 \sqrt{X(X-2j)} \left( 2 \left( 1 - 2 \gamma \alpha  \right) t - 1 \right) = 2 \sqrt{X(X-2j)} \left( \frac{t}{t^+} - 1 \right) \]
which means that $(f+g)(X) < 0$ if $t \leq t^+$.\\

To complete the proof, we now prove that $h(X)>\alpha$. Note that for every $X \in (2j,\alpha+j)$,
\[ h'(X) = 2 - \frac{2X^3 - 6j X^2 + 3 j^2 X + j^2(\alpha+j)}{(X(X-2j))^{\frac{3}{2}}} \] 
and
\[ h''(X) = \frac{3j^2(\alpha X - j(\alpha+j))}{(X(X-2j))^{\frac{5}{2}}} \geq \frac{3j^2(j(\alpha-j))}{(X(X-2j))^{\frac{5}{2}}} > 0 \]
(here we have used $X > 2j$). Hence $h'$ is strictly increasing on $(2j,\alpha+j)$ so
\[ \forall X \in (2j,\alpha+j) \qquad h'(X) < h'(\alpha+j) = 2 \left( 1 - \frac{\alpha}{\sqrt{\alpha^2 - j^2}}  \right) < 0. \]
So $h$ is strictly decreasing on this interval. Hence
\[ \forall X \in (2j,\alpha+j) \qquad h(X) > h(\alpha+j) = 2\alpha - \sqrt{\alpha^2 - j^2} > \alpha. \qedhere \]
\end{proof}

When $t > t^+$, the function $f+g$ may vanish, in which case the system will not be semitoric. In order to prove that it is hypersemitoric, it only remains to exclude the presence of degenerate points that are not parabolic.

\begin{lm}
\label{lm:CP2-only-parabolic}
The only possible degenerate rank one points of $F_t$ are parabolic points.
\end{lm}

\begin{proof}
It suffices to show that the conditions in Lemma \ref{lm:para_polar} are always satisfied (note that the $S^1$-action generated by $J$ is free at every point in the set $M_j^{\text{red}} \setminus (\{z_1=0\} \cup \{z_3=0\})$, since the non-free points all lie on $\{z_3=0\}$). The context here is slightly different from the one of this lemma, but it is easy to reconcile both settings by performing the change of variable $Y = \frac{X-2j}{\alpha-j}$ in order to write $H_t$ as a smooth function of $(j,Y) \in (0,\alpha) \times (0,1)$. Here we keep working with the variables $(j,X)$ for the sake of simplicity.

So we will prove that the system 
\begin{equation} \label{eqn:CP2-parabolic}\begin{cases} f = \varepsilon g, \\ 
f' = \varepsilon g', \\ 
f'' = \varepsilon g'' \, \text{ or }\, 
\frac{\partial f}{\partial j} = \varepsilon \frac{\partial g}{\partial j} \end{cases}
\end{equation}
has no solution. Here for the sake of simplicity we still write $u'$ for the partial derivative of a function $u$ with respect to $X$.

We know from Lemma \ref{lm:sols_eps1} that $f-g$ has a unique zero in $(2j,\alpha+j)$, so the system 
\[ \begin{cases} f = g, \\ f' = g' \end{cases} \]
never has any solution, which means, in view of Lemma \ref{lm:type_rankone}, that rank one points corresponding to $\varepsilon = 1$ are never degenerate. Therefore if the system \eqref{eqn:CP2-parabolic} has a solution, it must be for $\varepsilon = -1$. 

Moreover, we know from Lemma \ref{lm:count_roots} that $f + g$ has at most two zeros in $(2j,\alpha+j)$ (counted with multiplicity), so the system 
\[ \begin{cases} f = - g, \\ f' = - g', \\ f'' = - g'', \end{cases} \]
has no solution in $(2j,\alpha+j)$. So it only remains to prove that the system
\[ \begin{cases} f = - g, \\ f' = - g', \\ \frac{\partial f}{\partial j} = - \frac{\partial g}{\partial j}, \end{cases} \]
has no solution in $(2j,\alpha+j)$. Since we know from Lemma \ref{lm:fplusg} that the subsystem formed by the first two lines has no solution in $(2j,\alpha+j)$ when $t \leq t^+$, we may assume in the rest of the proof that $t > t^+$.

We compute 
\[ \frac{\partial f}{\partial j} = \frac{2 X}{\sqrt{X(X-2j)}} \left( 1 - 2t - 12 \gamma \delta t j + 8 \gamma \delta t X \right), \quad \frac{\partial g}{\partial j} = 8\gamma t (\alpha + 2j - 4X),  \]
so the latter system amounts to
\begin{equation}
\label{eq:para_system}
\begin{cases}
(2t-1-4 \gamma \delta t (X-j))\sqrt{X(X-2j)} + 4 \gamma t (2X^2 - (\alpha+4j)X + j(\alpha+j)) = 0,\\[1em]  4 \gamma t (4X - \alpha - 4j) \sqrt{X(X-2j)} 
- \left( 8 \gamma \delta t X^2 - (2t - 1 + 16 \gamma \delta t j) X + j (2t - 1 + 4\gamma \delta t j)\right) = 0,\\[1em]
  X \left( 1 - 2t - 12 \gamma \delta t j + 8 \gamma \delta t X \right)  + 4\gamma t (\alpha + 2j - 4X) \sqrt{X(X-2j)} = 0.
\end{cases}
\end{equation}
Adding the last two lines leads to
\[ 1 - 2 t - 8\gamma t \sqrt{X(X-2j)} - 4\gamma\delta t j + 4 \gamma \delta t X = 0 \]
which simplifies to
\[ 2t-1-4 \gamma \delta t (X-j) = 8\gamma t \sqrt{X(X-2j)}. \]
Substituting in the first line, we get
\[ 8 \gamma t X (X-2j) + 4 \gamma t (2X^2 - (\alpha+4j)X + j(\alpha+j)) = 0 \]
i.e., after dividing by $4 \gamma t$,
\[ 4 X^2 - (\alpha+8j)X + j(\alpha+j) = 0. \]
This equation has two roots 
\[ X^{\pm} = j + \frac{\alpha}{8} \pm \frac{1}{8} \sqrt{\alpha^2 + 48j^2}.  \]
Clearly $X^- < 2j$ so the only possible solution of the system in $(2j,\alpha+j)$ would be $X^+$. Note that we always have $2 j < X^+ < \alpha + j$ (using that $j < \alpha$). The rest of the proof will be devoted to showing that $X^+$ is not a solution of the system \eqref{eq:para_system}.

Multiplying the first line of the system~\eqref{eq:para_system} by $-8\gamma\delta j t$, its second line by $\delta(1-2t+\gamma\delta \alpha t - 2\gamma\delta j t)$, and its third line by $\delta(1-2t+\gamma\delta \alpha t)$ and summing all these quantities, we obtain that any solution of the system satisfies $P(X) = 0$ with $P$ some degree two polynomial in $X$ that we do not write here for the sake of clarity. But one readily checks that 
\[ \begin{split} P(X^+) & = \gamma^2 t^2 \left( \alpha^2 + 4 j^2 + \alpha \sqrt{\alpha^2 + 48 j^2} \right) \delta^2 - \frac{\gamma t}{4} (2t-1) \left( 7\alpha + 3 \sqrt{\alpha^2 + 48 j^2}  \right) \delta \\
& + (2t-1)^2 + 2 \gamma^2 t^2 \left(\alpha^2 - 8 j^2 + \alpha \sqrt{n^2 + 48 j^2} \right). \end{split} \]
We now show that $P(X^+) > 0$, which will conclude the proof. In order to do so, note that $P(X^+)$ is a quadratic polynomial in $\delta$ with positive leading coefficient, with minimum attained at 
\[ \delta_0 = \frac{(2t-1) \left( 7\alpha + 3 \sqrt{\alpha^2 + 48 j^2}  \right)}{8\gamma t \left( \alpha^2 + 4 j^2 + \alpha \sqrt{\alpha^2 + 48 j^2} \right)}. \]
Now,
\[ \frac{1}{2\gamma \alpha} - \delta_0 = \frac{(10t-3) \alpha \sqrt{\alpha^2 + 48 j^2} + 16 j^2 t + 18 \alpha^2 t - 7 \alpha^2 }{8\gamma t \alpha \left( \alpha^2 + 4 j^2 + \alpha \sqrt{\alpha^2 + 48 j^2} \right)}. \]
Since $t > t^+ \geq \frac{1}{2}$, $10t-3 > 0$ and $18 \alpha^2 t - 7 \alpha^2 > 2 \alpha^2 > 0$, and we conclude that
\[ \frac{1}{2\gamma \alpha} > \delta_0, \]
which means that $P(X^+)$, seen as a function of $\delta$, is strictly increasing on $(\frac{1}{2\gamma \alpha},+\infty)$, which is the range of $\delta$. Furthermore, when $\delta = \frac{1}{2\gamma \alpha}$,
\begin{equation} P(X^+) = \frac{16 \gamma^2 \alpha^2 t^2 (  \alpha^2 - 8 j^2 + \alpha\sqrt{48 j^2 + \alpha^2} ) + k(j,\alpha)}{8 \alpha^2} \label{eq:P_Xplus} \end{equation}
with 
\[ k:  [0,+\infty)^2 \to \R, \quad (j,\alpha) \mapsto 8 j^2 t^2 + 20 \alpha^2 t^2 - 25 \alpha^2 t + 8 \alpha^2 + (3-4t) \alpha t \sqrt{\alpha^2 + 48 j^2}. \]
We will prove that the quantity in Equation \eqref{eq:P_Xplus} is positive, and since the fact that $j < \alpha$ implies that
\[\alpha^2 - 8 j^2 + \alpha\sqrt{48 j^2 + \alpha^2} > \alpha^2 - j^2 > 0,\]
it is sufficient to show that $k(j,\alpha) > 0$. If $t \leq \frac{3}{4}$, then 
\[ k(j,\alpha) \geq \alpha^2 (20 t^2 - 25 t + 8) > 0. \]
Assume now that $\frac{3}{4} < t \leq 1$. Then $3 - 4 t < 0$ and $\sqrt{\alpha^2 + 48 j^2} \leq \alpha + 4 j \sqrt{3}$ so
\[ k(j,\alpha)  \geq 8 j^2 t^2 + 20 \alpha^2 t^2 - 25 \alpha^2 t + 8 \alpha^2 + (3-4t) \alpha t (\alpha + 4 j \sqrt{3}) =: Q(\alpha) \]
where 
\[ Q(\alpha) =  \left( 16 t^2 - 22 t + 8 \right) \alpha^2 + 4  j t \sqrt{3} (3 - 4t) \alpha + 8 j^2 t^2.  \]
The discriminant of $Q$ is 
\begin{equation} \Delta = 16 j^2 t^2 \left( 16 t^2 - 28 t + 11 \right),  \label{eq:discrim_proof_para}\end{equation}
and one readily checks that $\Delta < 0$ when $\frac{3}{4} < t \leq 1$. Thus, $Q$ is of the sign of $16 t^2 - 22 t + 8$, which is always positive. Therefore, $k(j,\alpha) > 0$.

Hence when $\delta = \frac{1}{2\gamma \alpha}$, $P(X^+) >0$. The above considerations imply that $P(X^+) > 0$ for all $\delta > \frac{1}{2\gamma \alpha}$.
\end{proof}

We are now ready to prove Proposition~\ref{prop:summary_CP2}.

\begin{proof}[Proof of Proposition \ref{prop:summary_CP2}]
We already proved the assertions about the $j=0$ case at the beginning of this section. So now we focus on the case $j > 0$.

Having Lemmas \ref{lm:crit_Hred_polar} and \ref{lm:type_rankone} in mind and combining Lemmas \ref{lm:sols_eps1}, \ref{lm:eps1_elliptic} and \ref{lm:fplusg}, we see that for $t \leq t^+$, $H_t^{\text{red},j}$ has one elliptic critical point which is a local maximum. 

Combining Lemmas \ref{lm:count_roots}, \ref{lm:sols_eps1}, \ref{lm:eps1_elliptic} and \ref{lm:elliptic_hyperbolic_points} (and using Lemmas \ref{lm:crit_Hred_polar} and \ref{lm:type_rankone} again) we see that when $t > t^+$, $H_t^{\text{red},j}$ either has one elliptic critical point (a local maximum), or one hyperbolic and two elliptic critical points (a local minimum and a local maximum), or one elliptic (local maximum) and one degenerate critical point. In the latter case, this degenerate critical point is necessarily parabolic thanks to Lemma \ref{lm:CP2-only-parabolic}.
\end{proof}

\subsubsection{Parabolic points and flap}

In this section, we go beyond Proposition \ref{prop:summary_CP2} and study in detail the parabolic and hyperbolic-regular singularities of $F_t$ for $t \in (t^+,1]$, in order to prove the claims for such $t$ in Theorem \ref{thm:CP2_sys}. The most substantial part of this study consists in controlling the number of such parabolic singularities. We will prove in Lemma \ref{lm:at_least_one_para_CP2} below that this number is nonzero, in Lemma \ref{lm:number_para_CP2} that this number does not depend on $t \in (t^+,1]$, and in Proposition \ref{prop:only_one_para_CP2} that there is exactly one parabolic value in the region $J>0$. The latter is achieved by combining Corollary \ref{cor:one_para}, which shows that the $J$-value of any parabolic point must go to zero when $t$ goes to $t^+$ and follows from Lemma \ref{lm:jt_fplusg}, with Lemma \ref{lm:at_most_one_per_j}, which states that for a given $j$, there is at most one value of $t$ such that there exists a parabolic point with $J$-value $j$. Once we know that there is exactly one parabolic value with positive $J$, we prove that it is the endpoint of a triangular flap with elliptic corner $B$; this is the content of Proposition \ref{prop:flap_CP2}.

\begin{lm}
\label{lm:at_least_one_para_CP2}
Let $t \in (t^+,1]$. Then $F_t$ has at least one parabolic point $p$ with $J(p) > 0$. 
\end{lm}

\begin{proof}
We know from Proposition \ref{prop:summary_CP2} that $F_t$ has hyperbolic-regular critical points in $J^{-1}(0)$. Let $C \subset M$ be the connected component of the set of hyperbolic-regular points in $M$ containing these particular points. Since, by Proposition \ref{prop:summary_CP2}, $F_t$ has no hyperbolic-elliptic singularity, we know from \cite[Corollary 4.3]{HohPal21} that $F_t(C) \subset F_t(M)$ is homeomorphic to an open interval with distinct endpoints, these endpoints being parabolic points. By \cite[Lemma 4.2]{HohPal21}, these endpoints have different $J$-values, which proves the statement because of the symmetry \eqref{eq:sym_CP2}.
\end{proof}

Let $t \in (t^+,1]$, and note that there are finitely many parabolic values for $F_t$ (this is a consequence of the local normal form near a parabolic orbit, see \cite{KudMar,BolGugKud}). Let $N_t \geq 1$ be the number of parabolic values which have positive $J$-coordinates $x_1, \ldots, x_{N_t} > 0$.

\begin{lm}
\label{lm:number_para_CP2}
The number $N_t$ does not depend on $t \in (t^+,1]$, and $x_1, \ldots, x_{N_t}$ can be chosen to depend smoothly on $t$.
\end{lm}

\begin{proof}
Let $\eta > 0$ be such that $1+\eta < \frac{7+\sqrt{5}}{8}$. Recall from Lemma \ref{lm:type_rankone} that the degenerate (hence parabolic by Lemma \ref{lm:CP2-only-parabolic}) critical points of $H_t^{\text{red},j}$ are the solutions $X \in (2j,n+j)$ of the system $f(X) + g(X) = 0 = f'(X) + g'(X)$, where $f$ and $g$ are given in Equation \eqref{eq:fg_CP2}. Let $G=(f+g,f'+g')$ seen as a function of 
\[ (X,j,t) \in U = \{ (X,j,t) \in \R^3, \ 0 < j < \alpha, 2j < X < \alpha+j, t^+ < t < 1 + \eta \} \]
(for the sake of simplicity we still write $u'$ for the $X$ derivative of $u$). Let $(X_0,j_0,t_0)$ be such that $G(X_0,j_0,t_0) = 0$, and consider the partial Jacobian matrix
\[ A = \begin{pmatrix} \frac{\partial G_1}{\partial X}(X_0,j_0,t_0) & \frac{\partial G_1}{\partial j}(X_0,j_0,t_0) \\[2mm] \frac{\partial G_2}{\partial X}(X_0,j_0,t_0) & \frac{\partial G_2}{\partial j}(X_0,j_0,t_0) \end{pmatrix} = \begin{pmatrix} f'(X_0) + g'(X_0) & \frac{\partial f}{\partial j}(X_0) + \frac{\partial g}{\partial j}(X_0) \\[2mm] f''(X_0) + g''(X_0) & \frac{\partial f'}{\partial j}(X_0) + \frac{\partial g'}{\partial j}(X_0) \end{pmatrix}.  \]
Since $G(X_0,j_0,t_0) = 0$, we have that
\[ A = \begin{pmatrix} 0 & \frac{\partial f}{\partial j}(X_0) + \frac{\partial g}{\partial j}(X_0) \\[2mm] f''(X_0) + g''(X_0) & \frac{\partial f'}{\partial j}(X_0) + \frac{\partial g'}{\partial j}(X_0) \end{pmatrix} \]
and hence
\[ \det A = - \left( f''(X_0) + g''(X_0) \right) \left( \frac{\partial f}{\partial j}(X_0) + \frac{\partial g}{\partial j}(X_0) \right). \]
We can slightly adapt the proof of Lemma \ref{lm:CP2-only-parabolic} to show that none of these two factors vanishes for $(X_0,j_0,t_0) \in U$ with $G(X_0,j_0,t_0) = 0$; indeed, by our choice of $\eta$, the conclusion that the discriminant \eqref{eq:discrim_proof_para} is negative still holds when $1 \leq t < \eta$. Therefore $A$ is invertible so by the implicit function theorem, near $(X_0,j_0,t_0)$, the solutions of $G(X,j,t) = 0$ are of the form $(\phi(t),\psi(t),t)$ for some smooth $\phi,\psi$. Therefore $N_t$ does not depend on $t$.
\end{proof}

In view of the previous lemma, for $t \in (t^+,1]$, we let $N = N_t$ and consider $x_1, \ldots, x_N$ as functions of $t$.

\begin{lm}
\label{lm:at_most_one_per_j}
Let $j > 0$. There is at most one $t_j \in (t^+,1]$ such that there exists $\ell \in \{1, \ldots N\}$ with $x_{\ell}(t_j) = j$. 
\end{lm}

\begin{proof}
Recall that the parabolic points are obtained as the solutions of the system $f+g=0=f'+g'$. Recall that this system is given by the first two lines in Equation \eqref{eq:para_system}, namely: 
\[ 
\begin{cases}
(2t-1-4 \gamma \delta t (X-j))\sqrt{X(X-2j)}  + 4 \gamma t (2X^2 - (\alpha+4j)X + j(\alpha+j)) = 0\\[1em]
 4 \gamma t (4X - \alpha - 4j) \sqrt{X(X-2j)} 
- \left( 8 \gamma \delta t X^2 - (2t - 1 + 16 \gamma \delta t j) X + j (2t - 1 + 4\gamma \delta t j)\right) = 0.
\end{cases}
\]
This can be written as 
\begin{equation} \begin{cases} u(X) t + v(X) = 0, \\ w(X)  t + z(X) = 0 \end{cases} \label{eq:linear_sys_para}\end{equation}
with 
\[ u(X) = 2\left( 1 - 2 \gamma \delta (X-j) \right) \sqrt{X(X-2j)} + 4 \gamma \left( 2 X^2 - (\alpha + 4 j) X + j (\alpha + j) \right), \ v(X) = - \sqrt{X(X-2j)} \]
and
\[ w(X) = 4\gamma(4(X-j) - \alpha) \sqrt{X(X-2j)} - 8 \gamma \delta X^2 + 2 (1 + 8\gamma \delta j)X - 2(1 + 2\gamma \delta j)j, \ z(X) = j-X. \]
Note that $v$ and $w$ never vanish on $(2j,\alpha+j)$. Hence if $X \in (2j,\alpha + j)$ is a solution of \eqref{eq:linear_sys_para}, necessarily $u(X) \neq 0$ and $w(X) \neq 0$. Therefore, there exists a $t_j$ (not necessarily in $[0,1]$) such that \eqref{eq:linear_sys_para} has a solution $X_j$ if and only if $u(X_j) z(X_j) - v(X_j) w(X_j) = 0$, and in this case this $t_j$ is unique: $t_j = - \frac{v(X_j)}{u(X_j)} = - \frac{z(X_j)}{w(X_j)} $. 

 The condition $u(X_j) z(X_j) - v(X_j) w(X_j) = 0$ amounts to $X_j$ being a zero of the function
 \[ h: X \mapsto \delta (X (X - 2 j ))^{\frac{3}{2}} - \left( (X-j) (2 X^2 - 4jX - j^2) + j^{2} \alpha \right), \]
 so we look for such zeros in the interval $(2j,\alpha+j)$. We compute $h''$ and get 
 \[ h''(X) = 3\frac{\delta X (X - 2j) + \delta (X-j)^2 - 4 (X-j) \sqrt{X(X-2j)}}{\sqrt{X(X-2j)}} \]
 and we use the inequality $ab \leq \frac{1}{2}(a^2 + b^2)$ with $a = 2 (X-j)$ and $b = 2 \sqrt{X(X-2j)}$ to get 
 \[ 4(X-j)\sqrt{X(X-2j)} \leq 2 (X (X - 2j) +  (X-j)^2)< \delta ( X (X - 2j) +  (X-j)^2) \]
 where we used $\delta > 2$. Thus $h'' > 0$ so $h$ is strictly convex; since moreover $h(2j) = -j^2(\alpha-j) < 0$, we conclude that $h$ has at most one zero in $(2j,\alpha+j)$. By the above discussion, there is at most one $t_j$ such that \eqref{eq:linear_sys_para} has a solution.
\end{proof}

The next result is a refinement of Lemma \ref{lm:fplusg} that controls for which values of $j$ hyperbolic-regular and parabolic singularities of $F_t$ can or cannot appear in $J^{-1}(j)$.

\begin{lm}
\label{lm:jt_fplusg}
Let $t \in (t^+,1]$, and let 
\begin{equation} j_t = \frac{(\delta - 2)(2(1 - 4\gamma\alpha) t - 1) + \sqrt{16\gamma \alpha t \left( 2t-1 + \gamma \alpha t (\delta^2 - 4 \delta + 1) \right) - (2t-1)^2}}{4\gamma t (\delta^2 - 4\delta + 5)} \label{eq:jt}. \end{equation}
If $j \in (j_t,\alpha)$, the function $f+g$ does not vanish on $(2j,n+j)$.
\end{lm}

\begin{proof}
Let $h$ be as in the proof of Lemma \ref{lm:fplusg}. By slightly modifying the beginning of this proof, we obtain that for $X \in (2j,\alpha+j)$,
\[ (f+g)(X) < 2 \sqrt{X(X-2j)} \biggl( 2 \biggl( 1 - 2 \gamma \biggl( (\delta - 2) (X-j) + h(X) \biggr)  \biggr) t - 1 \biggr)  \]
with $h$ as in the proof of Lemma \ref{lm:fplusg}. Using that $h(X) > 2 \alpha - \sqrt{\alpha^2 - j^2}$, that $\delta - 2 > 0$ and that $X-j > j$, we deduce from the above inequality that
\[ (f+g)(X) < 2 \sqrt{X(X-2j)} \biggl( 2(1-2\gamma \alpha) t - 1 - 4\gamma t \biggl( (\delta - 2) j + \alpha - \sqrt{\alpha^2 - j^2} \biggr)  \biggr)  \]
i.e. $(f+g)(X) < 2 \sqrt{X(X-2j)} k(j)$ with
\begin{equation} k_t(j) = \frac{t}{t^+} - 1 - 4\gamma t \biggl( (\delta - 2) j + \alpha - \sqrt{\alpha^2 - j^2} \biggr). \label{eq:k_proof_jt}\end{equation}
Since we have that 
\[ \forall j \in (0,\alpha) \qquad k_t'(j) = - 4 \gamma t \left( \delta - 2 + \frac{j}{\sqrt{\alpha^2 - j^2}} \right) < 0 \]
and that
\[ k_t(0) = \frac{t}{t^+} - 1 > 0, \quad k_t(\alpha) = 2t(1 - 2 \gamma \alpha \delta) - 1 < -1 < 0  \]
(for $k_t(\alpha)$, we have used the explicit value of $t^+$ and the bound $\delta > \frac{1}{2\gamma \alpha}$), there exists a unique $j_t \in (0,\alpha)$ such that $k_t(j_t) = 0$, and $k_t$ is positive on $(0,j_t)$ and negative on $(j_t,\alpha)$. This implies that if $j > j_t$, $f+g < 0$ on $(2j,\alpha+j)$. So it only remains to compute $j_t$.

Write 
\[ k_t(j) = \frac{1}{t^+} \left( t - t^+ - 4\gamma t t^+ \biggl( (\delta - 2) j + \alpha  \biggr) + 4 \gamma t t^+ \sqrt{\alpha^2 - j^2} \right) \]
and define 
\[ \ell_t(j) = \frac{1}{t^+} \left( t - t^+ - 4\gamma t t^+ \biggl( (\delta - 2) j + \alpha  \biggr) - 4 \gamma t t^+ \sqrt{\alpha^2 - j^2} \right) = k_t(j) - 8 \gamma t \sqrt{\alpha^2 - j^2}. \]
Then $j_t$ is a root of the degree two polynomial 
\[ P(j) = {t^+}^2 k_t(j) \ell_t(j) = \left(t - t^+ - 4 \gamma t t^+ \left( (\delta-2)j + \alpha \right)\right)^2 - 16\gamma^2 t^2 {t^+}^2 (\alpha^2 - j^2). \]
Replacing $t^+$ by its value, we write $P$ as $P(j) = \frac{Q(j)}{4(1 - 2 \gamma \alpha)^2}$ where
\[ Q(j) = 16 \gamma^2 t^2 (\delta^2 - 4 \delta + 5) j^2 + 8 \gamma t (\delta-2)(1 - 2 (1-4\gamma\alpha)t)j + (1 - 2 (1-6\gamma\alpha)t)(1 - 2 (1-2\gamma\alpha)t). \]
Since $16 \gamma^2 t^2 (\delta^2 - 4 \delta + 5) > 0$ and 
\[ P'(j_t) = k_t(j_t) \ell_t'(j_t) + k_t'(j_t) \ell_t(j_t) = - 8 \gamma t k_t'(j_t) \sqrt{\alpha^2 - j_t^2}  > 0, \]
we conclude that $j_t$ is the biggest root of $P$, and the result follows from a straightforward computation.
\end{proof}

As a corollary of Lemma \ref{lm:jt_fplusg}, we can control the behaviour of the positive $J$-coordinates $x_{\ell}$ of the parabolic values when $t$ decreases to $t^+$. 

\begin{cor}
\label{cor:one_para}
For every $\ell \in \{1, \ldots, N\}$, $x_{\ell}(t) \underset{t \to t^+}{\longrightarrow} 0$. 
\end{cor}

\begin{proof}
Recall that parabolic points are determined by the system $f + g = 0 = f'+g'$. Hence Lemma \ref{lm:jt_fplusg} implies that for every $\ell \in \{1, \ldots, N\}$, $x_{\ell}(t) \leq j_t$. But one readily checks that when $t$ goes to $t^+$, $j_t$ goes to zero.
\end{proof}

\begin{prop}
\label{prop:only_one_para_CP2}
For any $t \in (t^+,1]$, there is exactly one parabolic value in the region $J > 0$ (i.e. $N = 1$).
\end{prop}

\begin{proof}
Assume that $N \geq 2$. Let $t_0 \in (t^+,1]$; without loss of generality, we may assume that $x_2(t_0) > x_1(t_0)$. By Corollary \ref{cor:one_para}, $x_2(t) \underset{t \to t^+}{\longrightarrow} 0$, so since $x_2$ depends continuously on $t$, there exists $t_1 \in (t^+,1]$, $t_1 \neq t_0$, such that $x_2(t_1) = x_1(t_0)$. This contradicts Lemma \ref{lm:at_most_one_per_j}.
\end{proof}

This allows us to prove that for $t \geq t^+$, $F_t$ has exactly one flap, and the size of this flap can be bounded using Lemma \ref{lm:jt_fplusg}.

\begin{prop}
\label{prop:flap_CP2}
For $t \in (t^+,1]$, $F_t$ has one triangular flap with two parabolic corners (with first coordinates $x_1(t) > 0$ and $-x_{1}(t)$) and one elliptic-elliptic corner $F_t(B) = (0,2\gamma\delta t \alpha^2 + (1-2t)\alpha)$. Moreover,
\begin{equation} x_1(t) \leq j_t \leq j_1 = \frac{(\delta - 2)(1 - 8\gamma\alpha) + \sqrt{16\gamma \alpha \left( 1 + \gamma \alpha (\delta^2 - 4 \delta + 1) \right) - 1}}{4\gamma (\delta^2 - 4\delta + 5)} \label{eq:bound_flap_CP2}\end{equation}
where $j_t$ is as in Equation \eqref{eq:jt}. In particular, when $\delta$ goes to $+\infty$, the size of the flap goes to zero.
\end{prop}

\begin{proof}
Let $t \in (t^+,1]$. Proposition \ref{prop:only_one_para_CP2}, together with the symmetry \eqref{eq:sym_CP2}, shows that $F_t$ has exactly two parabolic values with first coordinates $x_1(t)$ and $-x_1(t)$. The parabolic value with first coordinate $x_1(t)$ is the rightmost endpoint of two curves $\mathcal{H}$ and $\mathcal{E}$, respectively consisting of hyperbolic-regular and elliptic-regular values (see for instance the discussion before Definition 2.45 in \cite{HohPal21}). We also know by Proposition \ref{prop:summary_CP2} that there is exactly one hyperbolic-regular value at the level $j=0$; this hyperbolic-regular value also lies on a curve consisting of hyperbolic-regular values, which must end at a parabolic point (recall the proof of Lemma \ref{lm:at_least_one_para_CP2}), and therefore necessarily it must lie on $\mathcal{H}$.  

Similarly, by the normal form for elliptic-elliptic singularities from Theorem \ref{thm:eliasson}, the elliptic-elliptic value $F_t(B)$ which lies on $j=0$ is attached to two curves of elliptic-regular values, one of which, call it $\mathcal{E}'$, lies in the region $J  > 0$ because of the symmetry \eqref{eq:sym_CP2}. So for small $j > 0$, we are in the situation described in Proposition \ref{prop:summary_CP2} for which $\Hred$ has three critical points, and the critical point with value $\mathcal{E}' \cap \{ x = j \}$ is a local minimum of $\Hred$.  For $j < x_1(t)$ close to $x_1(t)$, we are also in the same situation, and the critical point with value $\mathcal{E} \cap \{ x = j \}$ is a local minimum of $\Hred$. So either $\mathcal{E}$ and $\mathcal{E}'$ coincide or these curves must end at some points $(x_{\mathcal{E}},y_{\mathcal{E}})$ and $(x_{\mathcal{E}'},y_{\mathcal{E}'})$ respectively, with $0 < x_{\mathcal{E}'} < x_{\mathcal{E}} < x_1(t)$; we claim that the latter cannot happen. Indeed, by the normal form for elliptic-regular points from Theorem \ref{thm:eliasson}, $(x_{\mathcal{E}'},y_{\mathcal{E}'})$ cannot be an elliptic-regular value (otherwise it would be attached to nearby regular values on both sides of $x_{\mathcal{E}'}$). So by Proposition \ref{prop:summary_CP2}, for $j = x_{\mathcal{E}'}$, $\Hred$ must have two critical points, and $(x_{\mathcal{E}'},y_{\mathcal{E}'})$ must be a parabolic value, which is impossible since $x_{\mathcal{E}'} \neq x_1(t)$. 

The inequality $x_1(t) \leq j_t$ directly follows from Lemma \ref{lm:jt_fplusg}. The second inequality comes from the fact that $j_t$ is increasing on $(t^+,1]$. Indeed, let $t,s \in (t^+,1]$ with $s > t$, and let $k_{t}, k_{s}$ be as in Equation \eqref{eq:k_proof_jt}. Recall that for $u \in (t^+,1]$, $j_u$ is the unique solution of $k_u(j_u) = 0$, and that $k_u$ is positive on $(0,j_u)$ and negative on $(j_u,\alpha)$. From $k_t(j_t) = 0$ we extract
\[ 4\gamma \biggl( (\delta - 2) j + \alpha - \sqrt{\alpha^2 - j^2} \biggr) = \frac{1}{t^+} - \frac{1}{t} \]
which yields
\[ k_s(j_t) = \frac{s}{t} - 1 > 0. \]
So necessarily, $j_s > j_t$.
\end{proof}

In Figures \ref{fig:bound_flap_t} and \ref{fig:bound_flap_delta}, we compare the bound on the size of the flap obtained in Proposition \ref{prop:flap_CP2} with the actual value of $x_1(t)$, computed by numerically solving $f+g = 0 = f'+g'$.

\begin{figure}
\begin{center}
\includegraphics[scale=0.5]{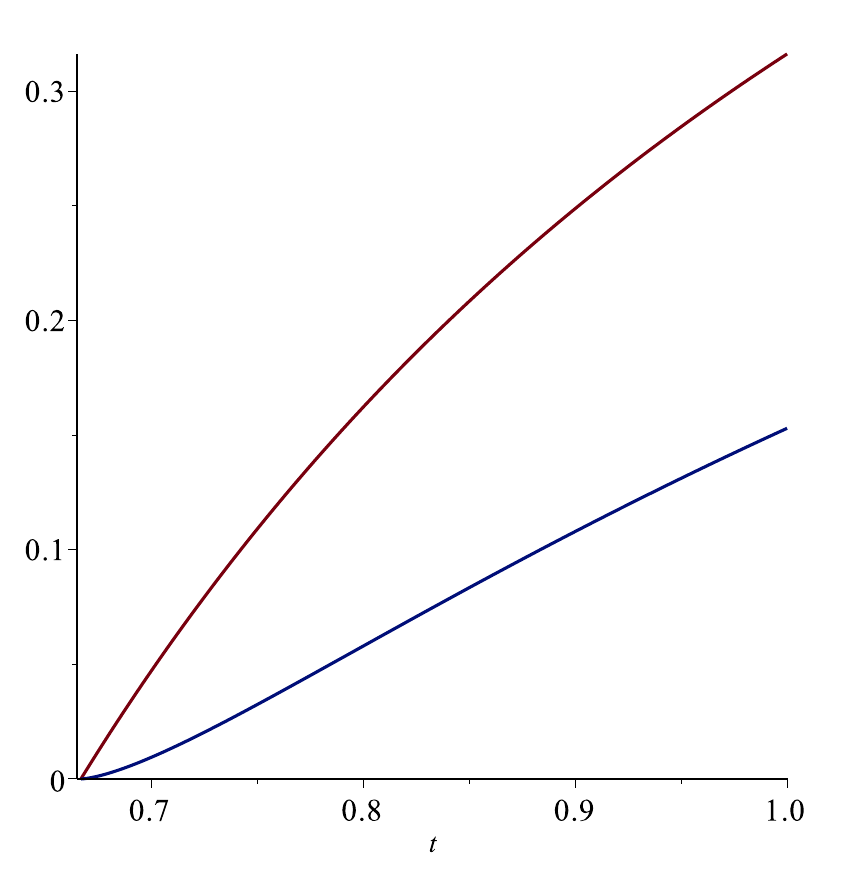}
\end{center}
\caption{Illustration of the inequality \eqref{eq:bound_flap_CP2}. We plot $x_1(t)$ (lower curve, in blue) and $j_t$ (upper curve, in red) as functions of $t$ for $\alpha = 1$, $\gamma = \frac{1}{8}$ and $\delta = 5$, in the interval $(t^+,1] = (\frac{2}{3},1]$.}
\label{fig:bound_flap_t}
\end{figure}

\begin{figure}
\begin{center}
\includegraphics[scale=0.5]{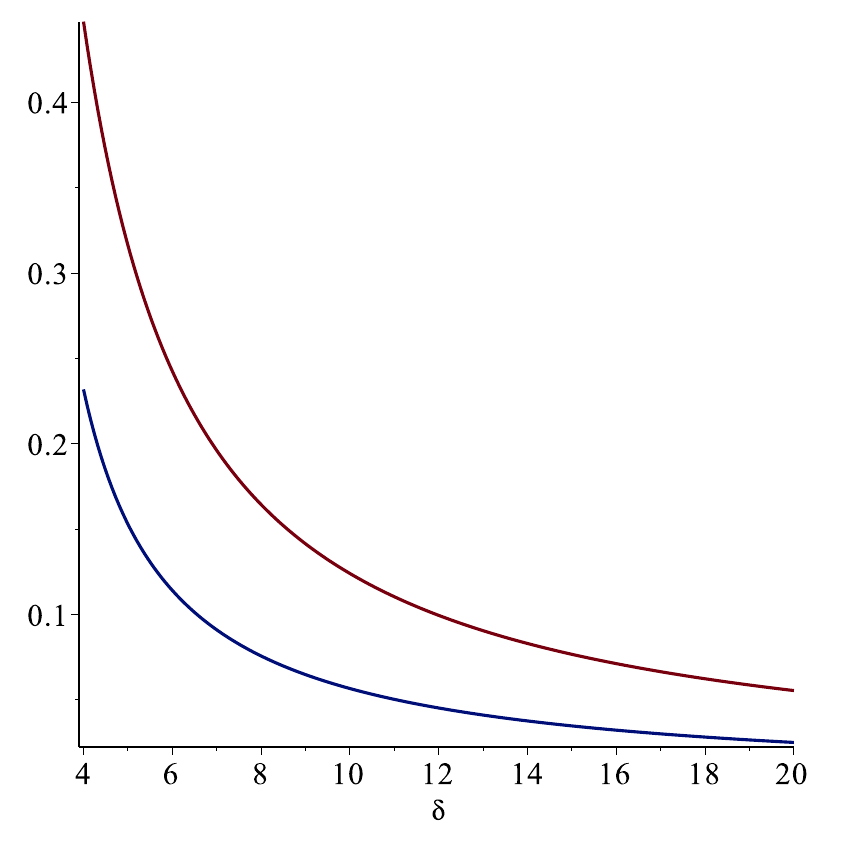}
\end{center}
\caption{Illustration of the inequality \eqref{eq:bound_flap_CP2}. We plot $x_1(1)$ (lower curve, in blue) and $j_1$ (upper curve, in red) as functions of $\delta$ for $\alpha = 1$ and $\gamma = \frac{1}{8}$, in the interval $(\frac{1}{2\gamma\alpha},20] = (4,20]$.}
\label{fig:bound_flap_delta}
\end{figure}

The intersection of the curve of elliptic-regular (respectively hyperbolic-regular) values of the flap with the right-half plane is the graph of a smooth function $j \in \mathcal{J}_t = (0,x_1(t)) \mapsto \xi^{(1)}_t(j)$ (respectively $j \in \mathcal{J}_t  \mapsto \xi^{(2)}_t(j)$); see \cite[Lemma 4.2]{HohPal21} for the hyperbolic-regular case, whose proof can easily be adapted to the elliptic-regular case. In order to get a better understanding of the flap, one can study these functions.

\begin{lm}
Let $t \in (t^+,1]$. The functions $\xi_t^{(\ell)}$, $\ell \in \{1,2\}$, are strictly increasing on $\mathcal{J}_t$.
\end{lm}

\begin{proof}
Let $\ell \in \{1,2\}$. Note that $\xi_t^{(\ell)}(j) = H_t^{\text{red},j}(\rho_{\ell}(j),\pi)$ where $\rho_{\ell}(j) = \sqrt{X_{\ell}(j)}$ and $X_1(j),X_2(j) \in (2j,\alpha+j)$ are the two distinct solutions of $f+g=0$ with $f,g$ as in Equation \eqref{eq:fg_CP2}, see Lemma \ref{lm:elliptic_hyperbolic_points}.

 Since $(\rho_{\ell}(j),\pi)$ is a critical point of $H_t^{\text{red},j}$,
\[ \forall j \in \mathcal{J}_t \qquad (\xi_t^{(\ell)})'(j) = \rho_{\ell}'(j) \frac{\partial H_t^{\text{red},j}}{\partial X}(\rho_{\ell}(j),\pi) + \frac{\partial H_t^{\text{red},j}}{\partial j}(\rho_{\ell}(j),\pi) = \frac{\partial H_t^{\text{red},j}}{\partial j}(\rho_{\ell}(j),\pi). \]
Recall from Equation \eqref{eq:Htred_CP2} that 
\[ H_t^{\text{red},j}(\rho,\theta) = 2\gamma \delta t \alpha^2 + (1-2t)(\alpha+j-\rho^2) - 2 \gamma \delta t (\rho^2 - j)^2 + 4 \gamma t (\alpha + j - \rho^2) \cos \theta \sqrt{\rho^2(\rho^2 - 2j)},\]
so that 
\[ \begin{split} \frac{\partial H_t^{\text{red},j}}{\partial j}(\rho,\theta) & = 1 - 2 t + 4 \gamma \delta t (\rho^2-j) + 4 \gamma t \cos \theta \sqrt{\rho^2(\rho^2-2j)} - \frac{ 4 \gamma t \rho^2 (\alpha+j - \rho^2) \cos \theta}{\sqrt{\rho^2(\rho^2-2j)}}  \\
& = \frac{ (1 - 2 t + 4 \gamma \delta t (\rho^2-j)) \sqrt{\rho^2(\rho^2-2j)} + 4 \gamma t \rho^2(\rho^2-2j) \cos \theta  - 4 \gamma t \rho^2 (\alpha+j - \rho^2) \cos \theta}{\sqrt{\rho^2(\rho^2-2j)}} \\
& = \frac{ (1 - 2 t + 4 \gamma \delta t (\rho^2-j)) \sqrt{\rho^2(\rho^2-2j)} + 4 \gamma t \rho^2 \cos \theta \left( 2\rho^2 - \alpha - 3j \right)}{\sqrt{\rho^2(\rho^2-2j)}}. \end{split}  \]
Using the fact that $(f+g)(X_{\ell}(j)) = 0$, we finally obtain that 
\[ (\xi_t^{(\ell)})'(j) = \frac{\partial H_t^{\text{red},j}}{\partial j}(\rho_{\ell}(j),\pi) = \frac{  4 \gamma t j \left( \alpha + j - X_{\ell}(j) \right)}{\sqrt{X_{\ell}(j)(X_{\ell}(j)-2j)}} > 0 \]
since $X_{\ell}(j) \in (2j,\alpha+j)$.
\end{proof}

More can be said about these curves. For instance, we already saw in Lemma \ref{lm:elliptic_hyperbolic_points} that $\xi_t^{(2)} > \xi_t^{(1)}$. One may try to prove that $\xi_t^{(1)}$ (respectively $\xi_t^{(2)}$) is a strictly concave (respectively convex) function of $j$, as can be guessed from Figures \ref{fig:image_CP2_alpha1_gamma1over8_delta5} and \ref{fig:image_CP2_alpha1_gamma1over5_delta3}, but the computations are very tedious and not necessary for the present paper.

\begin{rmk}
It would be interesting to investigate the behavior of parabolic points in a general one-parameter family of hypersemitoric systems. Since these systems are not the main topic of this paper, we will not try to do so here.    
\end{rmk}

\subsection{Height invariant}
\label{subsect:height_CP2}

In this section we compute the height invariant of the system in Theorem \ref{thm:CP2_sys} when it is semitoric with one focus-focus singularity (i.e. when $t^- < t <t^+$), and prove that by varying the different parameters, we can obtain any possible value for this invariant. 

In order to do so, we use the polar coordinates $(\rho,\theta)$ introduced at the beginning of Section \ref{subsect:other_rankone_CP2}. Then the reduced symplectic form is $\rho \ \dd\rho \wedge \dd\theta = \frac{1}{2} \dd X \wedge \dd\theta$, and the symplectic volume of $J^{-1}(0)$ is $\pi \alpha$. Recall that 
\[ H_t^{\text{red},0} = 2\gamma \delta t \alpha^2 + (1-2t)(\alpha-X) - 2 \gamma \delta t X^2 + 4 \gamma t X (\alpha - X) \cos \theta,  \]
and that $H_t(B) = 2\gamma \delta t \alpha^2 + (1-2t)\alpha$. Let $t \in (t^-,t^+)$, and recall that the height invariant of $(M,\omega,F_t)$ is the symplectic volume of $S = \{  H_t^{\text{red},0} < H_t(B) \}$ divided by $2\pi$. 

A point $(X,\theta) \in (0,\alpha) \times [0,2\pi)$ belongs to $S$ if and only if 
\[ \cos \theta < \frac{1-2t + 2\gamma\delta t X}{4\gamma t (\alpha-X)} =: G(X). \]
One readily checks that $G(X)$ belongs to $[-1,1]$ if and only if
\[ \frac{\frac{t}{t^+} - 1}{2\gamma t (\delta - 2)} \leq X \leq  \frac{\frac{t}{t^-} - 1}{2\gamma t (\delta + 2)} \]
(recall that $\delta > 2$). Since $t < t^+$, the left-hand side is negative so this condition amounts to
\[ X \leq \frac{\frac{t}{t^-} - 1}{2\gamma t (\delta + 2)} =: X^+. \]
Moreover $G(X) > 1$ for $X \in (X^+,\alpha)$. Hence
\[ S = \left\{ (X,\theta), \ 0 < X < X^+, \ \arccos(G(X)) < \theta < 2\pi - \arccos(G(X)) \right\}  \sqcup \left( (X^+,\alpha) \times [0,2\pi) \right) \]
so we obtain that
\[ 2\pi h = \pi(\alpha-X^+) + \pi X^+ - \int_0^{X^+} \arccos(G(X)) \ \dd X = \pi \alpha - \int_0^{X^+} \arccos(G(X)) \ \dd X.  \]
When $t \to t^-$, $X^+ \to 0$ so $h \to \frac{\alpha}{2}$ which is the maximal possible value for the height invariant; this was of course expected because when $t < t^-$ the value $F_t(B)$ lies on the top boundary of the image $F_t(M)$.

To get the minimal height invariant, we first let $t$ go to $t^+$ and check that in this case,
\[ h \underset{t \to t^+}{\longrightarrow} h^-(\delta) = \frac{\alpha}{2} \sqrt{\frac{\delta - 2}{\delta + 2}}.  \]
Since this height invariant is continuous with respect to $t$, by letting $t$ vary in $(t^-,t^+)$ we obtain every possible value in $(h^-(\alpha,\delta),\frac{\alpha}{2})$ for the height invariant. Moreover, the function $h^-$ is strictly increasing on $(\frac{1}{2\gamma\alpha},+\infty)$ and goes to $\frac{\alpha}{2} \sqrt{\frac{1 - 4 \gamma \alpha}{1 + 4\gamma\alpha}}$ when $\delta \to \frac{1}{2\gamma\alpha}$, which in turn is strictly decreasing with respect to $\gamma \in (0,\frac{1}{4\alpha})$ and goes to zero when $\gamma \to \frac{1}{4\alpha}$. Hence by also varying the parameters $\gamma$ and $\delta$ we obtain every possible value in $(0,\frac{\alpha}{2})$ for the height invariant.

\subsection{Proof of Theorem \ref{thm:CP2_sys}}
\label{sec:CP2-proof}

Using the above ingredients, we can now prove the main theorem of this section, Theorem \ref{thm:CP2_sys}.
 
The contents of Sections \ref{subsec:rank_zero_CP2} and \ref{subsec:rankone_Z2_CP2} and of Proposition \ref{prop:summary_CP2} (together with Lemma \ref{lm:CP2_zero} for $t=0$) show that $(M,\omega,F_t)$ is of toric type when $0 \leq t < t^-$, and semitoric with one focus-focus singularity (the point $B$) when $t^- < t < t^+$. When this is the case, we showed in Section \ref{subsect:height_CP2} above that by an appropriate choice of $\gamma$, $\delta$ and $t$, we can produce every possible value of the height invariant.

The aforementioned results also prove that this system is hypersemitoric (recall Definition \ref{dfn:hypersemi}) when $t^+ < t \leq 1$. Proposition \ref{prop:flap_CP2} proves that when this is the case, the system has exactly one triangular flap with two parabolic corners and one elliptic-elliptic corner at $F_t(B)$.

\section{Explicit semitoric systems of type (3)}
\label{sec:type_3}

In this section, we investigate the semitoric polygons of type (3a), (3b) and (3c), see Definition \ref{def:min-poly-123}. For each of these types, we provide completely explicit systems in the case that the underlying $S^1$-action has a $\Z_k$-sphere with $k \geq 2$, thus obtaining all possible unmarked semitoric polygons of type (3) (except for $n = 2$ which was already obtained in the literature, see the discussion before Theorem \ref{thm:intro_list_explicit}). In fact, with these explicit systems we obtain a substantial subset of the set of marked semitoric polygons of type (3).

These explicit systems, obtained in Theorems \ref{thm:system_3c}, \ref{thm:system_3c_general}, \ref{thm:system_3a}, \ref{thm:system_3a_general} and \ref{thm:system_3b_general}, could in fact all be implemented in a single family if we changed $\alpha$ to $-\alpha$ in the definition of type (3c) and allowed $\alpha$ to vanish, but we prefer to treat them separately for two reasons. First, we prefer to think of the scaling parameters $\alpha, \beta$ as positive quantities. Second, these three types are very different in nature, since the focus-focus value lies below the image of the $\Z_k$-sphere in type (3c) and below an elliptic-elliptic value in type (3b), while for type (3a) there is no such restriction.

In fact, this makes the construction of type (3c) systems inherently difficult, since this implies that they cannot be part of semitoric transition families, see Corollary \ref{cor:type_(3bc)_notfam}. Hence we apply the strategy described in Section \ref{sec:strategy} and construct them as part of half-semitoric transition families. The systems that we obtain for types (3a) and (3b) are then constructed in a similar way, and this is why we begin this section by the case of type (3c). Note that in principle, it should be possible to obtain an explicit semitoric transition family containing a system of type (3a) with any given marked semitoric polygon, but this would require a different construction and we leave this for future work.

Even though the semitoric polygons of type (3) with $n=3$ (hence with a $\Z_2$-sphere) are not minimal, it is already quite interesting to obtain explicit systems with these polygons, since the above discussion is still relevant. Moreover, when $n=3$ the formulas defining our explicit systems are less intricate and the proofs are easier to follow than in the $n=4$ case. Furthermore, in the $n=3$ case we are able to compute the height invariant of the systems that we construct, whereas for $n \geq 4$, these computations are too involved (except for type (3b)). This is why in Sections \ref{subsec:3c} and \ref{subsec:3a}, we chose to present separately the cases $n=3$ (Theorems \ref{thm:system_3c} and \ref{thm:system_3a}) and $n \geq 4$ (Theorems \ref{thm:system_3c_general} and \ref{thm:system_3a_general}).

\subsection{Explicit systems of type (3c)}
\label{subsec:3c}

\subsubsection{An explicit system with a $\Z_2$-sphere}
\label{subsubsec:3c_Z2}

We consider the polygon of type (3c) in Figure \ref{fig:min_type3c} with $n = 3$. To visualize the $\Z_2$-sphere of the underlying $S^1$-action, it is easier to consider the polygon obtained from the one in Figure \ref{fig:min_type3c} by changing the cut direction. This yields the polygon in Figure \ref{fig:poly_3c_down}; the $\Z_2$-sphere corresponds to the top right edge of this polygon.

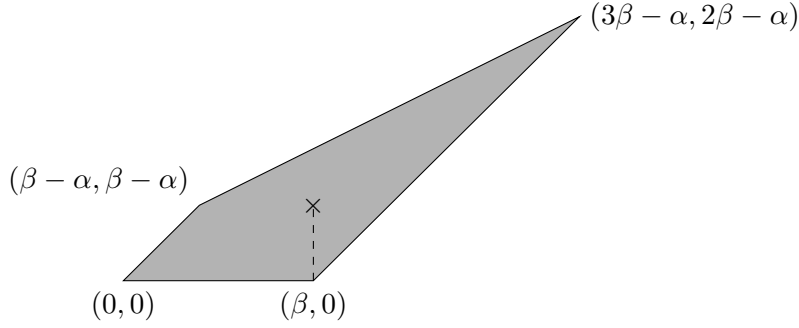
\begin{figure}

\begin{center}
\begin{tikzpicture}
\filldraw[draw=black, fill=gray!60] (0,0) -- (1,1) -- (6,3.5) -- (2.5,0) -- cycle; 

\draw (2.5,1) node {$\times$}; 
\draw [dashed] (2.5,0) -- (2.5,1); 

\draw (0,0) node[below] {$(0,0)$};
\draw (6,3.5) node[right] {$(3\beta - \alpha,2\beta - \alpha)$};
\draw (2.5,0) node[below] {$(\beta, 0)$};
\draw (1,1) node[above left] {$(\beta - \alpha, \beta - \alpha)$};
\end{tikzpicture} 

\end{center} 
\caption{A representative of the marked semitoric polygon of type (3c) with cut down for $n=3$.} 
\label{fig:poly_3c_down}
\end{figure}      

This polygon is Delzant, so according to the strategy explained in Section \ref{subsect:general}, we consider a representative of the associated isomorphism class of toric systems, namely:
\[ J = \frac{1}{2} (|z_1|^2 + |z_2|^2), \quad H = \frac{1}{2} |z_1|^2 \]
on the symplectic reduction $(M,\omega)$ of $\C^4$ by 
\[ N = \frac{1}{2} \left( |z_1|^2 + |z_3|^2 + |z_4|^2, |z_2|^2 + |z_4|^2 \right) \]
at level $(2\beta - \alpha, \beta)$. Note that $(M,\omega) = \left( W_1(\beta-\alpha,\beta), \omega_{W_1(\beta-\alpha,\beta)} \right)$, which can be deduced from inspecting the polygon in Figure \ref{fig:poly_3c_down}.

For our purpose, it is more convenient
to place the $\Z_2$-sphere $\{ z_3 = 0\}$ horizontally, which can be achieved by considering instead $(M,\omega,(J,H_0))$ where $H_0 = -\frac{1}{2} |z_3|^2$, whose image is displayed in Figure \ref{fig:F0_3c}.

\begin{figure}

\begin{center}
\begin{tikzpicture}
\filldraw[draw=black, fill=gray!60] (0,-1) -- (1,0) -- (6,0) -- (2.5,-3.5) -- cycle; 

\draw (0,-1) node[below left] {$(0,-(\beta - \alpha))$};
\draw (6,0) node[above] {$(3\beta - \alpha,0)$};
\draw (2.5,-3.5) node[below] {$(\beta, -(2 \beta - \alpha))$};
\draw (1,0) node[above] {$(\beta - \alpha, 0)$};

\draw (1,-2.2) node[left] {$z_1 = 0$};
\draw (0.5,-0.4) node[left] {$z_2 = 0$};
\draw (3,0) node[above] {$z_3 = 0$};
\draw (4,-2) node[right] {$z_4 = 0$};

\end{tikzpicture} 

\end{center} 
\caption{The image of $M$ by $F_0 = (J,H_0)$ given in Theorem \ref{thm:system_3c}. Each of the edges equals $F_0(\{ z_{\ell} = 0 \})$ where $\ell \in \{1,2,3,4\}$ is indicated near the corresponding edge.} 
\label{fig:F0_3c}
\end{figure}
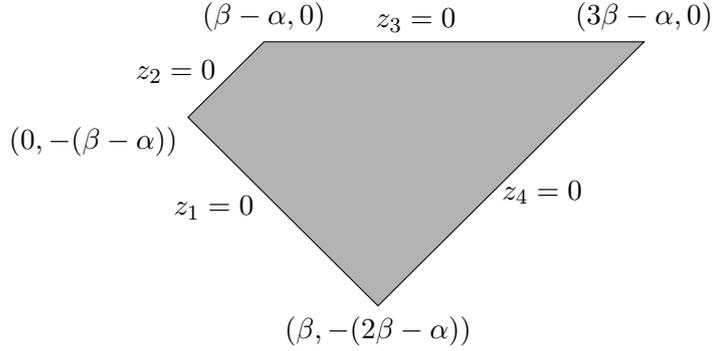  

The point that we want to submit to a Hamiltonian-Hopf bifurcation is the one corresponding to $z_1 = 0 = z_4$, namely $D = [0,\sqrt{2\beta},\sqrt{2(2\beta - \alpha)},0]$. Note that 
\[ J = \beta + \frac{1}{2} |z_1|^2 - \frac{1}{2} |z_4|^2 \]
so in view of Proposition \ref{prop:comm_q1_ff} we consider $\mathcal{X} = \Re(z_1 \bar{z}_2 \bar{z}_3^2 z_4)$, so that $D$ is of focus-focus type for $(J,\mathcal{X})$. Following the general strategy of Section \ref{subsect:general}, we consider 
\[ H_t = (1-2t) H_0 + 2 \gamma t (X + \delta R^2) - 2 \gamma \delta t (2\beta - \alpha)^2  \]
where $R = \frac{1}{2} (|z_1|^2 + |z_4|^2)$. This system has the following properties.

\begin{thm}[Type (3c), $n=3$]
\label{thm:system_3c}
Let $\beta > 0$ and $0 < \alpha < \beta$. Let $0 < \gamma < \frac{1}{4(2\beta - \alpha) \sqrt{2\beta}}$ and let $\delta > \max\left(\frac{1}{2(2\beta-\alpha)\gamma}, \frac{6\beta - \alpha}{\sqrt{2\beta}}\right)$. Let 
\[ J = \frac{1}{2} (|z_1|^2 + |z_2|^2), \qquad H_t = \frac{(2t-1)}{2} |z_3|^2  + 2 \gamma t (\mathcal{X} + \delta R^2) - 2 \gamma \delta t (2\beta - \alpha)^2. \]
Then $(M,\omega,F_t = (J,H_t))$ is 
\begin{itemize}
    \item of toric type when $0 \leq t < t^-$;
    \item semitoric with one focus-focus point ($D = [0,\sqrt{2\beta},\sqrt{2(2\beta - \alpha)},0]$) when $t^- < t < t^+$,
\end{itemize}
where 
\[ t^- = \frac{1}{2(1 + 2 \gamma (2 \beta - \alpha) \sqrt{2\beta} )}, \qquad t^+ = \frac{1}{2(1 - 2 \gamma (2 \beta - \alpha) \sqrt{2\beta} )}. \]
For $t\in (t^-,t^+)$ the semitoric polygon of the system is the one shown in Figure~\ref{fig:min_type3c} with $n=3$, so it is of type (3c) with parameters $\alpha$, $\beta$, and $n=3$.
Moreover, for any $h_0 \in (0,h^+)$, where
\[ h^+ = \frac{16 \beta (2\beta - \alpha) \arctan\left( \sqrt{\frac{\alpha + 6\beta}{2\beta - \alpha}} \right) + 16\beta^2 \arctan\left( \frac{\sqrt{(\alpha + 6\beta)(2\beta - \alpha)}}{\alpha + 2\beta} \right) - (6\beta - \alpha) \sqrt{(\alpha + 6\beta)(2\beta - \alpha)}}{16 \pi \beta},  \]
there exists a choice of $\gamma \in \left(0,\frac{1}{4(2\beta - \alpha) \sqrt{2\beta}}\right)$, $\delta \in \left(\max\left(\frac{1}{2(2\beta-\alpha)\gamma}, \frac{6\beta - \alpha}{\sqrt{2\beta}}\right), +\infty\right)$ and $t \in (t^-,t^+)$ such that the height invariant of the system is $h = h_0$. 
\end{thm}

The image of the momentum map $F_t$ is displayed in Figure \ref{fig:image_3c_1_2_1over30_6}, and a representative of the marked semitoric polygon of the systems for $t^- < t < t^+$ is displayed in Figure \ref{fig:poly_3c_down}. Recall that the height invariant of a system whose semitoric polygon is of type (3c) (see Figure \ref{fig:min_type3c}) can take every possible value in $(0,\beta - \frac{\alpha}{2})$; the range of height invariants that the system in the above statement produces is a strict sub-interval of this interval. In fact, one readily checks that for all $\alpha, \beta > 0$ with $\alpha < \beta$, 
\[ 0.59 \approx \left( 1 - \frac{3\sqrt{3}}{4\pi} \right) < \frac{h^+}{\beta - \frac{\alpha}{2}} <  2 - \left( \frac{16\arctan(\sqrt{7}) + 5 \sqrt{7}}{8\pi} \right) \approx 0.7 \]
(by noticing that $\frac{h^+}{\beta - \frac{\alpha}{2}}$ is a stricly increasing function of $\alpha \in (0,\beta)$). So with this system, we get more than half the possible values of the height invariant. The moral reason for not obtaining every possible height invariant is that the inequality $\delta > \frac{6\beta - \alpha}{\sqrt{2\beta}}$ prevents us from choosing $\delta$ such that the focus-focus point is arbitrarily close to the $\Z_2$-sphere, contrary to what happened for the system studied in the previous section, see Section \ref{subsect:height_CP2}.

\begin{figure}
\begin{center}
\includegraphics[scale=0.4]{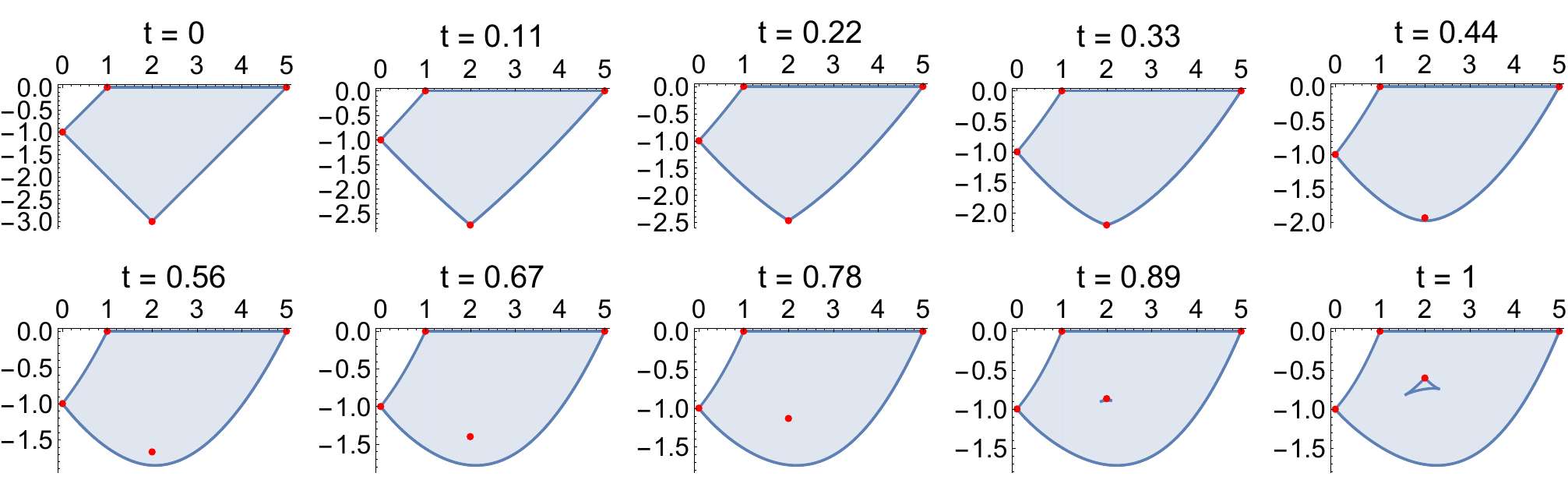}
\end{center}
\caption{Image of $(J,H_t)$ as in Theorem \ref{thm:system_3c} with $\alpha=1$, $\beta = 2$, $\gamma = \frac{1}{30}$ and $\delta = 6$. Note that for this choice of parameters, $t^- = \frac{5}{14}$ and $t^+ = \frac{5}{6}$.}
\label{fig:image_3c_1_2_1over30_6}
\end{figure}

\begin{rmk}
We conjecture that this system is hypersemitoric with one triangular flap when $t^+ < t \leq 1$, like the system in \ref{thm:CP2_sys}. However, we do not try to prove this here because it is not the main topic of the paper and would add complicated computations; such a proof would use the ideas introduced in the proof of Theorem \ref{thm:CP2_sys}.   
\end{rmk}

\begin{rmk}
When $\frac{\sqrt{2\beta}}{2(2\beta-\alpha)(6\beta-\alpha)} < \gamma < \frac{1}{4(2\beta - \alpha) \sqrt{2\beta}}$, we have that $\frac{1}{2(2\beta-\alpha)\gamma} < \frac{6\beta - \alpha}{\sqrt{2\beta}}$. In this case, for $\frac{1}{2(2\beta-\alpha)\gamma} < \delta < \frac{6\beta - \alpha}{\sqrt{2\beta}}$, the point $D$ is still focus-focus for $t \in (t^-,t^+)$, but $F_t$ can have hyperbolic-regular points, forming a two-sided flap (see also Remark \ref{rmk:bound_delta_3c}). This is illustrated in Figures \ref{fig:image_3c_1_2_1over26_4point5} and \ref{fig:image_3c_1_2_1over26_4point5_t089}.   
\end{rmk}

\begin{figure}
\begin{center}
\includegraphics[scale=0.4]{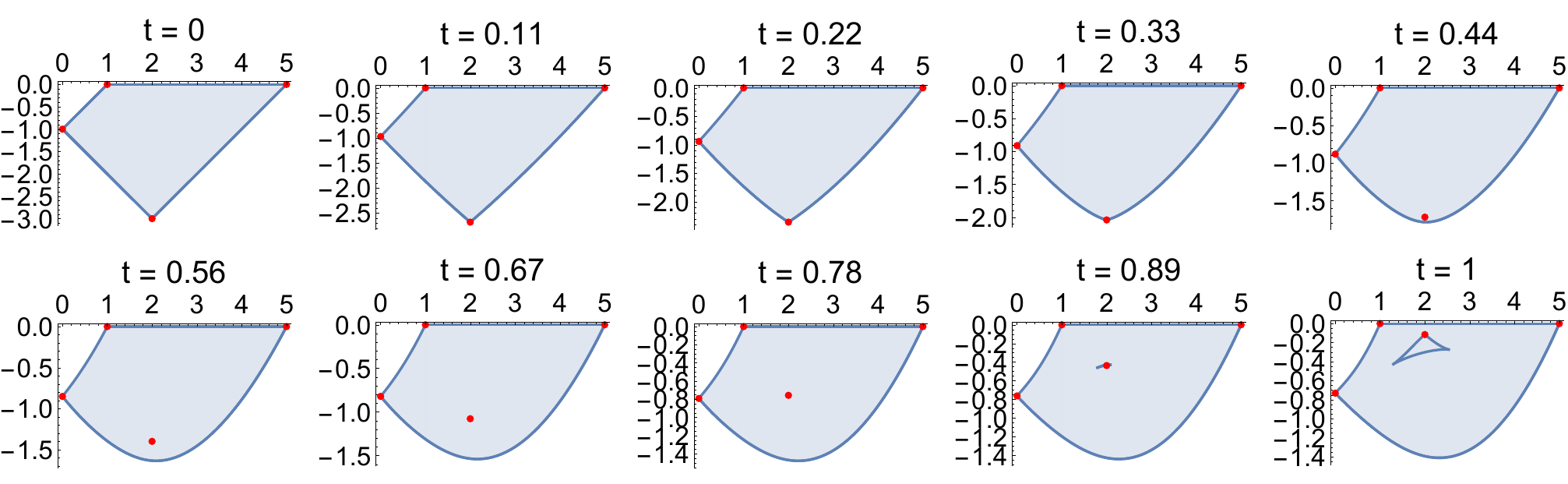}
\end{center}
\caption{Image of $(J,H_t)$ as in Theorem \ref{thm:system_3c} with $\alpha=1$, $\beta = 2$, $\gamma = \frac{1}{26}$ and $\delta = 4.5$. Note that for this choice of parameters, $t^- = \frac{13}{38} \approx 0.34$ and $t^+ = \frac{13}{14} \approx 0.93$. Observe that hyperbolic-regular values appear before $t = t^+$; this is more visible in Figure \ref{fig:image_3c_1_2_1over26_4point5_t089}, where we only plot $F_t(M)$ for $t=0.89$. }
\label{fig:image_3c_1_2_1over26_4point5}
\end{figure}

\begin{figure}
\begin{center}
\includegraphics[scale=0.6]{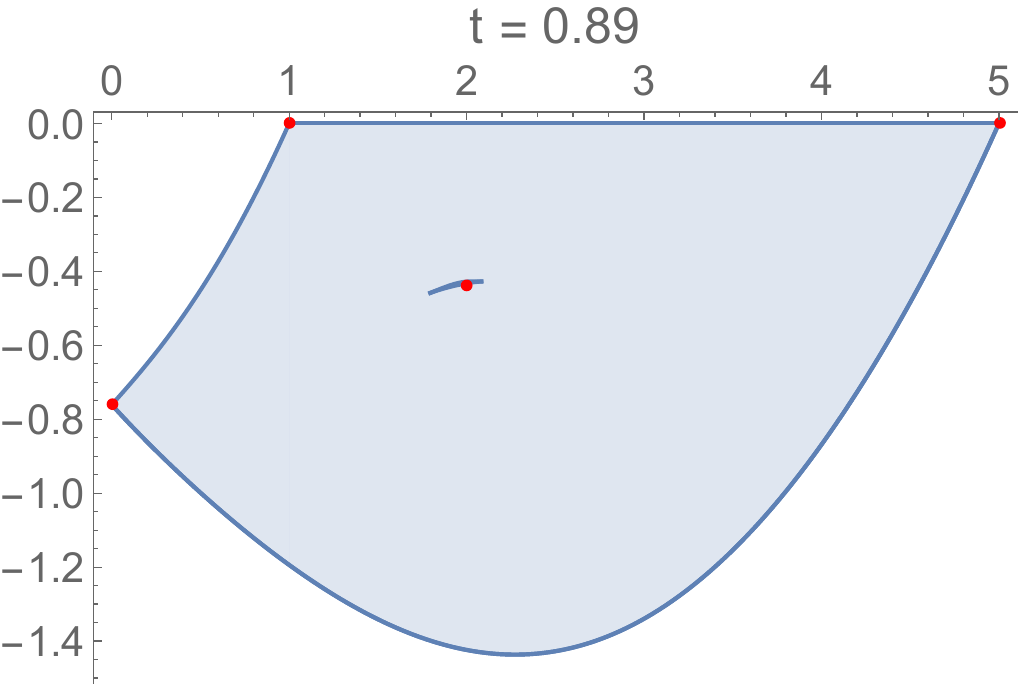}
\end{center}
\caption{The $t=0.89$ part of Figure \ref{fig:image_3c_1_2_1over26_4point5}. Recall that for this system $t^+ \approx 0.93$.}
\label{fig:image_3c_1_2_1over26_4point5_t089}
\end{figure}

\begin{proof}[Proof of Theorem \ref{thm:system_3c}]

We only give the key ideas of the proof, since a similar detailed proof is given for the system in Section \ref{sec:CP2}. We emphasize on one crucial aspect that was not present in the latter and influences the choice of $\delta$, and that we discuss again in Remark \ref{rmk:bound_delta_3c}. The system for $t=0$ can easily be shown to be of toric type as in Lemma \ref{lm:CP2_zero}, so in this proof we work with $t \in (0,t^+)$.

\paragraph{Rank zero points.} The rank zero singular points of $F_t$ are $A = [0,0,\sqrt{2(\beta - \alpha)},\sqrt{2\beta}]$, $B = [\sqrt{2(\beta-\alpha)},0,0,\sqrt{2\beta}]$, $C = [\sqrt{2(2\beta - \alpha)},\sqrt{2\beta},0,0]$ and $D = [0,\sqrt{2\beta},\sqrt{2(2\beta - \alpha)},0]$. We check the type of the transition point $D$, and leave the computations for the points $A$, $B$ and $C$ to the reader, in the spirit of Section \ref{subsec:rank_zero_CP2}.

Near $D$, we use the action of $N$ to assume that $z_2 = x_2 \in \R^+$ and $z_3 = x_3 \in \R^+$ and obtain local complex coordinates $z_1, z_4$ such that $\omega = \frac{i}{2} (\dd z_1 \wedge \dd \bar{z}_1 + \dd z_4 \wedge \dd \bar{z}_4)$. Then $x_2 = \sqrt{2\beta - |z_4|^2}$ and $x_3 = \sqrt{2(2\beta-\alpha) - |z_1|^2 - |z_4|^2}$, so in these local coordinates, $J = \beta + \frac{1}{2} |z_1|^2 - \frac{1}{2} |z_4|^2$ and 
\[ H_t = (2t-1)(2\beta - \alpha) - 2 \gamma \delta t (2\beta - \alpha)^2 + \frac{1-2t}{2}(|z_1|^2 + |z_4|^2) + 4 \gamma t (2\beta - \alpha) \sqrt{2\beta} \Re(z_1 z_4) + O(3) \]
so $H_t$ is of the form \eqref{eq:H_quad} with $\mu_1 = 4\gamma t (2\beta - \alpha) \sqrt{2\beta}$ and $\mu_2 = \mu_3 = \frac{1-2t}{2}$. By Proposition \ref{prop:comm_q1_ff}, $D$ is of focus-focus type if $|1-2t| < 4\gamma t (2\beta - \alpha) \sqrt{2\beta}$ and of elliptic-elliptic type if $|1-2t| > 4\gamma t (2\beta - \alpha) \sqrt{2\beta}$. The former amounts to $t \in (t^-, t^+)$ with $t^{\pm}$ as in Theorem \ref{thm:system_3c}. 

\paragraph{Rank one points.} We prove, similarly as in the beginning of Section \ref{subsec:rank_one_CP2}, that there are no rank one points with $z_1 = 0$ or $z_2 = 0$ or $z_4 = 0$ (the $\Z_2$-sphere $\{z_3 = 0\}$ of course consists of rank one points, except for the two points $B$ and $C$). One readily checks, as in Section \ref{subsec:rankone_Z2_CP2}, that the rank one points on the $\Z_2$-sphere $\{z_3 = 0\}$ are all non-degenerate; so now we focus on the rank one points with $z_3 \neq 0$.

In order to do so, for $j \in (0,3\beta-\alpha)$, we introduce cylindrical coordinates on $M_j^{\text{red}} \setminus \left( \{z_1 = 0 \} \cup \{ z_3 = 0\} \right)$ as in Section \ref{subsect:rankone_reduced}: we write $z_1 = \rho e^{i\theta}$ and, using the actions of $N$ and $J$, we choose representatives $z_2 = x_2 \in \R^+$, $z_3 = x_3 \in \R^+$ and $z_4 = x_4 \in \R^+$. Then $\rho \in (\sqrt{X^-(j)},\sqrt{X^+(j)})$ where $X^-(j) = \max(0,2(j-\beta))$ and $X^+(j) = \min(2j, j + \beta - \alpha)$ and
\[ x_2 = \sqrt{2j - \rho^2}, \quad x_3 = \sqrt{2(\beta - \alpha + j - \rho^2)}, \quad x_4 = \sqrt{2(\beta - j) + \rho^2}, \]
so we obtain that
\[ \begin{split} \Hred & = (2t - 1) (\beta - \alpha + j - \rho^2) - 2 \gamma \delta t (2\beta - \alpha)^2 \\ & + 2\gamma t \left( 2 \rho (\beta - \alpha + j - \rho^2) \cos \theta \sqrt{(2j - \rho^2)(2(\beta - j) + \rho^2)} + \delta (\rho^2 + \beta - j)^2 \right). \end{split}  \]
Hence $\Hred$ is of the form \eqref{eqn:Hred} with $b(X) = - 4 \gamma t (\beta - \alpha + j - X)$, $c(X) = X (2j - X)(2(\beta - j) + X)$ and
\[ a(X) = (2t - 1) (\beta - \alpha + j - X) + 2 \gamma \delta t (X + \beta - j)^2  - 2 \gamma \delta t (2\beta - \alpha)^2. \]
This allows us to compute the functions $f$ and $g$ of Equation \eqref{eq:fandg_general} as
\[ f(X) = 2 \left( 1 - 2 t + 4 \gamma \delta t (X + \beta - j) \right) \sqrt{X (2j - X)(2(\beta - j) + X)}, \qquad g(X) = -4\gamma t P(X) \]
where $P$ is given by
\[ P(X) =  5 X^3 + (3 \alpha + 5 \beta - 19 j ) X^2 + 4 (\alpha \beta + 5 j^2 - 2 j \beta - 2 j \alpha - \beta^2) X + 4 j (\beta - j) (\beta - \alpha + j). \]

We claim that 
\begin{itemize}
    \item $f+g$ vanishes exactly once in $(X^-(j),X^+(j))$;
    \item $f-g$ vanishes exactly once in $(X^-(j),X^+(j))$ if $0 < j < \beta - \alpha$;
    \item $f-g$ does not vanish in $(X^-(j),X^+(j))$ if $\beta - \alpha \leq j < 3\beta - \alpha$.
\end{itemize}
Recall from Lemma \ref{lm:crit_Hred_polar} that these zeroes correspond to critical points of $\Hred$. One then readily checks using Lemma \ref{lm:type_rankone} that the zero of $f+g$ is an elliptic point corresponding to a local minimum of $\Hred$ and that the zero of $f-g$ for $0 < j < \beta - \alpha$ is an elliptic point corresponding to a local maximum of $\Hred$. We will only give details for the third item of the above claim, since it is crucial when trying to avoid potential hyperbolic-regular points (and is the reason of one of the bounds on $\delta$). We leave the proof of the other claims, which can easily be adapted from the proofs in Section \ref{subsect:other_rankone_CP2}, to the reader.

As in the proof of Lemma \ref{lm:fplusg}, we write for $X \in (X^-(j),X^+(j))$
\[ (f-g)(X) = 2 \sqrt{c(X)} \left( 1 - 2 t \left( 1 - 2 \gamma \left( \delta (X + \beta - j) + \frac{P(X)}{2\sqrt{X (2j - X)(2(\beta - j) + X)}} \right) \right) \right) \]
so that $(f-g)(X) > 2 \sqrt{c(X)} \left( 1 - 2 t (1 - 2 \gamma h(X)) \right)$ with
\begin{equation} h(X) = \frac{6\beta - \alpha}{\sqrt{2\beta}} (X + \beta - j) + \frac{P(X)}{2\sqrt{X (2j - X)(2(\beta - j) + X)}} \label{eq:h_3c} \end{equation}
(here we have used the bound $\delta > \frac{6\beta - \alpha}{\sqrt{2\beta}}$). One readily checks that when $j$ varies in $[\beta-\alpha,3\beta-\alpha)$ and $X$ varies in $(X^-(j),X^+(j))$, $h(X)$ is bounded below by its limit at the point $j=\beta, X=0$. Hence for all $X \in (X^-(j),X^+(j))$, $h(X) \geq (2\beta - \alpha) \sqrt{2\beta}$; this implies that 
\[ (f-g)(X) > 2 \sqrt{c(X)} \left( 1 - 2 t (1 - 2 \gamma (2\beta-\alpha) \sqrt{2\beta})) \right) = 2 \sqrt{c(X)} \left(1 - \frac{t}{t^+}\right) > 0 \]
since $t < t^+$.

\paragraph{Height invariant.}

The symplectic volume of $J^{-1}(\beta)$ is $\pi(2\beta - \alpha)$. The symplectic form is $\frac{1}{2} \dd X \wedge \dd \theta$. Moreover
\[ H_t^{\text{red},\beta} = (2t - 1) (2\beta - \alpha - X) - 2 \gamma \delta t (2\beta - \alpha)^2 + 2\gamma t X \left( 2 (2\beta - \alpha - X) \cos \theta \sqrt{(2\beta - X)} + \delta X \right).  \]
Let $t \in (t^-, t^+)$. Then $(X,\theta) \in \{H_t^{\text{red},\beta} < H_t(D)\}$ if and only if $0 < X < 2\beta - \alpha$ and
\[ 1 - 2t + 2 \gamma \delta t X + 4 \gamma t (2\beta - \alpha - X) \cos \theta \sqrt{2\beta - X} < 0   \]
which amounts to 
\[ \cos \theta < \frac{2t - 1 - 2\gamma \delta t X}{4\gamma t (2\beta - \alpha - X) \sqrt{2\beta-X}} =: G(X). \]
Since $t < t^+$, we have that $2t - 1 < 4 \gamma t (2\beta - \alpha) \sqrt{2\beta}$ so, using that $\delta > 2\sqrt{2\beta}$, we obtain that
\[ G(X) < \frac{ (2\beta - \alpha - X) \sqrt{2\beta}}{ (2\beta - \alpha - X) \sqrt{2\beta-X}} < 1.  \]
Moreover, for every $X \in (0,2\beta - \alpha)$,
\[ G'(X) = \frac{Q(X)}{8 \gamma t (2\beta - X)^{\frac{3}{2}} (2\beta - \alpha  - X)^2}, \]
where
\[ Q(X) = 2 \gamma \delta t X^2 + \left( 3 (1 - 2t) + 2 \gamma \delta t (2\beta - \alpha)  \right) X + (2t-1)(6\beta-\alpha) - 8 \beta \gamma \delta t (2\beta - \alpha)  \]
is a degree two polynomial with positive leading coefficient. We claim that $Q(0) < 0$ and $Q(2\beta-\alpha) < 0$, which implies that $Q(X) < 0$ for every $X \in (0,2\beta-\alpha)$. Indeed  
\[ Q(0) = (2t-1) (6\beta-\alpha) - 8 \beta \gamma \delta t (2\beta-\alpha) < (6\beta - \alpha) \left( 2t - 1 - 4 \gamma t (2\beta-\alpha)\sqrt{2\beta} \right) < 0 \]
(here we have used that $\delta > \frac{6\beta-\alpha}{\sqrt{2\beta}}$ and that $t < t^+$). Moreover, using that $\delta > \frac{1}{2(2\beta-\alpha)\gamma}$ and $t < 1$,
\[ Q(2\beta-\alpha) =  2 \alpha \left( 2 t - 1 - 2\gamma\delta t (2\beta-\alpha) \right) < 4 \alpha (t-1) < 0.  \]

Hence $G$ is stricly decreasing on $(0,2\beta-\alpha)$. Moreover, $G$ goes to $-\infty$ when $X$ goes to $2\beta-\alpha$. Consequently, there exists a unique $X_0 \in (0,2\beta-\alpha)$ such that $G(X_0) = -1$, and $G(X) \in [-1,1]$ if $X \in (0,X_0]$ and $G(X) < -1$ if $X \in (X_0, 2\beta - \alpha)$. Therefore,
\[ 2 \pi h = \pi X_0 - \int_0^{X_0} \arccos(G(X)) \dd X.  \]
Then the exact same reasoning as in Section \ref{subsect:height_CP2} applies, and we compute that the maximal height invariant, obtained by first letting $t \to t^+$ and then letting $\gamma \to \frac{1}{4(2\beta-\alpha)\sqrt{2\beta}}$ and $\delta \to \frac{6\beta-\alpha}{\sqrt{2\beta}}$, equals 
\[ \frac{16 \beta (2\beta - \alpha) \arctan\left( \sqrt{\frac{\alpha + 6\beta}{2\beta - \alpha}} \right) + 16\beta^2 \arctan\left( \frac{\sqrt{(\alpha + 6\beta)(2\beta - \alpha)}}{\alpha + 2\beta} \right) - (6\beta - \alpha) \sqrt{(\alpha + 6\beta)(2\beta - \alpha)}}{16\pi\beta}. \]
\end{proof}

\begin{rmk}
\label{rmk:boundary_3c}
As explained in Section \ref{subsect:general} and as is the case for the system in the previous section (see Remark \ref{rmk:boundary_CP2}), our choice of parameters $\gamma$ and $\delta$ ensures that the points in the $\Z_2$-sphere always correspond to maxima of $H_t$ in their $J$-fiber. Indeed, a computation similar to the one of Remark \ref{rmk:boundary_CP2} shows that the difference $\varepsilon_t = H_t(w_1,w_2,0,w_4) - H_t(z_1,z_2,z_3,z_4)$ satisfies
\[ \begin{split} \varepsilon_t  & = \frac{1-2t}{2} |z_3|^2 - 2 \gamma t \left( \Re(z_1 \bar{z}_2 \bar{z_3}^2 z_4) - \frac{\delta}{2} |z_3|^2 \left( 2 (2\beta - \alpha) - \frac{1}{2} |z_3|^2 \right) \right) \\ 
& \geq \frac{1-2t}{2} |z_3|^2 - 2 \gamma t \frac{1}{2} |z_3|^2 \left(  |z_2| (|z_1|^2 + |z_4|^2) - \delta \left( 2 (2\beta - \alpha) - \frac{1}{2} |z_3|^2 \right) \right) \\
& \geq \frac{1-2t}{2} |z_3|^2 - 2 \gamma t \frac{1}{2} |z_3|^2 \left(  \sqrt{2\beta} (2(2\beta-\alpha) - |z_3|^2) - \delta  \left( 2 (2\beta - \alpha) - \frac{1}{2} |z_3|^2 \right) \right) \\
& = \frac{1}{2} |z_3|^2 \left( 1 - 2 t  - 2 \gamma t \left( 2 (2\beta - \alpha)(\sqrt{2\beta} - \delta) + |z_3|^2 \left(\frac{\delta}{2} - \sqrt{2\beta}\right)  \right)  \right). \end{split}  \]
In the second line we have used that $\Re(z_1 \bar{z}_2 \bar{z_3}^2 z_4) \leq \frac{1}{2} |z_2| |z_3|^2 (|z_1|^2 + |z_4|^2)$, and in the third line we have used that $|z_2| \leq \sqrt{2\beta}$. We have also used intensively that $\frac{1}{2} (|z_1|^2 + |z_4|^2) = 2 \beta - \alpha - \frac{1}{2} |z_3|^2$. Now, since $\gamma < \frac{1}{4(2\beta - \alpha) \sqrt{2\beta}}$ (or since $\alpha < \beta$), we have that 
\[ \delta > \max\left(\frac{1}{2(2\beta-\alpha)\gamma}, \frac{6\beta - \alpha}{\sqrt{2\beta}}\right) > 2 \sqrt{2\beta}; \]
since moreover $|z_3|^2 \leq 2(2\beta - \alpha)$, this yields
\[ \varepsilon_t \geq \frac{1}{2} |z_3|^2 \left( 1 - 2t \left( 1 - (2\beta - \alpha) \gamma \delta \right) \right). \]
Since $\delta > \frac{1}{2(2\beta-\alpha)\gamma}$, we obtain that for every $t \in [0,1]$, $\varepsilon_t \geq 0$, with equality if and only if $z_3 = 0$.
\end{rmk}

\begin{rmk}
\label{rmk:bound_delta_3c}
The bound for $\delta$ in Theorem \ref{thm:system_3c} follows from two different requirements. The first one is that the image of the $\Z_2$-sphere always lies on the top boundary of $F_t(M)$, and gives the bound $\delta > \frac{1}{2(2\beta - \alpha) \gamma}$ (see Remark \ref{rmk:boundary_3c}). The second one is that there is no hyperbolic-regular values (and hence no flap) in $F_t(M)$ when $t < t^+$; this is obtained from the non-vanishing of $f-g$ in the above proof, and gives the bound $\delta > \frac{6\beta - \alpha}{\sqrt{2\beta}}$. In practice, to find this second bound, we assumed that $\delta > C$ to obtain 
\[ h(X) = C (X + \beta - j) + \frac{P(X)}{2\sqrt{X (2j - X)(2(\beta - j) + X)}} \]
instead of the function $h$ in Equation \eqref{eq:h_3c}, and optimized the constant $C$ to ensure that $h \geq (2\beta - \alpha) \sqrt{2\beta}$. In the system from Theorem \ref{thm:CP2_sys}, both conditions were satisfied as soon as $\delta > \frac{1}{2\gamma \alpha}$, so this was less apparent.
\end{rmk}

\subsubsection{An explicit system with a higher isotropy sphere}

Following the same strategy also yields a system of type (3c) for any $n \geq 3$ that we describe now. We consider the symplectic reduction $(M,\omega) = \left( W_{n-2}(\beta-\alpha,\beta), \omega_{W_{n-2}(\beta-\alpha,\beta)} \right)$ of $\C^4$ by the action of 
\[ N = \frac{1}{2} \left( |z_1|^2 + |z_3|^2 + (n-2) |z_4|^2, |z_2|^2 + |z_4|^2 \right) \]
at level $\left( (n-1)\beta - \alpha, \beta \right)$, and we start with $F_0 = (J,H_0)$ with $J = \frac{1}{2} (|z_1|^2 + |z_2|^2)$ and $H_0 = -\frac{1}{2} |z_3|^2$. The image $F_0(M)$ is displayed in Figure \ref{fig:F0_3c_general}.

\begin{figure}

\begin{center}
\begin{tikzpicture}[scale=0.8]
\filldraw[draw=black, fill=gray!60] (0,-1) -- (1,0) -- (7,0) -- (2,-5) -- cycle; 

\draw (0,-1) node[below left] {$(0,-(\beta - \alpha))$};
\draw (7,0) node[above] {$(n\beta - \alpha,0)$};
\draw (2,-5) node[below] {$(\beta, -((n-1) \beta - \alpha))$};
\draw (1,0) node[above] {$(\beta - \alpha, 0)$};

\draw (0.8,-3) node[left] {$z_1 = 0$};
\draw (0.5,-0.4) node[left] {$z_2 = 0$};
\draw (4,0) node[above] {$z_3 = 0$};
\draw (4.8,-2.5) node[right] {$z_4 = 0$};

\end{tikzpicture} 

\end{center} 
\caption{The image of $M$ by $F_0 = (J,H_0)$ given in Theorem \ref{thm:system_3c_general}. Each of the edges equals $F_0(\{ z_{\ell} = 0 \})$ where $\ell \in \{1,2,3,4\}$ is indicated near the corresponding edge.} 
\label{fig:F0_3c_general}
\end{figure}
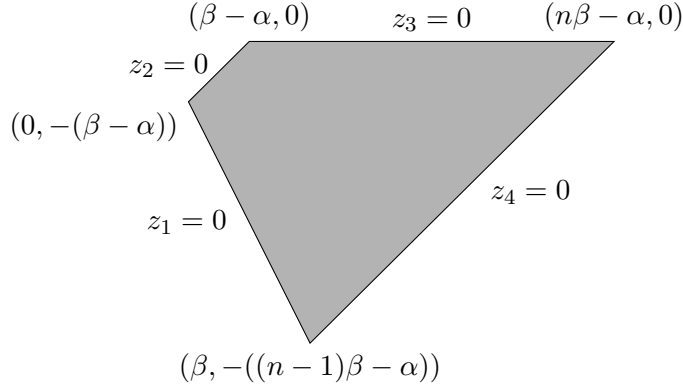  

Additionally, this time we consider
\[ \mathcal{X} = \Re(z_1 \bar{z}_2 \bar{z}_3^{n-1} z_4), \qquad  R = \frac{1}{2} \left( |z_1|^2 + (n-2) |z_4|^2 \right). \]
Using these ingredients, we obtain the system described in the following statement.

\begin{thm}[Type (3c), $n \geq 3$]
\label{thm:system_3c_general}
Let $n \geq 3$. Let $\beta > 0$ and $0 < \alpha < \beta$. Let $0 < \gamma < \frac{n-1}{2^{\frac{n+3}{2}}\left((n-1)\beta - \alpha\right)^{\frac{n-1}{2}} \sqrt{2\beta}}$ and let $\delta > \max\left(\frac{1}{2((n-1)\beta-\alpha)\gamma}, \frac{2^{\frac{n+1}{2}} \left((n-1)\beta - \alpha\right)^{\frac{n-3}{2}} \left( n (n-1) \beta - \alpha \right) }{(n-1)^2\sqrt{2\beta}}\right)$. Let 
\[ J = \frac{1}{2} (|z_1|^2 + |z_2|^2), \qquad H_t = \frac{(2t-1)}{2} |z_3|^2  + 2 \gamma t (\mathcal{X} + \delta R^2) - 2 \gamma \delta t ((n-1)\beta - \alpha)^2. \]
Then $(M,\omega,F_t = (J,H_t))$ is 
\begin{itemize}
    \item of toric type when $0 \leq t < t^-$;
    \item semitoric with one focus-focus point (the point $D=[0,\sqrt{2\beta},\sqrt{2((n-1)\beta - \alpha)},0]$) when $t^- < t < t^+$,
\end{itemize}
where 
\[ t^- = \frac{1}{2\left(1 + \frac{2^{\frac{n+1}{2}} \gamma \left((n-1) \beta - \alpha\right)^{\frac{n-1}{2}} \sqrt{2\beta}}{n-1} \right)}, \qquad t^+ = \frac{1}{2\left(1 - \frac{2^{\frac{n+1}{2}} \gamma \left((n-1) \beta - \alpha\right)^{\frac{n-1}{2}} \sqrt{2\beta}}{n-1} \right)}. \]
Moreover, for $t\in (t^-,t^+)$ the semitoric polygon of the system is the one shown in Figure~\ref{fig:min_type3c}, so it is of type (3c) with parameters $\alpha$, $\beta$, and $n$.
\end{thm}

The image of the momentum map $F_t$ is displayed in Figures \ref{fig:image_3c_n4_1_2_1over120_20} and \ref{fig:image_3c_n5_1over2_1_1over100_26}. As for the $n=3$ case (see above), we conjecture that for every $n \geq 3$, the system is hypersemitoric with one triangular flap with elliptic-elliptic corner $D$ when $t > t^+$.

We did not investigate the possible values of the height invariant that we can obtain with this family of systems for $n \geq 4$ since the computations are more involved that in the $n=3$ case. However, one can rather easily evaluate this height invariant numerically since, as in the proof of Theorem \ref{thm:system_3c}, it satisfies 
\[ 2 \pi h = \pi X_0 - \int_0^{X_0} \arccos(G(X)) \dd X  \]
where this time,
\[ G(X) = \frac{(n-1)\left( 2t - 1 - \gamma \delta t (n-1) X \right)}{4\gamma t \sqrt{(2\beta-X)\left( 2 \left( (n-1) \beta - \alpha \right) - (n-1) X \right)^{n-1}}} \]
and $X_0 \in (0,2\beta)$ is the unique solution of the equation $G(X) = -1$. For instance, in Figure \ref{fig:hplus_n4}, we plot the value of the maximal possible height invariant of the system when $n=4$ and $\beta=2$.

\begin{figure}
\begin{center}
\includegraphics[scale=0.4]{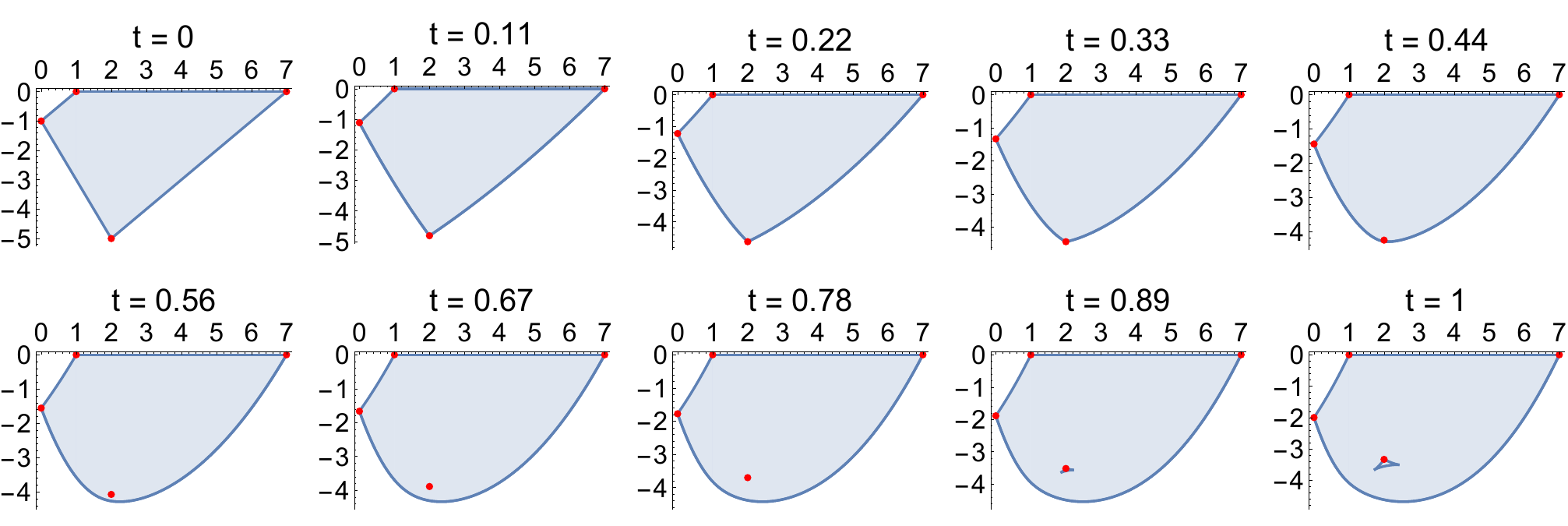}
\end{center}
\caption{Image of $F_t = (J,H_t)$ as in Theorem \ref{thm:system_3c_general} with $n=4$, $\alpha=1$, $\beta = 2$, $\gamma = \frac{1}{120}$ and $\delta = 20$. Note that for this choice of parameters, $t^- = \frac{9}{2(9+\sqrt{10})} \approx 0.37$ and $t^+ = \frac{9}{2(9-\sqrt{10})} \approx 0.77$.}
\label{fig:image_3c_n4_1_2_1over120_20}
\end{figure}

\begin{figure}
\begin{center}
\includegraphics[scale=0.4]{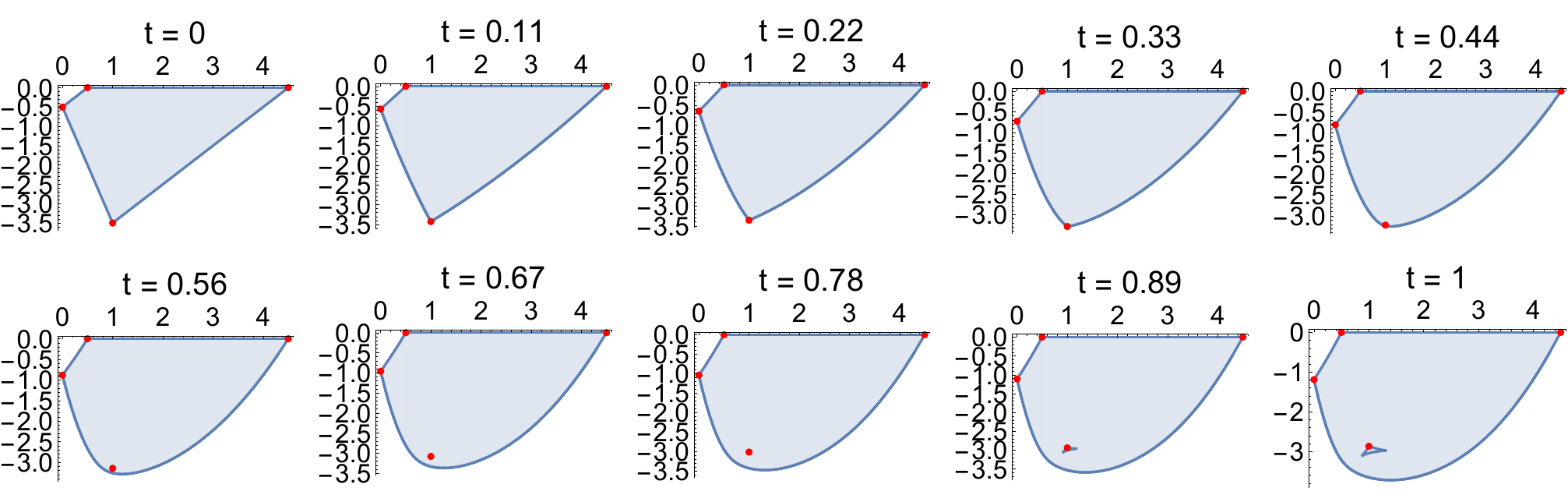}
\end{center}
\caption{Image of $F_t = (J,H_t)$ as in Theorem \ref{thm:system_3c_general} with $n=5$, $\alpha=\frac{1}{2}$, $\beta = 1$, $\gamma = \frac{1}{100}$ and $\delta = 26$. For these values of the parameters, $t^- = \frac{100}{200 + 49\sqrt{2}} \approx 0.37$ and $t^+ = \frac{100}{200 - 49\sqrt{2}} \approx 0.77$.}
\label{fig:image_3c_n5_1over2_1_1over100_26}
\end{figure}

\begin{figure}
\begin{center}
\includegraphics[scale=0.5]{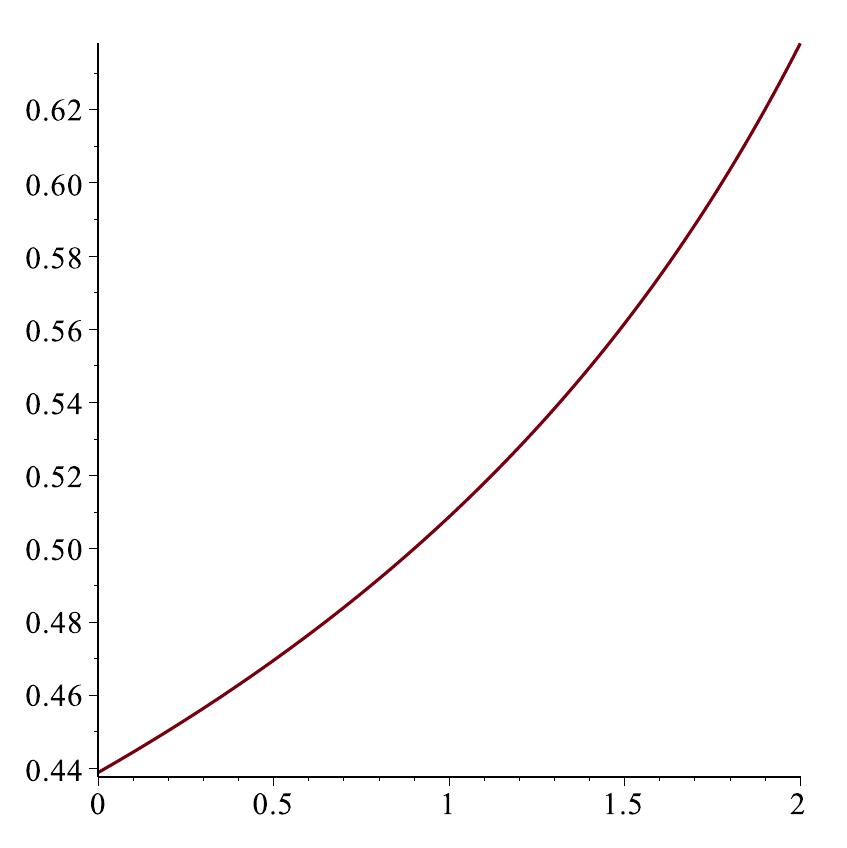}
\end{center}
\caption{Maximal value of the height invariant that the system in Theorem \ref{thm:system_3c_general} can attain for $n=4$ and $\beta = 2$, as a function of $\alpha \in (0,\beta)$.}
\label{fig:hplus_n4}
\end{figure}

The proof of Theorem \ref{thm:system_3c_general} is similar to the proof of Theorem \ref{thm:system_3c}, so we leave it to the reader. However it is useful to explain how to find this system, and in particular how to come up with the constraints on the parameter $\delta$. As in the previous section (see Remark \ref{rmk:bound_delta_3c}), these constraints come from two requirements.

First, we want to prevent the appearance of hyperbolic-regular points for $t \leq t^+$, which, as in the proof of Theorem \ref{thm:system_3c}, amounts to choose $\delta$ such that $f-g$ does not vanish in $(X^-(j) = \max(0,2(j-\beta)),X^+(j)=\min(2j,\frac{2(\beta-\alpha+(n-2)j)}{n-1}))$ if $\beta - \alpha \leq j < n\beta - \alpha$. Here we used once again the notation $f,g$ from Equation \eqref{eq:fandg_general}. As in the proof of this theorem, we obtain that for $X \in (X^-(j),X^+(j))$, $(f-g)(X) = (n-1) \sqrt{c(X)} \left( 1 - 2 t (1 - \gamma h(X)) \right)$ with
\begin{equation} h(X) = \delta \left( (n-2)(\beta-j) + (n-1) X \right) + \frac{c'(X)}{(n-1)\sqrt{c(X)}} \label{eq:h_3c_general} \end{equation}
where 
\[ c(X) = X (2j-X) \left( 2(\beta-j) + X \right) \left( 2 (\beta - \alpha + (n-2)j) - (n-1)X \right)^{n-1}. \]
As in the $n=3$ case, one conjectures (and proves a posteriori) that when $\delta$ is bigger than some well-chosen $\delta_0$, $h$ is bounded below by its limit at the point $j=\beta, X=0$, which equals
\[ \frac{2^{\frac{n+1}{2}} \left((n-1) \beta - \alpha\right)^{\frac{n-1}{2}} \sqrt{2\beta}}{n-1}. \]
So then $(f-g)(X) > (n-1) \sqrt{c(X)} \left(1 - \frac{t}{t^+}\right) > 0$. We find $\delta_0$ by asking that $h'_{|j = \beta}(0) = 0$; this yields
\[ \delta_0 = \frac{2^{\frac{n+1}{2}} \left((n-1)\beta - \alpha\right)^{\frac{n-3}{2}} \left( n (n-1) \beta - \alpha \right) }{(n-1)^2\sqrt{2\beta}}. \]

Second, we want the $\Z_{n-1}$-sphere $\{ z_3 = 0 \}$ to be sent to the top boundary of $F_t(M)$, for every $t \in [0,1]$. The same computations as in Remark \ref{rmk:boundary_3c} lead to the fact that the quantity $\varepsilon_t = H_t(w_1,w_2,0,w_4) - H_t(z_1,z_2,z_3,z_4)$ satisfies
\[ \begin{split} \varepsilon_t & \geq \frac{1}{2} |z_3|^2 \bigg( 1 - 2 t  \\ & \left.   - 2 \gamma t \left( 2 ((n-1)\beta - \alpha)\left(\sqrt{2\beta} \left( (n-1)\beta - \alpha) \right)^{\frac{n-3}{2}} - \delta\right) + |z_3|^2 \left(\frac{\delta}{2} - \sqrt{2\beta} \left( (n-1)\beta - \alpha) \right)^{\frac{n-3}{2}}\right)  \right)  \right). \end{split}  \]
But one readily checks that $\delta > \delta_0$ implies 
\[ \delta > 2\sqrt{2\beta} \left( (n-1)\beta - \alpha \right)^{\frac{n-3}{2}} \]
which, since $|z_3|^2 \leq 2 \left( (n-1)\beta - \alpha \right)$, yields
\[ \varepsilon_t  \geq \frac{1}{2} |z_3|^2 \left( 1 - 2 t \left( 1  - \left( (n-1)\beta - \alpha \right) \gamma \delta \right)  \right) . \]
Since we want $\varepsilon_t$ to be non-negative for every $t \in [0,1]$, we require that
\[ \delta > \frac{1}{2((n-1)\beta-\alpha)\gamma}. \]

\subsection{Explicit systems of type (3a)}
\label{subsec:3a}

\subsubsection{An explicit system with a $\Z_2$-sphere}
\label{subsubsec:3a_Z2}

We consider the polygon of type (3a) displayed in Figure \ref{fig:min_type3a} with $n = 3$; we draw this polygon for this particular value of $n$ in Figure \ref{fig:poly_3a_up}. 

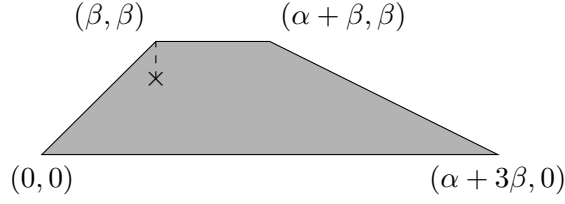
\begin{figure}

\begin{center}

\begin{tikzpicture}
\filldraw[draw=black, fill=gray!60] (0,0) -- (6,0) -- (3,1.5) -- (1.5,1.5) -- cycle; 

\draw (1.5,1) node {$\times$}; 
\draw [dashed] (1.5,1) -- (1.5,1.5);

\draw (0,0) node[below] {$(0,0)$};
\draw (6,0) node[below] {$(\alpha + 3 \beta,0)$};
\draw (3,1.5) node[above right] {$(\alpha + \beta,\beta)$};
\draw (1.5,1.5) node[above left] {$(\beta, \beta)$};
\end{tikzpicture} 

\end{center} 
\caption{A representative of the marked semitoric polygon of type (3a) with cut up for $n=3$.} 
\label{fig:poly_3a_up}
\end{figure} 

Here we have two choices: we can either start with this polygon or the one with cut down, displayed in Figure \ref{fig:poly_3a_down} (both are Delzant). For the (3c) case (see Section \ref{subsubsec:3c_Z2} above), it was necessary to start with the latter, so we do the same here (and also because the computations are nicer). 

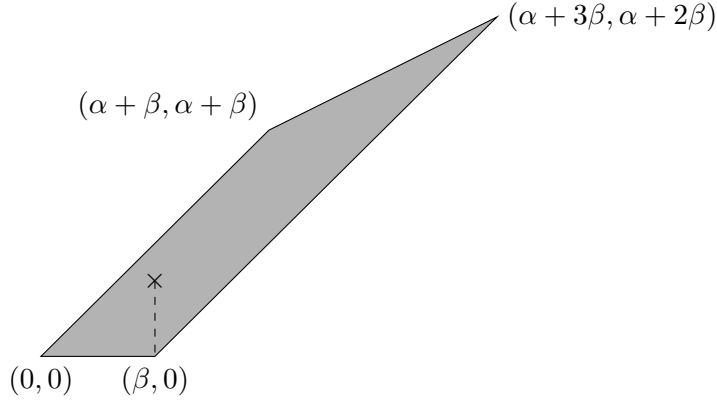
\begin{figure}

\begin{center}

\begin{tikzpicture}
\filldraw[draw=black, fill=gray!60] (0,0) -- (3,3) -- (6,4.5) -- (1.5,0) -- cycle; 

\draw (1.5,1) node {$\times$}; 
\draw [dashed] (1.5,1) -- (1.5,0);

\draw (0,0) node[below] {$(0,0)$};
\draw (6,4.5) node[right] {$(\alpha + 3 \beta, \alpha + 2\beta)$};
\draw (3,3) node[above left] {$(\alpha + \beta, \alpha + \beta)$};
\draw (1.5,0) node[below] {$(\beta, 0)$};
\end{tikzpicture} 

\end{center} 
\caption{A representative of the marked semitoric polygon of type (3a) with cut down for $n=3$.} 
\label{fig:poly_3a_down}
\end{figure} 

The manifold is the symplectic reduction of $\C^4$ by
\[ N = \frac{1}{2} \left( |z_1|^2 + |z_3|^2 + |z_4|^2, |z_2|^2 + |z_4|^2 \right) \]
at level $(\alpha + 2\beta, \beta)$, namely $\left( W_1(\alpha + \beta,\beta), \omega_{W_1(\alpha + \beta,\beta)} \right)$, and 
\[ J = \frac{1}{2} (|z_1|^2 + |z_2|^2), \qquad H = \frac{1}{2} |z_1|^2. \]
The $\Z_2$-sphere is $\{z_3 = 0\}$. So instead we choose $H_0 = - \frac{1}{2} |z_3|^2$, and obtain the image displayed in Figure \ref{fig:F0_3a}.

\begin{figure}

\begin{center}
\begin{tikzpicture}
\filldraw[draw=black, fill=gray!60] (0,-3) -- (3,0) -- (6,0) -- (1.5,-4.5) -- cycle; 

\draw (0,-3) node[left] {$(0,-(\alpha + \beta))$};
\draw (6,0) node[above right] {$(\alpha + 3\beta,0)$};
\draw (1.5,-4.5) node[below] {$(\beta, -(\alpha + 2\beta))$};
\draw (3,0) node[above left] {$(\alpha + \beta, 0)$};

\draw (4,-2.2) node[right] {$z_4 = 0$};
\draw (0.7,-4) node[left] {$z_1 = 0$};
\draw (1.1,-1.5) node[above] {$z_2 = 0$};
\draw (4.5,0) node[above] {$z_3 = 0$};

\end{tikzpicture} 

\end{center} 
\caption{The image of $M$ by $F_0 = (J,H_0)$, see Theorem \ref{thm:system_3a}. Each of the edges equals $F_0(\{ z_{\ell} = 0 \})$ where $\ell \in \{1,2,3,4\}$ is indicated near the corresponding edge.} 
\label{fig:F0_3a}
\end{figure}
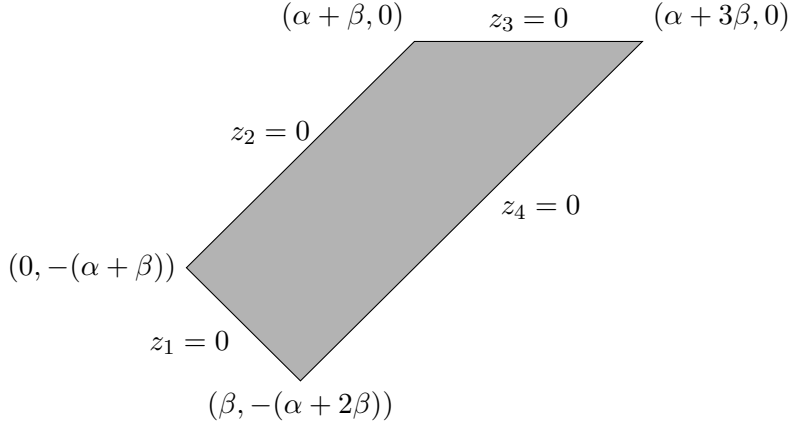

The transition point will correspond to $z_1 = 0 = z_4$, and $J = \beta + \frac{1}{2}(|z_1|^2 - |z_4|^2)$. So in view of Proposition \ref{prop:comm_q1_ff}, we want to take $\mathcal{X} = \Re(z_1 \bar{z}_2 \bar{z}_3^2 z_4)$. Following the strategy from Section \ref{subsect:general}, we consider 
\[ H_t = (1-2t) H_0 + 2\gamma t (\mathcal{X} + \delta R^2) - 2\gamma \delta t (\alpha + 2\beta)^2, \quad R = \frac{1}{2} (|z_1|^2 + |z_4|^2). \]

This suggests that everything works as if we could replace $\alpha$ by $-\alpha$ in Theorem \ref{thm:system_3c}. As the following statement shows, this is almost the case, except for the computation of the height invariant. 

\begin{thm}[Type (3a), $n=3$]
\label{thm:system_3a}
Let $\alpha, \beta > 0$. Let $0 < \gamma < \frac{1}{4(\alpha + 2\beta) \sqrt{2\beta}}$ and let $\delta > \max\left(\frac{1}{2(\alpha + 2\beta)\gamma}, \frac{\alpha + 6\beta}{\sqrt{2\beta}}\right)$. Let 
\[ J = \frac{1}{2} (|z_1|^2 + |z_2|^2), \qquad H_t = \frac{(2t-1)}{2} |z_3|^2  + 2 \gamma t (\mathcal{X} + \delta R^2) - 2 \gamma \delta t (\alpha + 2\beta)^2. \]
Then $(M,\omega,F_t = (J,H_t))$ is 
\begin{itemize}
    \item of toric type when $0 \leq t < t^-$;
    \item semitoric with one focus-focus point ($D = [0,\sqrt{2\beta},\sqrt{2(\alpha + 2\beta)},0]$) when $t^- < t < t^+$,
\end{itemize}
where 
\[ t^- = \frac{1}{2(1 + 2 \gamma (\alpha + 2 \beta) \sqrt{2\beta} )}, \qquad t^+ = \frac{1}{2(1 - 2 \gamma (\alpha + 2 \beta) \sqrt{2\beta} )}. \]
For $t\in (t^-,t^+)$ the semitoric polygon of the system is the one shown in Figure~\ref{fig:min_type3a} with $n=3$, so it is of type (3a) with parameters $\alpha$, $\beta$, and $n=3$.
Moreover, for any $h_0 \in (0,h^+)$, where
\[ h^+ = \begin{cases} \frac{16 \beta (\alpha + 2\beta) \arctan\left( \sqrt{\frac{6\beta - \alpha}{\alpha + 2\beta}} \right) + 16\beta^2 \arctan\left( \frac{\sqrt{(6\beta - \alpha)(\alpha + 2\beta)}}{2\beta - \alpha} \right) - (\alpha + 6\beta) \sqrt{(6\beta - \alpha)(\alpha + 2\beta)}}{16 \pi \beta} \text{ if } \alpha < 2\beta, \\
\left( \frac{3}{2} - \frac{2}{\pi} \right) \beta \text{ if } \alpha = 2\beta, \\
\beta - \frac{(\alpha+6\beta)\sqrt{(6\beta-\alpha)(\alpha+2\beta)} + 16\beta^2\arctan\left( \frac{\sqrt{(6\beta-\alpha)(\alpha+2\beta)}}{\alpha-2\beta} \right) - 16\beta(\alpha+2\beta)\arctan\left( \sqrt{\frac{6\beta-\alpha}{\alpha+2\beta}} \right) }{16\pi\beta} \text{ if } 2\beta < \alpha < 6\beta, \\
\beta \text{ if } \alpha \geq 6\beta,  \end{cases} \]
there exists a choice of $\gamma \in \left(0,\frac{1}{4(\alpha + 2\beta) \sqrt{2\beta}}\right)$, $\delta \in \left(\max\left(\frac{1}{2(\alpha + 2\beta)\gamma}, \frac{\alpha + 6\beta}{\sqrt{2\beta}}\right), +\infty\right)$ and $t \in (t^-,t^+)$ such that the height invariant of the system is $h = h_0$.
\end{thm}

The image of the momentum map of this system is displayed in Figure \ref{fig:image_3a_1_1_1over24_6}; the marked semitoric polygon of the system for $t^- < t < t^+$ is of type (3a) with parameters $\alpha, \beta$ and $n=3$, as shown in Figure \ref{fig:min_type3a}. As for the system in Theorem \ref{thm:system_3c}, we conjecture that this system is hypersemitoric with one triangular flap with elliptic corner $D$ for $t \in (t^+,1)$.

A remarkable feature of this system is that for some values of the scaling parameters $(\alpha,\beta)$, we can obtain every possible value for the height invariant of the system. Moreover, as in the previous example, one readily checks that $\frac{h^+}{\beta} > \left( 1 - \frac{3\sqrt{3}}{4\pi} \right) \approx 0.59$.

\begin{figure}
\begin{center}
\includegraphics[scale=0.4]{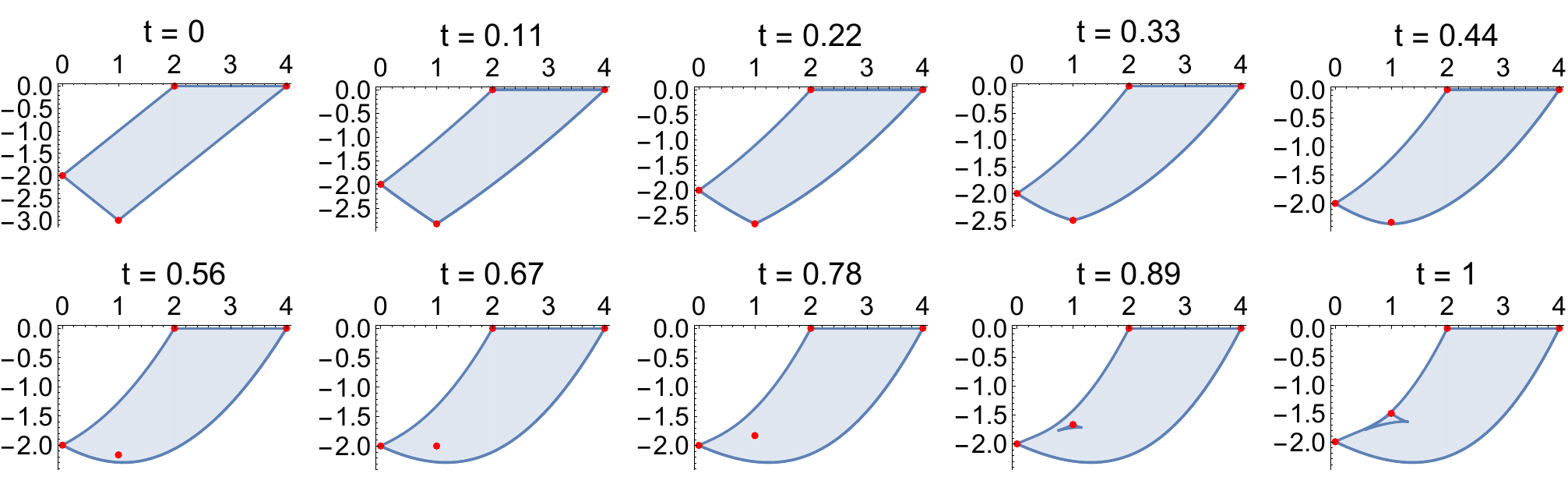}
\end{center}
\caption{Image of $(J,H_t)$ as in Theorem \ref{thm:system_3a} with $\alpha=1$, $\beta = 1$, $\gamma = \frac{1}{24}$ and $\delta = 6$. For this choice of parameters, $t^- = \frac{2}{4+\sqrt{2}} \approx 0.37$ and $t^+ = \frac{2}{4-\sqrt{2}} \approx 0.77$.}
\label{fig:image_3a_1_1_1over24_6}
\end{figure}

\begin{proof}[Proof of Theorem \ref{thm:system_3a}]

Since the proof is similar to the proof of Theorem \ref{thm:system_3c}, we only explain the differences between these two, in particular the ones regarding the height invariant. Again, we work with $t \in (0,t^+)$.

\paragraph{Rank zero points.} As in the previous section, we only write the details for the point $D = [0,\sqrt{2\beta},\sqrt{2(\alpha + 2\beta)},0]$. We use the action of $N$ to assume that $z_2 = x_2 \in \R^+$ and $z_3 = x_3 \in \R^+$ and obtain local complex coordinates $z_1, z_4$ near $D$ with $\omega = \frac{i}{2} (\dd z_1 \wedge \dd \bar{z}_1 + \dd z_4 \wedge \dd \bar{z}_4)$. This time we obtain that $J = \beta + \frac{1}{2} |z_1|^2 - \frac{1}{2} |z_4|^2$ and 
\[ H_t = (2t-1)(\alpha + 2\beta) - 2 \gamma \delta t (\alpha + 2\beta)^2 + \frac{1-2t}{2}(|z_1|^2 + |z_4|^2) + 4 \gamma t (\alpha + 2\beta) \sqrt{2\beta} \Re(z_1 z_4) + O(3). \]
Hence $H_t$ is as in \eqref{eq:H_quad} with $\mu_1 = 4\gamma t (\alpha + 2\beta) \sqrt{2\beta}$ and $\mu_2 = \mu_3 = \frac{1-2t}{2}$. Therefore, Proposition \ref{prop:comm_q1_ff} implies that $D$ is of focus-focus type if and only if $|1-2t| < 4\gamma t (\alpha + 2\beta) \sqrt{2\beta}$, that is if and only if $t \in (t^-, t^+)$ with $t^{\pm}$ as in Theorem \ref{thm:system_3a}. 

\paragraph{Rank one points.} Again, one readily checks that there are no rank one points with $z_1 = 0$ or $z_2 = 0$ or $z_4 = 0$, and that the rank one points on the $\Z_2$-sphere $\{z_3 = 0\}$ are all non-degenerate (and one can prove as in Remark \ref{rmk:boundary_3c} that this $\Z_2$-sphere is sent to the top boundary of $F_t(M)$). 
 
To study the rank one points with $z_3 \neq 0$, we once again use, for $j \in (0,3\beta+\alpha)$, cylindrical coordinates on $M_j^{\text{red}} \setminus \left( \{z_1 = 0 \} \cup \{ z_3 = 0\} \right)$ as in Section \ref{subsect:rankone_reduced}. Namely, we choose $z_2 = x_2 \in \R^+$, $z_3 = x_3 \in \R^+$, $z_4 = x_4 \in \R^+$ and write $z_1 = \rho e^{i\theta}$ with $\rho \in (\sqrt{X^-(j)},\sqrt{X^+(j)})$ where $X^-(j) = \max(0,2(j-\beta))$ and $X^+(j) = \min(2j, j + \alpha + \beta)$. Then
\[ x_2 = \sqrt{2j - \rho^2}, \quad x_3 = \sqrt{2(\alpha + \beta + j - \rho^2)}, \quad x_4 = \sqrt{2(\beta - j) + \rho^2}, \]
which yields
\[ \begin{split} \Hred & = (2t - 1) (\alpha + \beta + j - \rho^2) - 2 \gamma \delta t (\alpha + 2\beta)^2 \\ & + 2\gamma t \left( 2 \rho (\alpha + \beta + j - \rho^2) \cos \theta \sqrt{(2j - \rho^2)(2(\beta - j) + \rho^2)} + \delta (\rho^2 + \beta - j)^2 \right). \end{split}  \]
Therefore, the functions $f$ and $g$ of Equation \eqref{eq:fandg_general} are
\[ f(X) = 2 \left( 1 - 2 t + 4 \gamma \delta t (X + \beta - j) \right) \sqrt{X (2j - X)(2(\beta - j) + X)}, \qquad g(X) = -4\gamma t P(X) \]
where $P$ is the polynomial
\[ P(X) =  5 X^3 + (5 \beta - 3 \alpha - 19 j ) X^2 + 4 (2 j \alpha - \alpha \beta + 5 j^2 - 2 j \beta  - \beta^2) X + 4 j (\beta - j) (\alpha + \beta + j). \]
One readily checks that 
\begin{itemize}
    \item $f+g$ vanishes exactly once in $(X^-(j),X^+(j))$;
    \item $f-g$ vanishes exactly once in $(X^-(j),X^+(j))$ if $0 < j < \alpha + \beta$;
    \item $f-g$ does not vanish in $(X^-(j),X^+(j))$ if $\alpha + \beta \leq j < 3\beta - \alpha$.
\end{itemize}
Again, by Lemma \ref{lm:crit_Hred_polar} these zeroes correspond to critical points of $\Hred$, and one readily checks thanks to Lemma \ref{lm:type_rankone} that the zero of $f+g$ is an elliptic point corresponding to a local minimum of $\Hred$ and that the zero of $f-g$ for $0 < j < \alpha + \beta$ is an elliptic point corresponding to a local maximum of $\Hred$.

\paragraph{Height invariant.} Let $t \in (t^-, t^+)$. Then $(X,\theta) \in \{H_t^{\text{red},\beta} < H_t(D)\}$ if and only if $0 < X < 2\beta$ and
\[ \cos \theta < \frac{2t - 1 - 2\gamma \delta t X}{4\gamma t (\alpha + 2\beta - X) \sqrt{2\beta-X}} =: G(X). \]
Moreover, for every $X \in (0,2\beta)$,
\[ G'(X) = \frac{Q(X)}{8 \gamma t (2\beta - X)^{\frac{3}{2}} (\alpha + 2\beta - X)^2} \]
where
\[ Q(X) = 2 \gamma \delta t X^2 + \left( 3 (1 - 2t) + 2 \gamma \delta t (\alpha + 2\beta)  \right) X + (2t-1)(\alpha + 6\beta) - 8 \beta \gamma \delta t (\alpha + 2\beta).  \]
Note that, since $\delta > \frac{\alpha + 6 \beta}{\sqrt{2\beta}}$ and $t < t^+$, 
\[  Q(0) = (2t-1) (\alpha + 6 \beta) - 8 \gamma \delta t \beta (\alpha + 2 \beta) < (\alpha + 6 \beta) \left( 2t - 1 - 4 \gamma t (\alpha + 2\beta) \sqrt{2\beta} \right) < 0  \]
and that $Q(2\beta) = \alpha \left( 2t - 1 - 4 \gamma \delta t \beta \right)$. If $\alpha \leq 2\beta$, then by using again the bound on $\delta$ and the fact that $t < t^+$,
\[ Q(2\beta) < \alpha \left( 2t - 1 - 2 \gamma t (\alpha + 6\beta) \sqrt{2\beta} \right) < \alpha \left( 2t - 1 - 2 \gamma t (\alpha + 2\beta) \sqrt{2\beta} \right) < 0, \]
so as in the proof of Theorem \ref{thm:system_3c}, we obtain that $G$ is strictly decreasing, that $G(X) < 1$ for every $X \in (0,2\beta)$ and that there exists a unique $X_0 \in (0, 2\beta)$ such that $G(X) = -1$. Hence
\[ 2 \pi h = \pi X_0 - \int_0^{X_0} \arccos(G(X)) \dd X  \]
and we compute $h^+$ for this case as in the aforementioned proof. However if $\alpha > 2\beta$ the behavior of $G$ may change; in what follows we directly work with the limiting case $(\gamma,\delta,t) \to \left( \frac{1}{4(\alpha + 2\beta) \sqrt{2\beta}}, \frac{\alpha + 6\beta}{\sqrt{2\beta}}, 1 \right)$ to compute $h^+$. In this limit 
\[ G(X) = \frac{(\alpha + 6\beta) X - 4 \beta (\alpha + 2 \beta)}{2 (X - \alpha - 2\beta) \sqrt{2\beta}\sqrt{2\beta - X}}, \qquad Q(X) = X \left( (\alpha + 6 \beta) X + (\alpha - 6\beta)(\alpha + 2\beta) \right) \]
so $Q$ vanishes at $0$ and 
\[ X_1 = \frac{(6\beta - \alpha)(\alpha + 2\beta)}{\alpha + 6\beta}, \]
which is negative whenever $\alpha > 6\beta$. In this case $G$ is strictly increasing on $(0,2\beta)$ and $G(0) = 1$, so $G \geq 1$ on $(0,2\beta)$ and $h^+ = \beta$. So it remains to investigate the case $2\beta < \alpha \leq 6\beta$. In this case one readily checks that $X_1 \in (0,2\beta)$, so $G$ is strictly decreasing on $(0,X_1)$ and strictly increasing on $(X_1,2\beta)$. Moreover $G(X_1) > 0$, so $G(X) \in (-1,1)$ for $X \in (0,X_2)$ and $G(X) > 1$ for $X \in (X_2,2\beta)$ where $X_2 = \frac{(\alpha + 2\beta)(6\beta-\alpha)}{8\beta}$. Hence 
\[ 2\pi h^+ = 2\pi\beta - \int_0^{X_2} \arccos(G(X)) \dd X, \]
and a straightforward computation gives the result in the statement.
\end{proof}

\subsubsection{An explicit system with a higher isotropy sphere}

As for type (3c) in the previous section, we can generalize this system to obtain systems with marked semitoric polygons of type (3a) with $n \geq 4$ (including a $\Z_{n-1}$-sphere).

We consider the symplectic reduction $(M,\omega) = \left( W_{n-2}(\alpha + \beta,\beta), \omega_{W_{n-2}(\alpha + \beta,\beta)} \right)$ of $\C^4$ by the action of 
\[ N = \frac{1}{2} \left( |z_1|^2 + |z_3|^2 + (n-2) |z_4|^2, |z_2|^2 + |z_4|^2 \right) \]
at level $\left( \alpha + (n-1)\beta, \beta \right)$, and we start with $F_0 = (J,H_0)$ with $J = \frac{1}{2} (|z_1|^2 + |z_2|^2)$ and $H_0 = -\frac{1}{2} |z_3|^2$. The image $F_0(M)$ is displayed in Figure \ref{fig:F0_3a_general}.

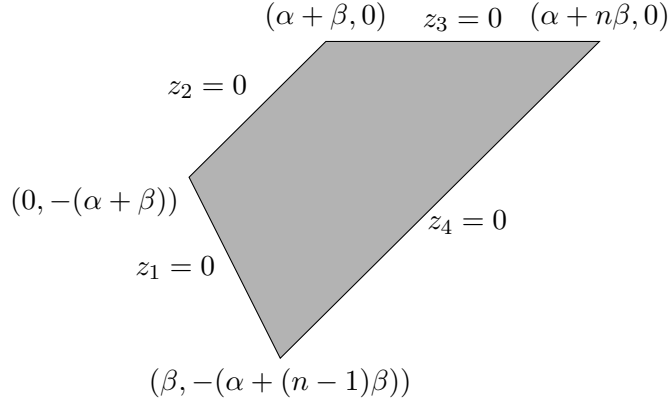
\begin{figure}

\begin{center}
\begin{tikzpicture}[scale=0.6]
\filldraw[draw=black, fill=gray!60] (0,-3) -- (3,0) -- (9,0) -- (2,-7) -- cycle; 

\draw (0,-3) node[below left] {$(0,-(\alpha + \beta))$};
\draw (9,0) node[above] {$(\alpha + n\beta,0)$};
\draw (2,-7) node[below] {$(\beta, -(\alpha + (n-1) \beta))$};
\draw (3,0) node[above] {$(\alpha + \beta, 0)$};

\draw (0.8,-5) node[left] {$z_1 = 0$};
\draw (1.5,-1) node[left] {$z_2 = 0$};
\draw (6,0) node[above] {$z_3 = 0$};
\draw (5,-4) node[right] {$z_4 = 0$};

\end{tikzpicture} 

\end{center} 
\caption{The image of $M$ by $F_0 = (J,H_0)$ given in Theorem \ref{thm:system_3a_general}. Each of the edges equals $F_0(\{ z_{\ell} = 0 \})$ where $\ell \in \{1,2,3,4\}$ is indicated near the corresponding edge.} 
\label{fig:F0_3a_general}
\end{figure}  

Moreover, we consider $\mathcal{X} = \Re(z_1 z_2 \bar{z}_3^{n-1} z_4) $ and $R = \frac{1}{2} \left( |z_1|^2 + (n-2) |z_4|^2\right)$.

\begin{thm}[Type (3a), $n \geq 3$]
\label{thm:system_3a_general}
Let $\alpha, \beta > 0$ and let $n \geq 3$. Let $0 < \gamma < \frac{n-1}{2^{\frac{n+3}{2}}\left((n-1)\beta + \alpha\right)^{\frac{n-1}{2}} \sqrt{2\beta}}$ and let
$\delta > \max\left(\frac{1}{2((n-1)\beta + \alpha)\gamma}, \frac{2^{\frac{n+1}{2}} \left((n-1)\beta + \alpha\right)^{\frac{n-3}{2}} \left( n (n-1) \beta + \alpha \right) }{(n-1)^2\sqrt{2\beta}}\right)$. Let 
\[ J = \frac{1}{2} (|z_1|^2 + |z_2|^2), \qquad H_t = \frac{(2t-1)}{2} |z_3|^2  + 2 \gamma t (\mathcal{X} + \delta R^2) - 2 \gamma \delta t ((n-1)\beta + \alpha)^2. \]
Then $(M,\omega,F_t = (J,H_t))$ is 
\begin{itemize}
    \item of toric type when $0 \leq t < t^-$;
    \item semitoric with one focus-focus point (the point $D = [0,\sqrt{2\beta},\sqrt{2(\alpha + (n-1)\beta)},0]$) when $t^- < t < t^+$,
\end{itemize}
where 
\[ t^- = \frac{1}{2\left(1 + \frac{2^{\frac{n+1}{2}} \gamma \left((n-1) \beta + \alpha\right)^{\frac{n-1}{2}} \sqrt{2\beta}}{n-1} \right)}, \qquad t^+ = \frac{1}{2\left(1 - \frac{2^{\frac{n+1}{2}} \gamma \left((n-1) \beta + \alpha\right)^{\frac{n-1}{2}} \sqrt{2\beta}}{n-1} \right)}. \]
Moreover,
for $t\in (t^-,t^+)$ the semitoric polygon of the system is the one shown in Figure~\ref{fig:min_type3a}, so it is of type (3a).
\end{thm}

The image of the momentum map for this system when $n=4$ and for some choice of parameters is displayed in Figure \ref{fig:image_3a_n4_2_1_1over90_15}.

\begin{figure}
\begin{center}
\includegraphics[scale=0.4]{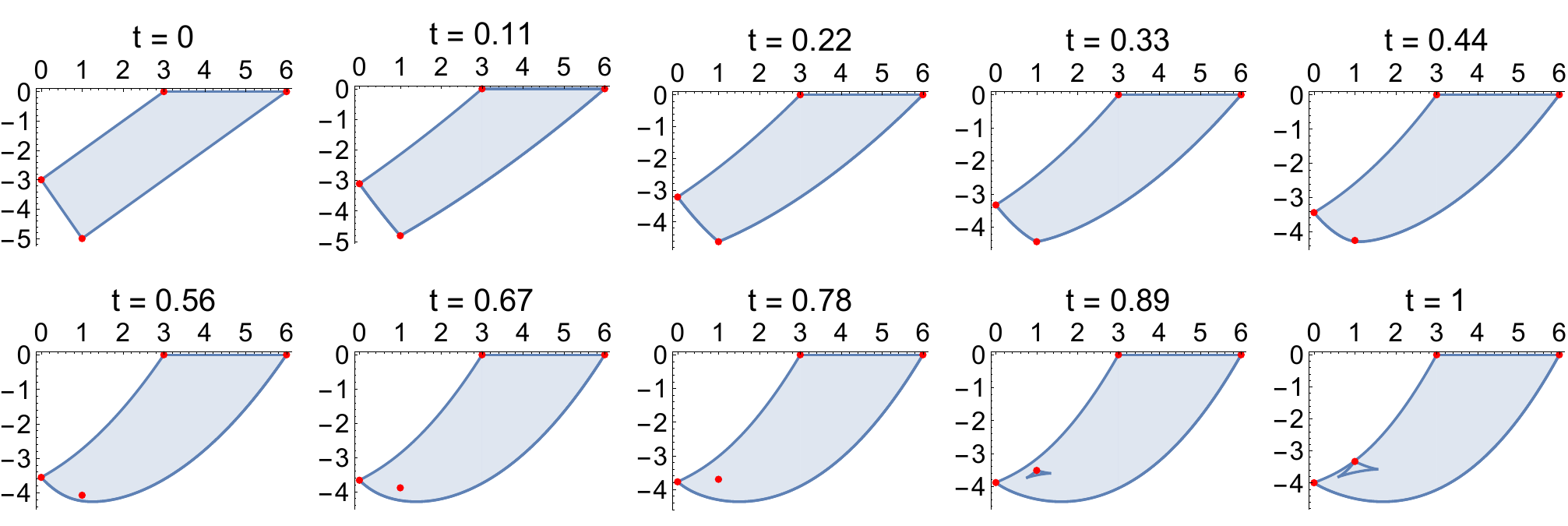}
\end{center}
\caption{Image of $(J,H_t)$ as in Theorem \ref{thm:system_3a_general} with $n=4$, $\alpha=2$, $\beta = 1$, $\gamma = \frac{1}{90}$ and $\delta = 15$. For these values of the parameters, $t^- = \frac{27}{2(27 + 4\sqrt{5})} \approx 0.38$ and $t^+ = \frac{27}{2(27 - 4\sqrt{5})} \approx 0.75$.}
\label{fig:image_3a_n4_2_1_1over90_15}
\end{figure}

\subsection{Explicit systems of type (3b)}
\label{subsec:explicit_3b}

In order to obtain explicit systems with polygons of type (3b), the idea is to consider the limit $\alpha \to 0$ in the systems of type (3a) (or of type (3c)) obtained in the previous sections. So we work in the manifold $(M,\omega) = \left( W_{n-2}(\beta,\beta), \omega_{W_{n-2}(\beta,\beta)} \right)$ obtained as the symplectic reduction of $\C^4$ by 
\[ N = \frac{1}{2} \left( |z_1|^2 + |z_3|^2 + (n-2) |z_4|^2, |z_2|^2 + |z_4|^2 \right) \]
at level $(2\beta,\beta)$, and obtain the following result (for the sake of concision we collect the cases $n = 3$ and $n \geq 4$ in a single statement).

\begin{thm}[Type (3b), $n \geq 3$]
\label{thm:system_3b_general}
Let $\beta > 0$ and let $n \geq 3$. Let 
\[ 0 < \gamma < \frac{1}{ 4 \left(2 \beta \right)^{\frac{n}{2}} (n-1)^{\frac{n-3}{2}}}, \qquad \delta > \max\left(\frac{1}{2(n-1)\beta\gamma}, n 2^{\frac{n}{2}} (n-1)^{\frac{n-5}{2}} \beta ^{\frac{n}{2}-1} \right).\] 
Let $\mathcal{X} = \Re(z_1 z_2 \bar{z}_3^{n-1} z_4) $, $R = \frac{1}{2} \left( |z_1|^2 + (n-2) |z_4|^2\right)$ and
\[ J = \frac{1}{2} (|z_1|^2 + |z_2|^2), \qquad H_t = \frac{(2t-1)}{2} |z_3|^2  + 2 \gamma t (\mathcal{X} + \delta R^2) - 2 \gamma \delta t (n-1)^2 \beta^2. \]
Then $(M,\omega,F_t = (J,H_t))$ is 
\begin{itemize}
    \item of toric type when $0 \leq t < t^-$;
    \item semitoric with one focus-focus point when $t^- < t < t^+$,
\end{itemize}
where 
\[ t^- = \frac{1}{2\left(1 + 2^{\frac{n}{2}+1} \gamma (n-1)^{\frac{n-3}{2}} \beta^{\frac{n}{2}} \right)}, \qquad t^+ = \frac{1}{2\left(1 - 2^{\frac{n}{2}+1} \gamma (n-1)^{\frac{n-3}{2}} \beta^{\frac{n}{2}} \right)}. \]
For $t\in (t^-,t^+)$ the semitoric polygon of the system is the one shown in Figure~\ref{fig:min_type3b}, so it is of type (3b) with parameters $\beta$ and $n$.
Moreover:
\begin{itemize}
    \item if $n = 3$, for any $h_0 \in (0,h^+)$ where 
\[ h^+ = \left( 1 - \frac{3\sqrt{3}}{4\pi} \right) \beta, \]
there exists a choice of $\gamma \in \left(0, \frac{1}{8 \beta \sqrt{2\beta}}\right)$, $\delta \in \left( \max\left( \frac{1}{4\beta \gamma}, 3 \sqrt{2\beta} \right), +\infty \right)$ and $t \in (t^-,t^+)$ such that the height invariant of the system is $h = h_0$;
    \item if $n = 4$, for any $h_0 \in (0,h^+)$ where 
\[ h^+ = \left( 1 - \frac{\ln (12-8\sqrt{2})}{\pi} \right) \beta,  \]
there exists a choice of $\gamma \in \left(0, \frac{1}{16 \beta^2 \sqrt{3}}\right)$, $\delta \in \left( \max\left( \frac{1}{6\beta \gamma}, \frac{16\beta}{\sqrt{3}} \right), +\infty \right)$ and $t \in (t^-,t^+)$ such that the height invariant of the system is $h = h_0$.
\end{itemize}

\end{thm}

The image of the momentum map for this system for a certain choice of parameters is shown in Figure \ref{fig:image_3b_2_1over50_7}; the marked semitoric polygon of the system for $t^- < t < t^+$ is of type (3b) with parameters $\beta$ and $n$, as shown in Figure \ref{fig:min_type3b}.

The computation of the maximal possible height invariant $h^+$ for $n \geq 5$ appears to be quite involved. It would be interesting to investigate this topic, but it would be even more interesting to find an explicit system for every marked semitoric polygon of type (3b) (and similarly for (3a) and (3c)). This calls for more work.

\begin{figure}
\begin{center}
\includegraphics[scale=0.4]{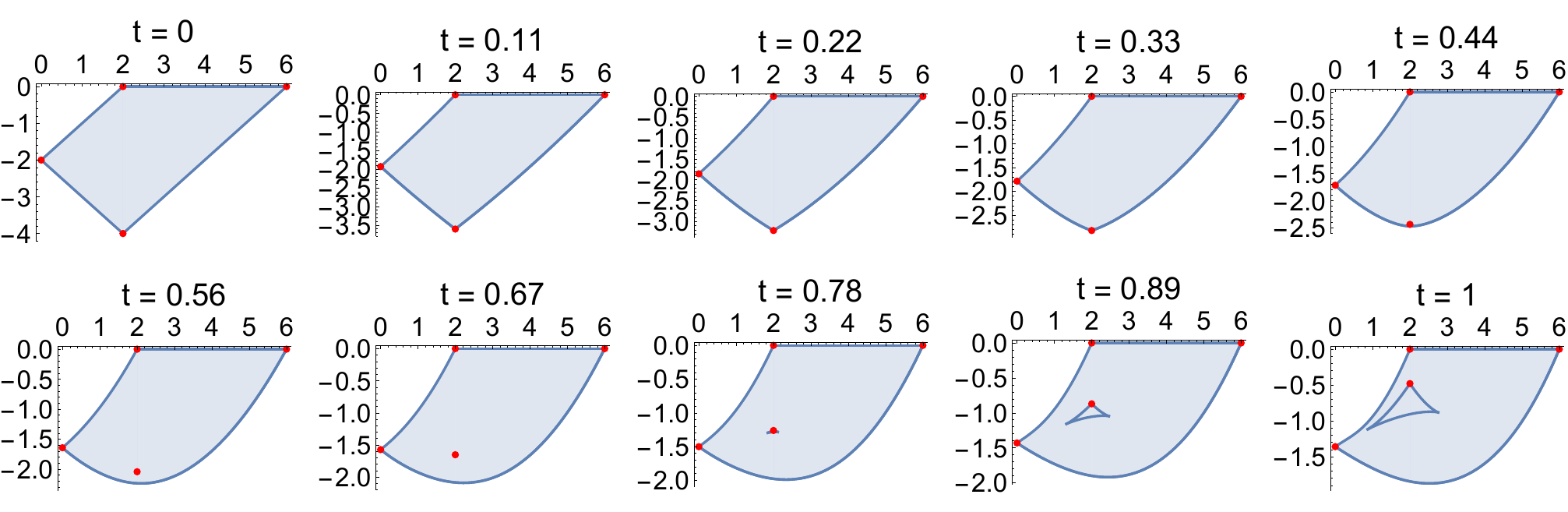}
\end{center}
\caption{Image of $(J,H_t)$ as in Theorem \ref{thm:system_3b_general} with $n=3$, $\beta = 2$, $\gamma = \frac{1}{50}$ and $\delta = 7$. For this choice of parameters, $t^- = \frac{25}{66} \approx 0.38$ and $t^+ = \frac{25}{34} \approx 0.74$.}
\label{fig:image_3b_2_1over50_7}
\end{figure}

\bibliographystyle{abbrv}
\bibliography{families}

\begin{thebibliography}{10}

\bibitem{AloHoh19}
J.~Alonso and S.~Hohloch.
\newblock Survey on recent developments in semitoric systems.
\newblock {\em RIMS Kokyuroku}, 2137, 2019.

\bibitem{AH-height}
J.~Alonso and S.~Hohloch.
\newblock The height invariant of a four-parameter semitoric system with two
  focus-focus singularities.
\newblock {\em J. Nonlinear Sci.}, 31(3):Paper No. 51, 32, 2021.

\bibitem{Ati}
M.~F. Atiyah.
\newblock Convexity and commuting {H}amiltonians.
\newblock {\em Bull. London Math. Soc.}, 14(1):1--15, 1982.

\bibitem{Auroux2009}
D.~Auroux.
\newblock Special {L}agrangian fibrations, wall-crossing, and mirror symmetry.
\newblock In {\em Surveys in differential geometry. {V}ol. {XIII}. {G}eometry,
  analysis, and algebraic geometry: forty years of the {J}ournal of
  {D}ifferential {G}eometry}, volume~13 of {\em Surv. Differ. Geom.}, pages
  1--47. Int. Press, Somerville, MA, 2009.

\bibitem{BolGugKud}
A.~Bolsinov, L.~Guglielmi, and E.~Kudryavtseva.
\newblock Symplectic invariants for parabolic orbits and cusp singularities of
  integrable systems.
\newblock {\em Philos. Trans. Roy. Soc. A}, 376(2131):20170424, 29, 2018.

\bibitem{BolFom}
A.~V. Bolsinov and A.~T. Fomenko.
\newblock {\em Integrable {H}amiltonian systems}.
\newblock Chapman \& Hall/CRC, Boca Raton, FL, 2004.
\newblock Geometry, topology, classification, Translated from the 1999 Russian
  original.

\bibitem{BHMM}
J.~Brugu{\'e}s, S.~Hohloch, P.~Mir, and E.~Miranda.
\newblock {Constructions of $b$-semitoric systems}.
\newblock {\em Journal of Mathematical Physics}, 64(7):072703, 07 2023.

\bibitem{Cha}
M.~Chaperon.
\newblock Normalisation of the smooth focus-focus: a simple proof.
\newblock {\em Acta Math. Vietnam.}, 38(1):3--9, 2013.

\bibitem{CDEW}
I.~Chiscop, H.~R. Dullin, K.~Efstathiou, and H.~Waalkens.
\newblock A {L}agrangian fibration of the isotropic 3-dimensional harmonic
  oscillator with monodromy.
\newblock {\em J. Math. Phys.}, 60(3):032103, 15, 2019.

\bibitem{ColVey}
Y.~Colin~de Verdi\`ere and J.~Vey.
\newblock Le lemme de {M}orse isochore.
\newblock {\em Topology}, 18(4):283--293, 1979.

\bibitem{CusBat}
R.~H. Cushman and L.~M. Bates.
\newblock {\em Global aspects of classical integrable systems}.
\newblock Birkh\"{a}user/Springer, Basel, second edition, 2015.

\bibitem{DHVN}
M.~Dauge, M.~A. Hall, and S.~V\~{u}~Ng\d{o}c.
\newblock Asymptotic lattices, good labellings, and the rotation number for
  quantum integrable systems.
\newblock {\em Discrete Contin. Dyn. Syst.}, 42(12):5683--5735, 2022.

\bibitem{DHN21}
S.~R. Dawson, H.~R. Dullin, and D.~M.~H. Nguyen.
\newblock Monodromy in prolate spheroidal harmonics.
\newblock {\em Stud. Appl. Math.}, 146(4):953--982, 2021.

\bibitem{HohMeu}
A.~De~Meulenaere and S.~Hohloch.
\newblock A family of semitoric systems with four focus-focus singularities and
  two double pinched tori.
\newblock {\em J. Nonlinear Sci.}, 31(4):Paper No. 66, 56, 2021.

\bibitem{Del}
T.~Delzant.
\newblock Hamiltoniens p{\'e}riodiques et images convexes de l'application
  moment.
\newblock {\em Bull. Soc. Math. France}, 116(3):315--339, 1988.

\bibitem{Pal22}
Y.~Du, G.~Kosmacher, Y.~Liu, J.~Massman, J.~Palmer, T.~Thieme, J.~Wu, and
  Z.~Zhang.
\newblock Packing densities of {D}elzant and semitoric polygons.
\newblock {\em SIGMA Symmetry Integrability Geom. Methods Appl.}, 19:Paper No.
  081, 42, 2023.

\bibitem{DufMol}
J.-P. Dufour and P.~Molino.
\newblock Compactification d'actions de {${\bf R}^n$} et variables action-angle
  avec singularit\'{e}s.
\newblock In {\em Symplectic geometry, groupoids, and integrable systems
  ({B}erkeley, {CA}, 1989)}, volume~20 of {\em Math. Sci. Res. Inst. Publ.},
  pages 151--167. Springer, New York, 1991.

\bibitem{Efs-book}
K.~Efstathiou.
\newblock {\em Metamorphoses of {H}amiltonian systems with symmetries}, volume
  1864 of {\em Lecture Notes in Mathematics}.
\newblock Springer-Verlag, Berlin, 2005.

\bibitem{EfsGia}
K.~Efstathiou and A.~Giacobbe.
\newblock The topology associated with cusp singular points.
\newblock {\em Nonlinearity}, 25(12):3409--3422, 2012.

\bibitem{Eliasson-thesis}
L.~H. Eliasson.
\newblock {\em {H}amiltonian systems with Poisson commuting integrals}.
\newblock PhD thesis, University of Stockholm, 1984.

\bibitem{Eli90}
L.~H. Eliasson.
\newblock Normal forms for {H}amiltonian systems with {P}oisson commuting
  integrals---elliptic case.
\newblock {\em Comment. Math. Helv.}, 65(1):4--35, 1990.

\bibitem{FiPaPe2016}
A.~Figalli, J.~Palmer, and A.~Pelayo.
\newblock Symplectic {$G$}-capacities and integrable systems.
\newblock {\em Ann. Sc. Norm. Super. Pisa Cl. Sci. (5)}, 18(1):65--103, 2018.

\bibitem{Ful}
W.~Fulton.
\newblock {\em Introduction to toric varieties}, volume 131 of {\em Annals of
  Mathematics Studies}.
\newblock Princeton University Press, Princeton, NJ, 1993.
\newblock The William H. Roever Lectures in Geometry.

\bibitem{GuiSte}
V.~Guillemin and S.~Sternberg.
\newblock Convexity properties of the moment mapping.
\newblock {\em Invent. Math.}, 67(3):491--513, 1982.

\bibitem{GulHoh}
Y.~Gullentops and S.~Hohloch.
\newblock Creating hyperbolic-regular singularities in the presence of an
  $\mathbb{S}^1$-symmetry.
\newblock Preprint, \url{https://arxiv.org/abs/2209.15631}, 2022.

\bibitem{GulHoh2}
Y.~Gullentops and S.~Hohloch.
\newblock Recent examples of hypersemitoric systems and first steps towards a
  classification: a brief survey.
\newblock Preprint, \url{https://arXiv.org/abs/2308.16346}, 2023.

\bibitem{HHM}
T.~V. Henriksen, S.~Hohloch, and N.~Martynchuk.
\newblock Towards hypersemitoric systems.
\newblock To appear in \textit{Conference Proceedings of RIMS Kokyuroku},
  \url{https://arxiv.org/abs/2307.04483}.

\bibitem{HohPal}
S.~Hohloch and J.~Palmer.
\newblock A family of compact semitoric systems with two focus-focus
  singularities.
\newblock {\em Journal of Geometric Mechanics}, 10(3):331--357, 2018.

\bibitem{HohPal21}
S.~Hohloch and J.~Palmer.
\newblock Extending compact {H}amiltonian {$S^1$}-spaces to integrable systems
  with mild degeneracies in dimension four.
\newblock Preprint, \url{https://arxiv.org/abs/2105.00523}, 2021.

\bibitem{HSS}
S.~Hohloch, S.~Sabatini, and D.~Sepe.
\newblock From compact semi-toric systems to {H}amiltonian {$S^1$}-spaces.
\newblock {\em Discrete Contin. Dyn. Syst.}, 35(1):247--281, 2015.

\bibitem{HSSS2}
S.~Hohloch, S.~Sabatini, D.~Sepe, and M.~Symington.
\newblock From {H}amiltonian $\mathbb{S}^1$-spaces to compact semi-toric
  systems.
\newblock \emph{In preparation}.

\bibitem{HSS-vertical}
S.~Hohloch, S.~Sabatini, D.~Sepe, and M.~Symington.
\newblock Faithful semitoric systems.
\newblock {\em SIGMA Symmetry Integrability Geom. Methods Appl.}, 14:Paper No.
  084, 66, 2018.

\bibitem{jaynes-cummings}
E.~Jaynes and F.~Cummings.
\newblock Comparison of quantum and semiclassical radiation theories with
  application to the beam maser.
\newblock {\em Proc. IEEE}, 51(1):89--109, 1963--1964.

\bibitem{KPP_fan}
D.~M. Kane, J.~Palmer, and {\'A}.~Pelayo.
\newblock Classifying toric and semitoric fans by lifting equations from {${\rm
  SL}_2(\mathbb{Z})$}.
\newblock {\em SIGMA Symmetry Integrability Geom. Methods Appl.}, 14:Paper No.
  016, 43, 2018.

\bibitem{KPP_min}
D.~M. {Kane}, J.~{Palmer}, and {\'A}.~{Pelayo}.
\newblock {Minimal models of compact symplectic semitoric manifolds.}
\newblock {\em {J. Geom. Phys.}}, 125:49--74, 2018.

\bibitem{karshon}
Y.~Karshon.
\newblock Periodic {H}amiltonian flows on four-dimensional manifolds.
\newblock {\em Mem. Amer. Math. Soc.}, 141(672):viii+71, 1999.

\bibitem{Kar02}
Y.~Karshon.
\newblock Maximal tori in the symplectomorphism groups of {H}irzebruch
  surfaces.
\newblock {\em Math. Res. Lett.}, 10(1):125--132, 2003.

\bibitem{KarTol01}
Y.~Karshon and S.~Tolman.
\newblock Centered complexity one {H}amiltonian torus actions.
\newblock {\em Trans. Amer. Math. Soc.}, 353(12):4831--4861, 2001.

\bibitem{Kudmar2}
E.~Kudryavtseva and N.~Martynchuk.
\newblock ${C}^\infty$ symplectic invariants of parabolic orbits and flaps in
  integrable hamiltonian systems.
\newblock Preprint, \url{https://arxiv.org/abs/2110.13758}, 2021.

\bibitem{KudMar}
E.~A. Kudryavtseva and N.~N. Martynchuk.
\newblock Existence of a smooth {H}amiltonian circle action near parabolic
  orbits and cuspidal tori.
\newblock {\em Regul. Chaotic Dyn.}, 26(6):732--741, 2021.

\bibitem{LFPfamilies}
Y.~Le~Floch and J.~Palmer.
\newblock Semitoric families.
\newblock To appear in \textit{Mem. Amer. Math. Soc.},
  \url{https://arxiv.org/abs/1810.06915}, accepted 2022.

\bibitem{LFP}
Y.~Le~Floch and A.~Pelayo.
\newblock Symplectic geometry and spectral properties of classical and quantum
  coupled angular momenta.
\newblock {\em J. Nonlinear Sci.}, 29(2):655--708, 2019.

\bibitem{LFVN}
Y.~Le~Floch and S.~V\~u Ng\d{o}c.
\newblock The inverse spectral problem for quantum semitoric systems.
\newblock Preprint, \url{https://arxiv.org/abs/2104.06704}, 115 pages, 2021.

\bibitem{Ler}
E.~Lerman.
\newblock Symplectic cuts.
\newblock {\em Math. Res. Lett.}, 2(3):247--258, 1995.

\bibitem{MW}
J.~Marsden and A.~Weinstein.
\newblock Reduction of symplectic manifolds with symmetry.
\newblock {\em Rep. Mathematical Phys.}, 5(1):121--130, 1974.

\bibitem{McDuffSal}
D.~McDuff and D.~Salamon.
\newblock {\em Introduction to symplectic topology}.
\newblock Oxford Graduate Texts in Mathematics. Oxford University Press,
  Oxford, third edition, 2017.

\bibitem{Mir}
E.~Miranda.
\newblock {\em On symplectic linearization of singular {L}agrangian
  foliations}.
\newblock PhD thesis, University of Barcelona, 2003.

\bibitem{miranda-zung}
E.~Miranda and N.~T. Zung.
\newblock Equivariant normal form for nondegenerate singular orbits of
  integrable {H}amiltonian systems.
\newblock {\em Ann. Sci. \'Ecole Norm. Sup. (4)}, 37(6):819--839, 2004.

\bibitem{Pal17}
J.~Palmer.
\newblock Moduli spaces of semitoric systems.
\newblock {\em J. Geom. Phys.}, 115:191--217, 2017.

\bibitem{PPT}
J.~Palmer, {\'A}.~Pelayo, and X.~Tang.
\newblock Semitoric systems of non-simple type.
\newblock Preprint, \url{https://arxiv.org/abs/1909.03501}, 2019.

\bibitem{Pe23}
{\'A}.~Pelayo.
\newblock Symplectic and inverse spectral geometry of integrable systems: A
  glimpse and open problems.
\newblock {\em Topology and its Applications}, page 108577, 2023.

\bibitem{PPRStoric}
{\'A}.~Pelayo, A.~R. Pires, T.~S. Ratiu, and S.~Sabatini.
\newblock Moduli spaces of toric manifolds.
\newblock {\em Geom. Dedicata}, 169:323--341, 2014.

\bibitem{PelRatVN-fibers}
A.~Pelayo, T.~S. Ratiu, and S.~V.~u. Ng\d{o}c.
\newblock Fiber connectivity and bifurcation diagrams of almost toric
  integrable systems.
\newblock {\em J. Symplectic Geom.}, 13(2):343--386, 2015.

\bibitem{PelRatVN}
{\'A}.~Pelayo, T.~S. Ratiu, and S.~V\~u Ng\d{o}c.
\newblock The affine invariant of proper semitoric integrable systems.
\newblock {\em Nonlinearity}, 30(11):3993--4028, 2017.

\bibitem{PeSa23}
{\'A}.~Pelayo and F.~Santos.
\newblock Moduli spaces of {D}elzant polytopes and symplectic toric manifolds.
\newblock Preprint, \url{https://arxiv.org/abs/}, 2023.

\bibitem{PT}
A.~Pelayo and X.~Tang.
\newblock V\~{u} {N}g\d{o}c's conjecture on focus-focus singular fibers with
  multiple pinched points.
\newblock {\em J. Fixed Point Theory Appl.}, 26(1):Paper No. 6, 34, 2024.

\bibitem{PVNinventiones}
{\'A}.~Pelayo and S.~V\~u Ng\d{o}c.
\newblock Semitoric integrable systems on symplectic 4-manifolds.
\newblock {\em Invent. Math.}, 177(3):571--597, 2009.

\bibitem{PVNacta}
{\'A}.~Pelayo and S.~V\~u Ng\d{o}c.
\newblock Constructing integrable systems of semitoric type.
\newblock {\em Acta Math.}, 206(1):93--125, 2011.

\bibitem{san-alvaro-spin}
{\'A}.~Pelayo and S.~V{\~u}~Ng\d{o}c.
\newblock Hamiltonian dynamics and spectral theory for spin-oscillators.
\newblock {\em Comm. Math. Phys.}, 309(1):123--154, 2012.

\bibitem{SZ}
D.~A. {Sadovski\'\i} and B.~I. {Z\^hilinski\'\i}.
\newblock Monodromy, diabolic points, and angular momentum coupling.
\newblock {\em Phys. Lett. A}, 256(4):235--244, 1999.

\bibitem{VNSepe}
D.~Sepe and S.~V\~u Ng\d{o}c.
\newblock Integrable systems, symmetries, and quantization.
\newblock {\em Lett. Math. Phys.}, 108(3):499--571, 2018.

\bibitem{Sym}
M.~Symington.
\newblock Four dimensions from two in symplectic topology.
\newblock In {\em Topology and geometry of manifolds ({A}thens, {GA}, 2001)},
  volume~71 of {\em Proc. Sympos. Pure Math.}, pages 153--208. Amer. Math.
  Soc., Providence, RI, 2003.

\bibitem{ToZel}
J.~A. Toth and S.~Zelditch.
\newblock {$L^p$} norms of eigenfunctions in the completely integrable case.
\newblock {\em Ann. Henri Poincar\'e}, 4(2):343--368, 2003.

\bibitem{VNpoly}
S.~V\~u Ng\d{o}c.
\newblock Moment polytopes for symplectic manifolds with monodromy.
\newblock {\em Adv. Math.}, 208(2):909--934, 2007.

\bibitem{VNWac}
S.~V\~{u}~Ng\d{o}c and C.~Wacheux.
\newblock Smooth normal forms for integrable {H}amiltonian systems near a
  focus-focus singularity.
\newblock {\em Acta Math. Vietnam.}, 38(1):107--122, 2013.

\bibitem{Vey}
J.~Vey.
\newblock Sur certains syst\`emes dynamiques s\'{e}parables.
\newblock {\em Amer. J. Math.}, 100(3):591--614, 1978.

\bibitem{Zung97}
N.~T. Zung.
\newblock A note on focus-focus singularities.
\newblock {\em Differential Geom. Appl.}, 7(2):123--130, 1997.

\bibitem{Zung02}
N.~T. Zung.
\newblock Another note on focus-focus singularities.
\newblock {\em Lett. Math. Phys.}, 60(1):87--99, 2002.

\bibitem{Zung03}
N.~T. Zung.
\newblock Symplectic topology of integrable {H}amiltonian systems. {II}.
  {T}opological classification.
\newblock {\em Compositio Math.}, 138(2):125--156, 2003.

\end{thebibliography}

\end{document}